\pdfoutput=1
\documentclass[leqno,11pt]{article}
\usepackage[utf8]{inputenc}
\usepackage[T1]{fontenc}
\usepackage[sb]{libertine}
\usepackage[varl,varqu]{zi4}
\usepackage[libertine,bigdelims,vvarbb,smallerops]{newtxmath} % for arXiv ONLY
\usepackage[supstfm=libertinesups,supscaled=1.2,raised=-.13em]{superiors} % for arXiv ONLY
\usepackage[left=2cm,right=2cm,top=2cm,bottom=2cm,footnotesep=0.15in,headsep=0in,headheight=0.25in]{geometry}
\usepackage[scr=boondox,cal=txupr]{mathalpha}

\usepackage{bm}
\usepackage[USenglish]{babel}
\usepackage{graphicx} 
\usepackage{scalerel}
\usepackage{mathstyle}
\usepackage[dvipsnames]{xcolor}
\usepackage{hyperref} 
\usepackage[explicit]{titlesec}
\usepackage{svg}
\usepackage{microtype} 

\usepackage{amsthm}
\usepackage{thmtools}
\usepackage{tikz-cd} 
\usetikzlibrary{arrows} 
\tikzset{
commutative diagrams/.cd,
arrow style=tikz,
diagrams={>=stealth}
}
\usetikzlibrary{matrix,decorations.pathreplacing,calc} 
\usepackage{listings} 
\usepackage[inline,shortlabels]{enumitem}
\usepackage{array,multirow}
\usepackage{footnote}
\makesavenoteenv{tabular}
\usepackage{marvosym}
\usepackage{iitem}
\usepackage{mathtools}
\usepackage{pict2e}
\usepackage{empheq}
\newenvironment{gat}{\begingroup\allowdisplaybreaks\csname gather*\endcsname}{\csname endgather*\endcsname\endgroup}
\newenvironment{al}{\begingroup\allowdisplaybreaks\csname align*\endcsname}{\csname endalign*\endcsname\endgroup}
\definecolor{Wblue2}{HTML}{438CD3}
\hypersetup{
colorlinks=true, 
linktoc=all,     
linkcolor=Wblue2,  
urlcolor=Wblue2,
anchorcolor=Wblue2,
citecolor=Wblue2, 
}
\titleformat{\section}
{\normalfont\Large\bfseries}{\thesection}{1em}{\hyperlink{sec-\thesection}{#1}
\addtocontents{toc}{\protect\hypertarget{sec-\thesection}{}}}
\titleformat{name=\section,numberless}
{\normalfont\Large\bfseries}{}{0pt}{#1}
\titleformat{\subsection}
{\normalfont\large\bfseries}{\thesubsection}{1em}{\hyperlink{subsec-\thesubsection}{#1}
\addtocontents{toc}{\protect\hypertarget{subsec-\thesubsection}{}}}
\titleformat{name=\subsection,numberless}
{\normalfont\large\bfseries}{\thesubsection}{0pt}{#1}
\titleformat{\subsubsection}
{\normalfont\normalsize\bfseries}{\thesubsubsection}{1em}{\hyperlink{subsubsec-\thesubsubsection}{#1}
\addtocontents{toc}{\protect\hypertarget{subsubsec-\thesubsubsection}{}}}
\titleformat{name=\subsubsection,numberless}
{\normalfont\normalsize\bfseries}{\thesubsubsection}{0pt}{#1}
\numberwithin{equation}{section}
\newcommand{\TheoremQED}{$\blacklozenge$}
\newcommand{\CorollaryQED}{\qedsymbol}
\newcommand{\DefinitionQED}{$\bullet$}
\newcommand{\RemarkQED}{$\clubsuit$}
\renewcommand{\qedsymbol}{$\blacksquare$}
\newcommand{\NotationQED}{$\circ$}
\declaretheorem[numberlike=equation,qed=\TheoremQED]{theorem}

\declaretheorem[numberlike=equation,style=definition,name=Proposition]{prop}
\declaretheorem[numberlike=equation,style=definition,name=Corollary,qed=\CorollaryQED]{cor}
\declaretheorem[numberlike=equation,name=Definition,style=definition,qed=\DefinitionQED]{definition}
\declaretheorem[numberlike=equation,style=remark,qed=\RemarkQED]{remark}
\declaretheorem[numberlike=equation,style=definition]{question}
\declaretheorem[numbered=no,name=Acknowledgments,style=remark,qed=\Cross]{ack*}
\declaretheorem[numberlike=equation,name=Notation,style=definition,qed=\NotationQED]{notation}
\declaretheorem[numberlike=equation,name=Ansatz,style=definition]{ansatz}
\declaretheorem[numberlike=equation,name=Data,style=definition]{data}
\setlist[description]{leftmargin=!,labelindent=1em}
\setlist[enumerate]{label={\rm (\arabic*)},ref=\arabic*}
\setlist[enumerate,2]{label={\rm (\alph*)},ref=\theenumi.\alph*}

\newcommand\restr[2]{{% we make the whole thing an ordinary symbol
\left.\kern-\nulldelimiterspace % automatically resize the bar with \right
#1 % the function
\littletaller % pretend it's a little taller at normal size
\right|_{#2} % this is the delimiter
}}
\newcommand{\littletaller}{\mathchoice{\vphantom{\big|}}{}{}{}}   
\def\<{\left\langle}
\def\>{\right\rangle}
\input pdfmsym
\pdfmsymsetscalefactor{12}
\makeatletter
\def\accent@skew{0}
\makeatother
\makeatletter
\def\accent@raise{0.485}
\makeatother

\newcommand{\lwhat}{\varwidehat}

\newcommand{\what}{\widehat}
\newcommand{\wtilde}{\widetilde}
\DeclareSymbolFont{chartermath}{OML}{mdbch}{m}{n}
\DeclareMathSymbol{\Pelta}{\mathalpha}{chartermath}{"01}
\renewcommand{\Delta}{\Pelta}
\DeclareFontFamily{U} {MnSymbolD}{}
\DeclareFontShape{U}{MnSymbolD}{m}{n}{
<-6> MnSymbolD5
<6-7> MnSymbolD6
<7-8> MnSymbolD7
<8-9> MnSymbolD8
<9-10> MnSymbolD9
<10-12> MnSymbolD10
<12-> MnSymbolD12}{}
\DeclareFontShape{U}{MnSymbolD}{b}{n}{
<-6> MnSymbolD-Bold5
<6-7> MnSymbolD-Bold6
<7-8> MnSymbolD-Bold7
<8-9> MnSymbolA-Bold8
<9-10> MnSymbolD-Bold9
<10-12> MnSymbolD-Bold10
<12-> MnSymbolD-Bold12}{}
\DeclareSymbolFont{MnSyD} {U} {MnSymbolD}{m}{n}
\DeclareFontFamily{U} {MnSymbolF}{}
\DeclareFontShape{U}{MnSymbolF}{m}{n}{
<-6> MnSymbolF5
<6-7> MnSymbolF6
<7-8> MnSymbolF7
<8-9> MnSymbolF8
<9-10> MnSymbolF9
<10-12> MnSymbolF10
<12-> MnSymbolF12}{}
\DeclareFontShape{U}{MnSymbolF}{b}{n}{
<-6> MnSymbolF-Bold5
<6-7> MnSymbolF-Bold6
<7-8> MnSymbolF-Bold7
<8-9> MnSymbolF-Bold8
<9-10> MnSymbolF-Bold9
<10-12> MnSymbolF-Bold10
<12-> MnSymbolF-Bold12}{}
\DeclareSymbolFont{MnSyF} {U} {MnSymbolF}{m}{n}
\DeclareMathSymbol{\ll}{\mathrel}{MnSyD}{80}
\DeclareMathSymbol{\gg}{\mathrel}{MnSyD}{81}
\DeclareMathSymbol{\Lum}{\mathop}{MnSyF}{80}
\renewcommand{\sum}{\Lum}
\renewcommand{\nabla}{\mathchoice
{\ensuremath{\rotatebox[origin=c]{-180}{$\displaystyle\Delta$}}}
{\ensuremath{\rotatebox[origin=c]{-180}{$\textstyle\Delta$}}}
{\ensuremath{\rotatebox[origin=c]{-180}{$\scriptstyle\Delta$}}}
{\ensuremath{\rotatebox[origin=c]{-180}{$\scriptscriptstyle\Delta$}}}}
\DeclareMathOperator{\actson}{{\mathchoice
{\rotatebox[origin=c]{-90}{$\displaystyle\circlearrowright$}}% \displaystyle
{\rotatebox[origin=c]{-90}{$\textstyle\circlearrowright$}}% \textstyle
{\rotatebox[origin=c]{-90}{$\scriptstyle\circlearrowright$}}% \scriptstyle
{\rotatebox[origin=c]{-90}{$\scriptscriptstyle\circlearrowright$}}% \scriptscriptstyle
}}

\DeclareMathOperator{\Ractson}{{\mathchoice
{\rotatebox[origin=c]{90}{$\displaystyle\circlearrowright$}}% \displaystyle
{\rotatebox[origin=c]{90}{$\textstyle\circlearrowright$}}% \textstyle
{\rotatebox[origin=c]{90}{$\scriptstyle\circlearrowright$}}% \scriptstyle
{\rotatebox[origin=c]{90}{$\scriptscriptstyle\circlearrowright$}}% \scriptscriptstyle
}}

\DeclareMathOperator{\Lactson}{{\mathchoice
{\rotatebox[origin=c]{-90}{$\displaystyle\circlearrowright$}}% \displaystyle
{\rotatebox[origin=c]{-90}{$\textstyle\circlearrowright$}}% \textstyle
{\rotatebox[origin=c]{-90}{$\scriptstyle\circlearrowright$}}% \scriptstyle
{\rotatebox[origin=c]{-90}{$\scriptscriptstyle\circlearrowright$}}% \scriptscriptstyle
}}
\DeclarePairedDelimiter\paren{\lparen}{\rparen}
\DeclarePairedDelimiter\sqparen{[}{]}
\DeclarePairedDelimiter{\Abs}{\|}{\|}
\DeclarePairedDelimiter{\Inner}{\langle}{\rangle}
\DeclarePairedDelimiter{\abs}{\lvert}{\rvert}
\DeclarePairedDelimiter{\set}{\lbrace}{\rbrace}
\newcommand{\bbC}{{\mathbb{C}}}
\newcommand{\bbK}{{\mathbb{K}}}
\newcommand{\bbN}{{\mathbb{N}}}
\newcommand{\bbR}{{\mathbb{R}}}
\newcommand{\bbT}{{\mathbb{T}}}
\newcommand{\bbZ}{{\mathbb{Z}}}
\newcommand{\SL}{\mathrm{SL}}
\newcommand{\SU}{\mathrm{SU}}

\newcommand{\Ric}{\mathrm{Ric}}
\newcommand{\sK}{\mathscr{K}}
\newcommand{\sM}{\mathscr{M}}
\DeclareMathOperator{\Hom}{Hom}
\newcommand{\id}{\mathrm{id}}
\newcommand{\U}{\mathrm{U}}
\renewcommand{\O}{\mathrm{O}}
\newcommand{\SO}{\mathrm{SO}}
\newcommand{\CP}{{\mathbb{C} P}} 
\newcommand{\Sp}{\mathrm{Sp}}
\DeclareMathOperator{\Fr}{Fr}
\newcommand{\del}{\partial}
\newcommand{\delbar}{\bar{\del}}
\DeclareMathOperator{\Diff}{Diff}
\DeclareMathOperator{\End}{End}
\DeclareMathOperator{\Hess}{Hess}
\newcommand{\Lip}{\mathrm{Lip}}
\DeclareMathOperator{\Fix}{Fix}
\DeclareMathOperator{\GL}{GL}
\DeclareMathOperator{\Gr}{Gr}
\DeclareMathOperator{\Hol}{Hol}
\renewcommand{\Im}{\operatorname{Im}}
\DeclareMathOperator{\Km}{Km}
\renewcommand{\Re}{\operatorname{Re}}
\DeclareMathOperator{\Sym}{Sym}
\DeclareMathOperator{\Tr}{Tr}
\DeclareMathOperator{\Vol}{Vol}
\newcommand{\cA}{\mathcal{A}}
\newcommand{\cB}{\mathcal{B}}
\newcommand{\cC}{\mathcal{C}}
\newcommand{\cE}{\mathcal{E}}
\newcommand{\cF}{\mathcal{F}}
\newcommand{\cG}{\mathcal{G}}
\newcommand{\cH}{\mathcal{H}}
\newcommand{\cK}{\mathcal{K}}
\newcommand{\cL}{\mathcal{L}}
\newcommand{\cM}{\mathcal{M}}
\newcommand{\cN}{\mathcal{N}}
\newcommand{\cO}{\mathcal{O}}
\newcommand{\cP}{\mathcal{P}}
\newcommand{\cS}{\mathcal{S}}
\newcommand{\fO}{{\mathfrak O}}

\newcommand{\Cone}{{\rm{Cone}}}
\newcommand{\GLr}[1]{\GL\paren*{#1,\mathbb{R}}}
\newcommand{\SLc}[1]{\SL\paren*{#1,\mathbb{C}}}
\newcommand{\bomega}{{\bm{\omega}}}
\newcommand{\onto}{\twoheadrightarrow}
\newcommand{\pr}{\mathrm{pr}}

\newcommand{\sing}{\mathrm{sing}}
\newcommand{\fa}{{\mathfrak a}}
\newcommand{\fs}{{\mathfrak s}}
\newcommand{\fq}{{\mathfrak q}}
\newcommand{\ft}{{\mathfrak t}}
\newcommand{\fw}{{\mathfrak w}}
\newcommand{\fS}{{\mathfrak S}}

\DeclareMathOperator{\dist}{dist}
\DeclareMathOperator{\im}{im}
\DeclareMathOperator{\ind}{ind}
\DeclareMathOperator{\inj}{inj}
\DeclareMathOperator{\intrp}{\lrcorner}
\DeclareMathOperator{\spano}{span}
\DeclareMathOperator{\supp}{supp}

\newcommand{\qqfa}{\qquad\forall}
\newcommand{\rd}{{\rm d}}
\newcommand{\dV}{\rd V}
\newcommand*\Eval[3]{\left.#1\right\rvert_{#2}^{#3}}
\newcommand{\ka}{{K}\"ahler\ }
\newcommand{\hka}{hyper-{K}\"ahler\ }
\usepackage[useregional]{datetime2}
\DTMlangsetup[en-US]{showdayofmonth=false}
\usepackage[hyperref,backend=biber,hyperref=true,backref=true]{biblatex}
\usepackage{csquotes}

\usepackage{footnotebackref}
\usepackage{indentfirst}

\numberwithin{theorem}{subsection}
\numberwithin{prop}{subsection}
\numberwithin{cor}{subsection}
\numberwithin{lemma}{subsection}
\numberwithin{definition}{subsection}
\numberwithin{notation}{subsection}
\numberwithin{data}{subsection}
\numberwithin{ansatz}{subsection}
\numberwithin{remark}{subsection}
\renewcommand*{\thetheorem}{%
\ifnum\value{subsection}=0 %
\thesection
\else
\thesubsection
\fi
.\arabic{theorem}%
}
\renewcommand*{\theremark}{%
\ifnum\value{subsection}=0 %
\thesection
\else
\thesubsection
\fi
.\arabic{remark}%
}
\renewcommand*{\thenotation}{%
\ifnum\value{subsection}=0 %
\thesection
\else
\thesubsection
\fi
.\arabic{notation}%
}

\title{The Kummer Construction of Calabi-Yau and Hyper-K\"{a}hler Metrics on the $K3$ Surface, and Large Families of Volume Non-collapsed Limiting Compact Hyper-K\"{a}hler Orbifolds}
\author{Thomas Jiang$^{\dagger}$% \\ {\small email: \href{mailto:thomas.jiang@duke.edu}{thomas.jiang@duke.edu}}
}
\date{
$^{\dagger}$email: \href{mailto:thomas.jiang@duke.edu}{thomas.jiang@duke.edu}\\
Department of Mathematics, Duke University, Durham, NC, USA;\\[2ex]
August 2026}

\begin{document}
\maketitle
\begin{abstract}
Right after Yau's resolution of the Calabi conjecture in the late 1970s, physicists Page \cite{Page} and Gibbons-Pope \cite{Gibbons-Pope} conjectured that one may approximate Ricci-flat \ka metrics on the $K3$ surface with metrics having ``almost special holonomy'' constructed via ``resolving'' the $16$ orbifold singularities of a flat $\bbT^{4}/\bbZ_{2}$ with Eguchi-Hanson metrics. Constructions of such metrics with special holonomy from such a ``gluing'' construction of approximate special holonomy metrics have since been called ``Kummer constructions'' of special holonomy metrics, and their proposal was rigorously carried out in the 1990s by Kobayashi \cite{KobayashiModuli} and LeBrun-Singer \cite{LeBrun-Singer}, and in the 2010s by Donaldson \cite{DonaldsonKummer}.

In this paper, we provide two new rigorous proofs of Page-Gibbons-Pope's proposal based on singular perturbation and weighted function space analysis. Each proof is done from a different perspective:\begin{itemize}
\itemsep0em 
\item solving the complex Monge-Ampere equation (Calabi-Yau)
\item perturbing closed definite triples (hyper-K\"{a}hler)
\end{itemize} Both proofs yield Eguchi-Hanson metrics as ALE bubbles/rescaled limits. Moreover, our analysis in the former perspective yields estimates which \textit{improve} Kobayashi's estimates, and our analysis in the latter perspective results in the construction of \textit{large} families of Ricci-flat \ka metrics on the $K3$ surface, yielding the full $58$ dimensional moduli space of such.

Finally, as a byproduct of our analysis, we produce a plethora of large families of compact \hka \textit{orbifolds} which all arise as volume non-collapsed Gromov-Hausdorff limit spaces of the aforementioned constructed large families of Ricci-flat \ka metrics on the $K3$ surface. Moreover, these compact \hka orbifolds are explicitly exhibited as points in the ``holes'' of the moduli space of Ricci-flat K\"{a}hler-Einstein metrics on $K3$ under the period map.\end{abstract}

\tableofcontents

\section{\textsection \ Introduction \& Survey \ \textsection}\label{Introduction}

A \textbf{$K3$ surface} is the underlying smooth oriented 4-manifold of a connected compact simply connected complex 2-manifold with trivial canonical bundle. 

By a theorem of Kodaira, any pair of $K3$ surfaces are diffeomorphic to one another, whence we may speak of \textit{the} $K3$ surface (up to diffeomorphism).\footnote{Though, there do exist \textit{topological $K3$ surfaces}, which are real 4-manifolds homeomorphic to a $K3$ but \textit{not} diffeomorphic to it.} Hence any connected compact simply connected complex 2-fold with trivial canonical bundle is realized by a pair $\paren*{K3, J}$ with $J$ an \textit{integrable} complex structure, whence are called \textbf{complex $K3$ surfaces} with specified complex structure.

By a theorem of Siu, all complex $K3$ surfaces $\paren*{K3,J}$ are \textbf{K\"{a}hler}. That is, they admit \ka metrics, which are Riemannian metrics with holonomy contained in $\U\paren*{\frac{n}{2}}$ ($n$ being the real dimension of the underlying manifold). Yau's resolution of the Calabi conjecture therefore implies that any complex $K3$ surface $\paren*{K3,J}$ admits a unique \textbf{Ricci-flat K\"{a}hler metric} in \textit{any} given K\"{a}hler class. Moreover, all Ricci-flat \ka metrics have holonomy\footnote{Technically speaking, this is supposed to be the \textit{restricted holonomy}, which is the identity component of the holonomy group.} contained in $\SU\paren*{\frac{n}{2}}$, and Riemannian metrics with such holonomy are called \textbf{Calabi-Yau} and are automatically Ricci-flat K\"{a}hler. Furthermore, Riemannian metrics with holonomy further contained in $\Sp\paren*{\frac{n}{4}}$ (which is itself contained in $\SU\paren*{\frac{n}{2}}$) are called \textbf{hyper-K\"{a}hler}, and are automatically Ricci-flat.

A crucial difference between Calabi-Yau metrics and \hka metrics are that the former comes with a \textit{fixed} (metric compatible) complex structure, while the latter comes with a \textit{an entire 2-sphere/$S^{2}$s worth} of (metric compatible) complex structures. Hence a Calabi-Yau metric is Ricci-flat \textit{K\"{a}hler} WRT a \textit{single} complex structure and is a \textit{complex geometric} object, whereas a \hka metric is Ricci-flat \textit{K\"{a}hler} WRT a \textit{family} of complex structures, hence by virtue of this 2-sphere ambiguity, is a \textit{real geometric} object.

However, in real dimension $4$, we have the special isomorphism of Lie groups $\SU\paren*{2} \cong \Sp\paren*{1}$, whence being Calabi-Yau and \hka are \textit{equivalent}. Therefore, all Ricci-flat \ka metrics on the $K3$ surface are automatically both Calabi-Yau and hyper-K\"{a}hler. 

Moreover, the \textbf{local Torelli Theorem} \cite{JoyceBook} tells us that there is a local diffeomorphism \begin{al}
\cP: \sM^{K3}_{\Vol = 1} &\rightarrow \Gr^{+}\paren*{3,19}/\Gamma \coloneq \Gamma \backslash \SO\paren*{3,19}/\SO\paren*{3}\times \SO\paren*{19}\\
g &\mapsto \cH^{+}_{g}
\end{al} called the \textbf{period map} between the \textbf{moduli space} of all unit volume Ricci-flat \ka metrics on $K3$ (hence the full moduli space without any volume normalization is denoted as $\sM^{K3} = \bbR_{+}\times \sM^{K3}_{\Vol=1}$), and the Grassmannian of all positive $3$-planes in $\bbR^{3,19} \cong H^{2}\paren*{K3;\bbR}$ mod $\Gamma\leq \GL\paren*{22,\bbZ}$ the automorphisms of the lattice $H^{2}\paren*{K3;\bbZ}$ equipped with the intersection form $\paren*{[\alpha],[\beta]}\mapsto \int_{K3} \alpha\wedge\beta$. This local diffeomorphism sends each $g $ to $ \cH^{+}_{g}$, the $3$-dimensional subspace of \textbf{self-dual harmonic $2$-forms} WRT $g$.

However, this local diffeomorphism $\cP: \sM^{K3}_{\Vol = 1} \rightarrow \Gr^{+}\paren*{3,19}/\Gamma$ is \textit{not} surjective: it is only surjective onto $$\Gr^{+}\paren*{3,19}/\Gamma - \bigcup_{\substack{\Sigma \in H_{2}\paren*{K3;\bbZ}\\ \Sigma \cdot \Sigma = -2}} \set*{H \in \Gr^{+}\paren*{3,19}/\Gamma : [\alpha]\paren*{\Sigma}= 0, \forall [\alpha] \in H}$$ which is the complement of a \textit{codimension 3} subset.

In 1987, Kobayashi and Todorov \cite{KobayashiTodorov} proved that points in these codimension $3$ ``holes'' should correspond to Ricci-flat metrics with \textbf{orbifold singularities}. In fact, upon compactifying $\sM^{K3}$ WRT the \textbf{Gromov-Hausdorff topology} to get $\overline{\sM^{K3}}$, the results of general Riemannian convergence theory (which appeared shortly after Kobayashi-Todorov), and more specifically the results of Anderson \cite{Anderson} Bando-Kasue-Nakajima \cite{BKNa} and Nakajima \cite{Nakajima} may be immediately applied to tell us that any Gromov-Hausdorff limit of any sequence $\set*{g_{i}}\subset \sM^{K3}$ which is \textbf{volume non-collapsed} must be a compact \textbf{\hka orbifold} with finitely many isolated singular points. Moreover, such a sequence admits noncompact \hka \textbf{ALE spaces} as \textbf{bubbles/rescaled limits}, which model the convergence near the (orbifold) singularities. In fact, upon denoting the \textbf{partial Gromov-Hausdorff compactification} of $\sM^{K3}$ ($\sM^{K3}_{\Vol=1}$) by adding in all compact (unit volume) \hka orbifolds as $\overline{\sM^{K3}}^{\del}$ ($\overline{\sM^{K3}_{\Vol=1}}^{\del}$), it follows from Kobayashi-Todorov \cite{KobayashiTodorov} and Anderson \cite{AndersonModuli} that the period map extends to a \textit{bijection} $$\cP: \overline{\sM^{K3}_{\Vol=1}}^{\del} \rightarrow \Gr^{+}\paren*{3,19}/\Gamma$$ Hence when adding in all volume non-collapsed limits, these volume non-collapsed limits should correspond to points in the codimension $3$ subset $\bigcup_{\substack{\Sigma \in H_{2}\paren*{K3;\bbZ}\\ \Sigma \cdot \Sigma = -2}} \set*{H \in \Gr^{+}\paren*{3,19}/\Gamma : [\alpha]\paren*{\Sigma}= 0, \forall [\alpha] \in H}$ in which the original period map fails to be surjective. Moreover, \textbf{continuity} of the period map WRT the Gromov-Hausdorff topology across these codimension $3$ orbifold metrics was proved by Anderson in \cite{AndersonModuli} (with a recent shorter proof by Odaka-Oshima \cite{Odaka-Oshima}), who in fact showed that $\overline{\sM^{K3}_{\Vol=1}}^{\del}$ and $\Gr^{+}\paren*{3,19}/\Gamma$ are \textbf{isometric} as Riemannian manifolds where the former's metric is the \textbf{$L^{2}$ metric} (which induces a stronger topology than the Gromov-Hausdorff topology) and we view $\overline{\sM^{K3}_{\Vol=1}}^{\del}$ as the \textbf{$L^{2}$ completion} of $\sM^{K3}_{\Vol=1}$, and the latter with the (locally) symmetric metric coming from the Killing form of $\SO\paren*{3,19}$.

Since both Yau's theorem and the general convergence theorems above respectively only give us abstract existence and abstract convergence behavior of these metrics on $K3$, it is therefore instructive to find more explicit/``hands-on'' examples of these Ricci-flat metrics \& their moduli behavior.

After Yau's proof, physicists Page \cite{Page} and Gibbons-Pope \cite{Gibbons-Pope} suggested that a more concrete description of these Ricci-flat K\"{a}hler metrics on $K3$ may be obtained via first viewing $K3$ as a \textbf{Kummer surface} $\Km$, and then approximating those Ricci-flat metrics by gluing in 16 copies of a particular ALE \hka metric called the \textbf{Eguchi-Hanson metric} around each of the 16 singular points of the flat orbifold $\bbT^{4}/\bbZ_{2}$. Such a proposal has become known in the literature as the \textbf{Kummer construction} of (special, usually Ricci-flat) metrics, and will be referred to as such. This proposal was first done rigorously by Kobayashi \cite{KobayashiModuli} and LeBrun-Singer \cite{LeBrun-Singer} (following an earlier attempt by Topiwala \cite{Topiwala}) in the early to mid 1990s, and later in the 2010s by Donaldson \cite{DonaldsonKummer}.

Hence the original Kummer construction of Ricci-flat K\"{a}hler metrics on $K3$ is by now a classical topic in geometric analysis, having become the blueprint for and spawned multiple different variations and generalizations throughout the decades.

LeBrun-Singer's proof of the original Kummer construction uses \textbf{twistor theory}, while Kobayashi and Donaldson's proofs are more modern and uses PDE, specifically the \textbf{complex Monge-Ampere equation}. From the complex Monge-Ampere perspective, the former's proof is via the \textbf{continuity method}, and (with recent corrections by Lye \cite{Lye}, who went on to further study \textbf{stable geodesics} on the resulting metric on $K3$) consists primarily of deriving \textbf{a priori estimates} for the unknown scalar function appearing in the complex Monge-Ampere equation. Donaldson's proof, on the other hand, is based off \textbf{singular perturbation} and primarily uses the \textbf{implicit function theorem}. Proving the existence of Ricci-flat \ka metrics using complex Monge-Ampere requires a \textit{fixed} complex structure, whence without further work, can only be used to produce Calabi-Yau metrics. While being Calabi-Yau and \hka in real dimension 4 are equivalent, fixing a complex structure in the Kummer construction does \textit{not} yield an open subset/full parameter family of the $58 = \dim_{\bbR} \sM^{K3} = \dim_{\bbR}\sM^{K3}_{\Vol = 1} + 1$ dimensional\footnote{$\dim_{\bbR} \sM^{K3}_{\Vol = 1} = \dim_{\bbR} \Gr^{+}\paren*{3,19} = \dim_{\bbR} \SO\paren*{3,19} - \dim_{\bbR}\SO\paren*{3} - \dim_{\bbR}\SO\paren*{19} = 231 - 3 - 171 = 57$ since $\Gamma\leq \GL\paren*{22,\bbZ}$ is discrete.} moduli space of all Ricci-flat \ka metrics on $K3$. This is because the complex structure on $K3$ in the Kummer construction is constructed via patching \textit{the same choice} of complex structure on the 16 copies of the \hka Eguchi-Hanson spaces to the background complex structure on $\bbT^{4}/\bbZ_{2}$. And since such a choice of complex structure is a choice of direction in $S^{2}$, fixing the same direction in all 16 $S^{2}$s of complex structures on the Eguchi-Hansons does \textit{not} yield enough parameters to add up to $58$.

Moreover, neither Kobayashi/Lye nor Donaldson use \textbf{weighted function spaces} for their estimates when dealing with the 16 conical singularities. Indeed, Kobayashi and Lye's a priori estimates are all unweighted since their strategy is to reprove Yau's a priori estimates in his solution to the Calabi conjecture to have constants which are \textit{uniform} in the \textbf{gluing parameters} involved, and Donaldson exploits the conformal equivalence between the cone metric over $\bbR P^{3}$ and the cylindrical metric on $\bbR P^{3}$ to bypass the usage of polynomial weights. To the author's best knowledge, this machinery has yet to be used to give an alternative rigorous proof of Page-Gibbons-Pope's proposal, in either the Calabi-Yau or the hyper-K\"{a}hler perspectives.

In this paper, we rigorously carry out two more proofs of Page-Gibbons-Pope's proposal of constructing Ricci-flat \ka metrics on $K3$ using singular perturbation and weighted analysis, from both the Calabi-Yau and \hka perspectives. For the former, we find a Ricci-flat \ka metric in a given \ka class by solving a complex Monge-Ampere equation which necessitates having a \textit{fixed} background complex structure, and for the latter we produce a \hka metric via \textbf{closed definite triple perturbation} which reduces to solving a \textbf{nonlinear elliptic PDE \textit{system}}. Both PDEs admit \textbf{singular solutions} as limiting cases whence necessitating care in the analysis, and where weighted analysis proves extremely fruitful.

More specifically, in both the Calabi-Yau and \hka perspectives we\begin{itemize}
\itemsep0em 
\item produce Ricci-flat \ka metrics on the $K3$ surface which Gromov-Hausdorff converge to the flat orbifold $\bbT^{4}/\bbZ_{2}$.
\item explicitly exhibit Eguchi-Hanson metrics as ALE bubbles/rescaled limits near each singular orbifold point.
\end{itemize} Moreover, as a byproduct of the weighted analysis developed, in the \hka perspective we\begin{itemize}
\itemsep0em 
\item explicitly construct a plethora of compact \hka \textit{orbifolds} as explicit volume non-collapsed limits of \textbf{58-parameter} families of \hka metrics on $K3$, thus yielding a large subset of points in the codimension $3$ ``holes'' of $\Gr^{+}\paren*{3,19}/\Gamma$.
\item exhibit an \textbf{iterated volume non-collapsed} phenomena on which each of these constructed compact \hka orbifolds degenerate to one another.
\end{itemize}

That is, we both exhibit a multitude of elements in $\overline{\sM^{K3}}^{\del} - \sM^{K3}$ ($\overline{\sM^{K3}_{\Vol=1}}^{\del} - \sM^{K3}_{\Vol=1}$) and construct a full open neighborhood in $\overline{\sM^{K3}}^{\del}$ ($\overline{\sM^{K3}_{\Vol=1}}^{\del}$) around these limit compact \hka orbifolds, and then show how they degenerate to one another starting from a full $58$-parameter family of \hka metrics on $K3$.

In more detail, we prove:\begin{theorem}[Main Theorem I, see Theorem \ref{maintheorem1}]\label{MainTheoremA} Let $K3$ denote the $K3$ surface. Equip $K3$ with a fixed complex structure $J$ making it into a Kummer surface $\paren*{\Km, J}$.

Then $0 < \exists \epsilon_{0} \ll 1$ such that \textbf{there exists a 1-parameter family of Calabi-Yau metrics} $$\set*{\paren*{{g_{\epsilon}}, {\omega_{\epsilon}}}}_{\epsilon \in \paren*{0,\epsilon_{0}}}$$ on $K3$ which are Ricci-flat \ka WRT the fixed complex structure $J$. The solved for functions $\set*{f_{\epsilon}}_{\epsilon \in \paren*{0,\epsilon_{0}}}$ in the complex Monge-Ampere equation (thus $\omega_{\epsilon} = \omega_{0} + i\del\delbar f_{\epsilon}$ for some initial approximately Ricci-flat \ka form $\omega_{0}$) produced satisfy the uniform estimate $$\Abs*{f_{\epsilon}}_{C^{k}_{g_{\epsilon}} \paren*{\Km}} \leq C\epsilon^{4 - \lambda\paren*{\delta + 2} + \delta - k}$$ for $\delta \in \paren*{\max\set*{-2,-4\lambda},0}$ and $\lambda \in \paren*{0,\frac{2}{3}}$.

Moreover when $\epsilon \searrow 0$, $$\paren*{\Km, J, {g_{\epsilon}}, {\omega_{\epsilon}}} \xrightarrow{GH} \paren*{\bbT^{4}/\bbZ_{2}, J_{0}, g_{0}, \omega_{0}}$$ converges \textit{in the Gromov-Hausdorff topology} to the flat \ka orbifold $\paren*{\bbT^{4}/\bbZ_{2}, J_{0}, g_{0},\omega_{0}}$ with the singular orbifold K\"{a}hler metric $\paren*{g_{0}, \omega_{0}}$.

Moreover, we may decompose $K3$ into a union of sets $\Km^{reg} \cup \bigsqcup_{i =1}^{16}\Km^{b_{i}}$ such that \begin{enumerate}
\itemsep0em 
\item $\paren*{\Km^{reg}, J, {g_{\epsilon}}, {\omega_{\epsilon}}}$ converges in $C^{\infty}_{g_{0},loc}$ to the flat \ka manifold $\paren*{\bbT^{4}/\bbZ_{2} - \Fix\paren*{\bbZ_{2}}, J_{0}, g_{0}, \omega_{0}}$ with bounded curvature away from $\Fix\paren*{\bbZ_{2}}$ (\textbf{regular region}).
\item For each $i = 1,\dots, 16$, $\paren*{\Km^{b_{i}}, \frac{1}{\epsilon^{2}}{g_{\epsilon}}}$ (after a $\frac{1}{\epsilon}$-dilation) converges to $\paren*{T^{*}\CP^{1}, g_{EH,1}}$ in $C^{\infty}_{g_{EH,1},loc}$ (\textbf{ALE bubble region}).\end{enumerate}\end{theorem}
\begin{theorem}[Main Theorem II, see Theorem \ref{maintheorem2}]\label{MainTheoremB}Enumerate the $16$ singular points $\Fix\paren*{\bbZ_{2}} = \set*{p_{1},\dots, p_{16}}$ of $\bbT^{4}/\bbZ_{2}$.

Then $\forall I \subseteq \set*{1,\dots, 16}$, $0 < \exists \epsilon^{0}_{I} \ll 1$ such that \textbf{there exists a $10 + 3\abs*{I}$-parameter family of (orbifold) \hka triples and (orbifold) \hka metrics} $$\set*{\paren*{\bomega_{f, \epsilon_{I}, e_{I}}, g_{\bomega_{f, \epsilon_{I}, e_{I}}}}}_{\substack{f \in \text{flat tori moduli}\\ \epsilon_{i} \in \paren*{0,\epsilon_{I}^{0}} \\ e_{i} \in S^{2} \\ i \in I}}$$ on the \textbf{$I$-th partial smoothing $\Km^{\abs*{I}} = \Km^{\abs*{I}}_{f, \epsilon_{I}, e_{I}}$ of $\bbT^{4}/\bbZ_{2}$}, constructed by only smoothing the $\abs*{I}$ singular points $p_{i} \in \Fix\paren*{\bbZ_{2}}, i \in I$.

Moreover as $\epsilon^{0}_{I}\searrow 0$, $$\paren*{\Km^{\abs*{I}}_{f, \epsilon_{I}, e_{I}}, \bomega_{f, \epsilon_{I}, e_{I}}, g_{\bomega_{f, \epsilon_{I}, e_{I}}}} \xrightarrow{GH} \paren*{\bbT^{4}/\bbZ_{2}, \bomega_{f}, g_{\bomega_{f}}}$$ converges in the Gromov-Hausdorff topology to the flat \hka orbifold $\paren*{\bbT^{4}/\bbZ_{2}, \bomega_{f}, g_{\bomega_{f}}}$ with the singular orbifold \hka metric $\bomega_{f}, g_{\bomega_{f}}$. Moreover, we may decompose each compact orbifold $\Km^{\abs*{I}}$ into a union of sets $\Km^{\abs*{I},reg}_{f, \epsilon_{I}, e_{I}} \cup \bigsqcup_{i \in I}\Km^{\abs*{I},b_{i}}_{f, \epsilon_{I}, e_{I}}$ such that \begin{enumerate}
\itemsep0em 
\item $\paren*{\Km^{\abs*{I},reg}_{f, \epsilon_{I}, e_{I}}, \bomega_{f, \epsilon_{I}, e_{I}}, g_{\bomega_{f, \epsilon_{I}, e_{I}}}}$ converges in $C^{\infty}_{g_{\bomega_{0}},loc}$ to the flat \hka orbifold $\paren*{\bbT^{4}/\bbZ_{2} - \set*{p_{i}}_{i\in I}, \bomega_{0}, g_{\bomega_{0}}}$ with bounded curvature away from $p_{i} \in \cS$ for $i \in I$ (\textbf{regular region}).
\item For each $i \in I$, $\paren*{\Km^{\abs*{I},b_{i}}_{f, \epsilon_{I}, e_{I}}, \frac{1}{\epsilon_{i}^{2}} \bomega_{f, \epsilon_{I}, e_{I}}, \frac{1}{\epsilon_{i}^{2}}g_{\bomega_{f, \epsilon_{I}, e_{I}}}}$ (after a $\frac{1}{\epsilon_{i}}$-dilation) converges to $\paren*{T^{*}S^{2}, A^{e_{i}}\bomega_{EH,1}, g_{A^{e_{i}}\bomega_{EH,1}}}$ in $C^{\infty}_{g_{A^{e_{i}}\bomega_{EH,1}},loc}$ (\textbf{ALE bubble region}).
\end{enumerate}

In particular, when $I = \set*{1,\dots, 16}$, $0 < \exists \epsilon^{0} \ll 1$ such that \textbf{there exists a $10 + 3\cdot 16 = 58$-parameter family of \hka triples and \hka metrics} $$\set*{\paren*{\bomega_{f, \epsilon_{I}, e_{I}}, g_{\bomega_{f, \epsilon_{I}, e_{I}}}}}_{\substack{f \in \text{flat tori moduli}\\ \epsilon_{i} \in \paren*{0,\epsilon^{0}} \\ e_{i} \in S^{2} \\ i =1,\dots, 16}}$$ on the \textbf{K3 surface} $\Km^{16} = K3 = \Km_{f, \epsilon_{I}, e_{I}}$.

Moreover as $\epsilon^{0} \searrow 0$, $$\paren*{\Km_{f, \epsilon_{I}, e_{I}}, \bomega_{f, \epsilon_{I}, e_{I}}, g_{\bomega_{f, \epsilon_{I}, e_{I}}}} \xrightarrow{GH} \paren*{\bbT^{4}/\bbZ_{2}, \bomega_{f}, g_{\bomega_{f}}}$$ converges in the Gromov-Hausdorff topology to the flat \hka orbifold $\paren*{\bbT^{4}/\bbZ_{2}, \bomega_{f}, g_{\bomega_{f}}}$ with the singular orbifold \hka metric $\bomega_{f}, g_{\bomega_{f}}$. Moreover, we may decompose $K3$ into a union of sets $\Km_{f, \epsilon_{I}, e_{I}}^{reg} \cup \bigsqcup_{i =1}^{16}\Km_{f, \epsilon_{I}, e_{I}}^{b_{i}}$ such that \begin{enumerate}
\itemsep0em 
\item $\paren*{\Km_{f, \epsilon_{I}, e_{I}}^{reg}, \bomega_{f, \epsilon_{I}, e_{I}}, g_{\bomega_{f, \epsilon_{I}, e_{I}}}}$ converges in $C^{\infty}_{g_{\bomega_{0}},loc}$ to the flat \hka manifold $\paren*{\bbT^{4}/\bbZ_{2} - \cS, \bomega_{0}, g_{\bomega_{0}}}$ with bounded curvature away from $\cS$ (\textbf{regular region}).
\item For each $i =1,\dots, 16$, $\paren*{\Km_{f, \epsilon_{I}, e_{I}}^{b_{i}}, \frac{1}{\epsilon_{i}^{2}} \bomega_{f, \epsilon_{I}, e_{I}}, \frac{1}{\epsilon_{i}^{2}}g_{\bomega_{f, \epsilon_{I}, e_{I}}}}$ (after a $\frac{1}{\epsilon_{i}}$-dilation) converges to $\paren*{T^{*}S^{2}, A^{e_{i}}\bomega_{EH,1}, g_{A^{e_{i}}\bomega_{EH,1}}}$ in $C^{\infty}_{g_{A^{e_{i}}\bomega_{EH,1}},loc}$ (\textbf{ALE bubble region}).
\end{enumerate}

Therefore, letting $k \in \set*{0,\dots, 16}$ be a chosen nonnegative integer (the \textbf{depth}), upon considering the following strictly nested sequence $$\set*{1,\dots, 16} = I_{0}\supsetneq I_{1}\supsetneq \cdots \supsetneq I_{k}$$ with $I_{16} \coloneq \emptyset$, then we have the following sequence of \textbf{iterated volume non-collapsed degenerations} of our 58-parameter family of \hka metrics $\bomega_{f, \epsilon_{I_{0}}, e_{I_{0}}}, g_{\bomega_{f, \epsilon_{I_{0}}, e_{I_{0}}}}$ on the $K3$ surface $\Km^{\abs*{I_{0}}}_{f, \epsilon_{I_{0}}, e_{I_{0}}} =\Km_{f, \epsilon_{I_{0}}, e_{I_{0}}}$:\\

\hspace*{-2.698cm}\scalebox{0.93}{\begin{tikzcd}[ampersand replacement=\&]
\&\& {\paren*{\Km_{f, \epsilon_{I_{0}}, e_{I_{0}}}, g_{\bomega_{f, \epsilon_{I_{0}}, e_{I_{0}}}}}} \\
{\paren*{\Km^{\abs*{I_{1}}}_{f, \epsilon_{I_{1}}, e_{I_{1}}}, g_{\bomega_{f, \epsilon_{I_{1}}, e_{I_{1}}}}}} \& {\paren*{\Km^{\abs*{I_{2}}}_{f, \epsilon_{I_{2}}, e_{I_{2}}}, g_{\bomega_{f, \epsilon_{I_{2}}, e_{I_{2}}}}}} \& \cdots \& {\paren*{\Km^{\abs*{I_{k-1}}}_{f, \epsilon_{I_{k-1}}, e_{I_{k-1}}}, g_{\bomega_{f, \epsilon_{I_{k-1}}, e_{I_{k-1}}}}}} \& {\paren*{\Km^{\abs*{I_{k}}}_{f, \epsilon_{I_{k}}, e_{I_{k}}}, g_{\bomega_{f, \epsilon_{I_{k}}, e_{I_{k}}}}}} \\
\&\& {\paren*{\bbT^{4}/\bbZ_{2}, g_{\bomega_{f}}}}
\arrow["GH"', from=1-3, to=2-1]
\arrow["GH", from=1-3, to=2-2]
\arrow["GH", from=1-3, to=2-3]
\arrow["GH", from=1-3, to=2-4]
\arrow["GH", from=1-3, to=2-5]
\arrow["GH", from=2-1, to=2-2]
\arrow["GH", from=2-1, to=3-3]
\arrow["GH", from=2-2, to=2-3]
\arrow["GH", from=2-2, to=3-3]
\arrow["GH", from=2-3, to=2-4]
\arrow["GH", from=2-3, to=3-3]
\arrow["GH", from=2-4, to=2-5]
\arrow["GH"', from=2-4, to=3-3]
\arrow["GH", from=2-5, to=3-3]
\end{tikzcd}}

where we\begin{itemize}
\itemsep0em 
\item degenerate from our \hka $K3$ surface $\paren*{\Km_{f, \epsilon_{I_{0}}, e_{I_{0}}}, g_{\bomega_{f, \epsilon_{I_{0}}, e_{I_{0}}}}}$ to $\paren*{\Km^{\abs*{I_{j}}}_{f, \epsilon_{I_{j}}, e_{I_{j}}}, \bomega_{f, \epsilon_{I_{j}}, e_{I_{j}}}, g_{\bomega_{f, \epsilon_{I_{j}}, e_{I_{j}}}}}$ via only collapsing the gluing parameters $\epsilon_{i}>0$ corresponding to $i \in I_{0} - I_{j} = I_{j}^{c}$
\item degenerate from $\paren*{\Km^{\abs*{I_{j-1}}}_{f, \epsilon_{I_{j-1}}, e_{I_{j-1}}}, \bomega_{f, \epsilon_{I_{j-1}}, e_{I_{j-1}}}, g_{\bomega_{f, \epsilon_{I_{j-1}}, e_{I_{j-1}}}}}$ to $\paren*{\Km^{\abs*{I_{j}}}_{f, \epsilon_{I_{j}}, e_{I_{j}}}, \bomega_{f, \epsilon_{I_{j}}, e_{I_{j}}}, g_{\bomega_{f, \epsilon_{I_{j}}, e_{I_{j}}}}}$ via only collapsing the gluing parameters $\epsilon_{i}>0$ corresponding to $i \in I_{j-1} - I_{j}$
\item and degenerate from $\paren*{\Km^{\abs*{I_{j}}}_{f, \epsilon_{I_{j}}, e_{I_{j}}}, \bomega_{f, \epsilon_{I_{j}}, e_{I_{j}}}, g_{\bomega_{f, \epsilon_{I_{j}}, e_{I_{j}}}}}$ to $\paren*{\bbT^{4}/\bbZ_{2}, \bomega_{f}, g_{\bomega_{f}}}$ via collapsing all gluing parameters $\epsilon_{i} > 0, i \in I_{j}$.\end{itemize}\end{theorem}

\begin{remark}\label{Kobayashi/Lye and Kobayashi-Todorov comparison}

As remarked already, the uniform a priori estimates that Kobayashi \cite{KobayashiModuli} and Lye \cite{Lye} prove are for the complex Monge-Ampere equation, whence their estimates requires a fixed integrable background complex structure and hold only in the Calabi-Yau perspective. Using the notation of Theorem \ref{MainTheoremA}, the estimates they prove are $$\Abs*{f_{\epsilon}}_{C^{k}_{g_{\epsilon}} \paren*{\Km}} \leq C\epsilon^{2 - \frac{k}{2}}$$ Their proof follows Yau's proof of the Calabi conjecture via continuity method, whence upon first establishing openness, uses a \textbf{uniform Sobolev inequality}, a \textbf{uniform Poincare inequality}, and \textbf{Moser iteration} to get the \textbf{$C^{0}$-estimate}, then the \textbf{maximum principle} to get the \textbf{$C^{2}$-estimate}, then further local estimates to get H\"{o}lder and higher order estimates. Kobayashi in fact also proves slightly weaker uniform a priori estimates more generally when resolving \textit{any} compact Ricci-flat \ka orbifold whose resolutions yields $K3$, namely $\Abs*{f_{\epsilon}}_{C^{k}_{g_{\epsilon}} \paren*{\Km}} \leq C\epsilon^{\frac{13}{8} - \frac{k}{2}}$.

The estimates that we prove from the Calabi-Yau perspective in Theorem \ref{MainTheoremA}, namely $$\Abs*{f_{\epsilon}}_{C^{k}_{g_{\epsilon}} \paren*{\Km}} \leq C\epsilon^{4 - \lambda\paren*{\delta + 2} + \delta - k}$$ for $\delta \in \paren*{\max\set*{-2,-4\lambda},0}$ and $\lambda \in \paren*{0,\frac{2}{3}}$ is thus an improvement over Kobayashi/Lye's estimates for the case of resolving $\bbT^{4}/\bbZ_{2}$ as in the original Kummer construction, and is achieved via deriving an estimate of $f_{\epsilon}$ in certain \textbf{weighted H\"{o}lder spaces} (see Theorem \ref{bigtheorem1}).

Moreover, our weighted analysis has the advantage of carrying over to the \hka case, whence in Theorem \ref{bigtheorem2} we prove similar weighted H\"{o}lder estimates from the \hka perspective, which readily carry over to the large families of \hka orbifolds we construct in Theorem \ref{MainTheoremB}.

In fact, in \textit{Remark} \ref{ricciestimates} and \textit{Remark} \ref{HKA ricciestimates} we derive \textbf{uniform Ricci curvature estimates} for our constructed approximate metric, whence (since we clearly have uniform volume non-collapse and uniformly bounded diameter) we \textit{also} have uniform Poincare and uniform Sobolev inequalities in our setting. Whence we may also carry out a \textit{weighted} Moser iteration argument and rederive the analogue of Yau/Kobayashi/Lye's $C^{0}$ and higher order estimates.\end{remark}

\begin{remark}Similar analysis done here has been carried out in slightly different settings.

Biquard-Minerbe \cite{Biquard-Minerbe} did a non-compact Kummer construction with weighed spaces in order to construct noncompact \hka \textbf{ALF, ALG, and ALH spaces} (which respectively have cubic, quadratic, and linear volume growth), and was done via solving the complex Monge-Ampere equation.

Brendle-Kapouleas \cite{BK} did the original Kummer construction considered here on $K3$, but they glued in the 16 Eguchi-Hanson spaces with half of them having \textit{opposite orientation} (arranged in a checkers board pattern). As a result, Brendle-Kapouleas had to solve the full Einstein equation $\Ric = 0$ because the swapped orientation forces one to leave even the K\"{a}hler realm (since it prevents there being \textit{any} resulting integrable complex structure) whence forcing the (approximate) metric to have \textit{generic holonomy}. Note that in our case, when we vary the parameters of the construction from the \hka perspective (Theorem \ref{MainTheoremB}), while we no longer patch together complex structures, we \textit{still} have all Eguchi-Hanson spaces glued in with the \textit{same} orientation since definite triples require the data of an orietation (see Definition \ref{HKA definite triple definition}), thus we're still \textit{special holonomy}. Brendle-Kapouleas end up showing that the relevant linearization has an obstructed kernel (even up to several higher orders), thus nullifying the construction and dashing any hopes of using the original Kummer construction to construct Ricci-flat metrics on \textit{closed} manifolds with \textit{generic holonomy} (as opposed to Calabi-Yau, hyper-K\"{a}hler, etc., with \textit{special} holonomy $\SU\paren*{\frac{n}{2}}, \Sp\paren*{\frac{n}{4}}$).
\end{remark}

\begin{remark}\label{ModuliRemark} 

It is interesting to consider the \textit{full} compactification $\overline{\sM^{K3}_{\Vol=1}}$, whence to study the limit spaces of sequences in $\sM^{K3}$ which are \textbf{volume collapsed}. Indeed, since the period map extends to a bijection $\cP: \overline{\sM^{K3}_{\Vol=1}}^{\del} \rightarrow \Gr^{+}\paren*{3,19}/\Gamma$ when adding in all volume non-collapsed limits, as $\Gr^{+}\paren*{3,19}/\Gamma$ is still \textit{non-compact} we say that a sequence $\set*{g_{i}}\subset \sM^{K3}$ is volume collapsed to mean that $\set*{\cP\paren*{g_{i}}}$ diverges to infinity in $\Gr^{+}\paren*{3,19}/\Gamma$. 

Much recent work has been done in these and related directions, namely in classifying elements in the ``boundary at $\infty$''/the set of \textbf{volume collapsed limit spaces} $\overline{\sM^{K3}_{\Vol=1}} - \overline{\sM^{K3}_{\Vol=1}}^{\del}$, exhibiting open subsets of $\overline{\sM^{K3}_{\Vol=1}}$/explicit $58$ parameter families of volume collapsed sequences convergent to those collapsed limit spaces, and exhibiting more general noncompact \hka gravitational instantons as bubbles/models of the convergence near singularities and volume-collapsed regions.

We give an impressionistic survey of these new (and old) results:

Foscolo \cite{Foscolo} exhibited a full open subset/58 parameter family of sequences $\set*{g_{i}} \subset \sM^{K3}$ which exhibit \textbf{codimension 1 collapse} down to $\bbT^{3}/\bbZ_{2}$ which have \textit{both} \textbf{cyclic ALF} (i.e. \textbf{multi Taub-NUT}) and \textbf{dihedral ALF} \hka spaces as bubbles near (no more than 16 arbitrary) punctures and the 8 punctures corresponding to the 8 fixed points of $\bbZ_{2}\Lactson \bbT^{3}$, respectively. This is done via a similar Kummer-type/gluing construction and closed definite perturbation as considered here, where (roughly speaking) we are shrinking one of the $S^{1}$-factors of $\bbT^{4}/\bbZ_{2}$ down to zero (or rather, the $S^{1}$-fibers of a not necessarily trivial principal $\U(1)$-bundle over an open subset of a punctured $\bbT^{3}$). Moreover, the \textbf{charges} of these cyclic \& dihedral ALF spaces must satisfy a \textit{charge conservation} equation, which, by the residue theorem, guarantees the existence of a suitable harmonic function on an open subset of a punctured $\bbT^{3}$ with \textit{prescribed asymptotics} near the punctures, which is needed for the \textbf{Gibbons-Hawking ansatz} of $\U(1)$-invariant $4$-dimensional \hka metrics that is involved in the construction of all cyclic ALF spaces and the asymptotics of dihedral ALF spaces (up to exponentially small error). Foscolo's constructed family of metrics (in the case where one of the building blocks consists of the \textbf{Atiyah-Hitchin manifold}, or the \textbf{$D_{1}$ ALF space}) was also shown to admit a \textbf{strictly stable minimal sphere} which is \textit{not} holomorphic WRT any complex structure compatible with the metric, thus providing alternative examples of this phenomena on $K3$ which was originally shown to exist by Micallef-Wolfson \cite{Micallef-Wolfson}.

Gross-Wilson \cite{Gross-Wilson} constructed sequences $\set*{g_{i}}\subset \sM^{K3}$ on an \textbf{elliptic $K3$ surface} with 24 singular $I_{1}$ fibers\footnote{Elliptic $K3$ surfaces, which (usually) have a \textit{fixed} complex structure and always have singular fibers, have their singular fibers classified: the list of types of singular fibers are $I_{0}^{*}, II, III, IV, II^{*}, III^{*}, IV^{*}, I_{\nu}, I_{\nu}^{*}, \forall \nu \in \bbN$.} exhibiting \textbf{codimension 2 collapse} down to $\paren*{\CP^{1},g_{ML}}$ with the (singular) \textbf{McLean metric}. This is done via a gluing construction similar to the one considered here, where the model away from singular fibers provided by a \textbf{semi-flat metric} (on the generic elliptic curve fiber), and near the singular fibers by the \hka \textbf{Ooguri-Vafa} metric, hence is a bubble modeling the convergence near the 24 singular $I_{1}$ fibers (topologically a pinched torus). The proof is done via complex Monge-Ampere and the continuity method, hence all metrics are \ka WRT a fixed elliptic complex structure. Moreover, this setting is that of a \ka metric on $K3$ approaching the ``\textbf{large complex structure limit}'', which by a \hka rotation of any complex structure on $K3$ giving the existence of special Lagrangian fibrations, is equivalent to having a \textit{fixed} complex structure on an elliptic $K3$ and letting the \ka metric degenerate \textbf{adiabatically} (i.e. the volume of the fibers shrinking to zero)\footnote{Note that near the singular fiber region the scale of the circle which collapses down to develop the $I_{1}$ fiber collapses \textit{faster} than the full diameter scale of the collapsing fibers in the adiabatic limit, whence exhibiting \textbf{multi-scale collapse}}. The main guiding heuristic here, namely that \ka metrics WRT a complex structure in the ``large complex structure limit'' should degenerate in a certain way and (roughly) produce special Lagrangian (torus) fibrations, is called the ``\textbf{SYZ conjecture}''.

Gross-Tosatti-Zhang	\cite{Gross-Tosatti-Zhang} generalized Gross-Wilson's construction to general elliptic $K3$ surfaces, i.e. non-generic configurations of singular fibers. This was also done via complex Monge-Ampere, hence the resulting sequence of Ricci-flat metrics are \ka WRT a fixed elliptic complex structure. However, no description of the metric degeneration near the singular fibers/bubbles were explored. Chen-Viaclovsky-Zhang \cite{Chen-V-Zhang-Elliptic}, using closed definite triple perturbation (but still with a fixed elliptic complex structure) redid Gross-Tosatti-Zhang but with much more precise descriptions of metric degeneration near singular fibers, exhibiting \hka ALE, cyclic ALF, and \hka \textbf{ALG} spaces as bubbles modeling the convergence near the singular fibers.

Hein-Sun-Viaclovsky-Zhang \cite{HSVZ} constructed sequences $\set*{g_{i}}\subset \sM^{K3}$ which exhibits both \textbf{multi-scale/iterative} and \textbf{codimension 3 collapse} down to the unit interval $\paren*{[0,1],dt^{2}}$. Generically, the collapse is modeled on $S^{1}$-fiber bundles over $\bbT^{2}$ with the fiber shrinking at a \textit{faster} rate than the base, and away from the generic region the collapse is modeled on Taub-NUT and \textbf{ALH Tian-Yau} spaces appearing as bubbles. This was done on the purely \textit{metric} level (hence via closed definite triple perturbation), with no background complex structure.

Chen-Viaclovsky-Zhang \cite{Chen-V-Zhang-Torelli-ALG} in a different work proved Torelli-type theorems on \hka \textbf{ALG$^{\textbf{*}}$} spaces, and also exhibits codimension 2 collapse of \hka metrics on $K3$ down to $\paren*{\CP^{1},g_{ML}}$ with ALG$^{\textbf{*}}$ and Taub-NUT bubbles modeling the collapse near the singular fibers.

Around the same time as these developments, Sun-Zhang \cite{Sun-Zhang-HKA} classified \textit{all} possible collapsed limit spaces/elements of $\overline{\sM^{K3}_{\Vol=1}} - \overline{\sM^{K3}_{\Vol=1}}^{\del}$ (hence completing the full picture of the Gromov-Hausdorff compactification), and showed that $\overline{\sM^{K3}_{\Vol=1}} - \overline{\sM^{K3}_{\Vol=1}}^{\del} = \set*{\bbT^{3}/\bbZ_{2}, S^{2}, [0,1]}$, hence the flat $\bbT^{3}/\bbZ_{3}$ (codimension 1 collapse), a singular special \ka metric on $S^{2}$ (codimension 2 collapse), or the flat unit interval $[0,1]$ (codimension 3 collapse, in which further structure was examined by Honda-Sun-Zhang \cite{Honda-Sun-Zhang}). This was done via proving \textbf{uniqueness of tangent cones} near singularities of any collapsed limit.

Finally, we mention recent work of Sun-Zhang \cite{Sun-Zhang-CY-2024} who proved new collapsing behavior of \textit{arbitrary dimension} Calabi-Yau metrics in the ``\textbf{small complex structure limit}'' (hence the other extreme of Gross-Wilson/SYZ), specifically in the case of a family of smooth projective Calabi-Yau hypersurfaces viewed as a complex $1$-parameter family $\set*{X_{t}}_{t \in \Delta}$ of complex deformations of a simple normal crossing variety $X_{0}$. That is, they prove that one can carry over this algebraic degeneration as $\abs*{t}\searrow 0$ to the \textit{metric} level and obtain precise geometric information on this metric degeneration, namely that the family collapses all the way down to the unit interval $\paren*{[0,1],dt^{2}}$ in such a way as to exhibit multiscale/iterated collapse, with generic fibers an $S^{1}$-bundle over the smooth intersection divisor, and singular $S^{1}$-bundles with Taub-NUT $\times$ flat and Tian-Yau bubbles near the singular $S^{1}$ fibers and the endpoints of $[0,1]$, respectively, and with all $S^{1}$ fibers shrinking faster than their bases.\end{remark}

This paper is organized as follows: 

In Section \ref{Notation}, we list some common notations used throughout.

In Section \ref{Preliminaries}, we list some preliminary definitions used and give a derivation of the Monge-Ampere equation for (more general) K\"{a}hler-Einstein metrics.

In Section \ref{EH as building block}, we define the Eguchi-Hanson metric on the blowup of $\bbC^{2}/\bbZ_{2}$ and prepare it for the gluing process, as well as recall some facts about flat tori.

In Section \ref{The Kummer Construction}, we construct the ``Kummer surface'' $\Km$ by blowing up the orbifold points of $\bbT^{4}/\bbZ_{2}$ in a manner suitable for gluing, and we graft in the prepped Eguchi-Hanson metrics to get an approximately Ricci-flat K\"{a}hler metric $\omega_{\epsilon}$ on $\Km$.

In Section \ref{Weighted Setup}, we define the relevant weighted H\"{o}lder spaces that will be used in the construction.

In Section \ref{Nonlinear Setup}, we setup the nonlinear PDE problem of perturbing the approximately Ricci-flat $\omega_{\epsilon}$ to a genuinely Ricci-flat K\"{a}hler metric within its K\"{a}hler class as a fully nonlinear elliptic PDE of Monge-Ampere type, in a manner suitable for the implicit function theorem.

In Section \ref{Blow-up Analysis}, we do a blow-up analysis argument to invert the linearization of the nonlinear map defined in the previous section in the relevant weighted H\"{o}lder spaces. This is the heart of the argument, as we need the weighted H\"{o}lder spaces to get \textit{uniform} bounded operator norm on the inverse of the linearization.

In Section \ref{Finishing the proof}, we finish the proof of Theorem \ref{MainTheoremA} by applying the implicit function theorem to get our solution to the Monge-Ampere equation.

In Sections \ref{HKA EH as building block}, \ref{HKA The Kummer Construction}, \ref{HKA Weighted Setup}, and \ref{HKA Blow-up Analysis} we upgrade our analysis to established in the previous sections to deal with closed definite triple perturbation from the \hka perspective. As we are no longer solving for a scalar function but for a \textit{$2$-form} in our PDE, the analysis substantially differs.

Finally, in Section \ref{HKA Finishing the proof} we finish the proof of Theorem \ref{MainTheoremB} by applying the implicit function theorem to get our solution to the nonlinear elliptic PDE system involved, as well as discuss some consequences of what has been proved from two perspectives:\begin{itemize}
\itemsep0em 
\item the perspective of ``further probing'' the codimension $3$ subset of all (unit volume) volume non-collapsed limits of \hka metrics on $K3$ inside $\Gr^{+}\paren*{3,19}/\Gamma$ (\textit{Remark} \ref{HKA big main theorem 2 consistency with codimension 3 for depth 1} and \textit{Remark} \ref{HKA big main theorem 2 consistency with codimension 3 for depth k})
\item the broader perspective of general Riemannian convergence theory, specifically the work of Cheeger and Naber \cite{CN} (\textit{Remark} \ref{HKA codimension 4 conjecture remark}).
\end{itemize}

\begin{ack*} Thanks to Professor Yuji Odaka for referring me to \cite{KobayashiTodorov} and \cite{Odaka-Oshima} and for his encouragements.

Part of this paper consists of the author's honors undergraduate thesis written in April 2023 at Stony Brook University. 

I owe a tremendous debt to my undergraduate advisor Professor Simon Donaldson for supervising, introducing, and teaching me the beautiful geometric analysis subfields of special holonomy and gauge theory. I would not be where I am today without his tutelage and his kindness towards me. 

Thanks also to my other mathematics professors at Stony Brook for teaching me so much in class, during office hours, through emails, on walks around campus, and over meals.

Thanks also to Jakob Stein for proofreading this paper and catching some mistakes.

Lastly, I would like to thank Professor Mark Stern, Professor Robert Bryant, and my PhD advisor, Professor Mark Haskins, for their encouragement, expertise, and kindness toward me. \textit{Soli Deo gloria!}\end{ack*}

\section{\textsection \ (Global) Notation \ \textsection}\label{Notation}

\begin{notation}
\begin{itemize}
\item Everything is assumed $C^{\infty, \omega}$ when appropriate/unless explicitly mentioned otherwise, and all manifolds are connected (unless otherwise stated). 
\item For $g$ a Riemannian metric and $E \subseteq \paren*{M,g}$ some subset, we will let $\abs*{\cdot}^{g}_{E} \coloneq \dist_{g}\paren*{E,\cdot}$ be the \textit{metric} distance function WRT $g$ away from $E$.

Clearly $\abs*{\cdot}^{g}_{E} \in \Lip^{1}\paren*{M, [0,\infty)}$ is Lipschitz continuous with Lipschitz constant $1$, and is $C^{\infty}$/smooth away from $E$ and its cut locus.
\item $B_{R}^{g}\paren*{E}$ denotes the open ball $\set*{\abs*{\cdot}^{g}_{E} < R}$. Hence this is a \textit{metric} open tubular neighborhood of $E$.
\item $D_{R}^{g}\paren*{E}$ denotes the disk $\set*{\abs*{\cdot}^{g}_{E} \leq R}$. Hence this is a \textit{metric} closed tubular neighborhood of $E$.
\item We sometimes use $\restr{f}{p}$ to denote $f$ a function (tensor, section, etc.) evaluated at $p$, i.e. $f(p)$, to simplify notation.
\item $\abs*{\cdot}_{g}$ will denote a (usually $(p,q)$-tensor) norm that's induced from a given specified Riemannian metric $g$.
\item $\nabla_{g}$ will denote the Levi-Civita connection induced from a given Riemannian metric $g$.
\item $O_{g}$ will denote the usual Landau symbol but with all tensor norms measured with respect to the specified Riemannian metric $g$.
\end{itemize}	
\end{notation}

\section{\textsection \ Calabi-Yau Metrics on $K3$ and the Classical Kummer Construction \ \textsection}\label{CY analogue}

\subsection{\textsection \ K\"{a}hler-Einstein Metrics, the Complex Monge-Ampere Equation, and Calabi-Yau Metrics \ \textsection}\label{Preliminaries}

We collect preliminary definitions and derivations on K\"{a}hler-Einstein metrics, mostly following \cite{Walpuski}.

\begin{definition}\label{KahlerDef}
A \textbf{K\"{a}hler manifold} is a \textit{real smooth} $2n$ dimensional manifold $M^{2n}$ with the following data:

\begin{enumerate}
\itemsep0em 
\item an almost complex structure $J \in \Gamma\paren*{\Hom\paren*{TM, TM}}$, hence pointwise on each tangent space $J^{2} = - \id$			
\item a Riemannian metric $g \in \Gamma\paren*{\Sym^{2}T^{*}M}$			
\item a non-degenerate 2-form $\omega \in \Omega^{2}\paren*{M} \coloneq \Gamma\paren*{\Lambda^{2} T^{*}M}$ (nondegeneracy here means that, pointwise/on the vector space level, the map $V \rightarrow V^{*}$ via $v \mapsto \omega(v,\cdot)$ is an isomorphism, which is equivalent to $\frac{\omega^{n}}{n!}$ being a volume form).

\end{enumerate} such that the 3 pieces of data above are compatible in the following way: $$\omega\paren*{\cdot,\cdot} = g\paren*{J\cdot,\cdot}$$ and one of the following equivalent conditions are satisfied: \begin{itemize}
\itemsep0em 
\item $J$ is integrable (i.e. $M^{2n}$ is actually a complex $n$-manifold $\paren*{M^{n},J}$ \textit{WRT $J$}, aka $N_{J} = 0$ the Nijenhuis tensor of $J$ vanishes) and $d\omega = 0$ (i.e. $\omega$ is symplectic)
\item $\nabla_{g} J = 0$
\item $\nabla_{g} \omega = 0$
\item The \textbf{holonomy group $\Hol(g)$} of the Levi-Civita connection of $g$ is contained in $\U\paren*{n}$.

\end{itemize} Hence the data of $\paren*{M^{2n}, J, g, \omega}$ satisfying the compatibility condition and one of the 3 equivalent conditions a \textbf{K\"{a}hler manifold}, and we call $[\omega] \in H^{2}\paren*{M}$ the \textbf{K\"{a}hler class} and $\omega$ the \textbf{K\"{a}hler form}, and we also have $\frac{\omega^{n}}{n!} = \dV_{g}$.\end{definition}

\begin{remark}\label{almost Hermitian manifold} Note that if we just have the 3 pieces of data of $\paren*{J, g, \omega}$ satisfying the compatibility condition, but not one of the following equivalent conditions, we call this an \textbf{almost Hermitian manifold.} 
\end{remark}

\begin{remark}\label{2outof3}
More importantly, note that \textit{any cardinality 2 subset of $\set*{J, g, \omega}$ automatically determines/defines the leftover element in its complement} via the compatibility condition. Roughly speaking, this follows because $\U(n) = \O(2n) \cap \Sp(2n) = \O(2n) \cap \GL\paren*{n,\bbC} = \GL\paren*{n,\bbC} \cap \Sp\paren*{2n}$ (this may be made rigorous using the language of $G$-structures on the frame bundle $\Fr\paren*{TM}$). In other words, K\"{a}hler geometry is the \textit{intersection} of complex and Riemannian geometry, and complex and symplectic geometry, and symplectic and Riemannian geometry. This is sometimes called the \textbf{``2-out-of-3''} property and is one of the reasons why K\"{a}hler geometry is very rich.
\end{remark}

\begin{remark}
Henceforth on a \textit{complex} manifold $\paren*{M^{n},J}$ we will sometimes abuse notation and call $\omega$ the K\"{a}hler metric instead of $g$, precisely because 2-out-of-3 precisely means that on $\paren*{M^{n},J}$, the existence of a \ka form $\omega$ (compatible with $J$) \textit{automatically} gives us a Riemannian metric $g$ (compatible with $J$) by the compatibility condition $g\paren*{\cdot,\cdot} = \omega\paren*{\cdot, J\cdot}$.
\end{remark}

Now on any K\"{a}hler manifold $\paren*{M^{2n}, J, g, \omega}$, we have that the \textbf{canonical bundle} $K_{M} \coloneq \bigwedge^{n}_{\bbC}T^{*}M^{(1,0)}$ associated to the \textit{underlying complex manifold $\paren*{M^{n},J}$} has an induced Hermitian metric coming from the Hermitian metric $h \coloneq g - i\omega$ on $TM$ (which as \textit{complex} vector bundles is isomorphic to $TM^{(1,0)}$, a \textit{holomorphic} vector bundle), and the induced Levi-Civita connection on $K_{M}$ coming from the Levi-Civita connection of $g$ is precisely the \textbf{Chern connection} of the induced Hermitian metric on $K_{M}$, i.e. it is unitary and has $(0,1)$-part the $\delbar$ operator inducing the holomorphic structure on $K_{M}$. Not only that, but the curvature 2-form $F$ of this connection on $K_{M}$ is precisely $i \Ric_{\omega}$, where $\Ric_{\omega}\paren*{\cdot,\cdot} \coloneq \Ric_{g}\paren*{J\cdot, \cdot} \in \Omega^{2}\paren*{M}$. Hence since $c_{1}\paren*{M^{n},J} = -\frac{i}{2\pi} [F]$ by Chern-Weil theory\footnote{We have that on any complex manifold $\paren*{M^{n},J}$ that $c_{1}(M^{n},J) \coloneq c_{1}\paren*{\det_{\bbC} TM} = c_{1}\paren*{\det_{\bbC} TM^{(1,0)}} = - c_{1}\paren*{\det_{\bbC} T^{*}M^{(1,0)}}$, and Chern-Weil theory tells us that the first Chern class of a line bundle is $\frac{i}{2\pi}$ times the de Rham cohomology class of the curvature 2-form of \textit{any arbitrary} connection on that line bundle, thus giving us $c_{1}\paren*{M^{n},J} = -\frac{i}{2\pi} [F]$ for $F$ the curvature 2-form of \textit{any} connection on the line bundle $\det_{\bbC} T^{*}M^{(1,0)} \eqcolon \bigwedge^{n}_{\bbC}T^{*}M^{(1,0)} \eqcolon K_{M}$.}, this gives us that $c_{1}(M^{n},J) = \frac{1}{2\pi} [\Ric_{\omega}]$.

We summarize this and other properties needed in the sequel in the below proposition (proofs found in \cite{Walpuski}):

\begin{prop}\label{K\"{a}hlerprop}
Let $\paren*{M^{2n}, J, g, \omega}$ be a K\"{a}hler manifold. Then the following properties hold:

\begin{itemize}
\itemsep0em 
\item The first Chern class $c_{1}\paren*{M^{n},J} \in H^{2}\paren*{M}$ of $M$ is $c_{1}\paren*{M^{n},J} = \frac{1}{2\pi} [\Ric_{\omega}]$.

\item For any two K\"{a}hler forms $\omega_{1}, \omega_{0}$, the difference between their Ricci forms is $\Ric_{\omega_{1}} - \Ric_{\omega_{0}} = - i \del \delbar \log\paren*{\frac{\omega_{1}^{n}}{\omega_{0}^{n}}}$, with the division of two top degree forms having the obvious meaning.

\end{itemize}
\end{prop}

\begin{remark}
Note that $c_{1}\paren*{M^{n},J} = \frac{1}{2\pi} [\Ric_{\omega}]$ holds for \textit{any} K\"{a}hler form $\omega$ on $\paren*{M^{n},J}$, and hence $c_{1}\paren*{M^{n},J}$ \textit{does not depend} on the K\"{a}hler metric $\omega$. In fact, as suggested by the notation, $c_{1}\paren*{M^{n},J}$ only depends on the (integrable) complex structure $J$ (besides obviously depending on the underlying smooth manifold $M^{2n}$), whence is an underlying \textit{topological} invariant of the \textit{complex} manifold $\paren*{M^{n},J}$.
\end{remark}

From now on, \textit{fix the (integrable) complex structure $J$} on our smooth manifold $M^{2n}$, aka we're now working on a fixed \textit{complex} manifold $\paren*{M^{n},J}$.

We now want to examine the necessary and sufficient conditions for a given closed (boundaryless and compact) K\"{a}hler manifold $\paren*{M^{n}, J, g_{0}, \omega_{0}}$ to admit a \textbf{K\"{a}hler-Einstein metric} in the same K\"{a}hler class $[\omega_{0}]$, namely a K\"{a}hler 2-form $\omega \in [\omega_{0}]$ (hence real $(1,1)$ and positive) such that $\Ric_{\omega} = \lambda \omega$ for $\lambda \in \bbR$, whence the resulting Riemannian metric $g$ \textit{compatible with $J$} is also Einstein with constant $\lambda$. Clearly from Proposition \ref{K\"{a}hlerprop} we have that a necessary condition is that $c_{1}\paren*{M^{n},J} = \frac{\lambda}{2\pi} [\omega] = \frac{\lambda}{2\pi} [\omega_{0}]$ since $[\omega] = [\omega_{0}]$.

It is therefore natural to ask the converse:

\begin{question}
Suppose $c_{1}\paren*{M^{n},J} = \frac{\lambda}{2\pi} [\omega_{0}]$ for a fixed closed K\"{a}hler manifold $\paren*{M^{n}, J, g_{0}, \omega_{0}}$. Does there exist a K\"{a}hler form $\omega \in [\omega_{0}]$ such that $\Ric_{\omega} = \lambda \omega$?
\end{question}

Now by Proposition \ref{K\"{a}hlerprop}, we have that $c_{1}\paren*{M^{n},J} = \frac{\lambda}{2\pi} [\omega_{0}] \Rightarrow [\Ric_{\omega_{0}}] = \lambda [\omega_{0}] \Rightarrow \Ric_{\omega_{0}} = \lambda \omega_{0} + i\del \delbar \rho$ by the $\del \delbar$ lemma with $\rho \in C^{\infty}\paren*{M}$ unique up to a constant. We call $\rho$ the \textbf{Ricci potential} of $\omega_{0}$, and from now on we impose the following ``volume normalization condition'' to pin down the potential: 

$$ \int_{M} e^{\rho} \omega_{0}^{n} = \int_{M}\omega_{0}^{n} = n! \Vol_{g_{0}}\paren*{M}$$

Now again by the $\del \delbar$ lemma, any K\"{a}hler form $\omega$ which is in the same cohomology class as $\omega_{0}$ must be of the form $\omega = \omega_{0}+ i\del \delbar f$ for some $f \in C^{\infty}\paren*{M}$. Hence if $\exists \omega$ a K\"{a}hler form in $[\omega_{0}]$ such that $\Ric_{\omega} = \lambda \omega$, then by part 2 of Proposition \ref{K\"{a}hlerprop}, $\omega  = \omega_{0}+ i\del \delbar f$ must satisfy

\begin{al}
\Ric_{\omega} - \Ric_{\omega_{0}} &= \lambda\paren*{ \omega_{0}+ i\del \delbar f} - \lambda \omega_{0} - i\del \delbar \rho \\
&= i\del \delbar \paren*{\lambda f - \rho } \\
&= -i\del \delbar \log\paren*{\frac{\omega^{n}}{\omega_{0}^{n}}}
\end{al} giving us $$ i\del \delbar\paren*{\log\paren*{\frac{\omega^{n}}{\omega_{0}^{n}}} + \lambda f - \rho} = 0$$ or equivalently $$\log\paren*{\frac{\omega^{n}}{\omega_{0}^{n}}} + \lambda f - \rho = c \Rightarrow \frac{\omega^{n}}{\omega_{0}^{n}} = e^{-\lambda f} e^{\rho} e^{c}$$ with the constant $c \coloneq - \log\paren*{\frac{\int_{M} e^{-\lambda f} e^{\rho} \omega_{0}^{n}}{\int_{M} \omega_{0}^{n}}}$.

Now if $\lambda = 0$, then by the $\rho$ normalization, $c = 0$. If $\lambda \neq 0$, then $f$ solves $\frac{\omega^{n}}{\omega_{0}^{n}} = e^{-\lambda f}e^{\rho}e^{c}$ if and only if $f + \frac{c}{\lambda}$ solves $\frac{\omega^{n}}{\omega_{0}^{n}} = e^{-\lambda f}e^{\rho}$. Hence $\forall \lambda \in \bbR$, we end up with the following equation for $f \in C^{\infty}\paren*{M}$:

\begin{equation*}
\frac{\paren*{\omega_{0} + i\del \delbar f}^{n}}{\omega_{0}^{n}} = e^{-\lambda f}e^{\rho}
\end{equation*}

Therefore we have proved the following 

\begin{prop}\label{K\"{a}hlereinstein}
Let $\paren*{M^{n},J, g_{0},\omega_{0}}$ be a closed \textit{complex} manifold with \textit{fixed} background integrable complex structure $J$ and \textit{arbitrary} K\"{a}hler form $\omega_{0}$ compatible with $J$ (whence yielding a Riemannian metric $g_{0}$).

If $c_{1}\paren*{M^{n},J} = \lambda [\omega_{0}]$ for $\lambda \in \bbR$ with the Ricci potential normalization $ \int_{M} e^{\rho} \omega_{0}^{n} = \int_{M}\omega_{0}^{n}$ (since $c_{1}\paren*{M^{n},J} = \frac{\lambda}{2\pi} [\omega_{0}] \Rightarrow \Ric_{\omega_{0}} = \lambda \omega_{0} + i\del \delbar \rho$), then if $\exists f \in C^{\infty}(M)$ such that \begin{itemize}
\itemsep0em 
\item (\textit{Positivity}): $\omega \coloneq \omega_{0} + i\del \delbar f$ is a positive $(1,1)$-form (that is, $g\paren*{\cdot, \cdot} \coloneq \omega\paren*{\cdot, J\cdot}$ defines a real Riemannian metric)
\item (\textit{Complex Monge-Ampere}): $f$ solves $$\frac{\paren*{\omega_{0} + i\del \delbar f}^{n}}{\omega_{0}^{n}} = e^{-\lambda f}e^{\rho}$$\end{itemize} then $\omega_{0} + i\del \delbar f\eqcolon \omega \in [\omega_{0}]$ (which is automatically closed) is a K\"{a}hler form in the same cohomology class as the original K\"{a}hler metric $\omega_{0}$ yielding a K\"{a}hler-Einstein metric $g \coloneq \omega\paren*{\cdot, J\cdot}$ with Einstein constant $\lambda$ that's compatible with the \textit{fixed} original integrable complex structure $J$.

In fact, further imposing the additional condition\begin{itemize}
\itemsep0em 
\item (\textit{Integral Zero}): $$\int_{M} f \omega_{0}^{n} = 0$$
\end{itemize} makes $f$ \textit{unique}, whence the resulting $\omega \in [\omega_{0}]$ is the \textit{unique} K\"{a}hler form in the starting K\"{a}hler class $[\omega_{0}]$ yielding a K\"{a}hler-Einstein metric with Einstein constant $\lambda$ compatible with $J$.\end{prop}

\begin{remark}\label{autopositivity} The 2nd condition (complex Monge-Ampere) \textit{automatically implies} the 1st condition (positivity). Indeed, being a positive $(1,1)$-form is an \textit{open condition} (a continuous Hermitian matrix-valued function having positive eigenvalues at a point implies having positive eigenvalues on an open neighborhood at that point), whence the 2nd condition implies the 1st via covering $M$ with local coordinate charts, continuity of the expression for $\omega$ whence continuity of all local coordinate functions of $\omega$, and positivity of the RHS of the 2nd condition implying that the \textit{determinant} of the local coordinate component Hermitian matrix of $\omega$ being \textit{positive}.
\end{remark}

The equation $\frac{\paren*{\omega_{0} + i\del \delbar f}^{n}}{\omega_{0}^{n}} = e^{-\lambda f}e^{\rho}$ is a fully nonlinear scalar elliptic equation of complex \textit{Monge-Ampere} type, i.e. in local coordinates the LHS looks like $\det \paren*{\Hess_{\bbC} f}$. Therefore proving the \textit{existence} of a smooth solution $f$ to this equation (which by this proposition implies the existence of K\"{a}hler-Einstein metrics) is quite difficult, and becomes harder as the sign of $\lambda$ changes from negative to $0$ to positive. Aubin proved that a solution $f$ \textit{always} exists for $\lambda < 0$, Yau proved that a solution $f$ \textit{always} exists for the $\lambda = 0$ (i.e. Ricci-flat) case (called the \textbf{Calabi conjecture}), and very recently Chen-Donaldson-Sun \cite{CDS1} \cite{CDS2} \cite{CDS3} solved the $\lambda > 0$ case via proving that existence of KE metrics with $\lambda > 0$ is equivalent to an algebraic-geometric ``K-stability'' condition (since here there are \textit{obstructions} to existence).

Next, we note the following special case of the above (see \cite{Huy}):

\begin{prop}\label{calabiyau} Let $\paren*{M^{n},J, g_{0},\omega_{0}}$ be a fixed closed K\"{a}hler manifold. Suppose $K_{M}$ is trivial, hence there exists a nowhere vanishing \textbf{holomorphic volume form}/nowhere vanishing holomorphic $(n,0)$-form $\Omega$ (an example being the $K3$ surface) and $c_{1}\paren*{M^{n},J} = 0$.

Then if $\exists f \in C^{\infty}(M)$ such that \begin{itemize}
\itemsep0em 
\item (\textit{Positivity}): $\omega \coloneq \omega_{0} + i\del \delbar f$ is a positive $(1,1)$-form
\item (\textit{Complex Monge-Ampere}): $f$ solves $$\frac{\paren*{\omega_{0} + i\del \delbar f}^{n}}{\Omega \wedge \overline{\Omega}} = C$$ for some constant $C > 0 $.\footnote{Most often people work with the convention that $C = n!(-1)^{\frac{n(n-1)}{2}}\paren*{\frac{i}{2}}^{n}$ upon (say) rescaling $\Omega$, see Definition \ref{CY-definition}.}\end{itemize} then $\omega_{0} + i\del \delbar f\eqcolon\omega \in [\omega_{0}]$ is a K\"{a}hler form in the same cohomology class as the original K\"{a}hler metric $\omega_{0}$ yielding a \textbf{Ricci-flat K\"{a}hler} metric $g \coloneq \omega\paren*{\cdot, J\cdot}$ compatible with the \textit{fixed} original integrable complex structure $J$.

Further imposing the additional condition\begin{itemize}
\itemsep0em 
\item (\textit{Integral Zero}): $$\int_{M} f \omega_{0}^{n} = 0$$
\end{itemize} makes $f$ and whence $\omega \in [\omega_{0}]$ the \textit{unique} K\"{a}hler form yielding a Ricci-flat metric $g$ compatible with the original $J$.\end{prop}

Such Ricci-flat metrics which are \ka WRT a \textit{fixed} integrable complex structure form a special class of metrics, namely: \begin{definition}\label{CY-definition}
A \textbf{Calabi-Yau manifold} is a \textit{real smooth} $2n$-dimensional manifold $M^{2n}$ with the following data:\begin{enumerate}
\itemsep0em 
\item An integrable complex structure $J$, whence making $M^{2n}$ into a \textit{complex} manifold $\paren*{M^{n},J}$.
\item $g$ is a Riemannian metric which is \ka WRT $J$ with \ka form $\omega$, whence making $\paren*{M^{n},J}$ a \ka manifold $\paren*{M^{n}, J, g, \omega}$.
\item $\Omega$ a nowhere vanishing holomorphic $\paren*{n,0}$-form WRT $J$.
\end{enumerate} such that one of the following equivalent conditions are satisfied\begin{itemize}
\itemsep0em 
\item $\nabla_{g}\Omega = 0$ and $\Omega$ has the (pointwise) normalization $\abs*{\Omega}_{g} = 2^{\frac{n}{2}}$.
\item The holonomy group $\Hol(g)$ of the Levi-Civita connection of $g$ is contained in $\SU(n)$, and the associated parallel $\paren*{n,0}$-form is $\Omega$ and has normalization $\abs*{\Omega}_{g} = 2^{\frac{n}{2}}$
\item The pieces of data above are compatible in the following way: $$\frac{\omega^{n}}{n!} = (-1)^{\frac{n(n-1)}{2}}\paren*{\frac{i}{2}}^{n} \Omega \wedge \overline{\Omega}$$
\end{itemize}

Thus since this compatibility condition is precisely the complex Monge-Ampere equation, by Proposition \ref{calabiyau} we therefore have that \textit{Calabi-Yau metrics are Ricci-flat}.
\end{definition}

\begin{remark}\label{CY normalization}
The normalization $\abs*{\Omega}_{g} = 2^{\frac{n}{2}}$, which is equivalent to the specifically chosen constant factor $(-1)^{\frac{n(n-1)}{2}}\paren*{\frac{i}{2}}^{n}$ in the complex Monge-Ampere equation in Definition \ref{CY-definition}, is chosen specifically to make $\Re \Omega$ a \textbf{calibration}.
\end{remark}

\begin{notation}\label{Ricci flat Kahler vs CY notation}
Therefore, by \textit{Remark} \ref{CY normalization}, we will oftentimes interchange the terms ``Ricci-flat K\"{a}hler'' and ``Calabi-Yau''. Indeed, the latter is just the former but with a suitable normalization of the nowhere vanishing holomorphic volume form. Moreover, the former may be converted into the latter because \ka Ricci-flat implies holonomy contained in $\SU(n)$, whence by the \textbf{holonomy principle} we get a nowhere vanishing parallel $\paren*{n,0}$-form, which is closed by antisymmetrization of the Levi-Civita connection.
\end{notation}

\begin{remark}\label{SUn structures}
Just like \textit{Remark} \ref{almost Hermitian manifold}, we have that if we weaken Definition \ref{CY-definition} for our manifold $\paren*{M^{2n},J,g,\omega}$ to instead be almost Hermitian, but equipped with a \textit{complex} nowhere vanishing $\paren*{n,0}$-form $\Omega$ with normalization $\abs*{\Omega}_{g} = 2^{\frac{n}{2}}$ and compatibility condition $\frac{\omega^{n}}{n!} = (-1)^{\frac{n(n-1)}{2}}\paren*{\frac{i}{2}}^{n} \Omega \wedge \overline{\Omega}$, then the tuple $\paren*{M^{2n},J,g,\omega, \Omega}$ is called an \textbf{$\SU(n)$-structure on $M^{2n}$}.
\end{remark}

Lastly, we note a linear algebraic fact that will be conceptually useful later on (whose proof follows pointwise in $\bbR^{2n} = \bbC^{n}$):

\begin{prop}\label{holovolformdeterminescomplexstructure}
Let $M^{2n}$ be a smooth manifold. Suppose that $M^{2n}$ admits a nowhere vanishing $(n,0)$-form $\Omega$\footnote{More rigorously, suppose $M^{2n}$ admits an \textbf{$\SLc{n}$-structure} in the language of $G$-structures, with defining tensor $\Omega$.}. Then \begin{itemize}
\itemsep0em 
\item $\Omega$ \textit{uniquely} determines the almost complex structure $J_{\Omega}$ over $M$ via determining the $(1,0)$-forms
$$\Lambda^{1,0}_{\bbC} T^{*}_{p}M\coloneq \ker\paren*{\restr{\alpha}{p}\mapsto \restr{\Omega}{p}\wedge \restr{\alpha}{p}},\qqfa p \in M$$
\item $d\Omega = 0$ \textit{implies} that $J_{\Omega}$ is \textbf{integrable}\footnote{Note that the converse is \textit{not} true.}, whence $\paren*{M^{2n},\Omega}$ is in fact a \textit{complex} manifold $\paren*{M^{n},J_{\Omega}}$ WRT $J_{\Omega}$ and $\Omega$ becomes a \textit{holomorphic} nowhere vanishing $(n,0)$-form WRT $J_{\Omega}$. \end{itemize} Hence a complex manifold $\paren*{M^{n},J}$ admitting a nowhere vanishing \textit{holomorphic} volume form $\Omega$ automatically has trivial canonical bundle $K_{M}$ (as the global frame of the complex line bundle $K_{M}$ is literally given by $\Omega$).\end{prop}

\subsection{\textsection \ Eguchi-Hanson Metrics \& Flat Tori as ``Building Blocks'' \ \textsection}\label{EH as building block}

In this section, we first define \& state several important properties of the Eguchi-Hanson metric, and then recall some facts about flat tori.

First, we make the following observation which will be crucially used numerous times in the sequel.

\begin{prop}\label{CoveringSpacesInvariantForms}
Let $M\Ractson G$ be a free and proper action of a \textit{finite} group $G$ on a smooth manifold $M$, whence $\pi: M \onto M/G$ is a principal $G$-bundle with the base space $M/G$ a smooth manifold and hence is a (regular) $\abs*{G}$-fold covering. Then any smooth function (tensor, section, etc.) $f$ on $M/G$ may be lifted to a smooth $G$-invariant function (tensor, section, etc.) $\pi^{*}f$ on $M$.

Conversely, any smooth $G$-invariant function (tensor, section, etc.) on $M$ descends down to a smooth function (tensor, section, etc.) on $M/G$.

In particular, when $G = \bbZ_{2}$ and $M = \bbR^{n}$ or $M = \bbT^{n}$ and the action sends $1 \mapsto \id_{M}$ and $-1\mapsto -\id_{M}$ (aka via the standard coordinate involution $x\mapsto -x$), then any smooth function (tensor, section, etc.) on the \textit{free} part $\bbR^{n}/\bbZ_{2}- \set*{0}$ or $\bbT^{n}/\bbZ_{2}- \Fix\paren*{\bbZ_{2}}$ may be lifted to a smooth \textit{even} function (tensor, section, etc.) on $\bbR^{n}$ or $\bbT^{n}$. That is, the lifted $\wtilde{f}$ satisfies $\wtilde{f}(x) = \wtilde{f}(-x)$.

Conversely, any smooth even function (tensor, section, etc.) on $\bbR^{n}$ or $\bbT^{n}$ descends down to a smooth function (tensor, section, etc.) on the free part $\bbR^{n}/\bbZ_{2}- \set*{0}$ or $\bbT^{n}/\bbZ_{2}- \Fix\paren*{\bbZ_{2}}$.\end{prop}

Start with Euclidean $\bbR^{4}$. We may equip $\bbR^{4}$ with the Euclidean integrable complex structure $J_{0}$ coming from the tautological identification $\bbR^{4} = \bbC^{2}$.

Let's now look at the complex orbifold $\paren*{\bbC^{2}/\bbZ_{2}, J_{0}}$, or $\paren*{\bbC^{2},J_{0}}$ mod the action $\bbZ_{2} \actson \bbC^{2}$ via $1 \mapsto \id_{\bbC^{2}}, -1 \mapsto -\id_{\bbC^{2}}$, which is clearly holomorphic WRT $J_{0}$. Whence $\bbC^{2} \onto \bbC^{2}/\bbZ_{2}$ is a double cover outside of the orbifold singularity at the origin, i.e. $\bbC^{2} - \set*{0} \overset{2:1}{\onto} \bbC^{2}/\bbZ_{2}- \set*{0}$. Blow up the singularity at the origin with a \textbf{crepant resolution} (i.e. preserving the canonical bundle, i.e. pullback of canonical bundle downstairs is canonical bundle upstairs) to get $\paren*{\wtilde{\bbC^{2}/\bbZ_{2}}, J_{\wtilde{\bbC^{2}/\bbZ_{2}}}} \overset{\text{Biholo}}{\cong} \paren*{T^{*} \bbC P^{1}, J_{\cO(-2)}} = \cO(-2) = \paren*{T^{*}S^{2}, J_{\cO(-2)}}$\footnote{Since the smooth manifold $\paren*{T^{*}S^{2}, J_{\cO(-2)}}$ with that given integrable complex structure is precisely the complex manifold $T^{*}\CP^{1}$ with that same complex structure.} and a map $\pi: \paren*{T^{*} \bbC P^{1}, J_{\cO(-2)}} \onto \paren*{\bbC^{2}/\bbZ_{2}, J_{0}}$ which is a biholomorphism outside the origin, i.e. $\pi: T^{*} \bbC P^{1} - \bbC P^{1} \overset{\text{Biholo}}{\cong} \bbC^{2}/\bbZ_{2}- \set*{0}$ since $\pi^{-1}\paren*{0} \cong \bbC P^{1}$ is the exceptional divisor/zero section.

Therefore we have that $dz^{1} \wedge dz^{2}$, the standard nowhere vanishing $(2,0)$-form/nowhere vanishing \textit{holomorphic} volume form on $\paren*{\bbC^{2}, J_{0}}$, gets preserved under $\bbZ_{2}$ (as it's clearly even) and descends down to a nowhere vanishing holomorphic $(2,0)$-form on $\paren*{\bbC^{2}/\bbZ_{2}, J_{0}}$ (it extends over the origin by Hartog), and gets pulled back to a nowhere vanishing holomorphic $(2,0)$-form $\Omega_{T^{*} \bbC P^{1}}$ on $\paren*{T^{*} \bbC P^{1}, J_{\cO(-2)}}$ since the resolution $\pi: \paren*{T^{*} \bbC P^{1}, J_{\cO(-2)}} \onto \paren*{\bbC^{2}/\bbZ_{2}, J_{0}}$ is crepant. Thus $\Omega_{T^{*} \bbC P^{1}}$ by Proposition \ref{holovolformdeterminescomplexstructure} uniquely determines the (integrable) complex structure $J_{\cO(-2)}$ on $T^{*} \bbC P^{1}$ (since $d\Omega_{T^{*} \bbC P^{1}} = 0$).

Now define the smooth function $\phi_{0}: \paren*{0,\infty}\rightarrow \bbR$ via $\phi_{0}\paren*{r} \coloneq \frac{r^{2}}{2}$. Then setting $r\mapsto \abs*{\cdot}^{g_{0}}_{0} \coloneq \paren*{\abs*{z^{1}}^{2} + \abs*{z^{2}}^{2}}^{\frac{1}{2}}$ (the standard absolute value on $\bbC^{2}$), we have that $i\del\delbar \phi_{0} \eqcolon \omega_{0}$ gives us the standard \textit{Euclidean} flat \ka metric, where of course $g_{0}$ is the \textit{Euclidean} flat metric on $\bbR^{4} = \bbC^{2}$. Repeating the same arguments as above, we have this time around that $g_{0}, \omega_{0}$ gets pulled back up and smoothly defined on $T^{*}\CP^{1} - \CP^{1}$. Whence we may speak of $\abs*{\cdot}^{g_{0}}_{\CP^{1}} \coloneq \dist_{g_{0}}\paren*{\CP^{1},\cdot}$, the distance function away from the exceptional divisor/zero section \textit{measured WRT $g_{0}$}.

Now define the smooth function $\phi_{FS}: \paren*{0,\infty} \rightarrow \bbR$ via $\phi_{FS}\paren*{r} \coloneq \log r = \frac{1}{2}\log r^{2}$. Setting $\phi_{FS} = \phi_{FS}\paren*{\abs*{\cdot}^{g_{0}}_{0}}$, this function clearly descends down the quotient $\bbC^{2}-\set*{0} \onto \CP^{1}$, and $i\del\delbar \phi_{FS} \eqcolon \omega_{FS}$ yields us the usual \textbf{Fubini-Study} \ka metric $g_{FS}, \omega_{FS}$ on $\paren*{\CP^{1}, J_{\CP^{1}}}$.

Let $s > 0$. Now define the smooth function $\phi_{EH,s}: \paren*{0,\infty}\rightarrow \bbR$ via $$\phi_{EH,s}\paren*{r} \coloneq \frac{1}{2}\paren*{s \log r^{2} + \paren*{s^{2}+ r^{4}}^{\frac{1}{2}} - s\log\paren*{s+ \paren*{s^{2}+r^{4}}^{\frac{1}{2}}} }$$

Setting $r\mapsto \abs*{\cdot}^{g_{0}}_{0}$ but this time noting that $\paren*{\abs*{\cdot}^{g_{0}}_{0}}^{4}$ is \textit{holomorphic} (since $\paren*{\abs*{\paren*{z^{1},z^{2}}}^{g_{0}}_{0}}^{4} \coloneq \abs*{\paren*{z^{1},z^{2}}}^{4} = \abs*{z^{1}}^{4} + 2\abs*{z^{1}z^{2}}^{2} + \abs*{z^{2}}^{4}$ clearly satisfies $\delbar \abs*{z}^{4} = 0$ by the chain rule), we have that upon repeating the same arguments above (namely $\bbC^{2} \onto \bbC^{2}/\bbZ_{2} - \set*{0}$, then extending to $\bbC^{2}/\bbZ_{2}$ by Hartog, then lifting through the resolution $T^{*}\CP^{1}$) that $$\omega_{EH,s} \coloneq i\del\delbar \phi_{EH,s} = s\pi^{*}\paren*{\omega_{FS}} + i\del\delbar\paren*{\frac{1}{2}\paren*{\paren*{s^{2}+ r^{4}}^{\frac{1}{2}} - s\log\paren*{s+ \paren*{s^{2}+r^{4}}^{\frac{1}{2}}} }}$$ is a \ka form defined \textit{everywhere} on the total space of $\pi: \paren*{T^{*}\CP^{1}, J_{\cO\paren*{-2}}} \onto \paren*{\CP^{1}, J_{\CP^{1}}}$ with associated Riemannian metric $g_{EH,s}$, and where $r \coloneq \abs*{\cdot}^{g_{0}}_{\CP^{1}}$. For $\lambda \in \bbC^{\times}$, upon denoting by $R_{\lambda} \in \SU(2)$ the scaling action given by $R_{\lambda} \coloneq \lambda \id_{\bbC^{2}}$ sending $\paren*{z^{1},z^{2}}\mapsto \paren*{\lambda z^{2},\lambda z^{2}}$, which clearly descends down $\bbC^{2} - \set*{0} \onto \CP^{1}$ and lifts to a global $\bbC^{\times}$ action on $T^{*}\CP^{1}$ (hence $R_{\lambda} \in \Diff\paren*{T^{*}\CP^{1}}$), upon observing that $\phi_{EH,s} = s R_{\frac{1}{s^{\frac{1}{2}}}}^{*}\paren*{\phi_{EH,1}}$, we have $\forall s > 0$ that $$R_{s^{\frac{1}{2}}}^{*}\paren*{\frac{1}{s}g_{EH,s}} = g_{EH,1}$$

Hence, \textit{up to scaling} there is one \ka metric $g_{EH,1}, \omega_{EH,1}$ defined on $\paren*{T^{*}\CP^{1}, J_{\cO\paren*{-2}}}$, which we call the \textbf{Eguchi-Hanson metric}. We now summarize \& state the important properties of this metric in the following \begin{prop}\label{EHproperties} Let $s > 0$ be a positive real number, and let $g_{EH,s},\omega_{EH,s}$ be as above.\begin{itemize}
\itemsep0em 
\item $\paren*{T^{*}\CP^{1},J_{\cO\paren*{-2}}, g_{EH,s}, \omega_{EH,s}}$ is a complete Ricci-flat Riemannian $4$-manifold which is an \textbf{ALE\footnote{Asymptotically Locally Euclidean.} Calabi-Yau\footnote{say, after appropriate normalization of $\Omega_{T^{*} \bbC P^{1}}$} gravitational instanton\footnote{A complete non-compact Riemannian manifold with $\Abs*{\operatorname{Rm}_{g}}_{L^{2}} < \infty$.}} asymptotic at $\infty$ to $\bbR^{4}/\bbZ_{2}$ with rate $-4 < 0$. Hence $\exists K \subset T^{*}\CP^{1}$ a compact subset, a constant $R_{EH} > 0$ \textit{sufficiently large}, and a diffeomorphism $F: \bbR^{4}/\bbZ_{2} - D^{g_{0}}_{R_{EH}}\paren*{0} \rightarrow T^{*}\CP^{1} - K$ such that $$\abs*{\nabla_{g_{0}}^{k}\paren*{F^{*}g_{EH,s} - g_{0}}}_{g_{0}} = O\paren*{\paren*{\abs*{\cdot}^{g_{0}}_{0}}^{-4-k}}\qquad\text{and}\qquad \abs*{\nabla_{g_{0}}^{k}\paren*{F^{*}\omega_{EH,s} - \omega_{0}}}_{g_{0}} = O\paren*{\paren*{\abs*{\cdot}^{g_{0}}_{0}}^{-4-k}}$$

on $\bbR^{4}/\bbZ_{2} - D^{g_{0}}_{R_{EH}}\paren*{0} = \bbC^{2}/\bbZ_{2} - D^{g_{0}}_{R_{EH}}\paren*{0}$. In fact, this follows from the Taylor expansion $\abs*{\nabla^{k}_{g_{0}}\paren*{\phi_{EH,s}  - \phi_{0}}}_{g_{0}} \leq c_{0}(k) r^{-2-k}$ for $r > 0$ sufficiently large ($r>R_{EH}$) and $c_{0}(k) > 0, k \in \bbN_{0}$ positive constants.

\item In fact, this diffeomorphism outside a compact subset comes from $\pi: T^{*} \bbC P^{1} - \bbC P^{1} \overset{\text{Biholo}}{\cong} \frac{\bbC^{2} - \set*{0}}{\bbZ_{2}}$, i.e. $\pi: T^{*} \bbC P^{1} - D^{g_{0}}_{R_{EH}}\paren*{\bbC P^{1}} \overset{\text{Biholo}}{\cong} \bbC^{2}/\bbZ_{2} - D^{g_{0}}_{R_{EH}}(0)$.

\item From the $s \pi^{*}\paren*{\omega_{FS}}$ term in $\omega_{EH,s}$ and the fact that the Fubini-Study metric on $\CP^{1}$ is isometric to the round metric of $S^{2}\paren*{\frac{1}{2}}$, we have that the exceptional divisor/zero section $\CP^{1}\subset \paren*{T^{*}\CP^{1}, g_{EH,s}}$\footnote{Also called the \textbf{bolt} of $T^{*}\CP^{1}$.} has sectional curvature $\frac{4}{s}$, volume $s\pi$, diameter $s^{\frac{1}{2}}$, injectivity radius $O\paren*{s^{\frac{1}{2}}}$, and self intersection number $\CP^{1} \cdot \CP^{1} = -2$.\footnote{This is because the Euler characteristic of $S^{2}$ is $2$, $T^{*}\CP^{1}$ has the \textit{opposite} orientation of $T \CP^{1}$, and the Euler class of a (real) vector bundle is the self-intersection number of the zero section. Alternatively, this follows because of the self-intersection and $c_{1}\paren*{L^{*}} = -c_{1}\paren*{L}$ and $c_{1}\paren*{L} = e\paren*{L_{\bbR}}$ for complex line bundles $L$.}

In particular, as $s\to 0$, the curvature of the exceptional $\CP^{1}$ blows up, but its diameter/volume shrinks.

\item The Eguchi-Hanson metric $\paren*{T^{*}\CP^{1},J_{\cO\paren*{-2}}, g_{EH,s}, \omega_{EH,s}}$ is \textbf{$\U(2)$-invariant}. In particular, we immediately have that the Eguchi-Hanson metric is $\U(1)$-invariant\footnote{This is easily seen because of how $\bbC^{\times} \Lactson T^{*}\CP^{1}$. Namely, it starts via $R_{\lambda}$ on $\bbC^{2}$, which clearly descends down $\bbC^{2} - \set*{0}\onto \CP^{1}$ and lifts to $T^{*}\CP^{1}$. However, if $\lambda \in \U(1)$, then $\abs*{\lambda} = 1$, whence by definition of $\phi_{EH,s}$ being inputted with $\abs*{\cdot}^{g_{0}}_{\CP^{1}}$, $\phi_{EH,s}$ is clearly $\U(1)$-invariant and hence so are $\omega_{EH,s}$ and $g_{EH,s}$.}, i.e. $$R^{*}_{\lambda} \paren*{g_{EH,s}} = g_{EH,s},\qqfa \lambda \in \U(1) \subset \bbC^{\times}$$

\item $g_{EH,s}$ and $g_{EH,1}$ (hence similarly for $\omega_{EH,s}$, $\omega_{EH,1}$) are isometric up to scaling, i.e. $$R_{s^{\frac{1}{2}}}^{*}\paren*{\frac{1}{s}g_{EH,s}} = g_{EH,1}$$

\end{itemize}\end{prop}

We now want to ``prepare'' the Eguchi-Hanson metric by interpolating it with the Euclidean metric on an ``annulus'' via a suitable cutoff function on $T^{*} \bbC P^{1}$.

\begin{data}\label{constants C_{1} and c_{chi}(k) from cutoff}
Fix a constant $C_{1} > 1$. Define a smooth monotone cutoff function $\chi: \bbR_{\geq 0} \rightarrow [0,1]$ satisfying $\chi(x) = \begin{cases}
1 & 0 \leq  x \leq 1 \\
0 & 2 \leq x
\end{cases},\abs*{\nabla\chi}\leq C_{1}$ on $\set*{1\leq x\leq 2}$, hence $\chi$ satisfies $\supp \nabla^{k}\chi \subseteq \set*{1\leq x \leq 2}$ and $\abs*{\nabla^{k}\chi} \leq c_{\chi}(k)$ for some positive constants $c_{\chi}(k) > 0$ where $c_{\chi}(0) = 1$ and $c_{\chi}(1) = C_{1}$. WLOG let us also denote by $\chi$ the smooth monotone function on $T^{*}\CP^{1}$ by $\chi \coloneq \chi \circ \abs*{\cdot}^{g_{0}}_{\CP^{1}}$. Hence $\chi = \begin{cases}
1 & \set*{\abs*{\cdot}^{g_{0}}_{\CP^{1}} \leq 1}\\
0 & \set*{2 \leq \abs*{\cdot}^{g_{0}}_{\CP^{1}}}
\end{cases}$, $\supp \nabla^{k}\chi_{\epsilon} \subseteq \set*{1 \leq \abs*{\cdot}^{g_{0}}_{\CP^{1}} \leq 2}$, and $\abs*{\nabla^{k}\chi_{\epsilon}}_{g_{0}} \leq c_{\chi}(k)$.
\end{data} 

\begin{remark}\label{cutoff function dependence}
The resulting Ricci-flat metric which will be constructed in fact does \textit{not} depend on this cutoff function $\chi$. Indeed, though the \textit{initial} \ka form and metric that we will perturb from does depend on $\chi$, its \textit{cohomology class} is independent of $\chi$, and the resulting Ricci-flat \ka form is the \textit{unique} such form in the cohomology class of the initial \ka form.
\end{remark}

\begin{data}\label{constant epsilon and lambda introduction} Throughout this paper, let $\boxed{\epsilon > 0}$ be a positive number, the \textbf{gluing parameter}, and let $\boxed{0 < \lambda < 1}$ be a number between 0 and 1, the \textbf{gluing rate}.\footnote{Parameters feature quite prominently in gluing constructions, and this $\epsilon>0$ and $\lambda$ will be referred back to throughout the rest of this paper.}.
\end{data}

Let us first define $\chi_{\sigma} \coloneq \chi\paren*{\sigma \abs*{\cdot}^{g_{0}}_{\CP^{1}}}$, aka $\chi_{\sigma} \coloneq R^{*}_{\sigma} \chi$, given any $\sigma > 0$. Hence $\chi_{\sigma} = \begin{cases}
1 & \set*{\abs*{\cdot}^{g_{0}}_{\CP^{1}} \leq \frac{1}{\sigma}}\\
0 & \set*{\frac{2}{\sigma} \leq \abs*{\cdot}^{g_{0}}_{\CP^{1}}}
\end{cases}$, $\supp \nabla^{k}\chi_{\sigma} \subseteq \set*{\frac{1}{\sigma} \leq \abs*{\cdot}^{g_{0}}_{\CP^{1}} \leq \frac{2}{\sigma}}$, and $\abs*{\nabla^{k}\chi_{\sigma}}_{g_{0}} \leq c_{\chi}(k)\sigma^{k}$.

Define $\lwhat{\phi_{EH-0,\epsilon}} \coloneq \chi_{\epsilon^{\lambda}} \phi_{EH,1}+ \paren*{1 - \chi_{\epsilon^{\lambda}}}\phi_{0}$, and define $\lwhat{\omega_{EH-0,\epsilon}}\coloneq i\del\delbar\paren*{\lwhat{\phi_{EH-0,\epsilon}}}$. Immediate from the definitions and the fact that being a positive $\paren*{1,1}$-form is open (\textit{Remark} \ref{autopositivity}), and the Leibniz rule is \begin{prop}\label{preglueEH BG interpolation!} For $\epsilon > 0$ from Data \ref{constant epsilon and lambda introduction} \textit{sufficiently small} so that $\lwhat{\omega_{EH-0, \epsilon}}$ is a positive $\paren*{1,1}$-form, we have that $\lwhat{\omega_{EH-0, \epsilon}}$ is a \ka form on $\paren*{T^{*}\CP^{1}, J_{\cO\paren*{-2}}}$ compatible with $J_{\cO\paren*{-2}}$. Hence $\paren*{T^{*}\CP^{1}, J_{\cO\paren*{-2}}, \lwhat{\omega_{EH-0, \epsilon}}, \lwhat{g_{EH-0, \epsilon}}}$ is a \ka manifold such that $$\lwhat{\omega_{EH-0, \epsilon}} = \begin{cases}
\omega_{EH,1} & \text{on } \set*{\abs*{\cdot}^{g_{0}}_{\CP^{1}} \leq \frac{1}{\epsilon^{\lambda}} } \\
\omega_{0} & \text{on }\set*{\frac{2}{\epsilon^{\lambda}} \leq \abs*{\cdot}^{g_{0}}_{\CP^{1}}}
\end{cases}$$ By 2-out-of-3, all of the above hold verbatim for the associated Riemannian metric $\lwhat{g_{EH-0,\epsilon}}$. Moreover, $$\abs*{\nabla_{g_{0}}^{k}\paren*{\lwhat{g_{EH-0, \epsilon}} - g_{0}}}_{g_{0}} \leq c_{2}(k) \epsilon^{\paren*{4+k}\lambda}\qquad\text{and}\qquad\abs*{\nabla_{g_{0}}^{k}\paren*{\lwhat{\omega_{EH-0, \epsilon}} - \omega_{0}}}_{g_{0}} \leq c_{1}(k+2) \epsilon^{\paren*{4+k}\lambda}$$ on the annuli $\set*{\frac{1}{\epsilon^{\lambda}} \leq \abs*{\cdot}^{g_{0}}_{\CP^{1}} \leq \frac{2}{\epsilon^{\lambda}}}$ when $\epsilon>0$ is \textit{sufficiently small} (specifically $R_{EH} < \frac{1}{\epsilon^{\lambda}} \Leftrightarrow \epsilon < R_{EH}^{-\frac{1}{\lambda}}$), for some positive constants $c_{1}\paren*{k} > 0, k \in \bbN_{0}$ in which $\abs*{\nabla_{g_{0}}^{k}\paren*{\lwhat{\phi_{EH-0,\epsilon}} - \phi_{0}}}_{g_{0}} \leq c_{1}(k) \epsilon^{\paren*{2+k}\lambda}$, and for some positive constants $c_{2}(k) > 0, k \in \bbN_{0}$.
\end{prop}

Define $\chi_{\epsilon} \coloneq \chi\paren*{\frac{\abs*{\cdot}^{g_{0}}_{\CP^{1}}}{\epsilon^{\lambda}}} = R^{*}_{\frac{1}{\epsilon^{\lambda}}}\chi \eqcolon \chi_{\epsilon^{\lambda}}$. Hence $\chi_{\epsilon} = \begin{cases}
1 & \set*{\abs*{\cdot}^{g_{0}}_{\CP^{1}} \leq \epsilon^{\lambda}}\\
0 & \set*{2\epsilon^{\lambda} \leq \abs*{\cdot}^{g_{0}}_{\CP^{1}}}
\end{cases}$, $\supp \nabla^{k}\chi_{\epsilon} \subseteq \set*{\epsilon^{\lambda} \leq \abs*{\cdot}^{g_{0}}_{\CP^{1}} \leq 2\epsilon^{\lambda}}$, and $\abs*{\nabla^{k}\chi_{\epsilon}}_{g_{0}} \leq c_{\chi}(k)\frac{1}{\epsilon^{k\lambda }}$. Now define $\wtilde{\phi_{EH, \epsilon}} \coloneq \chi_{\epsilon} \phi_{EH,\epsilon^{2}} + \paren*{1-\chi_{\epsilon}} \phi_{0}$. Hence $\wtilde{\phi_{EH,\epsilon}} = \epsilon^{2}\chi\paren*{\frac{\abs*{\cdot}^{g_{0}}_{\CP^{1}}}{\epsilon^{\lambda}}}\phi_{EH,1}\paren*{\frac{\abs*{\cdot}^{g_{0}}_{\CP^{1}}}{\epsilon}} + \epsilon^{2}\paren*{1-\chi\paren*{\frac{\abs*{\cdot}^{g_{0}}_{\CP^{1}}}{\epsilon^{\lambda}}}}\phi_{0}\paren*{\frac{\abs*{\cdot}^{g_{0}}_{\CP^{1}}}{\epsilon}} = \epsilon^{2}\chi\paren*{\frac{\abs*{\cdot}^{g_{0}}_{\CP^{1}}}{\epsilon^{\lambda}}}\phi_{EH,1}\paren*{\frac{\abs*{\cdot}^{g_{0}}_{\CP^{1}}}{\epsilon}} + \paren*{1-\chi\paren*{\frac{\abs*{\cdot}^{g_{0}}_{\CP^{1}}}{\epsilon^{\lambda}}}}\phi_{0}\paren*{\abs*{\cdot}^{g_{0}}_{\CP^{1}}}$ by homogeneity of $\phi_{0}$.

Now define $\wtilde{\omega_{EH, \epsilon}} \coloneq i\del\delbar \paren*{\wtilde{\phi_{EH, \epsilon}}}$. Immediate from the definitions, openness of positive $\paren*{1,1}$-forms, Leibniz rule, the identity $\phi_{EH,\epsilon^{2}} = \epsilon^{2} R^{*}_{\frac{1}{\epsilon}}\paren*{\phi_{EH,1}}$, as well as a local ON frame computation is \begin{prop}\label{preglueEH} For $\epsilon > 0$ from Data \ref{constant epsilon and lambda introduction} \textit{sufficiently small} so that $\wtilde{\omega_{EH, \epsilon}}$ is a positive $\paren*{1,1}$-form and for gluing rates which satisfy $\lambda < \frac{2}{3}$, we have that $\wtilde{\omega_{EH, \epsilon}}$ is a \ka form on $\paren*{T^{*}\CP^{1}, J_{\cO\paren*{-2}}}$ compatible with $J_{\cO\paren*{-2}}$. Hence $\paren*{T^{*}\CP^{1}, J_{\cO\paren*{-2}}, \wtilde{\omega_{EH, \epsilon}}, \wtilde{g_{EH, \epsilon}}}$ is a \ka manifold such that\begin{itemize}
\itemsep0em 
\item $$\wtilde{\omega_{EH, \epsilon}} = \begin{cases}
\omega_{EH,\epsilon^{2}} & \text{on } \set*{\abs*{\cdot}^{g_{0}}_{\CP^{1}} \leq \epsilon^{\lambda}} \\
\omega_{0} & \text{on }\set*{2\epsilon^{\lambda} \leq \abs*{\cdot}^{g_{0}}_{\CP^{1}}}
\end{cases}\qquad\text{and}\qquad \wtilde{g_{EH, \epsilon}} = \begin{cases}
g_{EH,\epsilon^{2}} & \text{on } \set*{\abs*{\cdot}^{g_{0}}_{\CP^{1}} \leq \epsilon^{\lambda}} \\
g_{0} & \text{on }\set*{2\epsilon^{\lambda} \leq \abs*{\cdot}^{g_{0}}_{\CP^{1}}}
\end{cases}$$
\item $$\abs*{\nabla_{g_{0}}^{k}\paren*{\wtilde{g_{EH, \epsilon}} - g_{0}}}_{g_{0}} \leq c_{2}(k) \epsilon^{4-\paren*{4+k}\lambda}\qquad\text{and}\qquad\abs*{\nabla_{g_{0}}^{k}\paren*{\wtilde{\omega_{EH, \epsilon}} - \omega_{0}}}_{g_{0}} \leq c_{1}(k+2) \epsilon^{4-\paren*{4+k}\lambda}$$ on the annuli $\set*{\epsilon^{\lambda} \leq \abs*{\cdot}^{g_{0}}_{\CP^{1}} \leq 2\epsilon^{\lambda}} = \set*{\frac{1}{\epsilon^{1-\lambda}} \leq \frac{\abs*{\cdot}^{g_{0}}_{\CP^{1}}}{\epsilon} \leq \frac{2}{\epsilon^{1-\lambda}}}$ when $\epsilon>0$ is \textit{sufficiently small} (specifically $R_{EH} < \frac{1}{\epsilon^{1-\lambda}} \Longleftrightarrow \epsilon < R_{EH}^{-\frac{1}{1-\lambda}}$) in which $\abs*{\nabla_{g_{0}}^{k}\paren*{\wtilde{\phi_{EH, \epsilon}} - \phi_{0}}}_{g_{0}} \leq c_{1}(k) \epsilon^{4-\paren*{2+k}\lambda}$, for the same positive constants $c_{1}\paren*{k} > 0, k \in \bbN_{0}$ and $c_{2}(k) > 0, k \in \bbN_{0}$ from Proposition \ref{preglueEH BG interpolation!}.
\item $$\wtilde{\omega_{EH,\epsilon}} = \epsilon^{2} R^{*}_{\frac{1}{\epsilon}} \paren*{\lwhat{\omega_{EH-0, \epsilon}}}\qquad\text{and}\qquad\wtilde{g_{EH,\epsilon}} = \epsilon^{2} R^{*}_{\frac{1}{\epsilon}} \paren*{\lwhat{g_{EH-0, \epsilon}}}$$

hence $$\wtilde{\omega_{EH, \epsilon}} = \begin{cases}
\epsilon^{2} R^{*}_{\frac{1}{\epsilon}} \paren*{\omega_{EH,1}} & \text{on } \set*{\abs*{\cdot}^{g_{0}}_{\CP^{1}} \leq \epsilon^{\lambda}} \\
\omega_{0} = \epsilon^{2} R^{*}_{\frac{1}{\epsilon}} \paren*{\omega_{0}} & \text{on }\set*{2\epsilon^{\lambda} \leq \abs*{\cdot}^{g_{0}}_{\CP^{1}}}
\end{cases}\qquad\text{and}\qquad\wtilde{g_{EH, \epsilon}} = \begin{cases}
\epsilon^{2} R^{*}_{\frac{1}{\epsilon}} \paren*{g_{EH,1}} & \text{on } \set*{\abs*{\cdot}^{g_{0}}_{\CP^{1}} \leq \epsilon^{\lambda}} \\
g_{0} = \epsilon^{2} R^{*}_{\frac{1}{\epsilon}} \paren*{g_{0}} & \text{on }\set*{2\epsilon^{\lambda} \leq \abs*{\cdot}^{g_{0}}_{\CP^{1}}}
\end{cases}$$
\end{itemize} Thus we also have, \textit{for $\epsilon < 1$}, $$\dV_{\wtilde{g_{EH, \epsilon}}} = \begin{cases}
\dV_{g_{EH,\epsilon^{2}}} & \text{on } \set*{\abs*{\cdot}^{g_{0}}_{\CP^{1}} \leq \epsilon^{\lambda}} \\
\paren*{1 + O\paren*{\epsilon^{4-4\lambda}}}\dV_{g_{0}} & \text{on }\set*{\epsilon^{\lambda} \leq \abs*{\cdot}^{g_{0}}_{\CP^{1}}}
\end{cases}$$
\end{prop}

\begin{remark}\label{lambda < 2/3 bound} The reason why we now impose $\lambda < \frac{2}{3}$ on the gluing rate in Proposition \ref{preglueEH} is to ensure that the curvature of $\wtilde{g_{EH,\epsilon}}$ stays \textit{bounded} in the region $\set*{\epsilon^{\lambda} \leq \abs*{\cdot}^{g_{0}}_{\CP^{1}} \leq 2\epsilon^{\lambda}} = \set*{\frac{1}{\epsilon^{1-\lambda}} \leq \frac{\abs*{\cdot}^{g_{0}}_{\CP^{1}}}{\epsilon} \leq \frac{2}{\epsilon^{1-\lambda}}}$. Indeed, since the \ka potentials on that annulus satisfy $\abs*{\nabla_{g_{0}}^{k}\paren*{\wtilde{\phi_{EH, \epsilon}} - \phi_{0}}}_{g_{0}} \leq c_{1}(k) \epsilon^{4 - \paren*{2+k}\lambda}$, we need $4$ \ka potential derivatives to get to 2 derivatives on the metric, whence for $k\leq 4$ we thus want $4 - \paren*{2+k}\lambda > 0 \Longleftrightarrow \frac{4}{2+k} > \lambda$, whence $k=4$ yields the sharp inequality $\lambda < \frac{4}{6} = \frac{2}{3}$.\end{remark}

Lastly, we recall a few very standard facts about flat tori. \begin{prop}\label{flat tori facts} Let $\Lambda \subset \bbR^{n}$ be a full rank lattice, hence is equipped with a \textit{choice} of isomorphism $\Lambda \cong \bbZ^{n}$. Consider the $n$-torus $\bbT^{n} \coloneq \bbR^{n}/\Lambda$. \begin{itemize}
\itemsep0em 
\item $\paren*{\bbT^{n},g_{0}}$ is a flat Riemannian manifold (since the \textit{Euclidean} flat metric $g_{0}$ is clearly $\bbZ^{n}$-invariant)
\item The \textbf{moduli space} of all flat metrics on $\bbT^{n}$ is $\GL\paren*{n,\bbZ}\backslash \GL\paren*{n,\bbR}/\O(n)$, which has dimension $\frac{n\paren*{n+1}}{2}$.
\item The same holds for orbifolds $\bbT^{4}/G$ for $G$ a finite group acting isometrically on $\bbT^{4}$.
\end{itemize} Indeed, $\GL\paren*{n,\bbZ}\backslash \GL\paren*{n,\bbR}$ parametrizes the choice of isomorphism $\Lambda \cong \bbZ^{n}$, and $\O(n)$ is precisely the stablizer in $\GLr{n}$ of the Euclidean metric $g_{0}$.\end{prop}

\subsection{\textsection \ The ``Kummer Construction'' \ \textsection}\label{The Kummer Construction}

Recall that a $K3$ surface is the underlying smooth oriented 4-manifold of a connected compact simply connected complex 2-manifold with trivial canonical bundle. Since any two $K3$ surfaces are automatically diffeomorphic, differing constructions of any complex $2$-fold that are connected, compact, simply connected, and have trivial canonical bundle all yield the same $K3$ surface, up to diffeomorphism. We rigorously construct here one such definition of a $K3$ surface, namely as a \textbf{Kummer surface $\Km$}, which is the resolution of an abelian variety quotiented out by an involution (which is closed, connected, simply connected and has trivial canonical bundle), and then graft in our Eguchi-Hanson metric $\wtilde{\omega_{EH, \epsilon}}$ from the previous section to get a family of K\"{a}hler metrics $\omega_{\epsilon}$ on $\sK$.

First consider the $4$-torus $\bbT^{4} = \bbR^{4}/\Lambda$ where $\Lambda \cong \bbZ^{4}$ is a full rank lattice in $\bbR^{4}$. Equip $\bbR^{4}$ with the standard Euclidean \ka structure $J_{0}, g_{0}, \omega_{0}$, getting $\paren*{\bbC^{2},J_{0}, g_{0}, \omega_{0}}$. Each of the 3 tensors $J_{0}, g_{0}, \omega_{0}$ are $\bbZ^{n}$-invariant, hence by Proposition \ref{CoveringSpacesInvariantForms} descends down to $\bbT^{4}$ yielding a \textbf{flat \ka $4$-torus $\paren*{\bbT^{4}, J_{0}, g_{0}, \omega_{0}}$}. Furthermore, the flat \ka datum $J_{0}, g_{0}, \omega_{0}$ are all $\bbZ_{2}$-invariant under the standard involution/$\bbZ_{2}$-action $z \mapsto -z$ on $\bbT^{4} = \bbC^{2}/\Lambda$. In fact, the same argument above works verbatim on $dz^{1} \wedge dz^{2}$. Hence, upon passing to the quotient orbifold $\bbT^{4}/\bbZ_{2}$, we in fact get a \textbf{flat \ka orbifold $\paren*{\bbT^{4}/\bbZ_{2}, J_{0}, g_{0}, \omega_{0}}$} with 16 isolated orbifold point-singularities, each modeled on $\bbR^{4}/\bbZ_{2} = \bbC^{2}/\bbZ_{2}$\footnote{Note that topologically we have that $\bbR^{4}/\bbZ_{2} \cong \Cone\paren*{\bbR P^{3}}$, and it is precisely the fact that the metric on $\Cone\paren*{\bbR P^{3}}$ is conformally equivalent to the metric on a cylinder over $\bbR P^{3}$ that Donaldson exploits in \cite{DonaldsonKummer} to bypass the usage of weighted H\"{o}lder spaces. Part of the reason why this works is because the cylindrical metric of $\bbR P^{3}$ has \textit{bounded geometry} as opposed to the conical metric with its cone singularity.}, and that the underlying \textbf{complex orbifold $\paren*{\bbT^{4}/\bbZ_{2}, J_{0}}$} admits a nowhere vanishing \textit{holomorphic} volume form $dz^{1} \wedge dz^{2}$ which uniquely determines $J_{0}$ by Proposition \ref{holovolformdeterminescomplexstructure}. We perform a crepant resolution $\paren*{\wtilde{\bbT^{4}/\bbZ_{2}}, J}$ by blowing up each of the 16 orbifold singularities \textit{and where we denote the resulting integrable complex structure induced by the pullback of $dz^{1}\wedge dz^{2}$ by $J$}, and the result is a connected compact simply connected complex $2$-fold with trivial canonical bundle, whence yielding a $K3$ surface we denote by $\Km$. Specifically, since each orbifold singularity is modeled on $\paren*{\bbC^{2}/\bbZ_{2}, J_{0}}$, we have that the resolution at that singular point is \textit{precisely the same resolution} $\pi: \paren*{T^{*} \bbC P^{1}, J_{\cO(-2)}} \onto \paren*{\bbC^{2}/\bbZ_{2}, J_{0}}$ for the Eguchi-Hanson space. In fact, it is this fact which led Page, Gibbons, and Pope \cite{Page} \cite{Gibbons-Pope} to conjecture that one may ``graft in'' 16 ``small Eguchi-Hanson metrics'' around the 16 exceptional divisors of $\wtilde{\bbT^{4}/\bbZ_{2}}$ onto the flat metric on its complement. Thus we also have a projection map $\pi: \paren*{\wtilde{\bbT^{4}/\bbZ_{2}}, J} \onto \paren*{\bbT^{4}/\bbZ_{2}, J_{0}}$ which is a biholomorphism on the complement of the 16 exceptional $\bbC P^{1}$s.

In more detail, let $\cS \coloneq \sing\paren*{\bbT^{4}/\bbZ_{2}}$ be the set of 16 singular points, aka $\Fix\paren*{\bbZ_{2}}$.

Thus, let $\abs*{\cdot}^{g_{0}}_{\cS} \coloneq \dist_{g_{0}}\paren*{\cS,\cdot}$ be the metric distance away from the singular points WRT $g_{0}$.

\begin{data}\label{constant epsilon patching and WLOG lattice}

Since this entire construction will work with any full rank lattice $\Lambda \cong \bbZ^{4}$, WLOG let $\Lambda = \bbZ^{4}$, hence the singular points are the points whose real and imaginary parts are $0 \mod \frac{1}{2}$, whence the distance between each of the 16 singular points WRT $g_{0}$ is $\frac{1}{2}$.

Let our sufficiently small $\epsilon > 0$ from Data \ref{constant epsilon and lambda introduction} and Proposition \ref{preglueEH} be even smaller so that $3\epsilon^{\lambda} < \frac{1}{2}$ (if it wasn't already small enough to also ensure this).
\end{data}

Now consider 16 copies of the closed tubular neighborhood of radius $3\epsilon^{\lambda}$ around the exceptional divisor on $\paren*{T^{*} \bbC P^{1}, J_{\cO(-2)}}$, i.e. $\bigsqcup_{i \in \set*{1,\dots, 16}} T^{*} \bbC P^{1} - \set*{\abs*{\cdot}^{g_{0}}_{\CP^{1}} > 3\epsilon^{\lambda}} = \bigsqcup_{i \in \set*{1,\dots, 16}} \set*{\abs*{\cdot}^{g_{0}}_{\CP^{1}} \leq 3\epsilon^{\lambda}} = \bigsqcup_{i \in \set*{1,\dots, 16}} D^{g_{0}}_{3\epsilon^{\lambda}}\paren*{\CP^{1}}$.

Now define \begin{al}
\paren*{\Km_{\epsilon}, J, g_{\epsilon}, \omega_{\epsilon}} &\coloneq \paren*{\wtilde{\bbT^{4}/\bbZ_{2}}, J, g_{\epsilon}, \omega_{\epsilon}}\\
&\coloneq \frac{\paren*{\bbT^{4}/\bbZ_{2}  , J_{0}, g_{0}, \omega_{0}} \bigsqcup_{i \in \set*{1,\dots, 16}} \paren*{T^{*} \bbC P^{1} - \set*{\abs*{\cdot}^{g_{0}}_{\CP^{1}} > 3\epsilon^{\lambda}}, J_{\cO(-2)}, \wtilde{g_{EH,\epsilon}}, \wtilde{\omega_{EH, \epsilon}}} }{\sim}
\end{al} with the equivalence relation being the identification $\set*{\abs*{\cdot}^{g_{0}}_{\cS} \leq 3\epsilon^{\lambda}} \sim \set*{\abs*{\cdot}^{g_{0}}_{\CP^{1}} \leq 3\epsilon^{\lambda}}$. That is, each of the 16 connected components of $\set*{\abs*{\cdot}^{g_{0}}_{\cS} \leq 3\epsilon^{\lambda}}$, which is $\set*{\abs*{\cdot}^{g_{0}}_{p} \leq 3\epsilon^{\lambda}}$ for each $p \in \cS$, gets identified with a single copy of $\set*{\abs*{\cdot}^{g_{0}}_{\CP^{1}} \leq 3\epsilon^{\lambda}} \subset T^{*} \bbC P^{1}$, and this identification is done via first lifting/pulling back $\set*{\abs*{\cdot}^{g_{0}}_{p} \leq 3\epsilon^{\lambda}} \subset \bbT^{4}/\bbZ_{2} $ to $\bbT^{4} $ via the (branched) double cover/orbifold projection, then to (a proper subset of the fundamental domain in) $\bbC^{2}$ via the quotient $\bbC^{2} \onto \bbC^{2}/\Lambda = \bbT^{4}$, translating so that $p \mapsto 0$, then pushing down via $\bbC^{2} \onto \bbC^{2}/\bbZ_{2}$ onto $\bbC^{2}/\bbZ_{2}$, then lifting/pulling back via the crepant resolution map $T^{*} \bbC P^{1} \onto \bbC^{2}/\bbZ_{2}$ to get $\set*{\abs*{\cdot}^{g_{0}}_{\CP^{1}} \leq 3\epsilon^{\lambda}} \subset T^{*} \bbC P^{1}$. Thus the \textit{underlying} projection map $\pi: \Km_{\epsilon} \onto \bbT^{4}/\bbZ_{2}$ is given by the identity on $\set*{3\epsilon^{\lambda} < \abs*{\cdot}^{g_{0}}_{\cS}}$, and on $\set*{\abs*{\cdot}^{g_{0}}_{\cS} \leq 3\epsilon^{\lambda}} \sim \set*{\abs*{\cdot}^{g_{0}}_{\CP^{1}} \leq 3\epsilon^{\lambda}}$ via the crepant resolution map $\pi: \paren*{T^{*}\CP^{1}, J_{\cO\paren*{-2}}}\onto \paren*{\bbC^{2}/\bbZ_{2},J_{0}}$.

This gives us our underlying Kummer $K3$ surface $\Km$, \textit{or $\Km_{\epsilon}$ whenever we want to track the dependence on $\epsilon$}, \textit{as well as the background integrable complex structure $J$}. Indeed, since $\paren*{\Km_{\epsilon}, J} = \paren*{\wtilde{\bbT^{4}/\bbZ_{2}}, J}$ is a crepant resolution of $\paren*{\bbT^{4}/\bbZ_{2}, J_{0}}$, we have that there exists a nowhere vanishing holomorphic $(2,0)$-form/holomorphic volume form $\Omega$ on $\paren*{\Km_{\epsilon}, J}$ (given by $dz^{1} \wedge dz^{2}$ on $\set*{\abs*{\cdot}^{g_{0}}_{\cS} > 3\epsilon^{\lambda}}$ and identifying $dz^{1} \wedge dz^{2}$ on $\set*{\abs*{\cdot}^{g_{0}}_{\cS} \leq 3\epsilon^{\lambda}}$ with $\Omega_{T^{*} \bbC P^{1}}$ on $\set*{\abs*{\cdot}^{g_{0}}_{\CP^{1}} \leq 3\epsilon^{\lambda}}$, since both holomorphic volume forms descend from the \textit{same} Euclidean holomorphic volume form on $\bbC^{2}$), and $\Omega$ uniquely determines $J$ and trivializes the canonical bundle from Proposition \ref{holovolformdeterminescomplexstructure}.

It now remains to define $\omega_{\epsilon}$. Letting $E_{i} \cong \CP^{1}$ denote one of the 16 exceptional divisors and $\pi^{-1}\paren*{\cS} = \sqcup_{i\in \set*{1,\dots, 16}}E_{i}$, we set $$\omega_{\epsilon} \coloneq \begin{cases}
\omega_{0} & \text{on } \Km_{\epsilon} - \sqcup_{i\in \set*{1,\dots, 16}} D^{g_{0}}_{3\epsilon^{\lambda}}\paren*{E_{i}} = \Km_{\epsilon} - \set*{\abs*{\cdot}^{g_{0}}_{\pi^{-1}\paren*{\cS}} \leq 3\epsilon^{\lambda}} = \set*{ 3\epsilon^{\lambda} \leq \abs*{\cdot}^{g_{0}}_{\pi^{-1}\paren*{\cS}}} \\
\wtilde{\omega_{EH, \epsilon}} & \text{on (each 16 connected components of)} \set*{\abs*{\cdot}^{g_{0}}_{\pi^{-1}\paren*{\cS}} \leq 3\epsilon^{\lambda}}
\end{cases}	$$ which is clearly well defined and smooth on all of $\Km_{\epsilon}$ by Proposition \ref{preglueEH}, and we now have our $K3$ surface $\paren*{\Km_{\epsilon}, J, g_{\epsilon}, \omega_{\epsilon}}$ with our K\"{a}hler form $\omega_{\epsilon}$ and associated Riemannian metric $g_{\epsilon}$ from 2-out-of-3, \textit{depending on the small $\epsilon > 0$.}

\begin{remark}\label{identification}
Here we have that $\abs*{\cdot}^{g_{0}}_{\pi^{-1}\paren*{\cS}} \coloneq \dist_{g_{0}}\paren*{\pi^{-1}\paren*{\cS},\cdot} = \dist_{g_{0}}\paren*{\cS,\pi(\cdot)} \eqcolon \abs*{\pi(\cdot)}^{g_{0}}_{\cS}$ because $\pi: \paren*{\Km_{\epsilon}, J} \onto \paren*{\bbT^{4}/\bbZ_{2}, J_{0}}$ is a \textit{biholomorphism} on the complement of the exceptional divisors $\pi^{-1}\paren*{\cS} = \sqcup_{i\in \set*{1,\dots, 16}}E_{i}$ and the projection $\bbT^{4}\onto\bbT^{4}/\bbZ_{2}$ is a double cover on the complement of the orbifold singularities $\cS$ along which $g_{0}$ descends, whence we have the flat metric $g_{0}$ on $\bbT^{4} - \cS \overset{2:1}{\onto}\bbT^{4}/\bbZ_{2} - \cS\overset{\text{Biholo}}{\cong}\Km_{\epsilon} - \pi^{-1}\paren*{\cS} $. 

In particular, $\forall r > 0$ we may identify $\set*{r < \abs*{\cdot}^{g_{0}}_{\pi^{-1}\paren*{\cS}}}\subset \Km_{\epsilon}$ with $\set*{r < \abs*{\cdot}^{g_{0}}_{\cS}} \subset \bbT^{4}/\bbZ_{2}- \cS$ via $\pi: \paren*{\Km_{\epsilon}, J} \onto \paren*{\bbT^{4}/\bbZ_{2}, J_{0}}$ and its lift in $\set*{r < \abs*{\cdot}^{g_{0}}_{\cS}} \subset \bbT^{4}- \cS$ via the double cover $\bbT^{4} - \cS \overset{2:1}{\onto}\bbT^{4}/\bbZ_{2} - \cS$, where it is of course understood that $\abs*{\cdot}^{g_{0}}_{\cS} \coloneq \dist_{g_{0}}\paren*{\cS,\cdot}$.

On a similar vein, we only did the identification $\set*{\abs*{\cdot}^{g_{0}}_{\pi^{-1}\paren*{\cS}} \leq 3\epsilon^{\lambda}} \sim \set*{\abs*{\cdot}^{g_{0}}_{\CP^{1}} \leq 3\epsilon^{\lambda}}$ in the underlying construction of $\Km_{\epsilon}$ to guarantee that our $\omega_{\epsilon}$ is well defined and smooth on all of $\Km_{\epsilon}$. Thus $\forall r \in \paren*{0,\frac{1}{2}}$ (since WLOG $\Lambda = \bbZ^{4}$) we may identify \textit{each of the 16 connected components of} $\set*{\abs*{\cdot}^{g_{0}}_{\pi^{-1}\paren*{\cS}} \leq r}\subset \Km_{\epsilon}$ with $\set*{\abs*{\cdot}^{g_{0}}_{\CP^{1}} \leq r}\subset T^{*}\CP^{1}$ via the same exact process as the original identification above (connected component starts out in $\bbT^{4}/\bbZ_{2}$, lift to $\bbT^{4}$, lift to $\bbC^{2}$, translate, project down to $\bbC^{2}/\bbZ_{2}$, lift to $T^{*}\CP^{1}$ via crepant resolution).
\end{remark}

Observe also the following properties which immediately follows from construction of $\omega_{\epsilon}$ and Proposition \ref{preglueEH}: \begin{prop}\label{approximatemetricproperties} For $\epsilon > 0$ satisfying Data \ref{constant epsilon patching and WLOG lattice}, we have that $\omega_{\epsilon}$ satisfies
$$\omega_{\epsilon} = \begin{cases}
\omega_{EH,\epsilon^{2}} & \text{on }\set*{\abs*{\cdot}^{g_{0}}_{\pi^{-1}\paren*{\cS}} \leq \epsilon^{\lambda}} \\
\omega_{0} & \text{on }\set*{2\epsilon^{\lambda} \leq \abs*{\cdot}^{g_{0}}_{\pi^{-1}\paren*{\cS}}} 
\end{cases}\qquad\text{and}\qquad g_{\epsilon} = \begin{cases}
g_{EH,\epsilon^{2}} & \text{on } \set*{\abs*{\cdot}^{g_{0}}_{\pi^{-1}\paren*{\cS}} \leq \epsilon^{\lambda}} \\
g_{0} & \text{on }\set*{2\epsilon^{\lambda} \leq \abs*{\cdot}^{g_{0}}_{\pi^{-1}\paren*{\cS}}}
\end{cases}$$ and, for sufficiently small $\epsilon > 0$ and $\lambda < \frac{2}{3}$ to satisfy Proposition \ref{preglueEH}, $$\abs*{\nabla_{g_{0}}^{k}\paren*{\omega_{\epsilon} - \omega_{0}}}_{g_{0}} \leq c_{1}(k+2) \epsilon^{4-\paren*{4+k}\lambda}\qquad\text{and}\qquad\abs*{\nabla_{g_{0}}^{k}\paren*{g_{\epsilon} - g_{0}}}_{g_{0}} \leq c_{2}(k) \epsilon^{4-\paren*{4+k}\lambda}$$ on $\set*{\epsilon^{\lambda}\leq \abs*{\cdot}^{g_{0}}_{\pi^{-1}\paren*{\cS}} \leq 2\epsilon^{\lambda}} = \set*{\frac{1}{\epsilon^{1-\lambda}}\leq \frac{\abs*{\cdot}^{g_{0}}_{\pi^{-1}\paren*{\cS}}}{\epsilon} \leq \frac{2}{\epsilon^{1-`\lambda}}}$.

Thus we also have, \textit{for $\epsilon < 1$} (which Data \ref{constant epsilon patching and WLOG lattice} guarantees), $$\dV_{g_{\epsilon}} = \begin{cases}
\dV_{g_{EH,\epsilon^{2}}} & \text{on } \set*{\abs*{\cdot}^{g_{0}}_{\pi^{-1}\paren*{\cS}} \leq \epsilon^{\lambda}} \\
\paren*{1 + O\paren*{\epsilon^{4-4\lambda}}}\dV_{g_{0}} & \text{on }\set*{\epsilon^{\lambda} \leq \abs*{\cdot}^{g_{0}}_{\pi^{-1}\paren*{\cS}}}
\end{cases}$$ and that the cohomology class $[\omega_{\epsilon}]\in  H^{2}\paren*{K3}$ is independent of the cutoff function $\chi$ from Data \ref{constants C_{1} and c_{chi}(k) from cutoff} (\textit{Remark} \ref{cutoff function dependence}).\end{prop}

\begin{remark}
Note that, as previously mentioned, a theorem of Kodaira implies that for any pair $\epsilon_{1}, \epsilon_{2} > 0$ of \textit{positive} real numbers, the resulting $\Km_{\epsilon_{1}}, \Km_{\epsilon_{2}}$ are diffeomorphic, and since we've fixed the underlying integrable complex structure $J$, they are in fact biholomorphic.
\end{remark}

\begin{remark}\label{GHremark}
From the construction of $\Km_{\epsilon}$ and $\omega_{\epsilon}$ itself, we have that if $\epsilon \searrow 0$, that as Riemannian manifolds $\paren*{\Km_{\epsilon}, J, g_{\epsilon}, \omega_{\epsilon}} \xrightarrow{GH} \paren*{\bbT^{4}/\bbZ_{2}, J_{0}, g_{0}, \omega_{0}}$ converges in the \textbf{Gromov-Hausdorff topology} (see \cite{CheegerBook}) to the orbifold $\paren*{\bbT^{4}/\bbZ_{2}, J_{0}, g_{0}, \omega_{0}}$ with the singular orbifold K\"{a}hler metric $\omega_{0}$ with conical $\bbC^{2}/\bbZ_{2}$ singularities at each of the 16 singular points. Moreover, this convergence satisfies \textbf{volume non-collapse}, that is, $\Vol_{g_{\epsilon}}\paren*{B_{r}^{g_{\epsilon}}\paren*{p}} \geq Cr^{4}$ uniformly in $p$, $r$, and $\epsilon$. Note that outside of the set $\cS$ of 16 singular points, we have that $\omega_{0}$ is the flat K\"{a}hler metric on $\bbT^{4}/\bbZ_{2} - \cS$ from before. See \cite{CheegerBook}, \cite{RongBook}.\end{remark}

\subsection{\textsection \ Weighted (H\"{o}lder) Spaces \ \textsection}\label{Weighted Setup}

We now define the weighted H\"{o}lder spaces that we will be using for our setting, first for $\paren*{T^{*}\bbC P^{1}, J_{\cO(-2)}, g_{EH,s}, \omega_{EH,s}}$ and $\paren*{\bbC^{2}/\bbZ_{2}, J_{0}, g_{0}, \omega_{0}}$, then for $\paren*{\bbT^{4}, J_{0}, g_{0}, \omega_{0}}$ and by extension $ \paren*{\bbT^{4} - \cS, J_{0}, g_{0}, \omega_{0}}$, $\paren*{\bbT^{4}/\bbZ_{2} - \cS, J_{0}, g_{0}, \omega_{0}}$, and lastly for $\paren*{\Km_{\epsilon}, J, g_{\epsilon}, \omega_{\epsilon}}$. As H\"{o}lder spaces spaces \textit{only depend on the Riemannian metric}, they do \textit{not} require an integrable complex structure nor a K\"{a}hler form (though by 2-out-of-3 both of these uniquely determine the metric), \textit{and hence are defined on the underlying smooth manifolds $\paren*{T^{*}S^{2}, g_{EH, s}}$, $\paren*{\bbR^{4}/\bbZ_{2}, g_{0}}$, $\paren*{\bbT^{4}, g_{0}}$, $\paren*{\bbT^{4} - \cS, g_{0}}$, $\paren*{\bbT^{4}/\bbZ_{2} - \cS, g_{0}}$, and $\paren*{\Km_{\epsilon}, g_{\epsilon}}$}.

\begin{notation}
Let $\alpha \in (0,1)$, $\delta \in \bbR$, $k \in \bbN_{0}$, and $\fs$ be a $\paren*{p,q}$-tensor. For any $j \geq 0$ let $\Abs*{\nabla_{g}^{j} \fs}_{C^{0}_{\delta,g}} \coloneq \Abs*{\rho^{-\delta + j}\nabla_{g}^{j} \fs}_{C^{0}_{g}}$ denote the weighted $C^{0}$-norm corresponding to a Riemannian metric $g$ for $\rho: M \rightarrow \paren*{0,\infty}$ a weight function. We oftentimes call $\delta$ the \textbf{weight} or \textbf{rate} parameter.
\end{notation}

\begin{definition}\label{weightedH\"{o}ldernormDef + sigma}
Let $\alpha \in (0,1)$, $\delta \in \bbR$, $k \in \bbN_{0}$, and $\fs$ be a $\paren*{p,q}$-tensor. Then \begin{al}
\Abs*{\fs}_{C^{k,\alpha}_{\delta, g}}&\coloneq \sum_{j \leq k} \Abs*{\nabla_{g}^{j} \fs}_{C^{0}_{\delta,g}} + \sup_{\substack{x \neq y\\ d_{g}\paren*{x,y} < \sigma_{g}\min\set*{\rho\paren*{x},\rho\paren*{y}}}}\min\set*{\rho(x),\rho(y)}^{-\delta + k + \alpha} \frac{\abs*{\restr{\nabla_{g}^{k}\fs}{x} - \restr{\nabla_{g}^{k}\fs}{y}}_{g}}{d_{g}\paren*{x , y}^{\alpha}} \\
&\coloneq \sum_{j \leq k} \Abs*{\rho^{-\delta + j}\nabla_{g}^{j} \fs}_{C^{0}_{g}} + \sup_{\substack{x \neq y\\ d_{g}\paren*{x,y} < \sigma_{g}\min\set*{\rho\paren*{x},\rho\paren*{y}}}}\min\set*{\rho(x),\rho(y)}^{-\delta + k + \alpha} \frac{\abs*{\restr{\nabla_{g}^{k}\fs}{x} - \restr{\nabla_{g}^{k}\fs}{y}}_{g}}{d_{g}\paren*{x , y}^{\alpha}}
\end{al} denotes the \textbf{weighted $C^{k,\alpha}_{\delta,g}$-norm} corresponding to a Riemannian metric $g$ for $\rho: M \rightarrow \paren*{0,\infty}$ a weight function. Here, $\sigma_{g}> 0$ is a constant chosen such that the inequality $$\inj_{g}\paren*{p} \geq \sigma_{g}\rho\paren*{p}>0$$ holds uniformly $\forall p \in M$, where $\inj_{g}\paren*{p}$ is the injectivity radius of $g$ about $p$. Whence we interpret the difference in the numerator of the (weighted) H\"{o}lder seminorm via parallel transport along the unique minimal geodesic connecting $x$ and $y$. Note that shrinking $\sigma_{g}$ induces an equivalent norm. 

Then we define the space of \textbf{weighted H\"{o}lder $\paren*{p,q}$-tensors $C^{k,\alpha}_{\delta,g}$} as $C^{k,\alpha}_{\delta,g} \coloneq \set*{ \fs \in C^{k,\alpha}_{g} :\Abs*{\fs}_{C^{k,\alpha}_{\delta,g}} < \infty }$\footnote{Note that the standard H\"{o}lder space $C^{k,\alpha}_{g}$ has the covariant derivatives, parallel transport, higher order tensor norms with respect to $g$.} but with the $\Abs*{\cdot}_{C^{k,\alpha}_{\delta,g}}$ norm instead. It is a Banach space whose elements, i.e. tensors $\fs$ with finite $\Abs*{\cdot}_{C^{k,\alpha}_{\delta,g}}$ norm, say $\Abs*{\fs}_{C^{k,\alpha}_{\delta,g}} \leq C$, decay roughly as $\nabla^{j}_{g} \fs = O_{g}\paren*{\rho^{\delta - j}}, j\leq k$. \end{definition}

Recalling $C_{1} > 0$ from Data \ref{constants C_{1} and c_{chi}(k) from cutoff}, on $T^{*}\bbC P^{1}$ we let $\wtilde{\rho_{0}}: T^{*}\bbC P^{1} \rightarrow [1,\infty)$ be a smooth monotone function satisfying $$\wtilde{\rho_{0}} = \begin{cases}
s^{\frac{1}{2}}\abs*{\cdot}^{g_{0}}_{\CP^{1}} & \text{ on }\set*{2s^{\frac{1}{2}} \leq \abs*{\cdot}^{g_{0}}_{\CP^{1}}}\\
s^{\frac{1}{2}} & \text{ on }\set*{\abs*{\cdot}^{g_{0}}_{\CP^{1}} \leq s^{\frac{1}{2}}}
\end{cases},\qquad\abs*{\nabla\wtilde{\rho_{0}}}_{g_{0}}\leq C_{1}\text{ on }\set*{s^{\frac{1}{2}}\leq \abs*{\cdot}^{g_{0}}_{\CP^{1}}\leq 2s^{\frac{1}{2}}}$$ On $\paren*{T^{*}\bbC P^{1}, J_{\cO(-2)}, g_{EH, s}, \omega_{EH, s}}$, we now define the space of weighted H\"{o}lder $\paren*{p,q}$-tensors $C^{k,\alpha}_{\delta,g_{EH, s}}\paren*{T^{*}\bbC P^{1}}$ as $C^{k,\alpha}_{\delta,g_{EH, s}}\paren*{T^{*}\bbC P^{1}} \coloneq \set*{ \fs \in C^{k,\alpha}_{g_{EH, s}}\paren*{T^{*}\bbC P^{1}} :\Abs*{\fs}_{C^{k,\alpha}_{\delta,g_{EH, s}}\paren*{T^{*}\bbC P^{1}}} < \infty }$ but with the $\Abs*{\cdot}_{C^{k,\alpha}_{\delta,g_{EH, s}}\paren*{T^{*}\bbC P^{1}}}$ norm instead using the above weight function $\wtilde{\rho_{0}}$. Note that since the injectivity radius of $g_{EH,s}$ is $O\paren*{s^{\frac{1}{2}}}$ and since the Eguchi-Hanson metric has bounded geometry (Proposition \ref{EHproperties}), $\sigma_{g_{EH,s}}>0$ thus exists.

\begin{remark}
Note that because our weight function $\wtilde{\rho_{0}}: T^{*}\CP^{1} \rightarrow [1,\infty)$ is smooth nondecreasing and satisfying $\wtilde{\rho_{0}} = \begin{cases}
s^{\frac{1}{2}}\abs*{\cdot}^{g_{0}}_{\CP^{1}} & \text{ on }\set*{2s^{\frac{1}{2}} \leq \abs*{\cdot}^{g_{0}}_{\CP^{1}}}\\
s^{\frac{1}{2}} & \text{ on }\set*{\abs*{\cdot}^{g_{0}}_{\CP^{1}} \leq s^{\frac{1}{2}}}
\end{cases}$, it is in particular \textit{unbounded from above}, and therefore we have that the space $C^{k,\alpha}_{\delta,g_{EH, s}}\paren*{T^{*}\bbC P^{1}}$ \textit{does not contain the constant tensors.} This is a crucial property for weighted H\"{o}lder spaces which allows for us to completely kill the kernel of the weighted Laplacian, since as we'll see, a suitable choice of weight parameter $\delta \in \bbR$ (specifically $\delta \in (-2,0)$) kills the harmonic functions \textit{but leaves only the constants left to deal with}.\end{remark}

In the exact same manner, for $\paren*{\bbC^{2}/\bbZ_{2}, J_{0}, g_{0}, \omega_{0}}$ we may define the space of weighted H\"{o}lder $\paren*{p,q}$-tensors $C^{k,\alpha}_{\delta,g_{0}}\paren*{\bbC^{2}/\bbZ_{2}}$ as $C^{k,\alpha}_{\delta,g_{0}}\paren*{\bbC^{2}/\bbZ_{2}}\coloneq \set*{ \fs \in C^{k,\alpha}_{g_{0}}\paren*{\bbC^{2}/\bbZ_{2}} :\Abs*{\fs}_{C^{k,\alpha}_{\delta,g_{0}}\paren*{\bbC^{2}/\bbZ_{2}}} < \infty }$ but with the $\Abs*{\cdot}_{C^{k,\alpha}_{\delta,g_{0}}\paren*{\bbC^{2}/\bbZ_{2}}}$ norm instead, where we abuse notation and denote the weight function used as $$\wtilde{\rho_{0}} =\abs*{\cdot}^{g_{0}}_{0}$$ Since the metric $g_{0}$ is flat, we may take $\sigma_{g_{0}} > 0$ to be arbitrarily small.

On $\paren*{\bbT^{4}, J_{0}, g_{0}, \omega_{0}}$, $\paren*{\bbT^{4} - \cS, J_{0}, g_{0}, \omega_{0}}$, and $\paren*{\bbT^{4}/\bbZ_{2} - \cS, J_{0}, g_{0}, \omega_{0}}$, we let $\rho_{0}: \bbT^{4} \rightarrow [0,1]$ be smooth, monotone, and satisfy $$\rho_{0} = \begin{cases}
1 & \text{ on }\set*{\frac{1}{5}\leq \abs*{\cdot}^{g_{0}}_{\cS}}\\
\abs*{\cdot}^{g_{0}}_{\cS} & \text{ on }\set*{\abs*{\cdot}^{g_{0}}_{\cS} \leq \frac{1}{8}}
\end{cases},\qquad\abs*{\nabla \rho_{0}}\leq \frac{35C_{1}}{3}\text{ on }\set*{\frac{1}{8}\leq \abs*{\cdot}^{g_{0}}_{\cS}\leq \frac{1}{5}}$$ For $\bbT^{4} - \cS$, we of course just restrict the $\rho_{0}$ above to that subset, whence is equal to $\abs*{\cdot}_{\cS}^{g_{0}}$ on $\set*{0< \abs*{\cdot}^{g_{0}}_{\cS} \leq \frac{1}{8}}$ instead. We now define the space of weighted H\"{o}lder $\paren*{p,q}$-tensors $C^{k,\alpha}_{\delta,g_{0}} \paren*{\bbT^{4} }$ as $C^{k,\alpha}_{\delta,g_{0}} \paren*{\bbT^{4} } \coloneq \set*{\fs \in C^{k,\alpha}_{g_{0}} \paren*{\bbT^{4} } : \Abs*{s}_{C^{k,\alpha}_{\delta,g_{0}} \paren*{\bbT^{4} }} < \infty}$ but with the $\Abs*{\cdot}_{C^{k,\alpha}_{\delta,g_{0}} \paren*{\bbT^{4}}}$ norm instead and with the above weight function $\rho_{0}$. Since the metric $g_{0}$ is the Euclidean metric, and since the 16 punctures $\cS$ constitute the source of incompleteness, $\sigma_{g_{0}} > 0$ clearly exists.

\begin{remark}
In contrast to $T^{*}\CP^{1}$ and $\bbC^{2}/\bbZ_{2}$, note that because the weight function $\rho_{0}$ has \textit{bounded} sup norm (due \textit{in part} because it is smooth and the underlying space is (a subset of a space which is) \textit{compact}), we have that $C^{k,\alpha}_{\delta,g_{0}} \paren*{\bbT^{4} }$ (specialized to the case of functions/$\paren*{0,0}$-tensors) \textit{contains the constant functions.} This is important because this produces a 1 dimensional subspace in the kernel of any differential operator on $\bbT^{4} $ with domain any $C^{k,\alpha}_{\delta,g_{0}} \paren*{\bbT^{4} }$, in particular the weighted Laplacian. Hence, upon choosing a suitable weight $\delta \in \bbR$ (again, specifically $\delta \in (2-\dim_{\bbR}M,0)$) to kill the harmonic functions, we \textit{still have to deal with the constant functions} in order to completely kill the kernel.\end{remark} Note also that this entire discussion \& definitions above hold verbatim for $\paren*{\bbT^{4} - \cS, J_{0}, g_{0}, \omega_{0}}$ and $\paren*{\bbT^{4}/\bbZ_{2} - \cS, J_{0}, g_{0}, \omega_{0}}$.

Now we define the weighted H\"{o}lder spaces for $\paren*{\Km_{\epsilon}, J, g_{\epsilon}, \omega_{\epsilon}}$. For that, we need $\epsilon > 0$ to now satisfy the following

\begin{data}\label{constant epsilon 1-16}
Let our sufficiently small $\epsilon > 0$ be from before (specifically Propositions \ref{preglueEH} and \ref{approximatemetricproperties} and Data \ref{constant epsilon patching and WLOG lattice}), but now even smaller to also ensure both $2\epsilon^{\lambda} < \frac{3}{25} < \frac{1}{8}$ and $2\epsilon < \epsilon^{\lambda}$, hence $\epsilon^{2} < \epsilon < 2\epsilon < \epsilon^{\lambda} < \frac{3}{50} < \frac{1}{16}$ (if it wasn't already small enough to also ensure this).
\end{data}

Let $G_{\epsilon}: [0,\infty)\rightarrow\sqparen*{\epsilon,1}$ be a smooth monotone function satisfying $$G_{\epsilon}(x) = \begin{cases}
1 & \frac{1}{5}\leq x\\
x & 2\epsilon\leq x \leq \frac{1}{8}\\
\epsilon & x \leq \epsilon
\end{cases},\qquad\abs*{\nabla G_{\epsilon}}\leq C_{1}\text{ on }\set*{\epsilon\leq x\leq 2\epsilon},\qquad\abs*{\nabla G_{\epsilon}}\leq \frac{35C_{1}}{3}\text{ on }\set*{\frac{1}{8}\leq x\leq \frac{1}{5}}$$ Note that $\set*{2\epsilon \leq x \leq \frac{1}{8}} \neq \emptyset$ because our $\epsilon > 0$ has been chosen up to this point to be small enough so that $\epsilon < \frac{1}{16}$ from Data \ref{constant epsilon 1-16}. Now define $\rho_{\epsilon} \coloneq G_{\epsilon} \circ \abs*{\cdot}^{g_{0}}_{\pi^{-1}\paren*{\cS}}$. Hence we have that $\rho_{\epsilon}:\Km_{\epsilon} \rightarrow [\epsilon,1]$ is smooth, monotone, and satisfies $$\rho_{\epsilon} = \begin{cases}
1 & \text{ on }\set*{\frac{1}{5}\leq \abs*{\cdot}^{g_{0}}_{\pi^{-1}\paren*{\cS}}}\\
\abs*{\cdot}^{g_{0}}_{\pi^{-1}\paren*{\cS}} & \text{ on }\set*{2\epsilon\leq \abs*{\cdot}^{g_{0}}_{\pi^{-1}\paren*{\cS}} \leq \frac{1}{8}}\\
\epsilon & \text{ on }\set*{\abs*{\cdot}^{g_{0}}_{\pi^{-1}\paren*{\cS}} \leq \epsilon}
\end{cases},\quad\abs*{\nabla \rho_{\epsilon}}_{g_{0}}\leq C_{1}\text{ on }\set*{\epsilon\leq \abs*{\cdot}^{g_{0}}_{\pi^{-1}\paren*{\cS}}\leq 2\epsilon},\quad\abs*{\nabla \rho_{\epsilon}}_{g_{0}}\leq \frac{35C_{1}}{3}\text{ on }\set*{\frac{1}{8}\leq \abs*{\cdot}^{g_{0}}_{\pi^{-1}\paren*{\cS}}\leq \frac{1}{5}}$$ We now define $C^{k,\alpha}_{\delta,g_{\epsilon}} \paren*{\Km_{\epsilon}}$ as $C^{k,\alpha}_{\delta,g_{\epsilon}} \paren*{\Km_{\epsilon}} \coloneq \set*{\fs \in C^{k,\alpha}_{g_{\epsilon}} \paren*{\Km_{\epsilon}} : \Abs*{s}_{C^{k,\alpha}_{\delta,g_{\epsilon}} \paren*{\Km_{\epsilon}}} < \infty}$ but with the $\Abs*{\cdot}_{C^{k,\alpha}_{\delta,g_{\epsilon}} \paren*{\Km_{\epsilon}}}$ norm instead and with the above weight function $\rho_{\epsilon}$.

\begin{remark}\label{curvatureestimates + K3 sigma} By Proposition \ref{approximatemetricproperties} and Proposition \ref{EHproperties}, we have that the curvature $\operatorname{Rm}_{g_{\epsilon}}$ is flat on $\set*{2\epsilon^{\lambda} \leq \abs*{\cdot}^{g_{0}}_{\pi^{-1}\paren*{\cS}}}$, blows up as $O_{g_{EH,\epsilon^{2}}}\paren*{\frac{1}{\epsilon^{2}}}$ on $\set*{\abs*{\cdot}^{g_{0}}_{\pi^{-1}\paren*{\cS}} \leq \epsilon^{\lambda}}$, and is \textit{uniformly bounded} on $\set*{\epsilon^{\lambda} \leq \abs*{\cdot}^{g_{0}}_{\pi^{-1}\paren*{\cS}} \leq 2\epsilon^{\lambda}}$ by Proposition \ref{approximatemetricproperties} and \textit{Remark} \ref{lambda < 2/3 bound}. Therefore from $\epsilon < 2\epsilon < \epsilon^{\lambda}$ from Data \ref{constant epsilon 1-16}, the monotonicity of $\rho_{\epsilon}$ restricted to $\set*{\epsilon \leq \abs*{\cdot}^{g_{0}}_{\pi^{-1}\paren*{\cS}} \leq \epsilon^{\lambda}}$, and the identity $\abs*{\fs}_{g_{\epsilon}} = \paren*{1+O\paren*{\epsilon^{4-4\lambda}}}\abs*{\fs}_{g_{0}}$ for $\paren*{p,q}$-tensors $\fs$ holding pointwise on $\set*{\epsilon^{\lambda} \leq \abs*{\cdot}^{g_{0}}_{\pi^{-1}\paren*{\cS}} \leq 2\epsilon^{\lambda}}$ (which follows from the exact same computation as for computing the volume form difference in Proposition \ref{preglueEH}), we therefore have that \textit{globally} on all of $\Km_{\epsilon}$ that $$\operatorname{Rm}_{g_{\epsilon}} = O_{g_{\epsilon}}\paren*{\rho_{\epsilon}^{-2}}$$ Combining this with the uniform volume non-collapse of $g_{\epsilon}$ (\textit{Remark} \ref{GHremark}), by a lemma of Cheeger (Corollary 3.2.2 of \cite{RongBook}) we have that $\exists \sigma > 0$ such that $$\sigma \rho_{\epsilon}\paren*{p} \leq \inj_{g_{\epsilon}}\paren*{p}$$ holds \textit{uniformly in $\epsilon$} and $\forall p \in \Km_{\epsilon}$. We therefore choose $\sigma_{g_{\epsilon}} \in \paren*{0,\sigma}$.\end{remark}

Moreover, we have the following proposition for (tensor) products: \begin{prop}\label{weighted products}
Let $\delta < l$ for some $l > 0$. Then $\exists C_{2} > 0$, \textit{independent of $\epsilon > 0$}, such that $$\Abs*{f\otimes g}_{C^{0,\alpha}_{\delta-l, g_{\epsilon}}\paren*{\Km_{\epsilon}}} \leq C_{2} \epsilon^{\delta - l}\Abs*{f}_{C^{0,\alpha}_{\delta-l, g_{\epsilon}}\paren*{\Km_{\epsilon}}} \Abs*{g}_{C^{0,\alpha}_{\delta-l, g_{\epsilon}}\paren*{\Km_{\epsilon}}}$$
\end{prop} \begin{proof}
First we have that $$\Abs*{f\otimes g}_{C^{0,\alpha}_{\delta-l, g_{\epsilon}}\paren*{\Km_{\epsilon}}} \leq C_{2} \Abs*{\rho_{\epsilon}^{\delta - l}}_{C^{0}}\Abs*{f}_{C^{0,\alpha}_{\delta-l, g_{\epsilon}}\paren*{\Km_{\epsilon}}} \Abs*{g}_{C^{0,\alpha}_{\delta-l, g_{\epsilon}}\paren*{\Km_{\epsilon}}}$$ by definition of the weighted H\"{o}lder norms on $\paren*{\Km_{\epsilon}, g_{\epsilon}}$, the fact that $\rho_{\epsilon}^{-\delta + l} = \rho_{\epsilon}^{\delta - l} \paren*{\rho_{\epsilon}^{-\delta + l}}^{2}$, the definition of the $C^{0}$-norm, and the definition of the induced metric on tensors. Whence, because by definition of the distance function $\rho_{\epsilon}$, we have that $\rho_{\epsilon} \geq \epsilon$ and $\delta - l<0$, whence the result follows immediately.\end{proof}

\begin{remark}
Lastly, just as for $\bbT^{4} - \cS$, we have that due to $\rho_{\epsilon}$ having \textit{bounded} sup norm (due \textit{in part} because it is smooth and $\Km_{\epsilon}$ is \textit{compact}), we have that $C^{k,\alpha}_{\delta,g_{\epsilon}} \paren*{\Km_{\epsilon}}$ (specialized to the case of functions/$\paren*{0,0}$-tensors) \textit{contains the constant functions}, whence produces a 1 dimensional subspace in the kernel of any differential operator on $\Km_{\epsilon}$ with domain any $C^{k,\alpha}_{\delta,g_{\epsilon}} \paren*{\Km_{\epsilon}}$, in particular the weighted Laplacian. Hence, upon choosing a suitable weight $\delta \in \bbR$ (again, specifically $\delta \in (2-\dim_{\bbR}M,0)$) to kill the harmonic functions, we \textit{still have to deal with the constant functions} in order to completely kill the kernel.
\end{remark}

\subsection{\textsection \ Nonlinear Setup \ \textsection}\label{Nonlinear Setup}

Let's now bring everything together and find our $\wtilde{\omega_{\epsilon}} \in [\omega_{\epsilon}]$ K\"{a}hler on $\paren*{\Km_{\epsilon}, J, g_{\epsilon}, \omega_{\epsilon}}$ whose associated $\wtilde{g_{\epsilon}}$ is Ricci-flat.

\begin{notation}
As we will be working with a \textit{fixed} $\epsilon > 0$, denote the sought after Ricci-flat K\"{a}hler form $\wtilde{\omega_{\epsilon}} \in [\omega_{\epsilon}]$ as $\omega$, \textit{keeping in mind that as this will end up being the unique Ricci-flat K\"{a}hler form in $[\omega_{\epsilon}]$, that \textit{it depends on the gluing parameter $\epsilon>0$} as the starting K\"{a}hler class $[\omega_{\epsilon}]$ depends on $\epsilon>0$}.
\end{notation}

By Proposition \ref{K\"{a}hlereinstein} from Section \ref{Preliminaries}, we now want to solve for $f \in C^{\infty}\paren*{\Km_{\epsilon}}$ a smooth function/$\paren*{0,0}$-tensor defining a K\"{a}hler form $\omega \coloneq \omega_{\epsilon} + i\del \delbar f \in [\omega_{\epsilon}]$ in the following \textbf{complex Monge-Ampere equation with singular limit}\footnote{The singular limit comes from when $\epsilon \searrow 0$, which produces the flat \ka metric on the singular orbifold $\bbR^{4}/\bbZ_{2}$, see \textit{Remark} \ref{GHremark}.}:

$$\frac{\omega^{2}}{\omega_{\epsilon}^{2}} = \frac{\paren*{\omega_{\epsilon} + i\del \delbar f}^{2}}{\omega_{\epsilon}^{2}} = e^{\phi_{\epsilon}}$$ where $\phi_{\epsilon}$ is the Ricci potential $\Ric_{\omega_{\epsilon}} = i\del \delbar \phi_{\epsilon}$ satisfying the volume normalization $\int_{\Km_{\epsilon}} e^{\phi_{\epsilon}} \omega_{\epsilon}^{2} = \int_{\Km_{\epsilon}} \omega_{\epsilon}^{2}$.	Because we have our nowhere vanishing holomorphic $(2,0)$-form $\Omega$ on $\paren*{\Km_{\epsilon}, J}$ (which determines $J$), by Proposition \ref{calabiyau} this is equivalent to solving $$\frac{\omega^{2}}{\Omega \wedge \overline{\Omega}}  = \frac{\paren*{\omega_{\epsilon} + i\del \delbar f}^{2}}{\Omega \wedge \overline{\Omega}} = C$$ for some constant $C > 0$. In fact, we therefore have that $\phi_{\epsilon} = \log\paren*{C \frac{\Omega \wedge \overline{\Omega}}{\omega_{\epsilon}^{2}}}$, and since $C > 0$ is a positive real constant, WLOG we may set $C = 1$ via multiplying the nowhere vanishing holomorphic $(2,0)$-form by $\frac{1}{C^{\frac{1}{2}}} $, i.e. replacing $\Omega \mapsto \frac{1}{C^{\frac{1}{2}}} \Omega$ (i.e. changing the trivialization of $K_{\Km_{\epsilon}}$), and thus $\phi_{\epsilon} = \log\paren*{\frac{\Omega \wedge \overline{\Omega}}{\omega_{\epsilon}^{2}}}$.

Now let $\Delta_{g}$ be the standard Hodge Laplacian $d d^{*} + d^{*} d$ acting on 0-forms (hence $d^{*} d$) with Hodge star coming from the Riemannian metric $g$ which itself is from a given K\"{a}hler form and integrable complex structure $\omega, J$ via 2-out-of-3. Note that this has \textbf{nonnegative} spectrum. Now since $\Delta_{g} f = -2n \frac{i\del \delbar f \wedge \omega^{n-1}}{\omega^{n}}$ on any $\dim_{\bbC} = n$ dimensional K\"{a}hler manifold $\paren*{M^{2n},J,g,\omega}$, we now setup our nonlinear problem as follows:

\textit{Specializing to functions/$\paren*{0,0}$-tensors}, define $C^{k,\alpha}_{\delta,g_{\epsilon}} \paren*{\Km_{\epsilon}}^{0} \coloneq \set*{f \in C^{k,\alpha}_{\delta,g_{\epsilon}} \paren*{\Km_{\epsilon}} : \int_{\Km_{\epsilon}} f \omega_{\epsilon}^{2} = 0} \subset C^{k,\alpha}_{\delta,g_{\epsilon}} \paren*{\Km_{\epsilon}}$. Because $\rho_{\epsilon}$ has \textit{bounded} sup norm, we have that $C^{k,\alpha}_{\delta,g_{\epsilon}} \paren*{\Km_{\epsilon}}^{0} \subset C^{k,\alpha}_{\delta,g_{\epsilon}} \paren*{\Km_{\epsilon}}$ is a \textit{closed} Banach subspace, inheriting the same Banach norm $\Abs*{\cdot}_{C^{k,\alpha}_{\delta,g_{\epsilon}} \paren*{\Km_{\epsilon}}}$.

Our nonlinear problem is therefore:\begin{al}
\cF_{\epsilon} :  C^{2,\alpha}_{\delta, g_{\epsilon}} \paren*{\Km_{\epsilon}}^{0} &\rightarrow C^{0,\alpha}_{\delta - 2, g_{\epsilon}} \paren*{\Km_{\epsilon}}^{0} \\
f &\mapsto \frac{\paren*{\omega_{\epsilon} + i\del \delbar f}^{2}}{\omega_{\epsilon}^{2}} - e^{\phi_{\epsilon}}		
\end{al}

where we solve \begin{al}
\cF_{\epsilon}\paren*{f} &\coloneq\frac{\paren*{\omega_{\epsilon} + i\del \delbar f}^{2}}{\omega_{\epsilon}^{2}} - e^{\phi_{\epsilon}}		 \\
&= \underbrace{\paren*{1 - e^{\phi_{\epsilon}}}}_{= \cF_{\epsilon}\paren*{0}} \underbrace{- \frac{1}{2}\Delta_{g_{\epsilon}}f}_{= D_{0}\cF_{\epsilon}\paren*{f}} + \underbrace{\frac{i\del \delbar f \wedge i\del \delbar f }{\omega_{\epsilon}^{2}}}_{\text{nonlinearity}} \\
&= 0
\end{al}

\begin{remark}
Note that since we are on a \textit{compact} manifold, for the domain and codomain weighted H\"{o}lder spaces $C^{2,\alpha}_{\delta, g_{\epsilon}} \paren*{\Km_{\epsilon}}, C^{0,\alpha}_{\delta - 2, g_{\epsilon}} \paren*{\Km_{\epsilon}}$ in the nonlinear problem above, by the Fredholm alternative for the Laplacian, \textit{we have to restrict to the closed subspace of functions with integral zero}, i.e. $\int_{\Km_{\epsilon}} f \omega_{\epsilon}^{2} = 0$, in order to guarantee the solvability of Laplace's equation $\Delta_{g_{\epsilon}} f = 0$ (or else, as remarked before, we have a nontrivial 1 dimensional subspace of the kernel consisting of the \textit{constant} functions, hence guaranteeing a nontrivial kernel).
\end{remark}

\begin{remark}
Likewise, notice also that by restricting to functions with integral zero, we \textit{automatically get uniqueness} of the resulting K\"{a}hler form $\omega \in [\omega_{\epsilon}]$ yielding a Ricci-flat K\"{a}hler metric (as explained by Propositions \ref{K\"{a}hlereinstein} and \ref{calabiyau}).
\end{remark}

Our method of attack to prove the existence of $f$ solving $\cF_{\epsilon}\paren*{f} = 0$ above is via the \textbf{implicit function theorem}, itself derived from the \textbf{Banach fixed point theorem}. For this to work, we need to invert the linearization $D_{0}\cF_{\epsilon}$ \textit{uniformly in $\epsilon > 0$}, have a Lipschitz bound on the nonlinearity, and have some good bounds on $\cF_{\epsilon}\paren*{0}$. We tackle the linearization in the next sections, and we start the next section with a preliminary step for the smallness of $\cF_{\epsilon}\paren*{0}$.

\begin{prop}\label{presmall}Recall the Ricci potential $\phi_{\epsilon} = \log\paren*{\frac{\Omega \wedge \overline{\Omega}}{\omega_{\epsilon}^{2}}}$ which also satisfies the volume normalization $\int_{\Km_{\epsilon}} e^{\phi_{\epsilon}} \omega_{\epsilon}^{2} = \int_{\Km_{\epsilon}} \omega_{\epsilon}^{2}$. 
Then we have that \textit{for $\epsilon > 0$ from Data \ref{constant epsilon 1-16} sufficiently small as in Proposition \ref{approximatemetricproperties}} and $c_{1}(k)$ from Propositions \ref{preglueEH} and \ref{approximatemetricproperties} along with $\lambda < \frac{2}{3}$: \begin{itemize}
\item $\supp \phi_{\epsilon} = \set*{\epsilon^{\lambda} \leq \abs*{\cdot}^{g_{0}}_{\pi^{-1}\paren*{\cS}} \leq 2\epsilon^{\lambda}}$
\item $\abs*{\nabla_{g_{0}}^{k}\phi_{\epsilon}}_{g_{0}} \leq c_{1}(k+2) \epsilon^{4-\paren*{4+k}\lambda}$ on the annulus $\set*{\epsilon^{\lambda} \leq \abs*{\cdot}^{g_{0}}_{\pi^{-1}\paren*{\cS}} \leq 2\epsilon^{\lambda}}$
\end{itemize}
\end{prop}

\begin{proof}
The support follows immediately because $\omega_{\epsilon}$ is the (scaled) Eguchi-Hanson on $\set*{\abs*{\cdot}^{g_{0}}_{\pi^{-1}\paren*{\cS}} \leq \epsilon^{\lambda}}$ and is flat on $\set*{2\epsilon^{\lambda} \leq \abs*{\cdot}^{g_{0}}_{\pi^{-1}\paren*{\cS}}}$, hence the Ricci potential $\phi_{\epsilon}$ of $\omega_{\epsilon}$ must vanish on those regions, and the support is the \textit{closure} of the complement of the vanishing set. Secondly, since $\omega_{0}$ is defined on the complement of the 16 exceptional divisors, it clearly is Ricci-flat and satisfies $\omega_{0}^{2} = \Omega \wedge \overline{\Omega}$. Hence \begin{al}
\abs*{\phi_{\epsilon}} &= \abs*{\log\paren*{\frac{\Omega \wedge \overline{\Omega}}{\omega_{\epsilon}^{2}}}} \\
&= \abs*{\log\paren*{\frac{\omega_{0}^{2}}{\omega_{\epsilon}^{2}}}} = \abs*{2\log\paren*{\frac{1}{1 + \frac{\fO}{\omega_{0}}}}}\\
&= 2\abs*{\log\paren*{1 + \frac{\fO}{\omega_{0}}}}\\
&\leq  c_{1}(2) \epsilon^{4-4\lambda}
\end{al} since $\abs*{\omega_{0}}_{g_{0}} = 2$, $\fO \coloneq \omega_{\epsilon} - \omega_{0}$ and on the annulus $\set*{\epsilon^{\lambda} \leq \abs*{\cdot}^{g_{0}}_{\pi^{-1}\paren*{\cS}} \leq 2\epsilon^{\lambda}}$ we have that $\abs*{\fO}_{g_{0}} \leq c_{1}(2) \epsilon^{4-4\lambda}$ when $\epsilon>0$ is sufficiently small (Proposition \ref{preglueEH}), and the Taylor expansion of $\log\paren*{1 + x}$ finishes it off. The estimate for higher derivatives follows immediately from $\abs*{\nabla^{k}_{g_{0}}\fO}_{g_{0}} \leq c_{1}(k+2) \epsilon^{4-\paren*{4+k}\lambda}$ and the exact same argument.
\end{proof}

\begin{remark}\label{ricciestimates}
Therefore since the Ricci potential $\phi_{\epsilon}$ satisfies $\Ric_{\omega_{\epsilon}} = i\del \delbar \phi_{\epsilon}$, we have that $$\abs*{\nabla_{g_{0}}^{k}\Ric_{\omega_{\epsilon}}}_{g_{0}} \leq c_{1}(k+4) \epsilon^{4-\paren*{6+k}\lambda}\qquad\text{and}\qquad\abs*{\nabla_{g_{0}}^{k}\Ric_{g_{\epsilon}}}_{g_{0}} \leq c_{3}(k) \epsilon^{4-\paren*{6+k}\lambda}$$ holds on the annulus $\set*{\epsilon^{\lambda} \leq \abs*{\cdot}^{g_{0}}_{\pi^{-1}\paren*{\cS}} \leq 2\epsilon^{\lambda}}$ for some positive constants $c_{3}(k) >0,k \in \bbN_{0}$, and $\Ric_{\omega_{\epsilon}}, \Ric_{g_{\epsilon}}$ vanishes outside $\set*{\epsilon^{\lambda} \leq \abs*{\cdot}^{g_{0}}_{\pi^{-1}\paren*{\cS}} \leq 2\epsilon^{\lambda}}$. In fact, using local ON frames WRT $g_{0}$ we see that $$\abs*{\Ric_{g_{\epsilon}}}_{g_{\epsilon}} \leq c_{3}(0) \epsilon^{4-6\lambda}$$ holds on the annulus $\set*{\epsilon^{\lambda} \leq \abs*{\cdot}^{g_{0}}_{\pi^{-1}\paren*{\cS}} \leq 2\epsilon^{\lambda}}$, whence we easily get that $$-3\Lambda g_{\epsilon}\leq \Ric_{g_{\epsilon}} \leq 3 \Lambda g_{\epsilon}$$ holds \textit{everywhere} on $\Km_{\epsilon}$ for $\Lambda \coloneq \frac{c_{3}(0)\epsilon^{4-6\lambda}}{3}$. Whence since $\lambda < \frac{2}{3}$, this thus tells us that $\paren*{\Km_{\epsilon}, J, g_{\epsilon}, \omega_{\epsilon}}$ has \textbf{uniformly bounded Ricci curvature}.
\end{remark}

\subsection{\textsection \ Blow-up analysis \ \textsection}\label{Blow-up Analysis}

We now want to invert, for suitable values of $\delta \in \bbR$ \textit{and for small enough $\epsilon > 0$ as from Data \ref{constant epsilon 1-16}}, the linearization $D_{0}\cF_{\epsilon} : C^{2,\alpha}_{\delta, g_{\epsilon}} \paren*{\Km_{\epsilon}}^{0} \rightarrow C^{0,\alpha}_{\delta - 2, g_{\epsilon}} \paren*{\Km_{\epsilon}}^{0}$ of $\cF_{\epsilon} : C^{2,\alpha}_{\delta, g_{\epsilon}} \paren*{\Km_{\epsilon}}^{0} \rightarrow C^{0,\alpha}_{\delta - 2, g_{\epsilon}} \paren*{\Km_{\epsilon}}^{0}$ at the origin. Since $D_{0}\cF_{\epsilon} = - \frac{1}{2}\Delta_{g_{\epsilon}}f$, it suffices to invert $\Delta_{g_{\epsilon}}: C^{2,\alpha}_{\delta, g_{\epsilon}} \paren*{\Km_{\epsilon}}^{0} \rightarrow C^{0,\alpha}_{\delta - 2, g_{\epsilon}} \paren*{\Km_{\epsilon}}^{0}$, which is a bounded linear operator. The key is to have the operator norm of the inverse be bounded \textit{uniformly in $\epsilon > 0$}. Thus we want to prove that for $\delta \in (-2,0)$\footnote{In fact, $\delta \in \paren*{\max\set*{-2,-4\lambda},0}$.} and $\epsilon > 0$ sufficiently small, $\Delta_{g_{\epsilon}}: C^{2,\alpha}_{\delta, g_{\epsilon}} \paren*{\Km_{\epsilon}}^{0} \rightarrow C^{0,\alpha}_{\delta - 2, g_{\epsilon}} \paren*{\Km_{\epsilon}}^{0}$ is a bounded linear isomorphism whose operator norm of the inverse is bounded by a constant \textit{independent of the parameter $\epsilon > 0$}.

First, we record the result of a computation that will be extensively used in the sequel. It follows from the induced metric on $\paren*{p,q}$-tensors, the equation for the Christoffel symbols $\Gamma^{k}_{ij} = \frac{1}{2} g^{pk} \paren*{\frac{\del g_{pj}}{\del x^{i}} + \frac{\del g_{pi}}{\del x^{j}} - \frac{\del g_{ij}}{\del x^{p}}} $, the local coordinate expression $\Delta_{g} f = - \frac{1}{\sqrt{\det g}} \frac{\del}{\del x^{a}} \paren*{\sqrt{\det g} \cdot g^{ab} \frac{\del f}{\del x^{b}}} $, and the fact that $\lambda \dist_{g}\paren*{p,q} = \dist_{\lambda^{2}g}\paren*{p,q}$ for any $\lambda>0$:

\begin{prop}\label{scalingproperties} Let $\paren*{M^{n},g}$ be an arbitrary Riemannian manifold. Let $C>0$ be a constant. Let $\rho$ be a smooth nonnegative function, let $\delta\in\bbR$, let $\fs$ be a $\paren*{p,q}$-tensor, let $\abs*{\cdot}_{g}$ denote all higher order tensor norms WRT the Riemannian metric $g$, and recall $\Abs*{\nabla_{g}^{j} \fs}_{C^{0}_{\delta,g}} \coloneq \Abs*{\rho^{-\delta + j}\nabla_{g}^{j} \fs}_{C^{0}_{g}}$.\begin{enumerate}
\item\label{scalingproperties covariant derivative rescale} Under a conformal rescaling $g\mapsto C^{2}g$, we have that $$\nabla_{C^{2}g} = \nabla_{g}$$

\item\label{scalingproperties pointwise tensor norm rescale} Under a conformal rescaling $g\mapsto C^{2}g$, we have that $$\abs*{\nabla^{j}\fs}_{C^{2}g} = \paren*{\frac{1}{C}}^{j+q - p}\abs*{\nabla^{j}\fs}_{g}$$

\item\label{scalingproperties scalar Laplacian rescale} Under a conformal rescaling $g\mapsto C^{2}g$, we have that $$\Delta_{C^{2}g} = \frac{1}{C^{2}} \Delta_{g}$$

\item\label{scalingproperties weighted H\"{o}lder tensor seminorm rescale} Denote by $\sqparen*{\fs}_{C^{k,\alpha}_{\delta,g}} \coloneq\sup_{\substack{x \neq y\\ d_{g}\paren*{x,y} < \sigma_{g}\min\set*{\rho\paren*{x},\rho\paren*{y}}}}\min\set*{\rho(x),\rho(y)}^{-\delta + k + \alpha} \frac{\abs*{\restr{\nabla_{g}^{k}\fs}{x} - \restr{\nabla_{g}^{k}\fs}{y}}_{g}}{d_{g}\paren*{x , y}^{\alpha}}$ the weighted H\"{o}lder seminorm on $\paren*{p,q}$-tensors (with all higher order tensor norms for $k \geq 1$ defined with respect to $g$). Immediately from (\ref{scalingproperties covariant derivative rescale}) and (\ref{scalingproperties pointwise tensor norm rescale}), under a conformal rescaling $g\mapsto C^{2}g$: $$ \sqparen*{\fs}_{C^{k,\alpha}_{\delta,C^{2}g}} = \paren*{\frac{1}{C}}^{\delta + q-p}\sqparen*{\fs}_{C^{k,\alpha}_{\delta,g}}$$ with the weight of the RHS $C^{k,\alpha}_{\delta,g}$ being $\frac{\rho}{C}$, i.e.
\begin{al}
\sup_{\substack{x \neq y\\ d_{C^{2}g}\paren*{x,y} < \sigma_{g}\min\set*{\rho\paren*{x},\rho\paren*{y}}}}\min\set*{\rho(x),\rho(y)}^{-\delta + k + \alpha} \frac{\abs*{\restr{\nabla_{g}^{k}\fs}{x} - \restr{\nabla_{g}^{k}\fs}{y}}_{C^{2}g}}{d_{C^{2}g}\paren*{x,y}^{\alpha}} \\ = \paren*{\frac{1}{C}}^{\delta + q-p}\sup_{\substack{x \neq y\\ d_{g}\paren*{x,y} < \sigma_{g}\min\set*{\frac{\rho\paren*{x}}{C},\frac{\rho\paren*{y}}{C}}}}\min\set*{\restr{\frac{\rho}{C}}{x},\restr{\frac{\rho}{C}}{y}}^{-\delta + k + \alpha} \frac{\abs*{\restr{\nabla_{g}^{k}\fs}{x} - \restr{\nabla_{g}^{k}\fs}{y}}_{g}}{d_{g}\paren*{x,y}^{\alpha}} 
\end{al} 

\item\label{scalingproperties weighted H\"{o}lder tensor norm rescale} Under a conformal rescaling $g\mapsto C^{2}g$, we have from (\ref{scalingproperties weighted H\"{o}lder tensor seminorm rescale}) that $$\Abs*{\fs}_{C^{k,\alpha}_{\delta,C^{2}g}} = \paren*{\frac{1}{C}}^{\delta + q-p} \Abs*{\fs}_{C^{k,\alpha}_{\delta,g}} \overset{g\leftrightarrow \frac{1}{C^{2}}g}{\Longleftrightarrow} \Abs*{\fs}_{C^{k,\alpha}_{\delta,g}} = \paren*{\frac{1}{C}}^{\delta + q-p} \Abs*{\fs}_{C^{k,\alpha}_{\delta,\frac{1}{C^{2}}g}}$$

where the weight on RHS of each equality is $\frac{\rho}{C}$.

\item\label{scalingproperties standard H\"{o}lder tensor norm rescale} Under a conformal rescaling $g\mapsto C^{2}g$, we have from (\ref{scalingproperties covariant derivative rescale}) and (\ref{scalingproperties pointwise tensor norm rescale}) that for the \textit{unweighted/standard} H\"{o}lder norm that \begin{al}
\Abs*{\fs}_{C^{k,\alpha}_{g}}  \xrightarrow{\text{conformal rescaling }g\mapsto C^{2}g} \Abs*{\fs}_{C^{k,\alpha}_{C^{2}g}} &= \sum_{j \leq k} \paren*{\frac{1}{C}}^{j+q-p}\Abs*{\nabla_{g}^{j} \fs}_{C^{0}_{g}} \\ &+ \paren*{\frac{1}{C}}^{k+q-p + \alpha}\sup_{\substack{x \neq y\\d_{g}\paren*{x,y} < \sigma_{g}\min\set*{\inj_{g}\paren*{x},\inj_{g}\paren*{y}}}} \frac{\abs*{\restr{\nabla_{g}^{k}\fs}{x} - \restr{\nabla_{g}^{k}\fs}{y}}_{g}}{d_{g}\paren*{x,y}^{\alpha}}
\end{al}

\item\label{scalingproperties Schauder Laplacian rescale} Transferring the standard Schauder estimate $\Abs*{f}_{C^{2,\alpha}_{g}} \leq \cC\paren*{\Abs*{f}_{C^{0}_{g}} + \Abs*{\Delta_{g}f}_{C^{0,\alpha}_{g}}}$ for functions/$\paren*{0,0}$-tensors on metric balls $\paren*{B^{g}_{r_{1}}\paren*{p}, B^{g}_{r_{2}}\paren*{p}}$ of radii $\paren*{r_{1},r_{2}}$\footnote{That is, the LHS of the inequality is on $B^{g}_{r_{1}}\paren*{p}$, and the RHS of the inequality is on $B^{g}_{r_{2}}\paren*{p}$, as these are \textit{interior} Schauder estimates.} (clearly with $r_{1} < r_{2}$) under the conformal rescaling $g\mapsto C^{2}g$, we end up getting via (\ref{scalingproperties scalar Laplacian rescale}) and (\ref{scalingproperties standard H\"{o}lder tensor norm rescale}):
\begin{al}
\Abs*{f}_{C^{2,\alpha}_{C^{2}g}} \leq \cC\paren*{\Abs*{f}_{C^{0}_{C^{2}g}} + \Abs*{\Delta_{C^{2}g}f}_{C^{0,\alpha}_{C^{2}g}}}
\end{al}
which is equal to	
\begin{al}
\Abs*{f}_{C^{0}_{g}} &+ \frac{1}{C} \Abs*{\nabla_{g}f}_{C^{0}_{g}} + \paren*{\frac{1}{C}}^{2}\Abs*{\nabla^{2}_{g} f}_{C^{0}_{g}} + \paren*{\frac{1}{C}}^{2 + \alpha}\sqparen*{\nabla^{2}_{g} f}_{C^{0,\alpha}_{g}}\\ &\leq \cC\paren*{\Abs*{f}_{C^{0}_{g}} + \paren*{\frac{1}{C}}^{2} \Abs*{\Delta_{g}f}_{C^{0}_{g}} + \paren*{\frac{1}{C}}^{2 + \alpha} \sqparen*{\Delta_{g} f}_{C^{0,\alpha}_{g}}  }
\end{al} on the \textit{same} metric balls $\paren*{B^{g}_{r_{1}}\paren*{p}, B^{g}_{r_{2}}\paren*{p}} = \paren*{B^{C^{2}g}_{Cr_{1}}\paren*{p}, B^{C^{2}g}_{Cr_{2}}\paren*{p}}$, but with different radii WRT the \textit{new} metric $C^{2}g$ (recall $\lambda \dist_{g}\paren*{p,q} = \dist_{\lambda^{2}g}\paren*{p,q}$ for any $\lambda>0$ and points $p,q$).

Conversely, if we already had the standard Schauder estimate for $C^{2}g$, namely $\Abs*{f}_{C^{2,\alpha}_{C^{2}g}} \leq \cC\paren*{\Abs*{f}_{C^{0}_{C^{2}g}} + \Abs*{\Delta_{C^{2}g}f}_{C^{0,\alpha}_{C^{2}g}}}$ on $\paren*{B^{C^{2}g}_{Cr_{1}}\paren*{p}, B^{C^{2}g}_{Cr_{2}}\paren*{p}}$, then unraveling it would yield the same estimate above on $\paren*{B^{g}_{r_{1}}\paren*{p}, B^{g}_{r_{2}}\paren*{p}} = \paren*{B^{C^{2}g}_{Cr_{1}}\paren*{p}, B^{C^{2}g}_{Cr_{2}}\paren*{p}}$.

\end{enumerate}
\end{prop}

Next, we record some useful \textit{uniform} estimates on our weight function $\rho_{\epsilon} \coloneq G_{\epsilon} \circ \abs*{\cdot}^{g_{0}}_{\pi^{-1}\paren*{\cS}}:\Km_{\epsilon} \rightarrow [\epsilon,1]$ as well as on the weight function $\wtilde{\rho_{0}}$ used on $T^{*}\CP^{1}$, all of which follow immediately by construction:

\begin{prop}\label{preciseweightfunction}						
The following hold \textit{uniformly in $\epsilon > 0$, i.e. $\forall \epsilon > 0$}:\begin{enumerate}[nosep]
\itemsep0em 
\item\label{preciseweightfunction first region} Restricted to $\set*{\abs*{\cdot}^{g_{0}}_{\pi^{-1}\paren*{\cS}} \leq 2\epsilon}$, we have the \textit{pointwise everywhere/uniform} estimate $$\epsilon \leq \rho_{\epsilon} \leq 2 \epsilon$$
\item\label{preciseweightfunction past 3/25 region} Restricted to $\set*{\frac{3}{25} < \abs*{\cdot}^{g_{0}}_{\pi^{-1}\paren*{\cS}}}$, we have that $\rho_{\epsilon}$ \textit{pointwise everywhere/uniformly} satisfies $$\frac{3}{25}\leq \rho_{\epsilon}\leq 1$$\end{enumerate}\end{prop}

\begin{prop}\label{preciseweightfunction EH almost constancy}
Recall $\wtilde{\rho_{0}}: T^{*}\CP^{1} \rightarrow [1,\infty)$ the smooth monotone function satisfying $\wtilde{\rho_{0}} = \begin{cases}
\abs*{\cdot}^{g_{0}}_{\CP^{1}} & \text{ on }\set*{2 \leq \abs*{\cdot}^{g_{0}}_{\CP^{1}}}\\
1 & \text{ on }\set*{\abs*{\cdot}^{g_{0}}_{\CP^{1}} \leq 1}
\end{cases}$. Let $R_{0} > 1$ be some fixed constant, and let $p \in \set*{1 < \abs*{\cdot}^{g_{0}}_{\CP^{1}} < R_{0}}$ be an \textit{arbitrary} point. 

Then $\exists \delta_{2,p} > 0$ depending on $p$ such that, restricted to $\set*{\abs*{p}^{g_{0}}_{\CP^{1}} - \delta_{2,p}\leq \abs*{\cdot}^{g_{0}}_{\CP^{1}} \leq \abs*{p}^{g_{0}}_{\CP^{1}} + \delta_{2,p}}$, we have the \textit{pointwise everywhere/uniform} estimate $$\wtilde{\rho_{0}}(p) - \frac{1}{1000}\wtilde{\rho_{0}}(p) \leq \wtilde{\rho_{0}} \leq \wtilde{\rho_{0}}(p) + \frac{1}{1000}\wtilde{\rho_{0}}(p)$$ That is, we have that $$\frac{999}{1000} \leq \frac{\wtilde{\rho_{0}}(q)}{\wtilde{\rho_{0}}(p)} \leq \frac{1001}{1000},\qqfa q \in \set*{\abs*{p}^{g_{0}}_{\CP^{1}} - \delta_{2,p}\leq \abs*{\cdot}^{g_{0}}_{\CP^{1}} \leq \abs*{p}^{g_{0}}_{\CP^{1}} + \delta_{2,p}}$$
\end{prop}

Now we arrive at the main isomorphism result in 3 parts: first a weighted Schauder estimate (which is first done locally), then an improved weighted Schauder estimate (which is the heart of the proof), and then the isomorphism result itself.

\begin{prop}[Local Weighted Schauder Estimate]\label{localWSE} Let $\epsilon > 0$ be from Data \ref{constant epsilon 1-16}. Then $\forall f \in C^{2,\alpha}_{\delta, g_{\epsilon}}\paren*{\Km_{\epsilon}}$ and $\forall p \in \Km_{\epsilon}$, there exists an $r = r_{p} > 0$ (hence may depend on $p$) such that: 
\begin{al}
\Abs*{f}_{C^{2,\alpha}_{\delta, g_{\epsilon}} \paren*{B^{g_{\epsilon}}_{r_{p}}\paren*{p}}} &\leq \cC_{1}\paren*{\Abs*{f}_{C^{0}_{\delta, g_{\epsilon}} \paren*{B^{g_{\epsilon}}_{2r_{p}}\paren*{p}}} + \Abs*{\Delta_{g_{\epsilon}}f}_{C^{0,\alpha}_{\delta - 2, g_{\epsilon}} \paren*{B^{g_{\epsilon}}_{2r_{p}}\paren*{p}}}}\\
&\coloneq \cC_{1}\paren*{\Abs*{\rho_{\epsilon}^{-\delta}f}_{C^{0}_{g_{\epsilon}} \paren*{B^{g_{\epsilon}}_{2r_{p}}\paren*{p}}} + \Abs*{\Delta_{g_{\epsilon}}f}_{C^{0,\alpha}_{\delta - 2, g_{\epsilon}} \paren*{B^{g_{\epsilon}}_{2r_{p}}\paren*{p}}}}
\end{al}
for $\cC_{1} > 0$ \textit{independent of $\epsilon$ and $p\in \Km_{\epsilon}$.}\end{prop}

\begin{proof}

The idea here is that, \textit{depending on which region our point $p$ is in}, we do a specific conformal rescaling on the metric (as well as rescale the function $f \mapsto kf$ by a positive scalar), use the fact that each (rescaled) region has \textbf{bounded geometry} to get \textit{standard} Schauder estimates with \textit{uniform} constants, then finish off with our precise uniform estimates on our weight functions (Proposition \ref{preciseweightfunction} and Proposition \ref{preciseweightfunction EH almost constancy}) via undoing the conformal rescaling \& examining how the resulting standard Schauder estimates transform via Proposition \ref{scalingproperties}. Since we do this on \textit{finitely} many regions, we take the max of all the uniform constants $\cC>0$ involved to get a global constant $\cC_{1}$. Thus let $p \in \Km_{\epsilon}$. We consider 3 cases: $p \in \set*{\abs*{\cdot}^{g_{0}}_{\pi^{-1}\paren*{\cS}} < 2\epsilon}$, $p \in \set*{\epsilon < \abs*{\cdot}^{g_{0}}_{\pi^{-1}\paren*{\cS}} < \frac{1}{8}}$, and $p \in \set*{\frac{3}{25} < \abs*{\cdot}^{g_{0}}_{\pi^{-1}\paren*{\cS}}}$. Clearly these 3 regions cover all of $\Km_{\epsilon}$.

\begin{enumerate}[wide, labelwidth=!, labelindent=0pt]
\item \textbf{Case 1; $p \in \set*{\abs*{\cdot}^{g_{0}}_{\pi^{-1}\paren*{\cS}} < 2\epsilon}$:} WLOG let us focus on only 1 of the 16 connected components of $\set*{\abs*{\cdot}^{g_{0}}_{\pi^{-1}\paren*{\cS}} < 2\epsilon}$, hence suppose that $\pi^{-1}\paren*{\cS} = E$ consists of only a single exceptional $E \cong \CP^{1}$. By construction (Proposition \ref{approximatemetricproperties} as well as Data \ref{constant epsilon 1-16} telling us that $2\epsilon < \epsilon^{\lambda}$), on this region $g_{\epsilon} = g_{EH,\epsilon^{2}}$. Recalling that $R_{\epsilon}^{*}\paren*{\frac{1}{\epsilon^{2}}g_{EH,\epsilon^{2}}} = g_{EH,1}$ from Proposition \ref{EHproperties} as well as the identification $\set*{\abs*{\cdot}^{g_{0}}_{\pi^{-1}\paren*{\cS}} \leq 3\epsilon^{\lambda}} \sim \set*{\abs*{\cdot}^{g_{0}}_{\CP^{1}} \leq 3\epsilon^{\lambda}}$ from the underlying construction of $\Km_{\epsilon}$, we have that upon doing a conformal rescaling $g_{\epsilon} \mapsto \frac{1}{\epsilon^{2}}g_{\epsilon}$ that $\paren*{\set*{\abs*{\cdot}^{g_{0}}_{\pi^{-1}\paren*{\cS}} < 2\epsilon}, \frac{1}{\epsilon^{2}} g_{\epsilon}}$ is isometric under $R_{\epsilon}$ to $\paren*{\set*{\abs*{\cdot}^{g_{0}}_{\CP^{1}} < 2}, g_{EH,1}} \subset \paren*{T^{*}\CP^{1}, g_{EH,1}}$, which is a precompact region of the Eguchi-Hanson space $\paren*{T^{*}\CP^{1},g_{EH,1}}$. Since $\paren*{T^{*}\CP^{1},g_{EH,1}}$ has \textbf{bounded geometry}\footnote{Indeed, the exceptional $\CP^{1}$ is no longer collapsing since we're in $g_{EH,1}$. Hence, roughly speaking, the conformal rescaling $g_{\epsilon} \mapsto \frac{1}{\epsilon^{2}}g_{\epsilon}$ ``rescaled the shrinking exceptional divisor to have unit volume'', whence dispensing with the only source of unbounded curvature (namely the shrinking exceptional divisor).}, we thus have, for $r_{2} > r_{1} > 0$ small enough so that $B^{g_{EH,1}}_{r_{2}}\paren*{p} \subset \paren*{\set*{\abs*{\cdot}^{g_{0}}_{\CP^{1}} < 2}, g_{EH,1}}$ (hence both $r_{1}, r_{2}$ may depend on $p$) where we abuse notation and denote the point where $p$ lands in $\set*{\abs*{\cdot}^{g_{0}}_{\CP^{1}} < 2}$ as $p$, the standard Schauder estimate $\Abs*{kf}_{C^{2,\alpha}_{g_{EH,1}}} \leq \cC_{0}\paren*{\Abs*{kf}_{C^{0}_{g_{EH,1}}} + \Abs*{k\Delta_{g_{EH,1}}f}_{C^{0,\alpha}_{g_{EH,1} } }}$ on the metric balls $\paren*{B^{g_{EH,1}}_{r_{1}}\paren*{p}, B^{g_{EH,1}}_{r_{2}}\paren*{p}}$ (where we have multiplied our function by a positive scalar $k > 0$, to be determined), \textit{where $\cC_{0}> 0$ is independent of $p$}.

Transferring everything back to $\paren*{\set*{\abs*{\cdot}^{g_{0}}_{\pi^{-1}\paren*{\cS}} < 2\epsilon}, \frac{1}{\epsilon^{2}} g_{\epsilon}}$ via applying $R_{\epsilon}$ to $\paren*{\set*{\abs*{\cdot}^{g_{0}}_{\CP^{1}} < 2}, g_{EH,1}}$ and using $\frac{1}{\epsilon^{2}}g_{EH,\epsilon^{2}} = R_{\frac{1}{\epsilon}}^{*}\paren*{g_{EH,1}}$ and the identification $\set*{\abs*{\cdot}^{g_{0}}_{\pi^{-1}\paren*{\cS}} \leq 3\epsilon^{\lambda}} \sim \set*{\abs*{\cdot}^{g_{0}}_{\CP^{1}} \leq 3\epsilon^{\lambda}}$, we have that our standard Schauder estimate is now $\Abs*{kf}_{C^{2,\alpha}_{\frac{1}{\epsilon^{2}}g_{\epsilon}}} \leq \cC_{0}\paren*{\Abs*{kf}_{C^{0}_{\frac{1}{\epsilon^{2}}g_{\epsilon}}} + \Abs*{k\Delta_{\frac{1}{\epsilon^{2}}g_{\epsilon}}f}_{C^{0,\alpha}_{\frac{1}{\epsilon^{2}}g_{\epsilon} } }}$ on the metric balls $\paren*{B^{\frac{1}{\epsilon^{2}}g_{\epsilon}}_{R_{1}}\paren*{p}, B^{\frac{1}{\epsilon^{2}}g_{\epsilon}}_{R_{2}}\paren*{p}}$ for some $0 < R_{1} < R_{2}$ (which may depend on $p$) such that $B^{\frac{1}{\epsilon^{2}}g_{\epsilon}}_{R_{2}}\paren*{p} \subset \set*{\abs*{\cdot}^{g_{0}}_{\pi^{-1}\paren*{\cS}} < 2\epsilon}$ and $B^{\frac{1}{\epsilon^{2}}g_{\epsilon}}_{R_{2}}\paren*{p} \subset R_{\epsilon} \paren*{B^{g_{EH,1}}_{r_{2}}\paren*{p}} $ (and vice versa for $R_{1}$). Now apply Proposition \ref{scalingproperties}(\ref{scalingproperties Schauder Laplacian rescale}) with $C = \frac{1}{\epsilon}$ and $g = g_{\epsilon}$, set $k = \epsilon^{-\delta}$, and upon noting that on $\set*{\abs*{\cdot}^{g_{0}}_{\pi^{-1}\paren*{\cS}} < 2\epsilon}$ we have the \textit{uniform} estimate $\epsilon \leq \rho_{\epsilon} \leq 2\epsilon$ (Proposition \ref{preciseweightfunction}(\ref{preciseweightfunction first region})), we get upon recalling the definition of the weighted H\"{o}lder norms our sought after \textit{weighted} Schauder estimate $\Abs*{f}_{C^{2,\alpha}_{\delta, g_{\epsilon}}} \leq \wtilde{\cC_{0}}\paren*{\Abs*{f}_{C^{0}_{\delta, g_{\epsilon}}} + \Abs*{\Delta_{g_{\epsilon}}f}_{C^{0,\alpha}_{\delta - 2, g_{\epsilon} } }}$ on the metric balls $\paren*{B^{\frac{1}{\epsilon^{2}}g_{\epsilon}}_{R_{1}}\paren*{p}, B^{\frac{1}{\epsilon^{2}}g_{\epsilon}}_{R_{2}}\paren*{p}} = \paren*{B^{g_{\epsilon}}_{\epsilon R_{1}}\paren*{p}, B^{g_{\epsilon}}_{\epsilon R_{2}}\paren*{p}}$ upon setting $r_{p} \coloneq \epsilon R_{1}$ and $R_{2} \coloneq 2 R_{1}$. Note that $\wtilde{\cC_{0}} > 0$ is still independent of $p$ \textit{as well as $\epsilon > 0$} since $\epsilon \leq \rho_{\epsilon} \leq 2\epsilon$ tells us that we may replace $\rho_{\epsilon}$ with $\epsilon$, up to introducing harmless factors of $2$ and $\frac{1}{2}$ into $\cC_{0}$.

\item \textbf{Case 2; $p \in \set*{\epsilon < \abs*{\cdot}^{g_{0}}_{\pi^{-1}\paren*{\cS}} < \frac{1}{8}}$:} Again WLOG let us focus only on 1 of the 16 connected components of $\set*{\epsilon < \abs*{\cdot}^{g_{0}}_{\pi^{-1}\paren*{\cS}} < \frac{1}{8}}$. Hence $g_{\epsilon}$ is an interpolation of $g_{EH,\epsilon^{2}}$ and $g_{0}$ on $\set*{\epsilon < \abs*{\cdot}^{g_{0}}_{\pi^{-1}\paren*{\cS}} < \frac{1}{8}}$, namely $$g_{\epsilon} = \begin{cases}
\wtilde{g_{EH, \epsilon}} & \text{on } \set*{\abs*{\cdot}^{g_{0}}_{\pi^{-1}\paren*{\cS}} \leq 3\epsilon^{\lambda}}\\
g_{0} & \text{on } \set*{ 3\epsilon^{\lambda} \leq \abs*{\cdot}^{g_{0}}_{\pi^{-1}\paren*{\cS}}} 
\end{cases}	$$ and where $\wtilde{g_{EH, \epsilon}} = \begin{cases}
g_{EH,\epsilon^{2}} & \text{on } \set*{\abs*{\cdot}^{g_{0}}_{\pi^{-1}\paren*{\cS}} \leq \epsilon^{\lambda}} \\
g_{0} & \text{on }\set*{2\epsilon^{\lambda} \leq \abs*{\cdot}^{g_{0}}_{\pi^{-1}\paren*{\cS}}}
\end{cases}$.

Now by \textit{Remark} \ref{identification}, we may identify $\set*{\epsilon < \abs*{\cdot}^{g_{0}}_{\pi^{-1}\paren*{\cS}} < \frac{1}{8}}$ with $\set*{\epsilon < \abs*{\cdot}^{g_{0}}_{\CP^{1}} < \frac{1}{8}}\subset T^{*}\CP^{1}$, and upon recalling $\wtilde{g_{EH,\epsilon}} = \epsilon^{2} R^{*}_{\frac{1}{\epsilon}} \paren*{\lwhat{g_{EH-0, \epsilon}}} \Leftrightarrow R_{\epsilon}^{*}\paren*{\frac{1}{\epsilon^{2}}\wtilde{g_{EH,\epsilon}}} = \lwhat{g_{EH-0, \epsilon}}$ and $g_{0} = \epsilon^{2}R^{*}_{\frac{1}{\epsilon}} \paren*{g_{0}}\Leftrightarrow R^{*}_{\epsilon}\paren*{\frac{1}{\epsilon^{2}}g_{0}} = g_{0}$ from Proposition \ref{preglueEH} directly implying by construction of $\lwhat{g_{EH-0, \epsilon}}$ from Proposition \ref{preglueEH BG interpolation!} (as well as by construction of $g_{\epsilon}$) that $R^{*}_{\epsilon}\paren*{\frac{1}{\epsilon^{2}}g_{\epsilon}} = \lwhat{g_{EH-0, \epsilon}}$, we have that upon doing a conformal rescaling $g_{\epsilon} \mapsto \frac{1}{\epsilon^{2}}g_{\epsilon}$ that $\paren*{\set*{\epsilon < \abs*{\cdot}^{g_{0}}_{\pi^{-1}\paren*{\cS}} < \frac{1}{8}}, \frac{1}{\epsilon^{2}} g_{\epsilon}}$ is isometric under $R_{\epsilon}$ to $\paren*{\set*{1 < \abs*{\cdot}^{g_{0}}_{\CP^{1}} < \frac{1}{8\epsilon}}, \lwhat{g_{EH-0, \epsilon}}} \subset \paren*{T^{*}\CP^{1}, \lwhat{g_{EH-0, \epsilon}}}$, which is a precompact region of the Riemannian manifold $\paren*{T^{*}\CP^{1},\lwhat{g_{EH-0, \epsilon}}}$. Now by Proposition \ref{preglueEH BG interpolation!}, since $\lwhat{g_{EH-0, \epsilon}} = \begin{cases}
g_{EH,1} & \text{on } \set*{\abs*{\cdot}^{g_{0}}_{\CP^{1}} \leq \frac{1}{\epsilon^{\lambda}} } \\
g_{0} & \text{on }\set*{\frac{2}{\epsilon^{\lambda}} \leq \abs*{\cdot}^{g_{0}}_{\CP^{1}}}
\end{cases}$ aka $\lwhat{g_{EH-0, \epsilon}}$ is an interpolation of $g_{EH,1}$ on a large compact region and with the flat Euclidean metric $g_{0}$ on the complement of a (slightly) larger (pre)compact region, it is immediate that the Riemannian manifold $\paren*{T^{*}\CP^{1},\lwhat{g_{EH-0, \epsilon}}}$ is complete with \textbf{bounded geometry}\footnote{Indeed, similar to the previous case for the Eguchi-Hanson space $\paren*{T^{*}\CP^{1}, g_{EH,1}}$ having bounded geometry, because we're flat outside $\set*{\abs*{\cdot}^{g_{0}}_{\CP^{1}} \leq \frac{2}{\epsilon^{\lambda}} }$ and Eguchi-Hanson $g_{EH,1}$ with parameter $1$ on $\set*{\abs*{\cdot}^{g_{0}}_{\CP^{1}} \leq \frac{1}{\epsilon^{\lambda}} }$ and interpolating in between, we may argue as before (namely the exceptional $\CP^{1}$ no longer shrinking) since there are no sources of unbounded curvature for either the flat Euclidean metric region (since its the Euclidean metric) or the interpolation annulus (since we're interpolating the Euclidean metric with $g_{EH,1}$ and $g_{EH,1}$ has bounded curvature everywhere, which follows from Proposition \ref{preglueEH BG interpolation!} and the same reasoning as \textit{Remark} \ref{curvatureestimates + K3 sigma}).}.

Next, note that upon transferring over the weight function $\rho_{\epsilon}$ via the above process, we end up getting $R^{*}_{\epsilon}\paren*{\rho_{\epsilon}}$, and since this satisfies $R^{*}_{\epsilon}\paren*{\rho_{\epsilon}} = \begin{cases}
\epsilon\abs*{\cdot}^{g_{0}}_{\CP^{1}} & \text{ on }\set*{2 \leq \abs*{\cdot}^{g_{0}}_{\CP^{1}}< \frac{1}{8\epsilon}}\\
\epsilon & \text{ on }\set*{\abs*{\cdot}^{g_{0}}_{\CP^{1}} \leq 1}
\end{cases}$, we immediately see that $R^{*}_{\epsilon}\paren*{\frac{1}{\epsilon}\rho_{\epsilon}} = \wtilde{\rho_{0}}$ on $\set*{1 < \abs*{\cdot}^{g_{0}}_{\CP^{1}} < \frac{1}{8\epsilon}}$, the weight function used to defined the weighted H\"{o}lder spaces on $T^{*}\CP^{1}$.

Now recall that our point $p$ is now $p \in \set*{1 < \abs*{\cdot}^{g_{0}}_{\CP^{1}} < \frac{1}{8\epsilon}}$, where we abuse notation and denote the point where $p$ lands in $\set*{1 < \abs*{\cdot}^{g_{0}}_{\CP^{1}} < \frac{1}{8\epsilon}}$ as $p$. Let $\delta_{2,p}> 0$ be the same constant as in Proposition \ref{preciseweightfunction EH almost constancy} but shrunken so that we have that $\set*{\abs*{p}^{g_{0}}_{\CP^{1}} - \delta_{2,p}\leq \abs*{\cdot}^{g_{0}}_{\CP^{1}} \leq \abs*{p}^{g_{0}}_{\CP^{1}} + \delta_{2,p}} \subset \set*{1 < \abs*{\cdot}^{g_{0}}_{\CP^{1}} < \frac{1}{8\epsilon}}$ by openness. Now let $\sigma \in \paren*{0,1}$ be chosen such that, upon letting $R \coloneq \wtilde{\rho_{0}}(p)$ for notation's sake, we have that $B_{2\sigma R}^{\lwhat{g_{EH-0, \epsilon}}}(p) \subset \set*{\abs*{p}^{g_{0}}_{\CP^{1}} - \delta_{2,p}\leq \abs*{\cdot}^{g_{0}}_{\CP^{1}} \leq \abs*{p}^{g_{0}}_{\CP^{1}} + \delta_{2,p}} \subset \set*{1 < \abs*{\cdot}^{g_{0}}_{\CP^{1}} < \frac{1}{8\epsilon}}$. Therefore, by Proposition \ref{preciseweightfunction EH almost constancy}, we have that $\frac{999}{1000} \leq \frac{\wtilde{\rho_{0}}(q)}{\wtilde{\rho_{0}}(p)} = \frac{\wtilde{\rho_{0}}(q)}{R} \leq \frac{1001}{1000}$ holds $\forall q \in B_{2\sigma R}^{\lwhat{g_{EH-0, \epsilon}}}(p)$. 

Now because $\paren*{T^{*}\CP^{1},\lwhat{g_{EH-0, \epsilon}}}$ has bounded geometry and we're on a precompact region $\paren*{\set*{1 < \abs*{\cdot}^{g_{0}}_{\CP^{1}} < \frac{1}{8\epsilon}}, \lwhat{g_{EH-0, \epsilon}}}$, we have the standard Schauder estimate $\Abs*{kf}_{C^{2,\alpha}_{\lwhat{g_{EH-0, \epsilon}}}} \leq \cK \paren*{ \Abs*{kf}_{C^{0}_{\lwhat{g_{EH-0, \epsilon}}}} + \Abs*{k \Delta_{\lwhat{g_{EH-0, \epsilon}}}f}_{C^{0,\alpha}_{\lwhat{g_{EH-0, \epsilon}}}} }$ on the metric balls $\paren*{B^{\lwhat{g_{EH-0, \epsilon}}}_{\sigma R}(p), B^{\lwhat{g_{EH-0, \epsilon}}}_{2\sigma R}(p)}$ (where we have multiplied our function by a positive scalar $k > 0$, to be determined), \textit{where $\cK> 0$ is independent of $p$}.

Performing a conformal rescaling $\lwhat{g_{EH-0, \epsilon}} \mapsto \frac{1}{R^{2}}\lwhat{g_{EH-0, \epsilon}}$ and then applying Proposition \ref{scalingproperties}(\ref{scalingproperties Schauder Laplacian rescale}) with $C = \frac{1}{R}$ and $g = \lwhat{g_{EH-0, \epsilon}}$ gives us \begin{al}
\Abs*{kf}_{C^{0}_{\lwhat{g_{EH-0, \epsilon}}}} &+ R \Abs*{\nabla_{\lwhat{g_{EH-0, \epsilon}}}kf}_{C^{0}_{\lwhat{g_{EH-0, \epsilon}}}} + R^{2}\Abs*{\nabla^{2}_{\lwhat{g_{EH-0, \epsilon}}} kf}_{C^{0}_{\lwhat{g_{EH-0, \epsilon}}}} + R^{2 + \alpha}\sqparen*{\nabla^{2}_{\lwhat{g_{EH-0, \epsilon}}} kf}_{C^{0,\alpha}_{\lwhat{g_{EH-0, \epsilon}}}}\\ &\leq \cK\paren*{\Abs*{kf}_{C^{0}_{\lwhat{g_{EH-0, \epsilon}}}} + R^{2} \Abs*{\Delta_{\lwhat{g_{EH-0, \epsilon}}}kf}_{C^{0}_{\lwhat{g_{EH-0, \epsilon}}}} + R^{2 + \alpha} \sqparen*{\Delta_{\lwhat{g_{EH-0, \epsilon}}} kf}_{C^{0,\alpha}_{\lwhat{g_{EH-0, \epsilon}}}}  }
\end{al} on the metric balls $\paren*{B^{\lwhat{g_{EH-0, \epsilon}}}_{\sigma R}(p), B^{\lwhat{g_{EH-0, \epsilon}}}_{2\sigma R}(p)} = \paren*{B^{\frac{1}{R^{2}}\lwhat{g_{EH-0, \epsilon}}}_{\sigma }(p), B^{\frac{1}{R^{2}}\lwhat{g_{EH-0, \epsilon}}}_{2\sigma }(p)}$. 

Now recalling that $\frac{999}{1000} \leq \frac{\wtilde{\rho_{0}}(q)}{R} \leq \frac{1001}{1000}$ holds $\forall q \in B_{2\sigma R}^{\lwhat{g_{EH-0, \epsilon}}}(p)$ and hence $\forall q \in B_{\sigma R}^{\lwhat{g_{EH-0, \epsilon}}}(p)$ as well, we get upon recalling the definition of the weighted H\"{o}lder norms the following weighted Schauder estimate $\Abs*{R^{\delta}kf}_{C^{2,\alpha}_{\delta, \lwhat{g_{EH-0, \epsilon}}}} \leq \wtilde{\cC_{1}} \paren*{ \Abs*{R^{\delta}kf}_{C^{0}_{\delta, \lwhat{g_{EH-0, \epsilon}}}} + \Abs*{R^{\delta}k \Delta_{\lwhat{g_{EH-0, \epsilon}}}f}_{C^{0,\alpha}_{\delta-2, \lwhat{g_{EH-0, \epsilon}}}} }$ on the metric balls $\paren*{B^{\lwhat{g_{EH-0, \epsilon}}}_{\sigma R}(p), B^{\lwhat{g_{EH-0, \epsilon}}}_{2\sigma R}(p)} = \paren*{B^{\frac{1}{R^{2}}\lwhat{g_{EH-0, \epsilon}}}_{\sigma }(p), B^{\frac{1}{R^{2}}\lwhat{g_{EH-0, \epsilon}}}_{2\sigma }(p)}$, where the weight function here is $\wtilde{\rho_{0}}$. Similarly as in the first case, $\wtilde{\cC_{1}}> 0$ is still independent of all the parameters involved because we've only modified $\cK$ by factors of $\frac{999}{1000}$ and $\frac{1001}{1000}$ since we've applied the uniform estimate $\frac{999}{1000} \leq \frac{\wtilde{\rho_{0}}(q)}{R} \leq \frac{1001}{1000}$ on those metric balls.

Now let us transfer everything back to $\paren*{\set*{\epsilon < \abs*{\cdot}^{g_{0}}_{\pi^{-1}\paren*{\cS}} < \frac{1}{8}}, \frac{1}{\epsilon^{2}} g_{\epsilon}}$ via applying $R_{\epsilon}$ to $\paren*{\set*{1 < \abs*{\cdot}^{g_{0}}_{\CP^{1}} < \frac{1}{8\epsilon}}, \lwhat{g_{EH-0, \epsilon}}}$ and using the identification $\set*{\epsilon < \abs*{\cdot}^{g_{0}}_{\pi^{-1}\paren*{\cS}} < \frac{1}{8}}\sim \set*{\epsilon < \abs*{\cdot}^{g_{0}}_{\CP^{1}} < \frac{1}{8}}\subset T^{*}\CP^{1}$, as well as using the relation $\frac{1}{\epsilon^{2}}g_{\epsilon} = R^{*}_{\frac{1}{\epsilon}} \paren*{\lwhat{g_{EH-0, \epsilon}}}$ and the relation $\frac{1}{\epsilon}\rho_{\epsilon} = R^{*}_{\frac{1}{\epsilon}}\paren*{\wtilde{\rho_{0}}}$. We get that our weighted Schauder estimate has now become $\Abs*{R^{\delta}kf}_{C^{2,\alpha}_{\delta, \frac{1}{\epsilon^{2}}g_{\epsilon}}} \leq \wtilde{\cC_{1}} \paren*{ \Abs*{R^{\delta}kf}_{C^{0}_{\delta, \frac{1}{\epsilon^{2}}g_{\epsilon}}} + \Abs*{R^{\delta}k \Delta_{\frac{1}{\epsilon^{2}}g_{\epsilon}}f}_{C^{0,\alpha}_{\delta-2, \frac{1}{\epsilon^{2}}g_{\epsilon} }} }$ on the metric balls $\paren*{B^{\frac{1}{\epsilon^{2}}g_{\epsilon}}_{R_{1}}(p), B^{\frac{1}{\epsilon^{2}}g_{\epsilon}}_{R_{2}}(p)}$ for some $0 < R_{1} < R_{2}$ (which may depend on $p$) such that $B^{\frac{1}{\epsilon^{2}}g_{\epsilon}}_{R_{2}}(p) \subset \set*{\epsilon < \abs*{\cdot}^{g_{0}}_{\pi^{-1}\paren*{\cS}} < \frac{1}{8}}$ and $B^{\frac{1}{\epsilon^{2}}g_{\epsilon}}_{R_{2}}(p) \subset R_{\epsilon}\paren*{B^{\lwhat{g_{EH-0, \epsilon}}}_{2\sigma R}(p)}$ (and vice versa for $R_{1}$), \textit{where the weight function used here is $\frac{1}{\epsilon}\rho_{\epsilon}$}. Now we undo the original conformal rescaling of $\frac{1}{\epsilon^{2}}g_{\epsilon}$ via performing a conformal rescaling $\frac{1}{\epsilon^{2}}g_{\epsilon} \mapsto \epsilon^{2}\frac{1}{\epsilon^{2}}g_{\epsilon} = g_{\epsilon}$ \textit{on the metric balls $\paren*{B^{\frac{1}{\epsilon^{2}}g_{\epsilon}}_{R_{1}}(p), B^{\frac{1}{\epsilon^{2}}g_{\epsilon}}_{R_{2}}(p)}$ hence yielding $\paren*{B^{\frac{1}{\epsilon^{2}}g_{\epsilon}}_{R_{1}}\paren*{p}, B^{\frac{1}{\epsilon^{2}}g_{\epsilon}}_{R_{2}}\paren*{p}} = \paren*{B^{g_{\epsilon}}_{\epsilon R_{1}}\paren*{p}, B^{g_{\epsilon}}_{\epsilon R_{2}}\paren*{p}}$}, applying Proposition \ref{scalingproperties}(\ref{scalingproperties weighted H\"{o}lder tensor norm rescale}) with $C = \epsilon$ and $g = \frac{1}{\epsilon^{2}}g_{\epsilon}$ on each of the 3 weighted H\"{o}lder norm terms \textit{where $g$ denotes the metric used in the weighted H\"{o}lder norm}, and then setting $k = \epsilon^{\delta}R^{-\delta}$, yielding the weighted Schauder estimate $\Abs*{f}_{C^{2,\alpha}_{\delta, g_{\epsilon}}} \leq \wtilde{\cC_{1}} \paren*{ \Abs*{f}_{C^{0}_{\delta, g_{\epsilon}}} + \Abs*{\Delta_{\frac{1}{\epsilon^{2}}g_{\epsilon}}f}_{C^{0,\alpha}_{\delta-2, g_{\epsilon} }} }$ on the metric balls $\paren*{B^{\frac{1}{\epsilon^{2}}g_{\epsilon}}_{R_{1}}(p), B^{\frac{1}{\epsilon^{2}}g_{\epsilon}}_{R_{2}}(p)}= \paren*{B^{g_{\epsilon}}_{\epsilon R_{1}}\paren*{p}, B^{g_{\epsilon}}_{\epsilon R_{2}}\paren*{p}}$. Now Proposition \ref{scalingproperties}(\ref{scalingproperties scalar Laplacian rescale}) tells us that $\Delta_{\frac{1}{\epsilon^{2}} g_{\epsilon}} = \epsilon^{2} \Delta_{g_{\epsilon}}$, and since Data \ref{constant epsilon 1-16} tells us that $\epsilon^{2} < \frac{1}{16} < 1$, we finally get our sought after weighted Schauder estimate $\Abs*{f}_{C^{2,\alpha}_{\delta, g_{\epsilon}}} \leq \wtilde{\cC_{1}} \paren*{ \Abs*{f}_{C^{0}_{\delta, g_{\epsilon}}} + \Abs*{\Delta_{g_{\epsilon}}f}_{C^{0,\alpha}_{\delta-2, g_{\epsilon} }} }$ on the metric balls $\paren*{B^{\frac{1}{\epsilon^{2}}g_{\epsilon}}_{R_{1}}\paren*{p}, B^{\frac{1}{\epsilon^{2}}g_{\epsilon}}_{R_{2}}\paren*{p}} = \paren*{B^{g_{\epsilon}}_{\epsilon R_{1}}\paren*{p}, B^{g_{\epsilon}}_{\epsilon R_{2}}\paren*{p}}$ upon setting $r_{p} \coloneq \epsilon R_{1}$ and $R_{2} \coloneq 2 R_{1}$.

\item \textbf{Case 3; $p \in \set*{\frac{3}{25} < \abs*{\cdot}^{g_{0}}_{\pi^{-1}\paren*{\cS}}}$:} By Data \ref{constant epsilon 1-16}, since $2\epsilon^{\lambda} < \frac{3}{25}$, by construction of $g_{\epsilon}$ (Proposition \ref{approximatemetricproperties}) we have that on this region $g_{\epsilon} = g_{0}$ the flat metric. This, combined with Proposition \ref{preciseweightfunction}(\ref{preciseweightfunction past 3/25 region}) which tells us that our weight function satisfies on this region the \textit{uniform} estimate $\frac{3}{25} \leq \rho_{\epsilon} \leq 1$, the proof of the weighted Schauder estimate is exactly as in the first case, but this time using the identification $\set*{\frac{3}{25} < \abs*{\cdot}^{g_{0}}_{\pi^{-1}\paren*{\cS}}} \sim \set*{\frac{3}{25} < \abs*{\cdot}^{g_{0}}_{\cS}} \subset \bbT^{4} - \cS\subset \bbT^{4}$ from \textit{Remark} \ref{identification} and without having to do any conformal rescaling.

In more detail, upon transferring everything over from $\paren*{\set*{\frac{3}{25} < \abs*{\cdot}^{g_{0}}_{\pi^{-1}\paren*{\cS}}}, g_{\epsilon}}$ to (the obviously precompact region) $\paren*{\set*{\frac{3}{25} < \abs*{\cdot}^{g_{0}}_{\cS}}, g_{0}}\subset \paren*{\bbT^{4}, g_{0}}$, and since flat tori clearly have bounded geometry, we thus automatically have the standard Schauder estimate $\Abs*{f}_{C^{2,\alpha}_{g_{0}}} \leq \cC_{1.5} \paren*{ \Abs*{f}_{C^{0}_{g_{0}}} + \Abs*{\Delta_{g_{0}}f}_{C^{0,\alpha}_{g_{0} }} }$ on the metric balls $\paren*{B^{g_{0}}_{R_{1}}\paren*{p}, B^{g_{0}}_{R_{2}}\paren*{p}}$ where $0 < R_{1} < R_{2}$ is chosen such that $B^{g_{0}}_{R_{2}}\paren*{p} \subset \set*{\frac{3}{25} < \abs*{\cdot}^{g_{0}}_{\cS}}$, and \textit{where $\cC_{1.5}> 0$ is independent of $p$}. Undoing the identification $\set*{\frac{3}{25} < \abs*{\cdot}^{g_{0}}_{\pi^{-1}\paren*{\cS}}} \sim \set*{\frac{3}{25} < \abs*{\cdot}^{g_{0}}_{\cS}} $ gives us that same estimate over $\set*{\frac{3}{25} < \abs*{\cdot}^{g_{0}}_{\pi^{-1}\paren*{\cS}}}$ on the metric balls $\paren*{B^{g_{0}}_{R_{1}}\paren*{p}, B^{g_{0}}_{R_{2}}\paren*{p}}$ (which are clearly preserved under the identification). Since we have our uniform estimate $\frac{3}{25} \leq \rho_{\epsilon} \leq 1$ on $\set*{\frac{3}{25} < \abs*{\cdot}^{g_{0}}_{\pi^{-1}\paren*{\cS}}}$, we have on any subset of $\set*{\frac{3}{25} < \abs*{\cdot}^{g_{0}}_{\pi^{-1}\paren*{\cS}}}$ that the weighted H\"{o}lder norms WRT $\rho_{\epsilon}$ and the standard H\"{o}lder norms are \textit{equivalent norms}, whence (up to changing the constant $\cC_{1.5}$ by harmless factors) we immediately get our sought after weighted Schauder estimate $\Abs*{f}_{C^{2,\alpha}_{\delta, g_{\epsilon}}} \leq \wtilde{\cC_{1.5}} \paren*{ \Abs*{f}_{C^{0}_{\delta, g_{\epsilon}}} + \Abs*{\Delta_{g_{\epsilon}}f}_{C^{0,\alpha}_{\delta-2, g_{\epsilon} }} }$ on the metric balls $\paren*{B^{g_{\epsilon}}_{R_{1}}\paren*{p}, B^{g_{\epsilon}}_{R_{2}}\paren*{p}}$ upon setting $r_{p} \coloneq R_{1}$ and $R_{2} \coloneq 2 R_{1}$, with $\wtilde{\cC_{1.5}}>0$ still independent of all parameters involved.

\end{enumerate} Setting $\cC_{1} \coloneq \max\set*{\wtilde{\cC_{0}},\wtilde{\cC_{1}},\wtilde{\cC_{1.5}}}$ and noting that these $3$ constants are both \textit{independent of $p$} as well as \textit{independent of $\epsilon > 0$} yields us our local weighted Schauder estimate $\forall p \in \Km_{\epsilon}$, as was to be shown.\end{proof}

\begin{cor}[Weighted Schauder Estimate]\label{WSE} Let $\epsilon > 0$ be from Data \ref{constant epsilon 1-16} and sufficiently small.
We have the following estimate $\forall f \in C^{2,\alpha}_{\delta, g_{\epsilon}}\paren*{\Km_{\epsilon}}$:
\begin{al}
\Abs*{f}_{C^{2,\alpha}_{\delta, g_{\epsilon}} \paren*{\Km_{\epsilon}}} &\leq \cC_{2}\paren*{\Abs*{f}_{C^{0}_{\delta, g_{\epsilon}} \paren*{\Km_{\epsilon}}} + \Abs*{\Delta_{g_{\epsilon}}f}_{C^{0,\alpha}_{\delta - 2, g_{\epsilon}} \paren*{\Km_{\epsilon}}}}\\
&\coloneq \cC_{2}\paren*{\Abs*{\rho_{\epsilon}^{-\delta}f}_{C^{0}_{g_{\epsilon}} \paren*{\Km_{\epsilon}}} + \Abs*{\Delta_{g_{\epsilon}}f}_{C^{0,\alpha}_{\delta - 2, g_{\epsilon}} \paren*{\Km_{\epsilon}}}}
\end{al}
for $\cC_{2} > 0$ \textit{independent of $\epsilon$.}
\end{cor}

\begin{proof}

Suppose for contradiction that the estimate fails. This implies by countable choice that $\forall i \in \bbN$, $\exists \epsilon_{i} \in \paren*{0,\frac{1}{i}}$ whence $\exists\epsilon_{i}\searrow 0$ and $\exists f_{i} \in  C^{2,\alpha}_{\delta, g_{\epsilon_{i}}}\paren*{\Km_{\epsilon_{i}}}$ a sequence such that (after normalizing) \begin{al}
\Abs*{f_{i}}_{C^{2,\alpha}_{\delta, g_{\epsilon_{i}}} \paren*{\Km_{\epsilon_{i}}}} &= 1\\
\Abs*{f_{i}}_{C^{0}_{\delta, g_{\epsilon_{i}}} \paren*{\Km_{\epsilon_{i}}}} + \Abs*{\Delta_{g_{\epsilon_{i}}}f_{i}}_{C^{0,\alpha}_{\delta - 2, g_{\epsilon_{i}}} \paren*{\Km_{\epsilon_{i}}}} &< \frac{1}{i}
\end{al} Now since $\Abs*{f_{i}}_{C^{2,\alpha}_{\delta, g_{\epsilon_{i}}} \paren*{\Km_{\epsilon_{i}}}} \coloneq \Abs*{f_{i}}_{C^{0}_{\delta, g_{\epsilon_{i}}} \paren*{\Km_{\epsilon_{i}}}} + \Abs*{\nabla_{g_{\epsilon_{i}}} f_{i}}_{C^{0}_{\delta, g_{\epsilon_{i}}} \paren*{\Km_{\epsilon_{i}}}} + \Abs*{\nabla^{2}_{g_{\epsilon_{i}}} f_{i}}_{C^{0}_{\delta, g_{\epsilon_{i}}} \paren*{\Km_{\epsilon_{i}}}} + \sqparen*{ f_{i}}_{C^{2,\alpha}_{\delta, g_{\epsilon_{i}}} \paren*{\Km_{\epsilon_{i}}}} = 1$ and $\Abs*{f_{i}}_{C^{0}_{\delta, g_{\epsilon_{i}}} \paren*{\Km_{\epsilon_{i}}}}\leq \frac{1}{i}$, we have by a simple inequality bash that either $\frac{1-\frac{1}{i}}{3} \leq \Abs*{\nabla_{g_{\epsilon_{i}}} f_{i}}_{C^{0}_{\delta, g_{\epsilon_{i}}} \paren*{\Km_{\epsilon_{i}}}}$ or $\frac{1-\frac{1}{i}}{3} \leq \Abs*{\nabla^{2}_{g_{\epsilon_{i}}} f_{i}}_{C^{0}_{\delta, g_{\epsilon_{i}}} \paren*{\Km_{\epsilon_{i}}}}$ or $\frac{1-\frac{1}{i}}{3} \leq \sqparen*{ f_{i}}_{C^{2,\alpha}_{\delta, g_{\epsilon_{i}}} \paren*{\Km_{\epsilon_{i}}}}$ holds $\forall i \in \bbN$, aka since $1-\frac{1}{i}$ is strictly monotone increasing in $i\in\bbN$ we have either $\frac{1}{6} \leq \Abs*{\nabla_{g_{\epsilon_{i}}} f_{i}}_{C^{0}_{\delta, g_{\epsilon_{i}}} \paren*{\Km_{\epsilon_{i}}}}$ or $\frac{1}{6} \leq \Abs*{\nabla^{2}_{g_{\epsilon_{i}}} f_{i}}_{C^{0}_{\delta, g_{\epsilon_{i}}} \paren*{\Km_{\epsilon_{i}}}}$ or $\frac{1}{6} \leq \sqparen*{ f_{i}}_{C^{2,\alpha}_{\delta, g_{\epsilon_{i}}} \paren*{\Km_{\epsilon_{i}}}}$ holds $\forall i \in \bbN-\set*{1}$. Hence $\exists p_{i} \in \Km_{\epsilon_{i}}$ and an $r_{p_{i}} > 0$ such that either $\frac{1}{6} \leq \abs*{\rho_{\epsilon_{i}}(p_{i})^{-\delta + 1}\nabla_{g_{\epsilon_{i}}} f_{i}(p_{i})}_{g_{\epsilon_{i}}}$ or $\frac{1}{6} \leq\abs*{\rho_{\epsilon_{i}}(p_{i})^{-\delta + 2}\nabla^{2}_{g_{\epsilon_{i}}} f_{i}(p_{i})}_{g_{\epsilon_{i}}}$ or $\frac{1}{6} \leq \sqparen*{ f_{i}}_{C^{2,\alpha}_{\delta, g_{\epsilon_{i}}} \paren*{B^{g_{\epsilon}}_{r_{p_{i}}}\paren*{p_{i}}}}$ holds $\forall i \in \bbN-\set*{1}$. But then we have by our local estimate Proposition \ref{localWSE} that $\Abs*{f_{i}}_{C^{2,\alpha}_{\delta, g_{\epsilon_{i}}} \paren*{B^{g_{\epsilon_{i}}}_{r_{p_{i}}}\paren*{p_{i}}}} \leq \cC_{1}\paren*{\Abs*{f_{i}}_{C^{0}_{\delta, g_{\epsilon_{i}}} \paren*{B^{g_{\epsilon_{i}}}_{2r_{i}}\paren*{p_{i}}}} + \Abs*{\Delta_{g_{\epsilon_{i}}}f_{i}}_{C^{0,\alpha}_{\delta - 2, g_{\epsilon_{i}}} \paren*{B^{g_{\epsilon_{i}}}_{2r_{p_{i}}}\paren*{p_{i}}}}}\leq \cC_{1}\paren*{\Abs*{f_{i}}_{C^{0}_{\delta, g_{\epsilon_{i}}} \paren*{\Km_{\epsilon_{i}}}} + \Abs*{\Delta_{g_{\epsilon_{i}}}f_{i}}_{C^{0,\alpha}_{\delta - 2, g_{\epsilon_{i}}} \paren*{\Km_{\epsilon_{i}}}}}\leq \frac{\cC_{1}}{i}\searrow 0$ as $i\nearrow\infty$. But then we would have either $\frac{1}{6} \leq \abs*{\rho_{\epsilon_{i}}(p_{i})^{-\delta + 1}\nabla_{g_{\epsilon_{i}}} f_{i}(p_{i})}_{g_{\epsilon_{i}}}\leq \Abs*{f_{i}}_{C^{2,\alpha}_{\delta, g_{\epsilon_{i}}} \paren*{B^{g_{\epsilon_{i}}}_{r_{p_{i}}}\paren*{p_{i}}}}\searrow 0$ or $\frac{1}{6} \leq\abs*{\rho_{\epsilon_{i}}(p_{i})^{-\delta + 2}\nabla^{2}_{g_{\epsilon_{i}}} f_{i}(p_{i})}_{g_{\epsilon_{i}}}\leq \Abs*{f_{i}}_{C^{2,\alpha}_{\delta, g_{\epsilon_{i}}} \paren*{B^{g_{\epsilon_{i}}}_{r_{p_{i}}}\paren*{p_{i}}}}\searrow 0$ or $\frac{1}{6} \leq \sqparen*{ f_{i}}_{C^{2,\alpha}_{\delta, g_{\epsilon_{i}}} \paren*{B^{g_{\epsilon}}_{r_{p_{i}}}\paren*{p_{i}}}}\leq \Abs*{f_{i}}_{C^{2,\alpha}_{\delta, g_{\epsilon_{i}}} \paren*{B^{g_{\epsilon_{i}}}_{r_{p_{i}}}\paren*{p_{i}}}}\searrow 0$ as $i\nearrow \infty$, aka a contradiction in all 3 cases.\end{proof}

We now want to remove the $C^{0}$ term on the RHS of our above weighted Schauder estimate. This allows us to get a \textit{uniform, i.e. independent of $\epsilon > 0$, bound on the inverse of $\Delta_{g_{\epsilon}}$}:

\begin{prop}[Improved Weighted Schauder Estimate]\label{IWSE} Let $\epsilon > 0$ be as from Data \ref{constant epsilon 1-16}. Let $\lambda<\frac{2}{3}$. \textit{Let the rate/weight parameter satisfy $\boxed{\delta \in \paren*{\max\set*{-2, -4\lambda}, 0}}$.}

Then for $\epsilon > 0$ sufficiently small there exists a constant $\cC_{3} > 0$ \textit{independent of $\epsilon>0$} such that we have the following estimate for $\Delta_{g_{\epsilon}}: C^{2,\alpha}_{\delta, g_{\epsilon}} \paren*{\Km_{\epsilon}}^{0} \rightarrow C^{0,\alpha}_{\delta-2, g_{\epsilon}} \paren*{\Km_{\epsilon}}^{0}$:
$$\Abs*{f}_{C^{2,\alpha}_{\delta, g_{\epsilon}} \paren*{\Km_{\epsilon}}} \leq \cC_{3} \Abs*{\Delta_{g_{\epsilon}}f}_{C^{0,\alpha}_{\delta - 2, g_{\epsilon}} \paren*{\Km_{\epsilon}}},\qqfa f \in C^{2,\alpha}_{\delta, g_{\epsilon}} \paren*{\Km_{\epsilon}}^{0}$$

Note that $f \in C^{2,\alpha}_{\delta, g_{\epsilon}} \paren*{\Km_{\epsilon}}^{0}$ i.e. the estimate is restricted to the closed subspace of functions with integral zero, i.e. $\int_{\Km_{\epsilon}} f \omega_{\epsilon}^{2} = 0$.
\end{prop}

\begin{remark}\label{closedrangedetector}
This inequality precisely means that when $\delta \in \paren*{\max\set*{-2, -4\lambda}, 0}$ and the gluing parameter $\epsilon > 0$ is sufficiently small, $\Delta_{g_{\epsilon}}: C^{2,\alpha}_{\delta, g_{\epsilon}} \paren*{\Km_{\epsilon}}^{0} \rightarrow C^{0,\alpha}_{\delta-2, g_{\epsilon}} \paren*{\Km_{\epsilon}}^{0}$ is injective, \textit{uniformly} bounded below, and has \textit{closed range}.
\end{remark}

\begin{proof} 
We prove this by contradiction via a ``blow-up'' analysis dichotomy. By our weighted Schauder estimate proved above (Corollary \ref{WSE}), we may instead prove $$\Abs*{f}_{C^{0}_{\delta, g_{\epsilon}} \paren*{\Km_{\epsilon}}} \leq \cC_{4}\Abs*{\Delta_{g_{\epsilon}}f}_{C^{0,\alpha}_{\delta - 2, g_{\epsilon}} \paren*{\Km_{\epsilon}}}$$ with the same assumptions ($\delta \in \paren*{\max\set*{-2, -4\lambda}, 0}$, etc.), whence $\cC_{3}\coloneq \cC_{2}\paren*{\cC_{4}+ 1}$. Now assume for contradiction that the statement for this latter inequality fails. Then this implies by countable choice that $\exists \epsilon_{i} \in \paren*{0,\frac{1}{i}}$ whence $\exists \epsilon_{i} \searrow 0$ and $\exists f_{i} \in  C^{2,\alpha}_{\delta, g_{\epsilon_{i}}} \paren*{\Km_{\epsilon_{i}}}^{0}$ a sequence such that (after normalizing) we get a tuple of equations \begin{al}
\Abs*{f_{i}}_{C^{0}_{\delta, g_{\epsilon_{i}}} \paren*{\Km_{\epsilon_{i}}}} &= 1\\
\Abs*{\Delta_{g_{\epsilon_{i}}}f_{i}}_{C^{0,\alpha}_{\delta - 2, g_{\epsilon_{i}}} \paren*{\Km_{\epsilon_{i}}}} &\leq \frac{1}{i}
\end{al} Note that our weighted Schauder estimate Corollary \ref{WSE} immediately gives us $$\Abs*{f_{i}}_{C^{2,\alpha}_{\delta, g_{\epsilon_{i}}} \paren*{\Km_{\epsilon_{i}}}} \leq 2\cC_{2}$$

Let us now examine the following two cases:

\begin{enumerate}[wide, labelwidth=!, labelindent=0pt]
\item \textbf{Case 1; bulk region:} Firstly, from \textit{Remark} \ref{identification} let us identify each region $\set*{\epsilon_{i}^{\lambda} < \abs*{\cdot}^{g_{0}}_{\pi^{-1}\paren*{\cS}}} \subset \Km_{\epsilon_{i}}$ with $\set*{\epsilon_{i}^{\lambda} < \abs*{\cdot}^{g_{0}}_{\cS}} \subset \bbT^{4}/\bbZ_{2} - \cS$ (hence we're in a \textit{fixed} ambient manifold). Similarly from that remark, let us denote the lift of $\set*{\epsilon_{i}^{\lambda} < \abs*{\cdot}^{g_{0}}_{\cS}} \subset \bbT^{4}/\bbZ_{2} - \cS$ to $\bbT^{4} - \cS$ via the double cover as $\lwhat{\set*{\epsilon_{i}^{\lambda} < \abs*{\cdot}^{g_{0}}_{\cS}}} \subset \bbT^{4} - \cS$, and denote the resulting lift of the metric $g_{\epsilon_{i}}$ as $\what{g_{\epsilon_{i}}}$. Now since both $\set*{\epsilon_{i}^{\lambda} < \abs*{\cdot}^{g_{0}}_{\cS}}$, $\lwhat{\set*{\epsilon_{i}^{\lambda} < \abs*{\cdot}^{g_{0}}_{\cS}}}$ are in \textit{fixed} ambient manifolds $\bbT^{4}/\bbZ_{2} - \cS$, $\bbT^{4} - \cS$ respectively, and since $\bbZ_{2}$ acts on $\bbT^{4} - \cS$ and on each region $\lwhat{\set*{\epsilon_{i}^{\lambda} < \abs*{\cdot}^{g_{0}}_{\cS}}} \subset \bbT^{4} - \cS$ in the same way, by construction (Proposition \ref{approximatemetricproperties}) we have that for all compact subsets $K\Subset \bbT^{4} - \cS$, $\exists i \gg 1$ such that $\forall j \geq i$ we have that $K \Subset \lwhat{\set*{\epsilon_{j}^{\lambda} < \abs*{\cdot}^{g_{0}}_{\cS}}}$ and $\restr{\what{g_{\epsilon_{j}}}}{K} \xrightarrow{C^{\infty}_{g_{0}}} \restr{g_{0}}{K}$.

Abuse notation and denote by $f_{i}$ to mean \textit{both} the restriction of each $f_{i} \in C^{2,\alpha}_{\delta, g_{\epsilon_{i}}} \paren*{\Km_{\epsilon_{i}}}^{0}$ to $\set*{\epsilon_{i}^{\lambda} < \abs*{\cdot}^{g_{0}}_{\pi^{-1}\paren*{\cS}}}$ as well as the transferal of $f_{i}$ to $\set*{\epsilon_{i}^{\lambda} < \abs*{\cdot}^{g_{0}}_{\cS}} \subset \bbT^{4}/\bbZ_{2} - \cS$, and denote the lift of each $f_{i}$ to $\lwhat{\set*{\epsilon_{i}^{\lambda} < \abs*{\cdot}^{g_{0}}_{\cS}}}$ as $\what{f_{i}}$. Hence $\what{f_{i}}$ is even/$\bbZ_{2}$-invariant by Proposition \ref{CoveringSpacesInvariantForms}.

Now since we have $\Abs*{f_{i}}_{C^{2,\alpha}_{\delta, g_{\epsilon_{i}}} \paren*{\Km_{\epsilon_{i}}}} \leq 2\cC_{2}$, we thus get $\Abs*{f_{i}}_{C^{2,\alpha}_{\delta, g_{\epsilon_{i}}} \paren*{\set*{\epsilon_{i}^{\lambda} < \abs*{\cdot}^{g_{0}}_{\cS}}}} \leq 2\cC_{2}$ and hence the same bound for each $\what{f_{i}}$, namely $$\Abs*{\what{f_{i}}}_{C^{2,\alpha}_{\delta, \what{g_{\epsilon_{i}}}} \paren*{\lwhat{\set*{\epsilon_{i}^{\lambda} < \abs*{\cdot}^{g_{0}}_{\cS}}}}} \leq 2\cC_{2}$$ where we also lift the weight function $\rho_{\epsilon_{i}}$ to $\lwhat{\rho_{\epsilon_{i}}}$.

Therefore, by Arzel\`{a}-Ascoli and sending $i\nearrow \infty$, upon relabeling the original sequence of $f_{i}$ we have that \begin{gat}
\what{f_{i}}\xrightarrow{C^{2,\alpha'}_{g_{0}, loc}}\what{f_{\infty}}\\
\Abs*{\what{f_{\infty}}}_{C^{2,\alpha}_{\delta, g_{0}} \paren*{\bbT^{4} - \cS}} \leq 2\cC_{2}
\end{gat} with $0 < \alpha' < \alpha < 1$, for some function $\what{f_{\infty}}$ on $\bbT^{4} - \cS$. Here, the $loc$ convergence means on compact subsets, i.e. the convergence $\xrightarrow{C^{2,\alpha'}_{g_{0}, loc}}$ denotes $\Abs*{\what{f_{i}} - \what{f_{\infty}}}_{C^{2,\alpha'}_{g_{0}} \paren*{K}} \searrow 0$ hence implies $\Abs*{\what{f_{i}}}_{C^{2,\alpha'}_{g_{0}} \paren*{K}} \to \Abs*{\what{f_{\infty}}}_{C^{2,\alpha'}_{g_{0}} \paren*{K}}$, holding for all compact subsets $K \Subset \bbT^{4} - \cS$. Moreover, since everything immediately works $\bbZ_{2}$-equivariantly, we have that $\what{f_{\infty}}$ is even/$\bbZ_{2}$-invariant and hence descends down to a function $f_{\infty}$ on $\bbT^{4}/\bbZ_{2} - \cS$ by Proposition \ref{CoveringSpacesInvariantForms}. Moreover, the weight function used in $C^{2,\alpha}_{\delta, g_{0}} \paren*{\bbT^{4} - \cS}$ is precisely the limiting weight function which satisfies $$\rho_{0} = \begin{cases}
1 & \text{ on }\set*{\frac{1}{5}\leq \abs*{\cdot}^{g_{0}}_{\cS}}\\
\abs*{\cdot}^{g_{0}}_{\cS} & \text{ on }\set*{0<\abs*{\cdot}^{g_{0}}_{\cS} \leq \frac{1}{8}}
\end{cases},\qquad\abs*{\nabla \rho_{0}}\leq \frac{35C_{1}}{3}\text{ on }\set*{\frac{1}{8}\leq \abs*{\cdot}^{g_{0}}_{\cS}\leq \frac{1}{5}}$$ aka in the limit we get that $\lim_{i\nearrow \infty}\lwhat{\rho_{\epsilon_{i}}} = \rho_{0}$ is precisely the weight function used to define $C^{k,\alpha}_{\delta,g_{0}} \paren*{\bbT^{4} - \cS}$ back in Section \ref{Weighted Setup}. 

Playing the same game with $\Abs*{\Delta_{g_{\epsilon_{i}}}f_{i}}_{C^{0,\alpha}_{\delta - 2, g_{\epsilon_{i}}} \paren*{\Km_{\epsilon_{i}}}} \leq \frac{1}{i}$, namely that this bound directly implies the bounds $\Abs*{\Delta_{g_{\epsilon_{i}}}f_{i}}_{C^{0,\alpha}_{\delta - 2, g_{\epsilon_{i}}} \paren*{\set*{\frac{1}{2}\epsilon_{i}^{\frac{1}{2}} < \abs*{\cdot}^{g_{0}}_{\cS}}}} \leq \frac{1}{i}$ and $\Abs*{\Delta_{\what{g_{\epsilon_{i}}}}\what{f_{i}}}_{C^{0,\alpha}_{\delta - 2, \what{g_{\epsilon_{i}}}} \paren*{\lwhat{\set*{\frac{1}{2}\epsilon_{i}^{\frac{1}{2}} < \abs*{\cdot}^{g_{0}}_{\cS}}}}} \leq \frac{1}{i}$, that we may take the limit as $i\nearrow \infty$ and get $$\Delta_{g_{0}} \what{f_{\infty}} = 0$$ on $\paren*{\bbT^{4} - \cS,g_{0}}$ (since $\Delta_{g_{0}} \what{f_{\infty}} = 0$ on all compact subsets of $\bbT^{4} - \cS$, hence on all of $\bbT^{4} - \cS$). Now as explained back in Section \ref{Weighted Setup}, $\Abs*{\what{f_{\infty}}}_{C^{2,\alpha}_{\delta, g_{0}} \paren*{\bbT^{4} - \cS}} \coloneq \Abs*{\what{f_{\infty}}}_{C^{0}_{\delta, g_{0}} \paren*{\bbT^{4} - \cS}} + \Abs*{\nabla_{g_{0}}\what{f_{\infty}}}_{C^{0}_{\delta, g_{0}} \paren*{\bbT^{4} - \cS}} + \Abs*{\nabla^{2}_{g_{0}}\what{f_{\infty}}}_{C^{0}_{\delta, g_{0}} \paren*{\bbT^{4} - \cS}} + \sqparen*{ \what{f_{\infty}}}_{C^{2,\alpha}_{\delta, g_{0}} \paren*{\bbT^{4} - \cS}} \leq 2\cC_{2} < \infty$  means that $\what{f_{\infty}}$ satisfies the decay $$\nabla^{j}_{g_{0}}\what{f_{\infty}} = O_{g_{0}}\paren*{\rho_{0}^{\delta - j}}, \qqfa j \leq 2$$ That is, by the above properties of $\rho_{0}$, $\what{f_{\infty}}$ is bounded on $\set*{\frac{1}{8} \leq \abs*{\cdot}^{g_{0}}_{\cS}}$ and satisfies $$\abs*{\nabla^{j}_{g_{0}}\what{f_{\infty}}}_{g_{0}} \leq C_{5,j}\paren*{\abs*{\cdot}^{g_{0}}_{\cS}}^{\delta - j}, \qqfa j \leq 2$$ on $\set*{0 < \abs*{\cdot}^{g_{0}}_{\cS}\leq \frac{1}{8}} \subset \bbT^{4} - \cS$ for some $C_{5,j} > 0$. 

Now $\Delta_{g_{0}} \what{f_{\infty}} = 0$ means that $\what{f_{\infty}}$ is harmonic on $\bbT^{4} - \cS$ WRT the flat metric $g_{0}$ descending from the Euclidean space $\paren*{\bbR^{4},g_{0}}$, hence by elliptic regularity $\what{f_{\infty}} \in C^{\infty}\paren*{\bbT^{4} - \cS}$. Moreover, via lifting $\bbT^{4} - \cS$ via $\bbR^{4} \onto \bbR^{4}/\Lambda \eqcolon \bbT^{4}$ to its fundamental domain in $\bbR^{4}$ (which ends up being a \textit{punctured} open set), we have that near each puncture/point of $\cS$, say on $\set*{0 < \abs*{\cdot}^{g_{0}}_{\cS} < \frac{63}{500}}$, that by virtue of being harmonic WRT $g_{0}$ the flat Euclidean metric, our $\bbZ_{2}$-invariant function $\what{f_{\infty}}$ admits a Laurent series expansion (no need for higher tensor norms since we're a $\paren*{0,0}$-tensor) $$\what{f_{\infty}} = A_{0} + O\paren*{\dist_{g_{0}}\paren*{\cS,\cdot}} + \frac{D_{0}}{\paren*{\dist_{g_{0}}\paren*{\cS,\cdot}}^{2}} + \sum_{m=1}^{\infty}D_{m}\frac{p_{m}\paren*{\dist_{g_{0}}\paren*{\cS,\cdot}}}{\paren*{\dist_{g_{0}}\paren*{\cS,\cdot}}^{2m + 2}}$$ which converges uniformly and absolutely on compact subsets of $\set*{0 < \dist_{g_{0}}\paren*{\cS,\cdot} < \frac{63}{500}}$, where $A_{0}, D_{0}, D_{m} \in \bbR$ are constants and $p_{m}$ is a homogeneous harmonic polynomial of degree $m$, and we recall from \textit{Remark} \ref{identification} that by definition $\abs*{\cdot}^{g_{0}}_{\cS} \coloneq \dist_{g_{0}}\paren*{\cS,\cdot}$.

However, $-2 < \delta < 0$ whence $\dist_{g_{0}}\paren*{\cS,\cdot}^{\delta} < \dist_{g_{0}}\paren*{\cS,\cdot}^{-2}$ holds \textit{on $\set*{0 < \dist_{g_{0}}\paren*{\cS,\cdot} < \frac{63}{500}}$ since $\dist_{g_{0}}\paren*{\cS,\cdot} < \frac{63}{500}$}. Whence, since we have that $\abs*{\what{f_{\infty}}} \leq C_{5,0} \paren*{\dist_{g_{0}}\paren*{\cS,\cdot}}^{\delta}$ on $\set*{0 < \dist_{g_{0}}\paren*{\cS,\cdot} \leq \frac{1}{8}}$ (the $\abs*{\cdot}_{g_{0}}$-norm on $0$-forms is just absolute value), since the above Laurent series expansion holds on $\set*{0 < \dist_{g_{0}}\paren*{\cS,\cdot} < \frac{63}{500}}$ hence holds on $\set*{0 < \dist_{g_{0}}\paren*{\cS,\cdot} \leq \frac{1}{8}} \subset \set*{0 < \dist_{g_{0}}\paren*{\cS,\cdot} < \frac{63}{500}}$, that $\abs*{\what{f_{\infty}}} \leq C_{5,0} \paren*{\dist_{g_{0}}\paren*{\cS,\cdot}}^{\delta}$ forces $D_{0} = 0$ and each $D_{m} = 0$, whence $$\what{f_{\infty}} = A_{0} + O\paren*{\dist_{g_{0}}\paren*{\cS,\cdot}}$$ on $\set*{0 < \dist_{g_{0}}\paren*{\cS,\cdot} < \frac{63}{500}}$.

Whence $\what{f_{\infty}}$ has \textit{removable singularities} at each point of $\cS$, whence extends to a smooth harmonic function $\what{f_{\infty}}$ on the flat $4$-torus $\paren*{\bbT^{4},g_{0}}$. Moreover, we have by our new decay $\what{f_{\infty}} = A_{0} + O\paren*{\dist_{g_{0}}\paren*{\cS,\cdot}}$ for $\what{f_{\infty}}$ on $\set*{\dist_{g_{0}}\paren*{\cS,\cdot} < \frac{63}{500}}$ that $\what{f_{\infty}}$ is bounded on $\set*{\dist_{g_{0}}\paren*{\cS,\cdot} < \frac{63}{500}}$ since $\dist_{g_{0}}\paren*{\cS,\cdot} < \frac{63}{500}$. Whence since $\what{f_{\infty}}$ is bounded on $\set*{\frac{1}{8}\leq \dist_{g_{0}}\paren*{\cS,\cdot}}$, we thus have that $\what{f_{\infty}}$ is a \textit{bounded} harmonic function on the compact 4-manifold without boundary $\paren*{\bbT^{4},g_{0}}$, whence $\what{f_{\infty}}$ is a \textbf{constant} by the maximum principle.

Now recall that each $f_{i} \in C^{2,\alpha}_{\delta, g_{\epsilon_{i}}} \paren*{\Km_{\epsilon_{i}}}^{0}$, whence the sequence of real numbers $\set*{\int_{\Km_{\epsilon_{i}}} f_{i} \omega_{\epsilon_{i}}^{2}} \subset \bbR$ is just the constant sequence $\set*{0}\subset \bbR$. Write each $\int_{\Km_{\epsilon_{i}}} f_{i} \omega_{\epsilon_{i}}^{2} = \int_{\set*{\abs*{\cdot}^{g_{0}}_{\pi^{-1}\paren*{\cS}} \leq \epsilon_{i}^{\lambda}}} f_{i} \omega_{\epsilon_{i}}^{2} + \int_{\set*{\epsilon_{i}^{\lambda} < \abs*{\cdot}^{g_{0}}_{\pi^{-1}\paren*{\cS}}}} f_{i} \omega_{\epsilon_{i}}^{2}$. Hence each $ \int_{\set*{\epsilon_{i}^{\lambda} < \abs*{\cdot}^{g_{0}}_{\pi^{-1}\paren*{\cS}}}} f_{i} \omega_{\epsilon_{i}}^{2} = - \int_{\set*{\abs*{\cdot}^{g_{0}}_{\pi^{-1}\paren*{\cS}} \leq \epsilon_{i}^{\lambda}}} f_{i} \omega_{\epsilon_{i}}^{2}$ from the integral zero assumption.

Since each $\Abs*{f_{i}}_{C^{0}_{\delta, g_{\epsilon_{i}}} \paren*{\Km_{\epsilon_{i}}}} = 1$, by definition of the weighted H\"{o}lder norms we have that each $f_{i} = O\paren*{\rho_{\epsilon_{i}}^{\delta}}$. Moreover, since we have that $\epsilon_{i} \leq \rho_{\epsilon_{i}}$, $\delta < 0$ gives us $\rho_{\epsilon_{i}}^{\delta} \leq \epsilon_{i}^{\delta}$ and hence each $f_{i} = O\paren*{\epsilon_{i}^{\delta}}$. Therefore upon using the identity $\frac{\omega^{2}}{2!} = \dV_{g}$ we have that \begin{al}
\abs*{\int_{\set*{\abs*{\cdot}^{g_{0}}_{\pi^{-1}\paren*{\cS}} \leq \epsilon_{i}^{\lambda}}} f_{i} \omega_{\epsilon_{i}}^{2}} &= 2\abs*{\int_{\set*{\abs*{\cdot}^{g_{0}}_{\pi^{-1}\paren*{\cS}} \leq \epsilon_{i}^{\lambda}}} f_{i} \dV_{g_{\epsilon_{i}}}}\\
&= O\paren*{\epsilon_{i}^{\delta}} \int_{\set*{\abs*{\cdot}^{g_{0}}_{\pi^{-1}\paren*{\cS}} \leq \epsilon_{i}^{\lambda}}} \dV_{g_{\epsilon_{i}}}\\
&= O\paren*{\epsilon_{i}^{\delta}} \Vol_{g_{0}}\paren*{\set*{\abs*{\cdot}^{g_{0}}_{\pi^{-1}\paren*{\cS}} \leq \epsilon_{i}^{\lambda}}}\\
&= O\paren*{\epsilon_{i}^{\delta}} \epsilon_{i}^{4\lambda}\\
&= O\paren*{\epsilon_{i}^{\delta + 4\lambda}}
\end{al} Thus since $-4\lambda < \delta \Longleftrightarrow 0 < \delta +4\lambda$, we therefore have that $\epsilon_{i}^{\delta+ 4\lambda} \searrow 0$ as $i\to \infty$, whence we have that $\int_{\set*{\epsilon_{i}^{\lambda} < \abs*{\cdot}^{g_{0}}_{\pi^{-1}\paren*{\cS}}}} f_{i} \omega_{\epsilon_{i}}^{2} = - \int_{\set*{\abs*{\cdot}^{g_{0}}_{\pi^{-1}\paren*{\cS}} \leq \epsilon_{i}^{\lambda}}} f_{i} \omega_{\epsilon_{i}}^{2}\to 0$. Next, lift to the $2$-fold cover and use the identity $\frac{\omega^{2}}{2!} = \dV_{g}$, whence each $$\int_{\set*{\epsilon_{i}^{\lambda} < \abs*{\cdot}^{g_{0}}_{\pi^{-1}\paren*{\cS}}}} f_{i} \omega_{\epsilon_{i}}^{2} = \int_{\lwhat{\set*{\epsilon_{i}^{\lambda} < \abs*{\cdot}^{g_{0}}_{\cS}}}} \what{f_{i}} \dV_{\what{g_{\epsilon_{i}}}}$$

Now I claim that $\int_{\lwhat{\set*{\epsilon_{i}^{\lambda} < \abs*{\cdot}^{g_{0}}_{\cS}}}} \what{f_{i}} \dV_{\what{g_{\epsilon_{i}}}} \to \int_{\bbT^{4} - \cS} \what{f_{\infty}} \dV_{g_{0}}$. Indeed, this follows immediately, as \textit{each} of our manifolds in question $\lwhat{\set*{\epsilon_{i}^{\lambda} < \abs*{\cdot}^{g_{0}}_{\cS}}}$ are open subsets of $\paren*{\bbT^{4},g_{0}}$ and Proposition \ref{approximatemetricproperties} gives us that each $\dV_{g_{\epsilon_{i}}} = \paren*{1 + O\paren*{\epsilon_{i}^{4-4\lambda}}}\dV_{g_{0}}$ on $\set*{\epsilon_{i}^{\lambda} \leq \abs*{\cdot}^{g_{0}}_{\pi^{-1}\paren*{\cS}}}$. 

Whence since $\cS$ (a finite set of points) has $\cH^{4}$-measure zero, we thus have \begin{al}
\int_{\lwhat{\set*{\epsilon_{i}^{\lambda} < \abs*{\cdot}^{g_{0}}_{\cS}}}} \what{f_{i}} \dV_{\what{g_{\epsilon_{i}}}} &\to \int_{\bbT^{4}} \what{f_{\infty}} \dV_{g_{0}}\\ \int_{\lwhat{\set*{\epsilon_{i}^{\lambda} < \abs*{\cdot}^{g_{0}}_{\cS}}}} \what{f_{i}} \dV_{\what{g_{\epsilon_{i}}}} &\to 0
\end{al} whence $$\int_{\bbT^{4}} \what{f_{\infty}} \dV_{g_{0}} = 0$$ As $\what{f_{\infty}}$ is a constant, we therefore conclude that $\what{f_{\infty}} = 0$.

Thus we have that $\what{f_{i}} \xrightarrow{C^{2,\alpha'}_{g_{0}, loc}} 0$ aka $\Abs*{\what{f_{i}}}_{C^{2,\alpha'}_{g_{0}} \paren*{K}} \searrow 0$ for all compact subsets $K \Subset\bbT^{4} - \cS$. I now claim that $\Abs*{\what{f_{i}}}_{C^{2,\alpha'}_{\delta, \what{g_{\epsilon_{i}}}} \paren*{K}} \searrow 0$ holds for all compact subsets $K \Subset\bbT^{4} - \cS$. Suppose not. Then for some fixed compact subset $K\Subset \bbT^{4} - \cS$, we have that $\Abs*{\what{f_{i}}}_{C^{2,\alpha'}_{\delta, \what{g_{\epsilon_{i}}}} \paren*{K}} \geq \cC_{2.5} > 0$ uniform in $i$. Now since for sufficiently large $i$ we have $K \Subset \lwhat{\set*{\epsilon_{i}^{\lambda} < \abs*{\cdot}^{g_{0}}_{\cS}}} \subset \bbT^{4} - \cS$, pick a sequence of points $p_{i} \in K \cap \lwhat{\set*{\epsilon_{i}^{\lambda} < \abs*{\cdot}^{g_{0}}_{\cS}}}$ such that $p_{i} \to p_{\infty} \in K$ and in which $$ \abs*{\lwhat{\rho_{\epsilon_{i}}}^{-\delta}\paren*{p_{i}} \lwhat{f_{i}}\paren*{p_{i}}} = \lwhat{\rho_{\epsilon_{i}}}^{-\delta}\paren*{p_{i}}\abs*{ \lwhat{f_{i}}\paren*{p_{i}}} \geq \cC_{2.5} > 0$$ But upon taking the limit $i\nearrow\infty$ we get $$\rho_{0}^{-\delta}\paren*{p_{\infty}}\abs*{ \lwhat{f_{\infty}}\paren*{p_{\infty}}} \geq \cC_{2.5} > 0$$ and since $p_{\infty} \in K \Subset \bbT^{4} - \cS$, by compactness of $K$ and continuity of $\rho_{0}$ we thus have that $\rho_{0}^{-\delta}\paren*{p_{\infty}} \neq 0$ and whence we get $$\abs*{ \lwhat{f_{\infty}}\paren*{p_{\infty}}} > 0$$ contradicting $\lwhat{f_{\infty}} = 0$.

Therefore, we have that $\Abs*{\what{f_{i}}}_{C^{2,\alpha'}_{\delta, \what{g_{\epsilon_{i}}}} \paren*{K}} \searrow 0$ for all compact subsets $K \Subset\bbT^{4} - \cS$, whence transferring everything back to $\Km_{\epsilon_{i}}$ by lifting compact subsets of $\set*{0 < \abs*{\cdot}^{g_{0}}_{\pi^{-1}\paren*{\cS}}} = \bbT^{4}/\bbZ_{2} - \cS$ and using $\Abs*{\cdot}_{C^{0}_{\delta, g_{\epsilon_{i}}}} \leq \Abs*{\cdot}_{C^{2,\alpha'}_{\delta,g_{\epsilon_{i}}}}$, we thus have that $$\Abs*{f_{i}}_{C^{0}_{\delta, g_{\epsilon_{i}}} \paren*{K}} \searrow 0$$ holds for all compact subsets $K \Subset \set*{0 < \abs*{\cdot}^{g_{0}}_{\pi^{-1}\paren*{\cS}}} \subset \Km_{\epsilon_{i}}$.

\item \textbf{Case 2; bubble region:} Recall that $\Abs*{f_{i}}_{C^{0}_{\delta, g_{\epsilon_{i}}} \paren*{\Km_{\epsilon_{i}}}} = 1$ yet \textbf{Case 1} above shows $\Abs*{f_{i}}_{C^{0}_{\delta, g_{\epsilon_{i}}} \paren*{K}} \searrow 0$ holds for all compact subsets $K \Subset \set*{0 < \abs*{\cdot}^{g_{0}}_{\pi^{-1}\paren*{\cS}}} \subset \Km_{\epsilon_{i}}$. Since singletons are compact, and $1 = \Abs*{f_{i}}_{C^{0}_{\delta, g_{\epsilon_{i}}} \paren*{\Km_{\epsilon_{i}}}} = \max\set*{\Abs*{f_{i}}_{C^{0}_{\delta, g_{\epsilon_{i}}} \paren*{\set*{\frac{1}{8} < \abs*{\cdot}^{g_{0}}_{\pi^{-1}\paren*{\cS}}}}}, \Abs*{f_{i}}_{C^{0}_{\delta, g_{\epsilon_{i}}} \paren*{\set*{\abs*{\cdot}^{g_{0}}_{\pi^{-1}\paren*{\cS}}\leq \frac{1}{8}}}}}$ we therefore must have that $$\Abs*{f_{i}}_{C^{0}_{\delta, g_{\epsilon_{i}}} \paren*{\set*{\abs*{\cdot}^{g_{0}}_{\pi^{-1}\paren*{\cS}}\leq \frac{1}{8}}}} = 1$$ WLOG we restrict to a single connected component of $\set*{\abs*{\cdot}^{g_{0}}_{\pi^{-1}\paren*{\cS}}\leq \frac{1}{8}}$ in which the above norm $=1$.

By construction and by Proposition \ref{preglueEH}, we have that on $\set*{\abs*{\cdot}^{g_{0}}_{\pi^{-1}\paren*{\cS}} \leq \frac{1}{8}}$ that $g_{\epsilon_{i}} = \wtilde{g_{EH, \epsilon_{i}}}$. Moreover, by Proposition \ref{preglueEH} and Proposition \ref{approximatemetricproperties} we have that $\wtilde{g_{EH,\epsilon_{i}}} = \epsilon_{i}^{2} R^{*}_{\frac{1}{\epsilon_{i}}} \paren*{\lwhat{g_{EH-0, \epsilon_{i}}}} \Longleftrightarrow R_{\epsilon_{i}}^{*}\paren*{\frac{1}{\epsilon_{i}^{2}} \wtilde{g_{EH,\epsilon_{i}}}} = \lwhat{g_{EH-0, \epsilon_{i}}}$. Therefore, just as in proving \textbf{Case 2} of Proposition \ref{localWSE}, upon performing a conformal rescaling $g_{\epsilon_{i}}\mapsto \frac{1}{\epsilon_{i}^{2}}g_{\epsilon_{i}}$, we have that each Riemannian manifold $\paren*{\set*{\abs*{\cdot}^{g_{0}}_{\pi^{-1}\paren*{\cS}} \leq \frac{1}{8}}, \frac{1}{\epsilon_{i}^{2}}g_{\epsilon_{i}}}$ is isometric under $R_{\epsilon_{i}}$ to the Riemannian submanifold $\paren*{\set*{\abs*{\cdot}^{g_{0}}_{\CP^{1}} \leq \frac{1}{8\epsilon_{i}}}, \lwhat{g_{EH-0, \epsilon_{i}}}} \subset \paren*{T^{*}\CP^{1},\lwhat{g_{EH-0, \epsilon_{i}}}}$.

Hence upon taking the limit as $i\nearrow \infty$, we have that, since we're in a \textit{fixed} ambient manifold and upon inspection Proposition \ref{preglueEH BG interpolation!} tells us that $\lwhat{g_{EH-0, \epsilon_{i}}}$ converges to $g_{EH,1}$ on compact subsets aka $\lwhat{g_{EH-0, \epsilon_{i}}} \xrightarrow{C^{\infty}_{g_{EH,1},loc}} g_{EH,1}$, the Riemannian manifolds $\paren*{\set*{\abs*{\cdot}^{g_{0}}_{\CP^{1}} \leq \frac{1}{8\epsilon_{i}}}, \lwhat{g_{EH-0, \epsilon_{i}}}}$ exhausts to $\paren*{T^{*}\CP^{1},g_{EH,1}}$ in the limit.

Transferring over the weight function $\rho_{\epsilon_{i}}$ on $\set*{\abs*{\cdot}^{g_{0}}_{\pi^{-1}\paren*{\cS}} \leq \frac{1}{8}}$ via the above isometry gives us $R^{*}_{\epsilon_{i}}\paren*{\rho_{\epsilon_{i}}}$ which satisfies $R^{*}_{\epsilon_{i}}\paren*{\frac{1}{\epsilon_{i}}\rho_{\epsilon_{i}}} = \begin{cases}
\frac{1}{\epsilon_{i}} & \text{ on }\set*{\frac{1}{5\epsilon_{i}} \leq \abs*{\cdot}^{g_{0}}_{\CP^{1}}}\\
\abs*{\cdot}^{g_{0}}_{\CP^{1}} & \text{ on }\set*{2\leq \abs*{\cdot}^{g_{0}}_{\CP^{1}} \leq \frac{1}{8\epsilon_{i}}}\\
1 & \text{ on }\set*{\abs*{\cdot}^{g_{0}}_{\CP^{1}} \leq 1}
\end{cases}$. Call $\wtilde{\rho_{\epsilon_{i}}} \coloneq R^{*}_{\epsilon_{i}}\paren*{\frac{1}{\epsilon_{i}}\rho_{\epsilon_{i}}}$ for notation.

Therefore, as $\Abs*{f_{i}}_{C^{0}_{\delta, g_{\epsilon_{i}}} \paren*{\Km_{\epsilon_{i}}}} \coloneq \Abs*{\rho_{\epsilon_{i}}^{-\delta}f_{i}}_{C^{0}_{g_{\epsilon_{i}}} \paren*{\Km_{\epsilon_{i}}}} = \Abs*{\rho_{\epsilon_{i}}^{-\delta}\epsilon_{i}^{\delta} \epsilon_{i}^{-\delta}f_{i}}_{C^{0}_{g_{\epsilon_{i}}} \paren*{\Km_{\epsilon_{i}}}}$, we must also scale our functions $f_{i}\mapsto \epsilon_{i}^{-\delta}f_{i}$ to preserve the weighted H\"{o}lder norm upon transferring to $\paren*{\set*{\abs*{\cdot}^{g_{0}}_{\CP^{1}} \leq \frac{1}{8\epsilon_{i}}}, \lwhat{g_{EH-0,\epsilon_{i}}}}$ (after first restricting to $\set*{\abs*{\cdot}^{g_{0}}_{\pi^{-1}\paren*{\cS}} \leq \frac{1}{8}}$). Thus call $\wtilde{f}_{i} \coloneq \epsilon_{i}^{-\delta}f_{i}$ for notation. Note that $\wtilde{f}_{i} = \restr{\wtilde{f}_{i}}{\set*{\abs*{\cdot}^{g_{0}}_{\CP^{1}} \leq \frac{1}{8\epsilon_{i}}}}$ is defined on $\set*{\abs*{\cdot}^{g_{0}}_{\CP^{1}} \leq \frac{1}{8\epsilon_{i}}}$.

Therefore, since Case (\ref{scalingproperties weighted H\"{o}lder tensor norm rescale}) of Proposition \ref{scalingproperties} gives $$\Abs*{f_{i}}_{C^{k,\alpha}_{\delta, g_{\epsilon_{i}}} \paren*{\set*{\abs*{\cdot}^{g_{0}}_{\pi^{-1}\paren*{\cS}}\leq \frac{1}{8}}}} = \Abs*{\wtilde{f_{i}}}_{C^{k,\alpha}_{\delta, \lwhat{g_{EH-0,\epsilon_{i}}}} \paren*{\set*{\abs*{\cdot}^{g_{0}}_{\CP^{1}} \leq \frac{1}{8\epsilon_{i}}}}}$$ with weight functions $\rho_{\epsilon_{i}}$ and $\wtilde{\rho_{\epsilon_{i}}}$ respectively, the inequality $\Abs*{f_{i}}_{C^{2,\alpha}_{\delta, g_{\epsilon_{i}}} \paren*{\Km_{\epsilon_{i}}}} \leq 2\cC_{2}$ implies we have the uniform bound $\Abs*{\wtilde{f_{i}}}_{C^{2,\alpha}_{\delta, \lwhat{g_{EH-0,\epsilon_{i}}}} \paren*{\set*{\abs*{\cdot}^{g_{0}}_{\CP^{1}} \leq \frac{1}{8\epsilon_{i}}}}} \leq 2\cC_{2}$ where again the weight function used is $\wtilde{\rho_{\epsilon_{i}}}$. Whence since $T^{*}\CP^{1} = \bigcup_{i \in \bbN} \set*{\abs*{\cdot}^{g_{0}}_{\CP^{1}} \leq \frac{1}{8\epsilon_{i}}}$ is a nested exhaustion of compact subsets, we therefore have by Arzel\`{a}-Ascoli and the convergence $\lwhat{g_{EH-0,\epsilon_{i}}} \xrightarrow{C^{\infty}_{g_{EH,1},loc}} g_{EH,1}$ that (upon relabeling the original sequence of $f_{i}$ once more) \begin{gat}
\wtilde{f_{i}} \xrightarrow{C^{2,\alpha'}_{g_{EH,1}, loc}} \wtilde{f_{\infty}}\\
\Abs*{\wtilde{f_{\infty}}}_{C^{2,\alpha}_{\delta, g_{EH,1}} \paren*{T^{*}\CP^{1}}} \leq 2\cC_{2}
\end{gat} with $0 < \alpha' < \alpha < 1$, for some function $\wtilde{f_{\infty}}$ on $T^{*}\CP^{1}$ and where the convergence $\xrightarrow{C^{2,\alpha'}_{g_{EH,1}, loc}}$ means that $\Abs*{\wtilde{f_{i}} - \wtilde{f_{\infty}}}_{C^{2,\alpha'}_{g_{EH,1}}\paren*{K}} \searrow 0$ hence $\Abs*{\wtilde{f_{i}}}_{C^{2,\alpha'}_{g_{EH,1}}\paren*{K}} \to \Abs*{\wtilde{f_{\infty}}}_{C^{2,\alpha'}_{g_{EH,1}}\paren*{K}}$, for all compact subsets $K\Subset T^{*}\CP^{1}$. Moreover, the weight function used in $C^{2,\alpha}_{\delta, g_{EH,1}} \paren*{T^{*}\CP^{1}}$ is precisely the limiting weight function which satisfies $\wtilde{\rho_{0}} = \begin{cases}
\abs*{\cdot}^{g_{0}}_{\CP^{1}} & \text{ on }\set*{2\leq \abs*{\cdot}^{g_{0}}_{\CP^{1}} }\\
1 & \text{ on }\set*{\abs*{\cdot}^{g_{0}}_{\CP^{1}} \leq 1}
\end{cases}$, aka in the limit (on compact subsets of $T^{*}\CP^{1}$) we get that $\lim_{i \nearrow \infty} \wtilde{\rho_{\epsilon_{i}}} = \wtilde{\rho_{0}}$ is precisely the weight function used to define the weighted H\"{o}lder spaces for $T^{*}\CP^{1}$.

Playing the same game with $\Abs*{\Delta_{g_{\epsilon_{i}}}f_{i}}_{C^{0,\alpha}_{\delta - 2, g_{\epsilon_{i}}} \paren*{\Km_{\epsilon_{i}}}} \leq \frac{1}{i}$, namely that this bound directly implies the bounds $\Abs*{\Delta_{\lwhat{g_{EH-0,\epsilon_{i}}}}f_{i}}_{C^{0,\alpha}_{\delta - 2, \lwhat{g_{EH-0,\epsilon_{i}}}} \paren*{\set*{\abs*{\cdot}^{g_{0}}_{\CP^{1}} \leq \frac{1}{8\epsilon_{i}}}}} \leq \frac{1}{i}$, that we may take the limit as $i\nearrow \infty$ and get $$\Delta_{g_{EH,1}} \wtilde{f_{\infty}} = 0$$ on $\paren*{T^{*}\CP^{1},g_{EH,1}}$, hence by elliptic regularity $\wtilde{f_{\infty}} \in C^{\infty}\paren*{T^{*}\CP^{1}}$.

Now $\Abs*{\wtilde{f_{\infty}}}_{C^{2,\alpha}_{\delta, g_{EH,1}} \paren*{T^{*}\CP^{1}}} \leq 2\cC_{2}$ means that $\wtilde{f_{\infty}}$ satisfies the decay $$\nabla_{g_{EH,1}}^{j}\wtilde{f_{\infty}} = O_{g_{EH,1}}\paren*{\wtilde{\rho_{0}}^{\delta - j}},\qqfa j\leq 2$$ hence by definition of $\wtilde{\rho_{0}}$, we have that $\wtilde{f_{\infty}}$ is bounded on $\set*{\abs*{\cdot}^{g_{0}}_{\CP^{1}} \leq 2}$ and satisfies $$\abs*{\nabla^{j}_{g_{EH,1}}\wtilde{f_{\infty}}}_{g_{0}} \leq C_{6,j} \paren*{\abs*{\cdot}^{g_{0}}_{\CP^{1}}}^{\delta - j},\qqfa j\leq 2$$ on $\set*{2\leq \abs*{\cdot}^{g_{0}}_{\CP^{1}}}$ for some $C_{6,j} > 0$.

But since $\delta < 0$, we have that $$\abs*{\wtilde{f_{\infty}}} \leq C_{6,0} \paren*{\abs*{\cdot}^{g_{0}}_{\CP^{1}}}^{\delta} \leq C_{6,0} 2^{\delta} < \infty$$ on $\set*{2\leq \abs*{\cdot}^{g_{0}}_{\CP^{1}}}$, and hence $\wtilde{f_{\infty}}$ is bounded on all of $T^{*}\CP^{1}$. Moreover, $\abs*{\wtilde{f_{\infty}}} \leq C_{6,0} \paren*{\abs*{\cdot}^{g_{0}}_{\CP^{1}}}^{\delta}$ and $\delta < 0$ means that $\wtilde{f_{\infty}}\to 0$ as $\abs*{\cdot}^{g_{0}}_{\CP^{1}} \nearrow \infty$. Thus we have that $\wtilde{f_{\infty}}$ is harmonic WRT $g_{EH,1}$, bounded on all of $T^{*}\CP^{1}$, and decays to $0$ at infinity, whence via maximum principle and connectedness of $T^{*}\CP^{1}$, we get $\wtilde{f_{\infty}} = 0$.

Thus we have that $\wtilde{f_{i}} \xrightarrow{C^{2,\alpha'}_{g_{EH,1}, loc}} 0$ aka $\Abs*{\wtilde{f_{i}}}_{C^{2,\alpha'}_{g_{EH,1}} \paren*{K}} \searrow 0$ for all compact subsets $K \Subset T^{*}\CP^{1}$. I now claim that $\Abs*{\wtilde{f_{i}}}_{C^{2,\alpha'}_{\delta, \lwhat{g_{EH-0, \epsilon_{i}}}} \paren*{K}} \searrow 0$ holds for all compact subsets $K \Subset T^{*}\CP^{1}$. Suppose not. Then for some fixed compact subset $K\Subset T^{*}\CP^{1}$, we have that $\Abs*{\wtilde{f_{i}}}_{C^{2,\alpha'}_{\delta, \lwhat{g_{EH-0, \epsilon_{i}}}} \paren*{K}} \geq \cC_{2.5} > 0$ uniform in $i$. Now since for sufficiently large $i$ we have $K \Subset \set*{\abs*{\cdot}^{g_{0}}_{\CP^{1}} \leq \frac{1}{8\epsilon_{i}}} \subset T^{*}\CP^{1}$, pick a sequence of points $p_{i} \in K \cap \set*{\abs*{\cdot}^{g_{0}}_{\CP^{1}} \leq \frac{1}{8\epsilon_{i}}}$ such that $p_{i} \to p_{\infty} \in K$ and in which $$ \abs*{\wtilde{\rho_{\epsilon_{i}}}^{-\delta}\paren*{p_{i}} \wtilde{f_{i}}\paren*{p_{i}}} = \wtilde{\rho_{\epsilon_{i}}}^{-\delta}\paren*{p_{i}}\abs*{ \wtilde{f_{i}}\paren*{p_{i}}} \geq \cC_{2.5} > 0$$ But upon taking the limit $i\nearrow\infty$ we get $$\wtilde{\rho_{0}}^{-\delta}\paren*{p_{\infty}}\abs*{ \wtilde{f_{\infty}}\paren*{p_{\infty}}} \geq \cC_{2.5} > 0$$ and since $p_{\infty} \in K \Subset T^{*}\CP^{1}$, by compactness of $K$ and continuity of $\wtilde{\rho_{0}}$ we thus have that $\wtilde{\rho_{0}}^{-\delta}\paren*{p_{\infty}} \neq 0$ and whence we get $$\abs*{ \wtilde{f_{\infty}}\paren*{p_{\infty}}} > 0$$ contradicting $\wtilde{f_{\infty}} = 0$. Therefore $\Abs*{\wtilde{f_{i}}}_{C^{2,\alpha'}_{\delta, \lwhat{g_{EH-0, \epsilon_{i}}}}\paren*{K}} \searrow 0$ whence from $\Abs*{\cdot}_{C^{0}_{\delta, \lwhat{g_{EH-0, \epsilon_{i}}}}} \leq \Abs*{\cdot}_{C^{2,\alpha'}_{\delta,\lwhat{g_{EH-0, \epsilon_{i}}}}}$, $\Abs*{\wtilde{f_{i}}}_{C^{0,\alpha'}_{\delta, \lwhat{g_{EH-0, \epsilon_{i}}}}\paren*{K}} \searrow 0$ holds for all compact subsets $K\Subset T^{*}\CP^{1}$. 

Now we further dichotomize: either \textit{globally} $\Abs*{\wtilde{f_{i}}}_{C^{0,\alpha'}_{\delta, \lwhat{g_{EH-0, \epsilon_{i}}}}\paren*{T^{*}\CP^{1}}} \searrow 0$, or not. If so, then transferring everything back to $\Km_{\epsilon_{i}}$ via $\Abs*{\wtilde{f_{i}}}_{C^{0,\alpha'}_{\delta, \lwhat{g_{EH-0, \epsilon_{i}}}} \paren*{\set*{\abs*{\cdot}^{g_{0}}_{\CP^{1}} \leq \frac{1}{8\epsilon_{i}}}}} = \Abs*{f_{i}}_{C^{0,\alpha'}_{\delta, g_{\epsilon_{i}}} \paren*{\set*{\abs*{\cdot}^{g_{0}}_{\pi^{-1}\paren*{\cS}}\leq \frac{1}{8}}}}$ gives us $$\Abs*{f_{i}}_{C^{0}_{\delta, g_{\epsilon_{i}}} } \searrow 0$$ holds on $\set*{\abs*{\cdot}^{g_{0}}_{\pi^{-1}\paren*{\cS}} \leq \frac{1}{8}} \subset \Km_{\epsilon_{i}}$. But this directly contradicts $\Abs*{f_{i}}_{C^{0}_{\delta, g_{\epsilon_{i}}} \paren*{\set*{\abs*{\cdot}^{g_{0}}_{\pi^{-1}\paren*{\cS}}\leq \frac{1}{8}}}} = 1$.

So instead suppose that $\Abs*{\wtilde{f_{i}}}_{C^{0,\alpha'}_{\delta, \lwhat{g_{EH-0, \epsilon_{i}}}}\paren*{T^{*}\CP^{1}}}$ does \textit{not} converge to $0$ as $i\nearrow\infty$. Then $\exists \wtilde{C_{6}} > 0$ uniform in $i$ such that $0 < \wtilde{C_{6}} \leq \Abs*{\wtilde{f_{i}}}_{C^{0,\alpha'}_{\delta, \lwhat{g_{EH-0, \epsilon_{i}}}}\paren*{T^{*}\CP^{1}}}$. Now since $\Abs*{\wtilde{f_{i}}}_{C^{0,\alpha'}_{\delta, g_{EH,1}}\paren*{K}} \searrow 0$ holds $\forall K \Subset T^{*}\CP^{1}$, we have upon writing $T^{*}\CP^{1}$ as a nested exhaustion of compact subsets that $\exists \set*{p_{i}} \subset T^{*}\CP^{1}$ a sequence of diverging points (escaping all compact subsets) such that $p_{i} \in \set*{\abs*{\cdot}^{g_{0}}_{\CP^{1}} \leq \frac{1}{8\epsilon_{i}}}$ and $$0 < \wtilde{C_{6}} \leq \abs*{\wtilde{\rho_{0}}\paren*{p_{i}}^{-\delta} \wtilde{f_{i}}\paren*{p_{i}}} = \wtilde{\rho_{0}}\paren*{p_{i}}^{-\delta}\abs*{ \wtilde{f_{i}}\paren*{p_{i}}}$$ Now since the points $\set*{p_{i}}$ are diverging, WLOG after assuming that they're all at least distance $\geq 2$ away from the zero section $\CP^{1}$ we thus have that $\wtilde{\rho_{0}}\paren*{p_{i}} \coloneq \abs*{p_{i}}^{g_{0}}_{\CP^{1}} \nearrow\infty$. 

Now by dichotomy, either $\sup_{i}\frac{1}{\epsilon_{i} \wtilde{\rho_{0}}\paren*{p_{i}}} < \infty$ or $\sup_{i}\frac{1}{\epsilon_{i} \wtilde{\rho_{0}}\paren*{p_{i}}} = \infty$.

Suppose the former, whence $\inf_{i} \epsilon_{i} \wtilde{\rho_{0}}\paren*{p_{i}} > 0$. Since $p_{i} \in \set*{\abs*{\cdot}^{g_{0}}_{\CP^{1}} \leq \frac{1}{8\epsilon_{i}}}$ we have by definition that actually $p_{i} \in \set*{\frac{\inf_{i} \epsilon_{i} \wtilde{\rho_{0}}\paren*{p_{i}}}{\epsilon_{i}} \leq \abs*{\cdot}^{g_{0}}_{\CP^{1}} \leq \frac{1}{8\epsilon_{i}}}$. Therefore transferring everything back to $\Km_{\epsilon_{i}}$ via $\Abs*{\wtilde{f_{i}}}_{C^{0,\alpha'}_{\delta, \lwhat{g_{EH-0, \epsilon_{i}}}} \paren*{\set*{\abs*{\cdot}^{g_{0}}_{\CP^{1}} \leq \frac{1}{8\epsilon_{i}}}}} = \Abs*{f_{i}}_{C^{0,\alpha'}_{\delta, g_{\epsilon_{i}}} \paren*{\set*{\abs*{\cdot}^{g_{0}}_{\pi^{-1}\paren*{\cS}}\leq \frac{1}{8}}}}$ tells us (upon recalling $\wtilde{f}_{i} \coloneq \epsilon_{i}^{-\delta}f_{i}$) that $$0<\wtilde{C_{6}} \leq \paren*{\abs*{p_{i}}^{g_{0}}_{\pi^{-1}\paren*{\cS}}}^{-\delta}\abs*{ f_{i}\paren*{p_{i}}}$$ where now $p_{i} \in \set*{\inf_{i}\epsilon_{i} \wtilde{\rho_{0}}\paren*{p_{i}} \leq \abs*{\cdot}^{g_{0}}_{\pi^{-1}\paren*{\cS}} \leq \frac{1}{8}}$. Since $\inf_{i} \epsilon_{i} \wtilde{\rho_{0}}\paren*{p_{i}} > 0$, whence $\set*{\inf_{i}\epsilon_{i} \wtilde{\rho_{0}}\paren*{p_{i}} \leq \abs*{\cdot}^{g_{0}}_{\pi^{-1}\paren*{\cS}} \leq \frac{1}{8}} \Subset \set*{0 < \abs*{\cdot}^{g_{0}}_{\pi^{-1}\paren*{\cS}}} \subset \Km_{\epsilon_{i}}$. Whence since $\rho_{\epsilon_{i}} = \abs*{\cdot}^{g_{0}}_{\pi^{-1}\paren*{\cS}} $ on $\set*{2\epsilon_{i}\leq \abs*{\cdot}^{g_{0}}_{\pi^{-1}\paren*{\cS}} \leq \frac{1}{8}}$, for sufficiently large $i$ we have that $2\epsilon_{i} \leq \inf_{i}\epsilon_{i} \wtilde{\rho_{0}}\paren*{p_{i}}$ and so we get $$0<\wtilde{C_{6}} \leq \rho_{\epsilon_{i}}\paren*{p_{i}}^{-\delta}\abs*{ f_{i}\paren*{p_{i}}} \leq \Abs*{f_{i}}_{C^{0}_{\delta, g_{\epsilon_{i}}} \paren*{\set*{\inf_{i}\epsilon_{i} \wtilde{\rho_{0}}\paren*{p_{i}} \leq \abs*{\cdot}^{g_{0}}_{\pi^{-1}\paren*{\cS}} \leq \frac{1}{8}}}}$$ However, \textbf{Case 1} establishes that $\Abs*{f_{i}}_{C^{0}_{\delta, g_{\epsilon_{i}}} \paren*{K}} \searrow 0$ holds for all compact subsets $K \Subset \set*{0 < \abs*{\cdot}^{g_{0}}_{\pi^{-1}\paren*{\cS}}} \subset \Km_{\epsilon_{i}}$, in particular for $K = \set*{\inf_{i}\epsilon_{i} \wtilde{\rho_{0}}\paren*{p_{i}} \leq \abs*{\cdot}^{g_{0}}_{\pi^{-1}\paren*{\cS}} \leq \frac{1}{8}}$, contradiction.

Ergo $\sup_{i}\frac{1}{\epsilon_{i} \wtilde{\rho_{0}}\paren*{p_{i}}} = \infty$.

Now consider the conformally rescaled Riemannian manifolds $\paren*{\set*{\abs*{\cdot}^{g_{0}}_{\CP^{1}} \leq \frac{1}{8\epsilon_{i}}}, \frac{1}{\wtilde{\rho_{0}}\paren*{p_{i}}^{2}}\lwhat{g_{EH-0, \epsilon_{i}}}}$. By Proposition \ref{preglueEH BG interpolation!} we have that $\frac{1}{\wtilde{\rho_{0}}\paren*{p_{i}}^{2}}\lwhat{g_{EH-0, \epsilon_{i}}} = \begin{cases}
\frac{1}{\wtilde{\rho_{0}}\paren*{p_{i}}^{2}}g_{EH,1} & \text{on } \set*{\abs*{\cdot}^{g_{0}}_{\CP^{1}} \leq \frac{1}{\epsilon_{i}^{\lambda}} } \\
\frac{1}{\wtilde{\rho_{0}}\paren*{p_{i}}^{2}}g_{0} & \text{on }\set*{\frac{2}{\epsilon_{i}^{\lambda}} \leq \abs*{\cdot}^{g_{0}}_{\CP^{1}}}
\end{cases}$. Upon pulling back by the scaling $R_{\wtilde{\rho_{0}}\paren*{p_{i}}}$, using the scaling of the Euclidean metric $R_{s}^{*}g_{0} = s^{2}g_{0}$ as well as $R_{s}^{*} \paren*{ \frac{1}{s^{2}} g_{EH,1}} = g_{EH,\frac{1}{s^{2}}}$ from Proposition \ref{EHproperties}, we have that $\paren*{\set*{\abs*{\cdot}^{g_{0}}_{\CP^{1}} \leq \frac{1}{8\epsilon_{i}}}, \frac{1}{\wtilde{\rho_{0}}\paren*{p_{i}}^{2}}\lwhat{g_{EH-0, \epsilon_{i}}}}$ is isometric under $R_{\wtilde{\rho_{0}}\paren*{p_{i}}}$ to $\paren*{\set*{\abs*{\cdot}^{g_{0}}_{\CP^{1}} \leq \frac{1}{8\epsilon_{i}\wtilde{\rho_{0}}\paren*{p_{i}}}}, \frac{1}{\wtilde{\rho_{0}}\paren*{p_{i}}^{2}}R^{*}_{\wtilde{\rho_{0}}\paren*{p_{i}}}\paren*{\lwhat{g_{EH-0, \epsilon_{i}}}}}$ where $ \frac{1}{\wtilde{\rho_{0}}\paren*{p_{i}}^{2}}R^{*}_{\wtilde{\rho_{0}}\paren*{p_{i}}}\paren*{\lwhat{g_{EH-0, \epsilon_{i}}}} = \begin{cases}
g_{EH, \frac{1}{\wtilde{\rho_{0}}\paren*{p_{i}}^{2}}} & \text{on } \set*{\abs*{\cdot}^{g_{0}}_{\CP^{1}} \leq \frac{1}{\epsilon_{i}^{\lambda}\wtilde{\rho_{0}}\paren*{p_{i}}} } \\
g_{0} & \text{on }\set*{\frac{2}{\epsilon_{i}^{\lambda}\wtilde{\rho_{0}}\paren*{p_{i}}} \leq \abs*{\cdot}^{g_{0}}_{\CP^{1}}}
\end{cases}$.

Thus, as $\set*{\wtilde{\rho_{0}}\paren*{p_{i}}}\subset\bbR_{>0}$ is a diverging sequence of real numbers, as $i\nearrow \infty$ the rescaled Eguchi-Hanson metric $g_{EH, \frac{1}{\wtilde{\rho_{0}}\paren*{p_{i}}^{2}}}$ gets blown down to the flat metric $g_{0}$ on $\bbC^{2}/\bbZ_{2}$ since the diameter of the $\CP^{1}$ bolt shrinks (Proposition \ref{EHproperties}). Ergo by Proposition \ref{preglueEH BG interpolation!} and $\sup_{i}\frac{1}{\epsilon_{i} \wtilde{\rho_{0}}\paren*{p_{i}}} = \infty$, in the limit as $i\nearrow\infty$ we have that $$\paren*{\set*{\abs*{\cdot}^{g_{0}}_{\CP^{1}} \leq \frac{1}{8\epsilon_{i}\wtilde{\rho_{0}}\paren*{p_{i}}}}, \frac{1}{\wtilde{\rho_{0}}\paren*{p_{i}}^{2}}R^{*}_{\wtilde{\rho_{0}}\paren*{p_{i}}}\paren*{\lwhat{g_{EH-0, \epsilon_{i}}}}} \xrightarrow{C^{\infty}_{g_{0},loc}} \paren*{\bbC^{2}/\bbZ_{2} - \set*{0}, g_{0}}$$ That is, $\paren*{\bbC^{2}/\bbZ_{2}, g_{0}}$ is the \textbf{tangent cone at infinity} to the (modified) Eguchi-Hanson metric $\paren*{T^{*}\CP^{1}, \lwhat{g_{EH-0, \epsilon_{i}}}}$.

Now by Case (\ref{scalingproperties weighted H\"{o}lder tensor norm rescale}) of Proposition \ref{scalingproperties} we have upon transferring under the isometry $R_{\wtilde{\rho_{0}}\paren*{p_{i}}}$ that $$\Abs*{\wtilde{f_{i}}}_{C^{k,\alpha}_{\delta, \lwhat{g_{EH-0, \epsilon_{i}}}} \paren*{\set*{\abs*{\cdot}^{g_{0}}_{\CP^{1}} \leq \frac{1}{8\epsilon_{i}}}}} = \Abs*{\wtilde{\rho_{0}}\paren*{p_{i}}^{-\delta}\wtilde{f_{i}}}_{C^{k,\alpha}_{\delta, \frac{1}{\wtilde{\rho_{0}}\paren*{p_{i}}^{2}}\lwhat{g_{EH-0, \epsilon_{i}}}} \paren*{\set*{\abs*{\cdot}^{g_{0}}_{\CP^{1}} \leq \frac{1}{8\epsilon_{i}\wtilde{\rho_{0}}\paren*{p_{i}}}}}}$$ where the weight function on the RHS is $R_{\wtilde{\rho_{0}}\paren*{p_{i}}}^{*}\paren*{\frac{1}{\wtilde{\rho_{0}}\paren*{p_{i}}}\wtilde{\rho_{\epsilon_{i}}}} = \begin{cases}
\frac{1}{\epsilon_{i}\wtilde{\rho_{0}}\paren*{p_{i}}} & \text{ on }\set*{\frac{1}{5\epsilon_{i}\wtilde{\rho_{0}}\paren*{p_{i}}} \leq \abs*{\cdot}^{g_{0}}_{\CP^{1}}}\\
\abs*{\cdot}^{g_{0}}_{\CP^{1}} & \text{ on }\set*{\frac{2}{\wtilde{\rho_{0}}\paren*{p_{i}}}\leq \abs*{\cdot}^{g_{0}}_{\CP^{1}} \leq \frac{1}{8\epsilon_{i}\wtilde{\rho_{0}}\paren*{p_{i}}}}\\
\frac{1}{\wtilde{\rho_{0}}\paren*{p_{i}}} & \text{ on }\set*{\abs*{\cdot}^{g_{0}}_{\CP^{1}} \leq \frac{1}{\wtilde{\rho_{0}}\paren*{p_{i}}}}
\end{cases}$. Therefore transferring over $\Abs*{\wtilde{f_{i}}}_{C^{2,\alpha}_{\delta, \lwhat{g_{EH-0,\epsilon_{i}}}} \paren*{\set*{\abs*{\cdot}^{g_{0}}_{\CP^{1}} \leq \frac{1}{8\epsilon_{i}}}}} \leq 2\cC_{2}$ and $\Abs*{\Delta_{\lwhat{g_{EH-0,\epsilon_{i}}}}f_{i}}_{C^{0,\alpha}_{\delta - 2, \lwhat{g_{EH-0,\epsilon_{i}}}} \paren*{\set*{\abs*{\cdot}^{g_{0}}_{\CP^{1}} \leq \frac{1}{8\epsilon_{i}}}}} \leq \frac{1}{i}$ under this identification, an application of Arzel\`{a}-Ascoli yields \begin{gat}
\wtilde{\rho_{0}}\paren*{p_{i}}^{-\delta}\wtilde{f_{i}} \xrightarrow{C^{2,\alpha'}_{g_{0},loc}} \overline{f_{\infty}}\\
\Abs*{\overline{f_{\infty}}}_{C^{2,\alpha}_{\delta,g_{0}}\paren*{\bbC^{2}/\bbZ_{2} - \set*{0}}} \leq 2\cC_{2}\\
\Delta_{g_{0}}\overline{f_{\infty}} = 0
\end{gat} for $0 < \alpha' < \alpha < 1$ and where the limit weight function is simply the radial weight function $\abs*{\cdot}^{g_{0}}_{0}$ on $\bbC^{2}/\bbZ_{2}$.

Ergo $\overline{f_{\infty}}$ is a harmonic function on $\bbC^{2}/\bbZ_{2}- \set*{0}$ that satisfies $\overline{f_{\infty}} = O\paren*{\paren*{\abs*{\cdot}^{g_{0}}_{0}}^{\delta}}$. But since $-2 < \delta < 0$, such a function must vanish by maximum principle.

However, upon transferring $p_{i} \in \set*{\abs*{\cdot}^{g_{0}}_{\CP^{1}} \leq \frac{1}{8\epsilon_{i}}}$ to $p_{i} \in \set*{\abs*{\cdot}^{g_{0}}_{\CP^{1}} \leq \frac{1}{8\epsilon_{i}\wtilde{\rho_{0}}\paren*{p_{i}}}}$, we get that $p_{i} \in \set*{\abs*{\cdot}^{g_{0}}_{\CP^{1}} = 1}$, whence in the limit $p_{\infty} \in \set*{\abs*{\cdot}^{g_{0}}_{0} = 1} \Subset \bbC^{2}/\bbZ_{2} - \set*{0}$ lands in a compact subset. Thus taking $i\nearrow\infty$ of $0 < \wtilde{C_{6}} \leq \abs*{\wtilde{\rho_{0}}\paren*{p_{i}}^{-\delta} \wtilde{f_{i}}\paren*{p_{i}}}$ yields $$0 < \wtilde{C_{6}} \leq \abs*{\overline{f_{\infty}}\paren*{p_{\infty}}}$$ contradicting $\overline{f_{\infty}} = 0$.\end{enumerate} Whence we have reached a contradiction to which the proof follows, as was to be shown.\end{proof}

With these in hand, we may finally prove

\begin{theorem}\label{biginverse} For $\delta \in \paren*{\max\set*{-2, -4\lambda}, 0}$ and $\epsilon > 0$ sufficiently small, the following is a \textit{bounded linear isomorphism} for both the domain and codomain $C^{2,\alpha}_{\delta, g_{\epsilon}} \paren*{\Km_{\epsilon}}, C^{0,\alpha}_{\delta-2, g_{\epsilon}} \paren*{\Km_{\epsilon}}$ restricted to the closed subspace of functions with integral zero, i.e. $\int_{\Km_{\epsilon}} f \omega_{\epsilon}^{2} = 0$:
$$\Delta_{g_{\epsilon}}: C^{2,\alpha}_{\delta, g_{\epsilon}} \paren*{\Km_{\epsilon}}^{0} \rightarrow C^{0,\alpha}_{\delta - 2, g_{\epsilon}} \paren*{\Km_{\epsilon}}^{0}$$
More crucially, we have that the operator norm of the inverse is bounded by a constant independent of $\epsilon > 0$:
$$ \Abs*{\Delta_{g_{\epsilon}}^{-1}} \leq K$$
\end{theorem}

\begin{proof}
The Laplacian in the unweighted H\"{o}lder spaces is Fredholm of index zero. Now for \textit{fixed} $\delta,\epsilon$, we have that $C^{k,\alpha}_{\delta, g_{\epsilon}} \paren*{\Km_{\epsilon}}, C^{2,\alpha} \paren*{\Km_{\epsilon}}$ have equivalent norms (they're already the same underlying vector space and $\rho_{\epsilon}$ is by construction uniformly bounded below by $\epsilon$ and above by $1$). Hence $\Delta_{g_{\epsilon}}: C^{2,\alpha}_{\delta, g_{\epsilon}} \paren*{\Km_{\epsilon}}^{0} \rightarrow C^{0,\alpha}_{\delta - 2, g_{\epsilon}} \paren*{\Km_{\epsilon}}^{0}$ is also Fredholm of index zero. As remarked in \textit{Remark} \ref{closedrangedetector}, the improved weighted Schauder estimate $\Abs*{f}_{C^{2,\alpha}_{\delta, g_{\epsilon}} \paren*{\Km_{\epsilon}}} \leq \cC_{3} \Abs*{\Delta_{g_{\epsilon}}f}_{C^{0,\alpha}_{\delta - 2, g_{\epsilon}} \paren*{\Km_{\epsilon}}}$ as well as the integral zero condition gives us that $\ker \Delta_{g_{\epsilon}} = \set*{0}$ when \textit{$\epsilon > 0$ is sufficiently small} and $\delta \in \paren*{\max\set*{-2, -4\lambda}, 0}$, hence $\ind \paren*{\Delta_{g_{\epsilon}}} = 0 $ implies surjective and hence bounded linear isomorphism, and $\Abs*{f}_{C^{2,\alpha}_{\delta, g_{\epsilon}} \paren*{\Km_{\epsilon}}} \leq \cC_{3} \Abs*{\Delta_{g_{\epsilon}}f}_{C^{0,\alpha}_{\delta - 2, g_{\epsilon}} \paren*{\Km_{\epsilon}}}$ also gives us our $\epsilon > 0$ independent bound on the operator norm of the inverse, as was to be shown.\end{proof}

\subsection{\textsection \ Finishing the Proof \ \textsection}\label{Finishing the proof}

Recall from Section \ref{Nonlinear Setup} our nonlinear problem: \begin{al}
\cF_{\epsilon} :  C^{2,\alpha}_{\delta, g_{\epsilon}} \paren*{\Km_{\epsilon}}^{0} &\rightarrow C^{0,\alpha}_{\delta - 2, g_{\epsilon}} \paren*{\Km_{\epsilon}}^{0} \\
f &\mapsto \frac{\paren*{\omega_{\epsilon} + i\del \delbar f}^{2}}{\omega_{\epsilon}^{2}} - e^{\phi_{\epsilon}}		
\end{al} where we solve \begin{al}
\cF_{\epsilon}\paren*{f} &\coloneq\frac{\paren*{\omega_{\epsilon} + i\del \delbar f}^{2}}{\omega_{\epsilon}^{2}} - e^{\phi_{\epsilon}}		 \\
&= \underbrace{\paren*{1 - e^{\phi_{\epsilon}}}}_{= \cF_{\epsilon}\paren*{0}} \underbrace{- \frac{1}{2}\Delta_{g_{\epsilon}}f}_{= D_{0}\cF_{\epsilon}\paren*{f}} + \underbrace{\frac{i\del \delbar f \wedge i\del \delbar f }{\omega_{\epsilon}^{2}}}_{\text{nonlinearity}} \\
&= 0
\end{al}

We want to apply the following \begin{theorem}[Implicit Function Theorem]\label{IFT}Let $X, Y$ be two Banach spaces. Let $\Phi : X \rightarrow Y$ be a $C^{1}$ map. Write $\Phi(x) = \Phi(0) + D_{0}\Phi(x) + \cN(x)$ with $\cN(x) \coloneq \Phi(x) - \Phi(0) - D_{0}\Phi(x)$ being the nonlinear term. Suppose we have 3 positive constants $L, r_{0}, N > 0$ such that \begin{enumerate}[nosep]
\itemsep0em 
\item $D_{0}\Phi : X \rightarrow Y$ admits a bounded linear \textit{right inverse} $R: Y\rightarrow X$ such that $\Abs*{R}_{\text{op}} \leq L$ ($L>0$ controls right inverse)
\item A bound $\Abs*{\cN(x) - \cN(y)}_{Y} \leq N \Abs*{x-y}_{X}\paren*{\Abs*{x}_{X} + \Abs*{y}_{X}}, \qqfa x,y \in D^{X}_{r_{0}}\paren*{0}$ (``Lipschitz'' bounds on nonlinearity) 
\item $\Abs*{\Phi(0)}_{Y} \leq \frac{r}{2L}$ with $r < \min\set*{r_{0}, \frac{1}{2NL}}$ (``smallness'' of $\Phi(0)$)
\end{enumerate}
Then $\Phi(x) = 0$ has a unique solution $x \in D^{X}_{2L\Abs*{\Phi(0)}_{Y}}\paren*{0}$.
\end{theorem}

\begin{proof} This is essentially the Banach fixed point theorem. We have that $\Phi(x) = 0 \Longleftrightarrow \cA(x) = x$ with $\cA(x) \coloneq -R\paren*{\Phi(0) + \cN(x)}$. We want to show that $\cA : D_{r}^{X}\paren*{0} \rightarrow D^{X}_{r}\paren*{0}$ is a contraction mapping. First we have that $\im \cA \subset D^{X}_{r}\paren*{0}$ because for $x \in D^{X}_{r}\paren*{0}$, \begin{al}
\Abs*{\cA(x)}_{X} &= \Abs*{R\paren*{\Phi(0) + \cN(x)}}_{X}\\
&\leq L \paren*{\Abs*{\Phi(0)}_{Y} + \Abs*{\cN(x)}_{Y}} \\
&\leq L \paren*{\Abs*{\Phi(0)}_{Y} + 4L^{2}N\Abs*{\Phi(0)}_{Y}^{2}} \\
&< L\Abs*{\Phi(0)}_{Y}\paren*{1 + 1}\\
&= 2L\Abs*{\Phi(0)}_{Y}
\end{al} Next we have that \begin{al}
\Abs*{\cA(x) - \cA(y)}_{X} &= \Abs*{R\paren*{ \cN(x) - \cN(y)}}_{X}\\
&\leq L\Abs*{\cN(x) - \cN(y)}_{Y} \\ 
&\leq LN \Abs*{x-y}_{X}\paren*{\Abs*{x}_{X} + \Abs*{y}_{X}}\\
&\leq 4NL^{2}\Abs*{\Phi(0)}_{Y}\Abs*{x-y}_{X} \\
&< \frac{4NL^{2}}{4NL^{2}}\Abs*{x-y}_{X}\\
&= \Abs*{x-y}_{Y}
\end{al} Hence we're done, as was to be shown. \end{proof}

We've now arrived at: \begin{theorem}[Big Theorem I]\label{bigtheorem1} Let $\lambda < \frac{2}{3}$ be so that the curvature of $g_{\epsilon}$ remains bounded (\textit{Remark} \ref{lambda < 2/3 bound}). Let $\delta \in \paren*{\max\set*{-2, -4\lambda}, 0}$ for Proposition \ref{IWSE}. Let $\epsilon > 0$ from Data \ref{constant epsilon and lambda introduction} be\begin{enumerate}
\item sufficiently small so that each pre-glued Eguchi-Hanson metric from Propositions \ref{preglueEH BG interpolation!} and \ref{preglueEH} are K\"{a}hler and satisfy their annular estimates, and $\epsilon < 1$ to satisfy the volume form estimates.
\item $3\epsilon^{\lambda} < \frac{1}{2}$ from Data \ref{constant epsilon patching and WLOG lattice} for the construction of $\Km_{\epsilon}$.
\item $2\epsilon^{\lambda} < \frac{3}{25}$ and $2\epsilon < \epsilon^{\lambda}$ from Data \ref{constant epsilon 1-16} for the construction of $\rho_{\epsilon}$.
\item sufficiently small for the smallness of $\cF_{\epsilon}(0)$ from Proposition \ref{presmall}.
\item sufficiently small for the weighted Schauder estimate from Corollary \ref{WSE}.
\item sufficiently small for the improved Schauder estimate from Proposition \ref{IWSE}.
\item sufficiently small for the the uniform bounded linear isomorphism $\Delta_{g_{\epsilon}}: C^{2,\alpha}_{\delta, g_{\epsilon}} \paren*{\Km_{\epsilon}}^{0} \rightarrow C^{0,\alpha}_{\delta - 2, g_{\epsilon}} \paren*{\Km_{\epsilon}}^{0}$ from Theorem \ref{biginverse}
\item $\epsilon^{\paren*{2+\delta}\paren*{1-\lambda}} < \frac{1}{4L^{2}C_{7}C_{8}}$ for $L, C_{7}, C_{8} > 0$ constants which will be defined in the proof.
\end{enumerate} Then for each such $\epsilon > 0$, the equation $\cF_{\epsilon}(f) = 0$ has a \textbf{unique} solution $f$ satisfying:

\begin{itemize}
\item $\int_{\Km_{\epsilon}} f \omega_{\epsilon}^{2} = 0$ (\textbf{integral zero})
\item $f \in C^{\infty}\paren*{\Km_{\epsilon}}$ (\textbf{smoothness})
\item $\Abs*{f}_{C^{2,\alpha}_{\delta, g_{\epsilon}} \paren*{\Km_{\epsilon}}} \leq 2L\Abs*{\cF_{\epsilon}\paren*{0}}_{C^{0,\alpha}_{\delta - 2, g_{\epsilon}} \paren*{\Km_{\epsilon}}} < C_{9}\epsilon^{4 - \lambda\paren*{\delta + 2}}$ for $C_{9} > 0$ chosen so that $C_{9} > 2L C_{8}$ (\textbf{smallness of norm})
\end{itemize}

Hence by Proposition \ref{K\"{a}hlereinstein}, \textit{Remark} \ref{autopositivity}, and Proposition \ref{calabiyau}, for each such $\epsilon > 0$ we have that $\omega_{\epsilon} + i\del\delbar f \in [\omega_{\epsilon}]$ is the \textbf{unique} K\"{a}hler form in the K\"{a}hler class of $\omega_{\epsilon}$ on $\paren*{\Km_{\epsilon}, J}$ which gives a Ricci-flat K\"{a}hler metric compatible with $J$. 
\end{theorem} \begin{proof} We setup the input needed for the implicit function theorem (Theorem \ref{IFT}). \begin{enumerate}[wide, labelwidth=!, labelindent=0pt]
\item We have from Theorem \ref{biginverse} that $-\frac{1}{2}\Delta_{g_{\epsilon}}: C^{2,\alpha}_{\delta, g_{\epsilon}} \paren*{\Km_{\epsilon}}^{0} \rightarrow C^{0,\alpha}_{\delta - 2, g_{\epsilon}} \paren*{\Km_{\epsilon}}^{0}$ is a bounded linear isomorphism with $\Abs*{2\Delta^{-1}_{g_{\epsilon}}} \leq 2K \eqcolon L$ and $L>0$ \textit{uniform in $\epsilon > 0$}.
\item Firstly, we have that for any $4$-form $\Theta$, that since $\Theta = \frac{\Theta}{\omega\wedge \omega} \omega \wedge \omega$ where $\frac{\Theta}{\omega\wedge \omega}$ is the coefficient \textit{function}, that combined with the identities $\dV_{g} = \frac{\omega^{2}}{2}$ and $\abs*{\dV_{g}}_{g} = 1$, we have that (where the tensor norms here are on $4$-forms) $\abs*{\Theta}_{g} = \abs*{\frac{\Theta}{\omega \wedge \omega} \omega \wedge \omega}_{g} = 2 \abs*{\frac{\Theta}{\omega \wedge \omega} \dV_{g}}_{g} = 2 \abs*{\frac{\Theta}{\omega \wedge \omega}}$. Whence it suffices to bound the $4$-form norm of $i\del\delbar f \wedge i\del\delbar f - i\del\delbar g \wedge i\del\delbar g$. 

Therefore, using the fact that $a \wedge a - b \wedge b = \paren*{a - b} \wedge \paren*{a + b}$ for $2$-forms $a,b$, Proposition \ref{weighted products} with $l=2$ on expressions with products, and the definition of the weighted H\"{o}lder norms, for the nonlinearity we have that there exists a \textit{uniform} constant $C_{7} > 0$ such that $$\Abs*{\cN(f) - \cN(g)}_{C^{0,\alpha}_{\delta - 2, g_{\epsilon}} \paren*{\Km_{\epsilon}}} \leq C_{7}\epsilon^{\delta - 2} \Abs*{f-g}_{C^{2,\alpha}_{\delta, g_{\epsilon}} \paren*{\Km_{\epsilon}}}\paren*{\Abs*{f}_{C^{2,\alpha}_{\delta, g_{\epsilon}} \paren*{\Km_{\epsilon}}} + \Abs*{g}_{C^{2,\alpha}_{\delta, g_{\epsilon}} \paren*{\Km_{\epsilon}}}}, \qqfa f,g \in D_{r_{0}}^{C^{2,\alpha}_{\delta, g_{\epsilon}} \paren*{\Km_{\epsilon}}}\paren*{0}$$ with $N \coloneq C_{7}\epsilon^{\delta - 2}$ and $r_{0} > 0$ to be determined later.

\item Proposition \ref{presmall} and the Taylor expansion of the exponential function gives us $\supp \cF_{\epsilon}\paren*{0} \subseteq \set*{\epsilon^{\lambda} \leq \abs*{\cdot}^{g_{0}}_{\pi^{-1}\paren*{\cS}} \leq 2\epsilon^{\lambda}}$ and $\abs*{\cF_{\epsilon}\paren*{0}} \leq c_{1}(2)\epsilon^{4-4\lambda}$ on its support. By definition of the weighted H\"{o}lder norm and the definition of $\rho_{\epsilon}$ restricted to $\set*{\epsilon^{\lambda} \leq \abs*{\cdot}^{g_{0}}_{\pi^{-1}\paren*{\cS}} \leq 2\epsilon^{\lambda}}$ combined with Data \ref{constant epsilon 1-16} (namely $2\epsilon < \epsilon^{\lambda}$) giving the bound $\epsilon^{\lambda} \leq \rho_{\epsilon} \leq 2\epsilon^{\lambda}$, we thus have that $\Abs*{\cF_{\epsilon}\paren*{0}}_{C^{0,\alpha}_{\delta - 2, g_{\epsilon}} \paren*{\Km_{\epsilon}}}\leq C_{8}\epsilon^{4-4\lambda-\lambda\paren*{\delta - 2}} = C_{8}\epsilon^{4 - 2\lambda - \delta\lambda}= C_{8}\epsilon^{4 - \lambda\paren*{\delta + 2}}$ for some \textit{uniform} constant $C_{8} > 0$. Moreover, since $0 < \lambda < 1$ and $\delta \in \paren*{\max\set*{-2,-4\lambda},0}$, we automatically have that $4 - \lambda\paren*{\delta + 2} > 0$.
\end{enumerate}

Set $r_{0} \coloneq C_{9} \epsilon^{4 - \lambda\paren*{\delta + 2}}$ with $C_{9} > 0$ chosen so that $C_{9} > 2LC_{8}$. We now need to show that $$\Abs*{\cF_{\epsilon}\paren*{0}}_{C^{0,\alpha}_{\delta - 2, g_{\epsilon}} \paren*{\Km_{\epsilon}}}\leq \frac{r}{2L}$$ with $r < \min\set*{r_{0}, \frac{1}{2NL}} = \min\set*{C_{9} \epsilon^{4 - \lambda\paren*{\delta + 2}}, \frac{\epsilon^{2-\delta}}{2LC_{7}}}$. We have that \begin{al}
\Abs*{\cF_{\epsilon}\paren*{0}}_{C^{0,\alpha}_{\delta - 2, g_{\epsilon}} \paren*{\Km_{\epsilon}}}\leq C_{8}\epsilon^{4 - \lambda\paren*{\delta + 2}} = \frac{2LC_{8}\epsilon^{4 - \lambda\paren*{\delta + 2}}}{2L}
\end{al} and so we choose our $r \coloneq 2LC_{8}\epsilon^{4 - \lambda\paren*{\delta + 2}}$. Clearly $r < r_{0}$, and it remains to show that $r < \frac{\epsilon^{2-\delta}}{2LC_{7}}$, i.e. $\epsilon^{4 - \lambda\paren*{\delta + 2} + \delta - 2} < \frac{1}{4L^{2}C_{7}C_{8}}$. However, $4 - \lambda\paren*{\delta + 2} + \delta - 2 = 2+\delta - \lambda\paren*{\delta + 2} = \paren*{2+\delta}\paren*{1-\lambda} > 0$ is \textit{positive} since $\lambda \in \paren*{0,\frac{2}{3}}$ and $\delta \in \paren*{\max\set*{-2,-4\lambda},0}$. Ergo making $\epsilon > 0$ small enough so that $\epsilon^{\paren*{2+\delta}\paren*{1-\lambda}} < \frac{1}{4L^{2}C_{7}C_{8}}$ forces $r < \frac{\epsilon^{2-\delta}}{2LC_{7}}$. Therefore, \textit{for each such $\epsilon > 0$} Theorem \ref{IFT} gives us the existence of a unique solution $f \in D^{C^{2,\alpha}_{\delta, g_{\epsilon}} \paren*{\Km_{\epsilon}}^{0}}_{2L\Abs*{\cF_{\epsilon}\paren*{0}}_{C^{0,\alpha}_{\delta - 2, g_{\epsilon}} \paren*{\Km_{\epsilon}}}}\paren*{0}$ to $\cF_{\epsilon}\paren*{f} = 0$, giving us existence, uniqueness, integral zero, and the bound $$\Abs*{f}_{C^{2,\alpha}_{\delta, g_{\epsilon}} \paren*{\Km_{\epsilon}}} \leq 2L\Abs*{\cF_{\epsilon}\paren*{0}}_{C^{0,\alpha}_{\delta - 2, g_{\epsilon}} \paren*{\Km_{\epsilon}}} < C_{9}\epsilon^{4 - \lambda\paren*{\delta + 2}}$$ since $r < r_{0}$. Note that the RHS remains \textit{bounded} as $\epsilon \ll 1$ since $\lambda < \frac{4}{2+\delta}$ holds since $1 < \frac{4}{2+\delta}$ from $\delta \in \paren*{\max\paren*{-2,-4\lambda},0}$. Smoothness follows from elliptic regularity, as our nonlinear equation has smooth coefficients \& $2$nd order elliptic linearization and the produced function $f$ is in $C^{2,\alpha}\subset C^{2}$. Lastly, Propositions \ref{K\"{a}hlereinstein} and \ref{calabiyau}, as well as \textit{Remark} \ref{autopositivity} regarding positivity of the resulting K\"{a}hler form being automatic, finish off the existence of our unique Ricci-flat K\"{a}hler metric for each sufficiently small $\epsilon > 0$, as was to be shown.\end{proof}

Now we have from the smallness of the norm $\Abs*{f}_{C^{2,\alpha}_{\delta, g_{\epsilon}} \paren*{\Km_{\epsilon}}} < C_{9}\epsilon^{4 - \lambda\paren*{\delta + 2}}$ of the solution to the nonlinear problem above, the definition of the weighted H\"{o}lder spaces, the global bound $\epsilon \leq \rho_{\epsilon}$ and the fact that $\delta \in \paren*{\max\set*{-2,-4\lambda},0}$ gives $\Abs*{f}_{C^{2,\alpha}_{\delta, g_{\epsilon}} \paren*{\Km_{\epsilon}}} < C_{9}\epsilon^{4 - \lambda\paren*{\delta + 2}} \Longrightarrow \Abs*{f}_{C^{0}\paren*{\Km_{\epsilon}}} < C_{9}\epsilon^{4 - \lambda\paren*{\delta + 2}+\delta}, \Abs*{\nabla_{g_{\epsilon}}f}_{C^{0}_{g_{\epsilon}}\paren*{\Km_{\epsilon}}} < C_{9}\epsilon^{3 - \lambda\paren*{\delta + 2} + \delta}, \Abs*{\nabla^{2}_{g_{\epsilon}}f}_{C^{0}_{g_{\epsilon}}\paren*{\Km_{\epsilon}}} < C_{9}\epsilon^{2 - \lambda\paren*{\delta + 2} + \delta} = C_{9}\epsilon^{\paren*{1-\lambda}\paren*{\delta+2}}$ with all higher order tensor norms defined with respect to $g_{\epsilon},\omega_{\epsilon}$. In particular, this implies that $\Abs*{i\del \delbar f}_{C^{0}_{g_{\epsilon}}\paren*{\Km_{\epsilon}}} < C_{10}\epsilon^{\paren*{1-\lambda}\paren*{\delta+2}}$ for some \textit{uniform} $C_{10} > 0$, whence $\Abs*{i\del \delbar f}_{C^{0}_{g_{\epsilon}}\paren*{\Km_{\epsilon}}}$ tends to zero as $\epsilon \searrow 0$.

Combining these with $\Abs*{\paren*{\omega_{\epsilon} +i\del \delbar f} - \omega_{\epsilon}}_{C^{0}_{g_{\epsilon}}\paren*{\Km_{\epsilon}}} = \Abs*{i\del \delbar f}_{C^{0}_{g_{\epsilon}}\paren*{\Km_{\epsilon}}} < C_{10}\epsilon^{\paren*{1-\lambda}\paren*{\delta+2}}$, \textit{Remark} \ref{GHremark}, the two convergences $g_{\epsilon} \xrightarrow{C^{\infty}_{g_{0},loc}\paren*{\set*{0 < \abs*{\cdot}^{g_{0}}_{\pi^{-1}\paren*{\cS}}}}} g_{0}$ and $R_{\epsilon}^{*}\paren*{\frac{1}{\epsilon^{2}} g_{\epsilon}} = \lwhat{g_{EH-0, \epsilon}} \xrightarrow{C^{\infty}_{g_{EH,1},loc}\paren*{\set*{\abs*{\cdot}^{g_{0}}_{\CP^{1}} \leq \frac{1}{8\epsilon}}}} g_{EH,1}$ used in the proof of Proposition \ref{IWSE}, Proposition \ref{EHproperties} regarding the curvature blow-up of the bolt of $g_{EH,\epsilon^{2}}$ and Propositions \ref{preglueEH} and \ref{approximatemetricproperties} on the construction of $g_{\epsilon}$, and the triangle inequality, we thus have proven Theorem \ref{MainTheoremA}, i.e.:

\begin{theorem}[Main Theorem I]\label{maintheorem1} Let $K3$ denote the $K3$ surface. Equip $K3$ with the fixed complex structure $J$ and compatible family of \ka metrics $\paren*{g_{\epsilon}, \omega_{\epsilon}}$ constructed in Section \ref{The Kummer Construction}, hence getting the family of \ka metrics $\paren*{\Km_{\epsilon}, J, g_{\epsilon}, \omega_{\epsilon}}$ on a family of Kummer surfaces $\paren*{\Km_{\epsilon}, J}$.

Then denoting by $\wtilde{\omega_{\epsilon}} \coloneq \omega_{\epsilon} + i \del \delbar f_{\epsilon} \in [\omega_{\epsilon}]$ and $\wtilde{g_{\epsilon}}$ the unique Ricci-flat K\"{a}hler metric in the K\"{a}hler class $[\omega_{\epsilon}]$ produced from Theorem \ref{bigtheorem1}, we have that $0 < \exists \epsilon_{0} \ll 1$ such that \textbf{there exists a 1-parameter family of Calabi-Yau\footnote{upon suitably normalizing the nowhere vanishing holomorphic $\paren*{2,0}$-form on each $\Km_{\epsilon}$} metrics} $$\set*{\paren*{\wtilde{g_{\epsilon}}, \wtilde{\omega_{\epsilon}}}}_{\epsilon \in \paren*{0,\epsilon_{0}}}$$ on $K3$ which are Ricci-flat \ka WRT the fixed complex structure $J$. The solved for functions $\set*{f_{\epsilon}}_{\epsilon \in \paren*{0,\epsilon_{0}}}$ produced from Theorem \ref{bigtheorem1} satisfy the uniform estimate $$\Abs*{f_{\epsilon}}_{C^{k}_{g_{\epsilon}} \paren*{\Km_{\epsilon}}} \leq C_{9}\epsilon^{4 - \lambda\paren*{\delta + 2} + \delta - k}$$ for $\delta \in \paren*{\max\set*{-2,-4\lambda},0}$ and $\lambda \in \paren*{0,\frac{2}{3}}$.

Moreover when $\epsilon \searrow 0$, $$\paren*{\Km_{\epsilon}, J, \wtilde{g_{\epsilon}}, \wtilde{\omega_{\epsilon}}} \xrightarrow{GH} \paren*{\bbT^{4}/\bbZ_{2}, J_{0}, g_{0}, \omega_{0}}$$ converges \textit{in the Gromov-Hausdorff topology} to the flat \ka orbifold $\paren*{\bbT^{4}/\bbZ_{2}, J_{0}, g_{0},\omega_{0}}$ with the singular orbifold K\"{a}hler metric $\paren*{g_{0}, \omega_{0}}$ which is equal to the flat K\"{a}hler metric on $\bbT^{4}/\bbZ_{2} - \cS$.

Moreover, we may decompose $K3$ into a union of sets\footnote{say, $\Km_{\epsilon}^{reg} \coloneq \set*{\frac{\epsilon^{\lambda}}{2} < \abs*{\cdot}^{g_{0}}_{\pi^{-1}\paren*{\cS}}}$ and each $\Km_{\epsilon}^{b_{p}}$ being the 16 connected components of $\set*{\abs*{\cdot}^{g_{0}}_{\pi^{-1}\paren*{\cS}}\leq \epsilon^{\lambda}}$.} $\Km_{\epsilon}^{reg} \cup \bigsqcup_{p \in \cS}\Km_{\epsilon}^{b_{p}}$ such that \begin{enumerate}
\itemsep0em 
\item $\paren*{\Km_{\epsilon}^{reg}, J, \wtilde{g_{\epsilon}}, \wtilde{\omega_{\epsilon}}}$ converges in $C^{\infty}_{g_{0},loc}$ to the flat \ka manifold $\paren*{\bbT^{4}/\bbZ_{2} - \cS, J_{0}, g_{0}, \omega_{0}}$ with bounded curvature away from $\cS$ (\textbf{regular region}).
\item For each $p \in \cS$, after applying a $\frac{1}{\epsilon}$-dilation $\paren*{R_{\frac{1}{\epsilon}}\paren*{\Km_{\epsilon}^{b_{p}}}, R_{\epsilon}^{*}\paren*{\frac{1}{\epsilon^{2}}\wtilde{g_{\epsilon}}}}$ converges to $\paren*{T^{*}\CP^{1}, g_{EH,1}}$ in $C^{\infty}_{g_{EH,1},loc}$ (\textbf{ALE bubble region}).\end{enumerate}\end{theorem}

\section{\textsection \ Hyper-K\"{a}hler Metrics on $K3$ and Volume Non-Collapsed Limit Spaces \ \textsection}\label{HKA analogue}

\subsection{\textsection \ Hyper-K\"{a}hler Metrics in $4$-dimensions \& Perturbing Closed Definite Triples \ \textsection}\label{HKA prelims}

Let us first recall the basics of hyper-K\"{a}hler metrics in dimension 4 which will be needed, mostly following Foscolo's paper \cite{Foscolo} and survey \cite{FoscoloS}.

First, recall that for $\paren*{M^{4},\mu_{0}}$ an oriented smooth $4$-manifold with a volume form $\mu_{0}$, that each $\Lambda^{2}T_{x}^{*}M$ has a non-degenerate quadratic form of signature $\paren*{3,3}$ via $\frac{\cdot \wedge \cdot}{\mu_{0}}$.

\begin{notation}\label{triple of 2 forms as columns!}
Now let $\bomega\coloneq \paren*{\omega_{1},\omega_{2},\omega_{3}} \in \Omega^{2}\paren*{M}\otimes \bbR^{3}$ denote a triple of \textit{real} $2$-forms, \textit{and from now on we consider these triples as \textbf{column} vectors}. This always yields us a (smooth) symmetric matrix-valued function $Q \in \Gamma\paren*{M, \Sym^{2}\paren*{\bbR^{3}}}$ via \textbf{componentwise division}: $$\frac{1}{2} \frac{\bomega\wedge \bomega^{T}}{\mu_{0}} = Q$$ That is, $\frac{1}{2} \omega_{i}\wedge \omega_{j} = Q_{ij}\mu_{0}$.
\end{notation}

\begin{definition}\label{HKA definite triple definition} On $\paren*{M^{4},\mu_{0}}$ an oriented smooth $4$-manifold with a volume form $\mu_{0}$, for a triple of 2-forms $\bomega$ the following are equivalent:\begin{itemize}
\itemsep0em 
\item $\spano\paren*{\bomega} \coloneq \spano\paren*{\omega_{1},\omega_{2},\omega_{3}} \subseteq \Lambda^{2}T_{x}^{*}M$ is a $3$-dimensional positive definite subspace for every $x \in M$. 
\item The associated $Q$ is a positive definitive matrix everywhere.\end{itemize}
In such an event, we call $\bomega$ a \textbf{definite triple}.\end{definition}

Now for $\bomega$ a definite triple, we may associate an \textbf{associated volume form} $\mu_{\bomega}$ and a new matrix called the \textbf{associated intersection matrix} $Q_{\bomega}$ via $\mu_{\bomega} \coloneq \paren*{\det Q}^{\frac{1}{3}}\mu_{0}$ and $Q_{\bomega} \coloneq \paren*{\det Q}^{-\frac{1}{3}}Q$. Hence what we've done is normalized the determinant of $Q$ to have unit determinant, and have appropriately normalized the given volume form so that we satisfy $$\frac{1}{2} \frac{\bomega\wedge \bomega^{T}}{\mu_{\bomega}} = Q_{\bomega}$$ as well as $$\mu_{\bomega} = \frac{1}{2 \Tr\paren*{Q_{\bomega}}} \Tr\paren*{\bomega\wedge \bomega^{T}}$$ Thus both $\mu_{\bomega}, Q_{\bomega}$ are \textit{independent} of the choice of volume form $\mu_{0}$. Moreover, the following are equivalent:\begin{itemize}
\itemsep0em 
\item A choice of conformal class $[g]$ on $M$.
\item The choice of a $3$-dimensional positive definite subspace of $\Lambda^{2} T_{x}^{*}M, \forall x \in M$.
\end{itemize}

Hence every definite triple $\bomega$ defines a Riemannian metric $g_{\bomega}$ via requiring $\restr{\spano\paren*{\bomega}}{x} = \Lambda_{g_{\bomega}}^{+}T_{x}^{*}M,\forall x \in M$ and $\dV_{g_{\bomega}} = \mu_{\bomega}$, where $\Lambda_{g_{\bomega}}^{+}, \Omega^{+}_{g_{\bomega}}$ denote the self-dual \textit{2-forms} WRT $*_{g_{\bomega}}$.

\begin{definition} On $\paren*{M^{4},\mu_{0}}$ with a definite triple $\bomega$, we say that \begin{itemize}
\itemsep0em 
\item $\bomega$ is \textbf{closed} or \textbf{hypersymplectic} when $d\bomega \coloneq \paren*{d\omega_{1},d\omega_{2},d\omega_{3}} = 0$.
\item $\bomega$ is an \textbf{$\SU(2)$-structure} when $Q_{\bomega} = \id$, i.e. $\frac{1}{2} \frac{\bomega\wedge \bomega^{T}}{\mu_{\bomega}} = \id$.
\item $\bomega$ is \textbf{hyper-K\"{a}hler} when it is both closed and an $\SU(2)$-structure.			
\end{itemize} Note that the 3rd point is equivalent to $\Hol\paren*{g_{\bomega}}\subset \Sp(1)$, \textit{whence \hka metrics are Ricci-flat}.\end{definition}

\begin{remark}\label{SU2 structure and associated volume form equation}
Note that when $\paren*{M^{4},\bomega}$ is an $\SU(2)$-structure, we have that $$\mu_{\bomega} = \frac{1}{6} \bomega^{T} \wedge \bomega = \frac{1}{6}\Tr\paren*{\bomega\wedge\bomega^{T}}$$ \textit{Note} the order swap, i.e. $\mu_{\bomega} = \frac{1}{6} \sum_{i=1}^{3} \omega_{i}\wedge \omega_{i}$. Indeed, either notice that $\Tr\paren*{\id} = 3$ or write out $\frac{1}{2}\bomega \wedge \bomega^{T} = \id\cdot \mu_{\bomega}$, giving us $\frac{1}{2} \omega_{i}\wedge \omega_{j} = \delta_{ij} \mu_{\bomega}$ whence $\mu_{\bomega} = \frac{1}{2} \omega_{i} \wedge \omega_{i}, \forall i \in \set*{1,2,3}$. Whence summing over all $3$ equivalent expressions for $\mu_{\bomega}$ and then dividing by $3$ gives us our sought after expression. Therefore, being an $\SU(2)$-structure may be equivalently stated as $\bomega$ satisfying one of the following \textit{equivalent} equations: $$Q_{\bomega} = \id \Longleftrightarrow \frac{1}{2} \frac{\bomega\wedge \bomega^{T}}{\mu_{\bomega}} = \id \Longleftrightarrow \frac{\bomega\wedge \bomega^{T}}{\bomega^{T} \wedge \bomega} = \frac{1}{3}\id $$
\end{remark}

\begin{remark}\label{CYvsHKA}
Being \hka is related to being Calabi-Yau (or Ricci-flat K\"{a}hler), as well as the usual definition of hyper-K\"{a}hler, as follows: First, \textit{choose a direction in $S^{2} = S^{2}(1) \subset \bbR^{3}$, say $e_{1}\in S^{2}\subset \bbR^{3}$}.

For $\paren*{M^{4},\bomega, g_{\bomega}}$ hyper-K\"{a}hler, we have that upon writing $\omega_{c} \coloneq \omega_{2}+ i\omega_{3}$, $\overline{\omega_{c}} \coloneq \omega_{2} - i\omega_{3}$, that this makes $\omega_{c}$ a complex 2-form, defining an almost complex structure $J_{1}$ ala Proposition \ref{holovolformdeterminescomplexstructure} via defining $\Lambda^{1,0}_{\bbC} T_{x}^{*}M \coloneq \ker\paren*{\restr{\alpha}{x}\mapsto \restr{\alpha}{x}\wedge \restr{\omega_{c}}{x}},\forall x \in M$. In fact, as the notation suggests, the almost complex structure $J_{1}$ to which $\omega_{c}$ is complex with respect to is \textit{precisely} the almost complex structure $J_{1}$ by which $\omega_{1}$ is compatible with $g_{\bomega}$. $d\omega_{c} = 0$ forces the differential ideal generated by the $(1,0)$-forms be closed, whence $J_{1}$ is integrable by Newlander-Nirenberg. Moreover, $\omega_{1}$ and $\omega_{c}$ are respectively real $(1,1)$ and holomorphic $(2,0)$ forms \textit{with respect to $J_{1}$}, and since $\omega_{1}$ is closed \& nondegenerate, this makes $\paren*{M^{4},J_{1},g,\omega_{1}}$ \textit{K\"{a}hler} with $g \coloneq g_{\bomega}$ the induced metric by 2-out-of-3. Moreover, $\frac{1}{2} \frac{\bomega\wedge \bomega^{T}}{\mu_{\bomega}} = Q_{\bomega} = \id$ gives\footnote{Indeed, $\frac{1}{2} \frac{\bomega\wedge \bomega^{T}}{\mu_{\bomega}} = \id \Longleftrightarrow \frac{\bomega\wedge \bomega^{T}}{\bomega^{T} \wedge \bomega} = \frac{1}{3}\id $ thus we have that $\omega_{i}\wedge \omega_{j} = \frac{1}{3}\delta_{ij}\paren*{\omega_{1}\wedge \omega_{1}+ \omega_{2}\wedge \omega_{2}+ \omega_{3}\wedge \omega_{3}}$, whence setting $i = j = 1$ and noting that $\omega_{c} \wedge \overline{\omega_{c}} = \omega_{2}\wedge \omega_{2}+ \omega_{3}\wedge \omega_{3}$ finishes it off.} us $\omega_{1}^{2} = \frac{1}{2}\omega_{c}\wedge \overline{\omega_{c}}$, aka by Proposition \ref{calabiyau}, $g$ is \textit{Ricci-flat}. Since the choice of direction $e_{1}\in S^{2}\subset \bbR^{3}$ was \textit{arbitrary}, we get that $g_{\bomega}$ is Ricci-flat \ka WRT to an $S^{2}$s worth of compatible integrable complex structures. In other words, $\paren*{J_{e},\omega_{e}}_{e \in S^{2}}$ are all compatible WRT the \textit{fixed} metric $g_{\bomega}$. Holonomially this is because $\Sp(1)\cong \SU(2)$ (and more generally $\Sp(m)\subseteq \SU(2m)$), and schematically $$\paren*{M^{4},\bomega, g_{\bomega}}\text{ hyper-K\"{a}hler}\Longrightarrow \paren*{M^{4}, J_{e}, g_{\bomega}, \omega_{e}} \text{ Ricci-flat K\"{a}hler},\forall e \in S^{2}$$

Conversely, only in $\dim_{\bbR}M = 4$ because $\Sp(1)\cong \SU(2)$, given $\paren*{M^{4}, J, g, \omega}$ Ricci-flat K\"{a}hler (whence $M^{4}$ is clearly orientable since it admits complex structures), we have that the real and imaginary parts of the associated nowhere vanishing holomorphic $(2,0)$-form/holomorphic volume form $\Omega$ (\textit{Remark} \ref{Ricci flat Kahler vs CY notation}) give us $\omega_{2} \coloneq \Re \Omega,\omega_{3} \coloneq \Im \Omega$ respectively which are closed, and we let $\omega_{1}\coloneq\omega$, whence giving us our \hka triple $\bomega$. Moreover, we still end up with $g_{\bomega}  = g$ our original given \ka \textit{metric}. Schematically, \begin{al}
\paren*{M^{4}, J, g, \omega}\text{ Ricci-flat K\"{a}hler} \Longrightarrow& \paren*{M^{4},\bomega , g_{\bomega} }\text{ hyper-K\"{a}hler}\\
&\text{where }\bomega \coloneq \paren*{\omega, \Re \Omega, \Im \Omega}\\
&\text{and } g_{\bomega} = g
\end{al}

Hence, when passing from Ricci-flat \ka $\paren*{M^{4}, J, g, \omega}$ to hyper-K\"{a}hler $\paren*{M^{4},\bomega , g_{\bomega}}$, \textit{we forget the given complex structure $J$ and work with the underlying oriented smooth $4$-manifold $M^{4}$.}\end{remark}

Now suppose $\paren*{M^{4},\bomega}$ is hypersymplectic/closed, and (interpreted componentwise) $\Abs*{Q_{\bomega} - \id}_{C^{0}\paren*{M^{4}}}<\sigma$ for $0< \sigma \ll 1$. We want to perturb $\bomega$ into a hyper-K\"{a}hler triple, whence by \textit{Remark} \ref{SU2 structure and associated volume form equation} we seek a triple of \textit{closed} 2-forms $\bm{\eta}$ such that $\bomega + \bm{\eta}$ is hyper-K\"{a}hler, i.e. $$\frac{\paren*{\bomega+\bm{\eta}}\wedge \paren*{\bomega+\bm{\eta}}^{T}}{\paren*{\bomega+\bm{\eta}}^{T} \wedge \paren*{\bomega+\bm{\eta}}} = \frac{1}{3}\id $$

Decompose $\bm{\eta} = \bm{\eta}^{+_{g_{\bomega}}} + \bm{\eta}^{-_{g_{\bomega}}}$ into self-dual and anti-self-dual parts WRT $g_{\bomega}$, respectively, whence $\bm{\eta}^{\pm_{g_{\bomega}}} \coloneq \frac{\bm{\eta} \pm *_{g_{\bomega}} \bm{\eta}}{2}$. Write the self-dual part as $\bm{\eta}^{+_{g_{\bomega}}} = A \bomega$ for some matrix $A \in \Gamma\paren*{M, \End\paren*{\bbR^{3}}}$.

Given any tuple of $2$-forms $\bm{\gamma}$, denote by $\bm{\gamma} \star \bm{\gamma} \coloneq \frac{1}{2} \frac{\bm{\gamma} \wedge\bm{\gamma}^{T}}{\mu_{\bomega}} \in \Gamma\paren*{M,\Sym^{2}\paren*{\bbR^{3}}} $ that particular symmetric $3\times 3$ matrix valued function on $M$. Rewrite the perturbation above via factoring out the product \textit{using the fact that the wedge product of a self-dual form and an anti-self-dual form vanishes}\footnote{Indeed, $\fa^{+} \wedge \fq^{-} = \paren*{*_{g}\fa^{+}} \wedge \paren*{*_{g} \fq^{-}} = - \fa^{+}\wedge \fq^{-}$.}, then multiplying the LHS by $1 = \frac{2\mu_{\bomega}}{2\mu_{\bomega}}$, then using the relation $\frac{1}{2} \frac{\bomega\wedge \bomega^{T}}{\mu_{\bomega}} = Q_{\bomega}$, and finally passing the denominator of the LHS to the RHS. We thus get $$Q_{\bomega} + Q_{\bomega} A^{T} + AQ_{\bomega} + AQ_{\bomega}A^{T} + \bm{\eta}^{-_{g_{\bomega}}} \star \bm{\eta}^{-_{g_{\bomega}}} = \frac{1}{3}\paren*{\Tr\paren*{Q_{\bomega}} + \Tr\paren*{AQ_{\bomega}} + \Tr\paren*{Q_{\bomega}A^{T}} + \Tr\paren*{AQ_{\bomega}A^{T}} + \Tr\paren*{\bm{\eta}^{-_{g_{\bomega}}} \star \bm{\eta}^{-_{g_{\bomega}}}}}\id$$

Therefore letting $M_{0} \coloneq M - \frac{1}{3}\Tr\paren*{M}\cdot \id$ be the trace-free part of any $3\times 3$ matrix $M$, our perturbation equation is thus $$\paren*{Q_{\bomega} A^{T} + AQ_{\bomega} + AQ_{\bomega}A^{T}}_{0} = \paren*{- Q_{\bomega} - \bm{\eta}^{-_{g_{\bomega}}} \star \bm{\eta}^{-_{g_{\bomega}}}}_{0}$$

Now consider the map \begin{al}
\Gamma\paren*{M,\End\paren*{\bbR^{3}}} &\rightarrow \Gamma\paren*{M,\Sym^{2}\paren*{\bbR^{3}}}\\
A &\mapsto Q_{\bomega} A^{T} + AQ_{\bomega} + AQ_{\bomega}A^{T}
\end{al} Note that the differential of this map at $0$ is $A\mapsto Q_{\bomega} A^{T} + AQ_{\bomega}$, and since $\Abs*{Q_{\bomega} - \id}_{C^{0}}<\sigma$ for $0< \sigma \ll 1$, \textit{pointwise} this linear map induces an isomorphism $\Sym^{2}\paren*{\bbR^{3}}\rightarrow \Sym^{2}\paren*{\bbR^{3}}$ for $\sigma\ll 1$ sufficiently small. Whence upon restricting to trace-free symmetric matrices we invoke the \textbf{inverse function theorem} to define a smooth map $\cF: \Sym_{0}^{2}\paren*{\bbR^{3}}\rightarrow \Sym_{0}^{2}\paren*{\bbR^{3}}$ such that $\paren*{Q_{\bomega} A^{T} + AQ_{\bomega} + AQ_{\bomega}A^{T}}_{0} = S \Longleftrightarrow A = \cF(S)$, whence we may rewrite the rewritten perturbation equation above as $$\bm{\eta}^{+_{g_{\bomega}}} = \cF\paren*{\paren*{- Q_{\bomega} - \bm{\eta}^{-_{g_{\bomega}}} \star \bm{\eta}^{-_{g_{\bomega}}}}_{0}}\bomega$$

\begin{remark}
In the event that $\bomega$ is hyper-K\"{a}hler, aka $Q_{\bomega} = \id$, we have that the 3 dimensional kernel of the linearization $A\mapsto Q_{\bomega} A^{T} + AQ_{\bomega}$ at the origin is in \textit{bijection} with \textbf{infinitesimal hyper-K\"{a}hler rotations}. 
\end{remark}

Now let $\cH^{+}_{g_{\bomega}}$ be the $3$-dimensional space of \textit{self-dual} harmonic \textit{2-forms} WRT $g_{\bomega}$. We have that $\cH^{+}_{g_{\bomega}} = \spano_{\bbR}\paren*{\bomega}$, aka $\cH^{+}_{g_{\bomega}}$ consists of \textit{constant} linear combinations of the $3$ components of $\bomega$, because $\bomega$ is closed and self-dual hence harmonic, and linearly independent because $\bomega$ is definite.

In addition to $\bm{\eta}$ being closed, let us now impose the following \begin{ansatz}\label{HKA perturbation ansatz}
Let $$\boxed{\bm{\eta} = d\bm{a}+ \bm{\zeta}}$$ where $\bm{a} \in \Omega^{1}\paren*{M}\otimes \bbR^{3}$ and $\bm{\zeta} \in \cH^{+}_{g_{\bomega}} \otimes \bbR^{3}$.
\end{ansatz}

With this ansatz, we may now finally rewrite the rewritten perturbation equation above as the following \textbf{nonlinear PDE system}: 
$$d_{g_{\bomega}}^{+}\bm{a} + \bm{\zeta} = \cF\paren*{\paren*{- Q_{\bomega} - d^{-}_{g_{\bomega}}\bm{a} \star d^{-}_{g_{\bomega}}\bm{a}}_{0}}\bomega$$ where $$\paren*{\bm{a}, \bm{\zeta}} \in \paren*{\Omega^{1}\paren*{M}\oplus \cH^{+}_{g_{\bomega}}} \otimes \bbR^{3}$$ Here of course $d_{g}^{\pm} \coloneq \frac{d \pm *_{g}d}{2}$ are the self-dual and anti-self-dual parts of the exterior derivative, and $d_{g}^{*} = -*_{g} d *_{g}$ is the adjoint of the exterior derivative acting on $1$-forms on an oriented Riemannian $4$-manifold.

However, the above nonlinear system of PDE is \textit{not} elliptic: a simple dimension count yields that it inputs an element of $\paren*{\Omega^{1}\paren*{M} \oplus \cH^{+}_{g}}\otimes \bbR^{3}$, which is $7\cdot 3 = 21$-dimensional, and outputs an element of $\Omega^{+}_{g}\paren*{M} \otimes \bbR^{3}$, which is $3\cdot 3 = 9$-dimensional. We therefore also impose the following \textbf{gauge-fixing condition}: $$d_{g_{\bomega}}^{*}\bm{a} = 0$$ 

Whence our nonlinear PDE system now becomes $$d_{g_{\bomega}}^{+}\bm{a} + \bm{\zeta} = \cF\paren*{\paren*{- Q_{\bomega} - d^{-}_{g_{\bomega}}\bm{a} \star d^{-}_{g_{\bomega}}\bm{a}}_{0}}\bomega, \qquad d_{g_{\bomega}}^{*}\bm{a} = 0$$ This time around, we input an element of $\paren*{\Omega^{1}\paren*{M} \oplus \cH^{+}_{g}}\otimes \bbR^{3}$, which is $7\cdot 3 = 21$-dimensional, and output an element of $\paren*{\Omega^{0}\paren*{M}\oplus \Omega^{+}_{g}\paren*{M}} \otimes \bbR^{3}$, which is $\paren*{3+4}\cdot 3 = 21$-dimensional. Moreover, the linearization is $$\paren*{D_{g_{\bomega}}\circ \pr_{1} + \paren*{0,\pr_{2}} }\otimes \bbR^{3}= \paren*{D_{g_{\bomega}}\circ \pr_{1} + 0 \oplus \id_{\cH^{+}_{g_{\bomega}}} }\otimes \bbR^{3}$$ with $D_{g_{\bomega}} \coloneq d_{g_{\bomega}}^{*} \oplus d_{g_{\bomega}}^{+}: \Omega^{1}\paren*{M}\rightarrow \Omega^{0}\paren*{M}\oplus \Omega_{g_{\bomega}}^{+}\paren*{M}$ the \textbf{leading order term} and $\id_{\cH^{+}_{g_{\bomega}}}: \cH^{+}_{g_{\bomega}}\rightarrow \cH^{+}_{g_{\bomega}}\subset \Omega^{+}_{g_{\bomega}}\paren*{M}$ the \textbf{lower order term}. A straightforward computation gives us that the principal symbol of $D_{g}: \Omega^{1}\paren*{M} \rightarrow \Omega^{0}\paren*{M} \oplus \Omega_{g}^{+}\paren*{M}$ is $$\sigma_{1}\paren*{D_{g}}\paren*{x,\xi} = \frac{1}{2} \paren*{- \xi^{\sharp}\intrp \cdot , *_{g} \paren*{\xi \wedge \cdot}}  : T^{*}_{x}M \rightarrow \bbR \oplus \Lambda^{+}_{g}T_{x}^{*}M,\qqfa \paren*{x,\xi} \in T^{*}M$$ which is clearly an isomorphism away from the zero section of $T^{*}M$ since $\dim_{\bbR}M = 4$ and $\sigma_{1}\paren*{D_{g}}\paren*{x,\xi}$ is clearly injective for $\xi \neq 0$, whence \textit{$D_{g}$ is a first order elliptic PDO} and our nonlinear PDE system above is in fact a \textbf{nonlinear \textit{elliptic} system}.

Moreover, we have the following standard facts from Hodge theory: \begin{prop}\label{HKA Dirac implies harmonic}
Let $\paren*{M^{4},g}$ be an oriented Riemannian $4$-manifold. Let $\fa \in \Omega^{1}\paren*{M}$ be any $1$-form. Then\footnote{\label{HKA Dirac operator remark}In fact, $D_{g}$ is the \textbf{Dirac operator} associated to $g$ when $g$ is a $4$-dimensional \hka metric.} $$D_{g}\fa = 0 \Longrightarrow \Delta_{g}\fa = 0$$\end{prop} \begin{proof}
Indeed, $D_{g}\fa = 0$ means that $d_{g}^{*}\fa = 0$ and $d^{+}_{g} \fa = \frac{d\fa + *_{g}d\fa}{2} = 0$ hence $*_{g}d\fa = - d\fa$. Now since $d^{*}_{g}\fa = 0$, we have that $\Delta_{g}\fa \coloneq \paren*{d d^{*}_{g} + d_{g}^{*}d}\fa = 0 + d_{g}^{*}d \fa$. But $*_{g}d\fa = - d\fa$ and $d_{g}^{*} = - *_{g}d*_{g}$ on $2$-forms means that $d_{g}^{*}d \fa = - *_{g}d *_{g} d\fa = *_{g} d^{2}\fa = 0$, as was to be shown.\end{proof}

\begin{prop}\label{HKA 4d Hodge theory}
Let $\paren*{M^{4},g}$ be an oriented Riemannian $4$-manifold. Suppose that $b_{1}\paren*{M} = 0$, and moreover suppose that $\paren*{M^{4},g}$ admits a Fredholm theory for $\Delta_{g} \coloneq \paren*{d + d_{g}^{*}}^{2}$ (e.g. when $M^{4}$ is closed). Then we have the following:\begin{itemize}
\itemsep0em 
\item $\Omega^{1}\paren*{M} = d\paren*{\Omega^{0}\paren*{M}} \oplus \mathring{\Omega}_{g}^{1}\paren*{M}$
\item $\Omega^{+}_{g}\paren*{M} = \cH^{+}_{g} \oplus d_{g}^{+}\paren*{\mathring{\Omega}_{g}^{1}\paren*{M}}$
\end{itemize} Here $\mathring{\Omega}_{g}^{1}\paren*{M} \coloneq \Omega^{1}\paren*{M} \cap \ker d^{*}_{g} = \set*{\fa \in \Omega^{1}\paren*{M} : d_{g}^{*}\fa = 0}$, and all decompositions are $L_{g}^{2}$-orthogonal.
\end{prop} \begin{proof}
We firstly have that $\Omega^{1}\paren*{M} = d\paren*{\Omega^{0}\paren*{M}} \oplus d_{g}^{*}\paren*{\Omega^{2}\paren*{M}} \oplus \cH_{g}^{1}$ by definition. But since $b_{1}\paren*{M} = 0$, $d_{g}^{*}\paren*{\Omega^{2}\paren*{M}} \subset \mathring{\Omega}_{g}^{1}\paren*{M}$, and $d\paren*{\Omega^{0}\paren*{M}}, \mathring{\Omega}_{g}^{1}\paren*{M}$ are $L^{2}_{g}$-orthogonal, $\Omega^{1}\paren*{M} = d\paren*{\Omega^{0}\paren*{M}} \oplus \mathring{\Omega}_{g}^{1}\paren*{M}$ follows. $\Omega^{+}_{g}\paren*{M} = \cH^{+}_{g} \oplus d_{g}^{+}\paren*{\mathring{\Omega}_{g}^{1}\paren*{M}}$ is thus immediate, as was to be shown.
\end{proof}

Therefore, Proposition \ref{HKA 4d Hodge theory} immediately tells us that $d_{g}^{+} \oplus \id_{\cH^{+}_{g}}: \mathring{\Omega}^{1}_{g}\paren*{M} \oplus \cH^{+}_{g} \onto \Omega^{+}_{g}\paren*{M}$ hence $\paren*{d_{g}^{+} \oplus \id_{\cH^{+}_{g}}}\otimes \bbR^{3}: \paren*{\mathring{\Omega}^{1}_{g}\paren*{M} \oplus \cH^{+}_{g}}\otimes \bbR^{3} \onto \Omega^{+}_{g}\paren*{M} \otimes \bbR^{3}$ is \textit{surjective}, and $$\ker\paren*{d_{g}^{+} \oplus \id_{\cH^{+}_{g}}} = \cH_{g}^{1}$$ the space of harmonic $1$-forms. But since Proposition \ref{HKA 4d Hodge theory} requires our $\paren*{M^{4},g}$ to satisfy $b_{1}\paren*{M} = 0$, we therefore have that $d_{g}^{+} \oplus \id_{\cH^{+}_{g}}$ hence $\paren*{d_{g}^{+} \oplus \id_{\cH^{+}_{g}}}\otimes \bbR^{3}$ is \textit{always} an isomorphism.

Henceforth when $\paren*{M^{4},\bomega}$ is hypersymplectic/closed, $M^{4}$ is \textit{closed} and $b_{1}\paren*{M^{4}} = 0$, and $\Abs*{Q_{\bomega} - \id}_{C^{0}\paren*{M^{4}}}<\sigma$ for $0< \sigma \ll 1$, we have that we may perturb $\bomega$ to a genuinely \hka triple $\bomega + d\bm{a}+ \bm{\zeta} = \bomega + d_{g_{\bomega}}^{+}\bm{a}+ \bm{\zeta}$ provided we may be able to prove the existence of a smooth solution to the following \textbf{nonlinear \textit{elliptic} system}: $$\Phi\paren*{\bm{a},\bm{\zeta}} = 0$$ where \begin{al}
\Phi: \paren*{\mathring{\Omega}_{g}^{1}\paren*{M}\oplus \cH^{+}_{g_{\bomega}}} \otimes \bbR^{3} &\rightarrow \Omega_{g_{\bomega}}^{+}\paren*{M}\otimes\bbR^{3}\\
\paren*{\bm{a},\bm{\zeta}} &\mapsto d_{g_{\bomega}}^{+}\bm{a}+ \bm{\zeta} - \cF\paren*{\paren*{- Q_{\bomega} - d^{-}_{g_{\bomega}}\bm{a} \star d^{-}_{g_{\bomega}}\bm{a}}_{0}}\bomega
\end{al} and with $D_{0}\Phi = \paren*{d_{g}^{+} \oplus \id_{\cH^{+}_{g}}}\otimes \bbR^{3}$ always an isomorphism (whence why we've already factored in the gauge-fixing term into the domain via restricting to $\mathring{\Omega}_{g}^{1}\paren*{M}$).

\begin{remark}
Note that throughout, \textit{contrary} to producing Calabi-Yau metrics \textit{within an initial \ka class} via solving the complex Monge-Ampere equation, it \textit{is necessary} to deform the cohomology classes of $\omega_{1},\omega_{2},\omega_{3}$ (whence the translation by harmonic 2-forms $\bm{\zeta}$ in the perturbation $\bm{\eta}$), because every hyper-K\"{a}hler triple must satisfy $$\frac{1}{2} \Inner*{[\bomega] \cup [\bomega]^{T}, [M]} = \begin{pmatrix}
\Vol_{g_{\bomega}}\paren*{M} & 0 & 0\\
0 & \Vol_{g_{\bomega}}\paren*{M} & 0\\
0 & 0 & \Vol_{g_{\bomega}}\paren*{M}
\end{pmatrix} $$
\end{remark}

\subsection{\textsection \ Eguchi-Hanson Metrics \& Flat Tori as ``Building Blocks'' \ \textsection}\label{HKA EH as building block}

Now let us gather together the same building blocks for the Kummer construction that has already been carried out, but this time from the hyper-K\"{a}hler perspective.

Let $g_{\bomega_{0}}, \bomega_{0}$ be the Euclidean \hka metric on $\bbR^{4}$, whence $g_{\bomega_{0}} = g_{0}$ and $\bomega_{0} = \paren*{\omega_{0}, \Re dz^{1} \wedge dz^{2}, \Im dz^{1} \wedge dz^{2}}$ by \textit{Remark} \ref{CYvsHKA}. Let $g_{EH,s},\omega_{EH,s}$ be the Eguchi-Hanson metric and \ka form constructed as from Section \ref{EH as building block}. The analogue of Proposition \ref{EHproperties} is

\begin{prop}\label{HKA EHproperties} Let $s > 0$ be a positive real number. Let $\paren*{T^{*}\CP^{1},J_{\cO\paren*{-2}}, g_{EH,s}, \omega_{EH,s}}$ be our Eguchi-Hanson \ka manifold.\begin{itemize}
\itemsep0em

\item By \textit{Remark} \ref{CYvsHKA}, we have that since $\paren*{T^{*}\CP^{1},J_{\cO\paren*{-2}}, g_{EH,s}, \omega_{EH,s}}$ is Ricci-flat K\"{a}hler, that the Eguchi-Hanson metric is also \textbf{hyper-K\"{a}hler} via $$\paren*{T^{*}\CP^{1},J_{\cO\paren*{-2}}, g_{EH,s}, \omega_{EH,s}} \Longrightarrow \paren*{T^{*}S^{2}, \bomega_{EH,s}, g_{\bomega_{EH,s}}} $$ where we \textit{forget the integrable complex structure $J_{\cO\paren*{-2}}$} (whence getting $T^{*}S^{2}$ instead of $T^{*}\CP^{1}$) and we have that $$\bomega_{EH,s} \coloneq \paren*{\omega_{EH,s}, \Re \Omega_{T^{*} \bbC P^{1}} , \Im \Omega_{T^{*} \bbC P^{1}}}$$ and $$g_{\bomega_{EH,s}} = g_{EH,s}$$

\item $\paren*{T^{*}S^{2}, \bomega_{EH,s}, g_{\bomega_{EH,s}}}$ is a complete Ricci-flat Riemannian $4$-manifold which is an \textbf{ALE \hka gravitational instanton} asymptotic at $\infty$ to $\bbR^{4}/\bbZ_{2}$ with rate $-4 < 0$. Hence $\exists K \subset T^{*}S^{2}$ a compact subset, a constant $R_{EH} > 0$ \textit{sufficiently large}, and a diffeomorphism $F: \bbR^{4}/\bbZ_{2} - D^{g_{\bomega_{0}}}_{R_{EH}}\paren*{0} \rightarrow T^{*}S^{2} - K$ such that $$\abs*{\nabla_{g_{\bomega_{0}}}^{k}\paren*{F^{*}g_{\bomega_{EH,s}} - g_{\bomega_{0}}}}_{g_{\bomega_{0}}} = O_{g_{\bomega_{0}}}\paren*{\paren*{\abs*{\cdot}^{g_{\bomega_{0}}}_{0}}^{-4-k}}\qquad\text{and}\qquad\abs*{\nabla_{g_{\bomega_{0}}}^{k}\paren*{F^{*}\bomega_{EH,s} - \bomega_{0} }}_{g_{\bomega_{0}}} = O_{g_{\bomega_{0}}}\paren*{\paren*{\abs*{\cdot}^{g_{\bomega_{0}}}_{0}}^{-4-k}}$$

on $\bbR^{4}/\bbZ_{2} - D^{g_{\bomega_{0}}}_{R_{EH}}\paren*{0}$, where the latter inequality is interpreted componentwise. 

\item In fact, this diffeomorphism outside a compact subset comes from $\pi: T^{*} \bbC P^{1} - \bbC P^{1} \overset{\text{Biholo}}{\cong} \frac{\bbC^{2} - \set*{0}}{\bbZ_{2}}$, i.e. $\pi: T^{*} \bbC P^{1} - D^{g_{\bomega_{0}}}_{R_{EH}}\paren*{\bbC P^{1}} \overset{\text{Biholo}}{\cong} \bbC^{2}/\bbZ_{2} - D^{g_{\bomega_{0}}}_{R_{EH}}(0)$, \textit{since all biholomorphisms are diffeomorphisms upon forgetting the complex structures}.

\item From the $s \pi^{*}\paren*{\omega_{FS}}$ term in $\omega_{EH,s} = \paren*{\bomega_{EH,s}}_{1}$ and the fact that the Fubini-Study metric on $\CP^{1}$ is isometric to the round metric of $S^{2}\paren*{\frac{1}{2}}$, we have that the exceptional divisor/zero section $S^{2}\subset \paren*{T^{*}S^{2}, g_{\bomega_{EH,s}}}$ has sectional curvature $\frac{4}{s}$, volume $s\pi$, diameter $s^{\frac{1}{2}}$, injectivity radius $O\paren*{s^{\frac{1}{2}}}$, and self intersection number $S^{2} \cdot S^{2} = -2$.

In particular, as $s\to 0$, the curvature of the exceptional $S^{2}$ blows up, but its diameter/volume shrinks.

\item The Eguchi-Hanson metric $\paren*{T^{*}S^{2}, \bomega_{EH,s}, g_{\bomega_{EH,s}}}$ is \textbf{$\U(2)$-invariant}. In particular, we immediately have that the Eguchi-Hanson metric is $\U(1)$-invariant\footnote{In fact, this $\U(1)$-invariance is what allows us to write the Eguchi-Hanson metric (from the \hka perspective) via the \textbf{Gibbons-Hawking ansatz}. The harmonic function used to get Eguchi-Hanson therefore is just a Dirac monopole with 2 singularities and \textit{zero mass} (or else we would get \textbf{ALF} or \textbf{Asymptotically Locally Flat}).}, i.e. $$R^{*}_{\lambda} \paren*{g_{\bomega_{EH,s}}} = g_{\bomega_{EH,s}},\qqfa \lambda \in \U(1) \subset \bbC^{\times}$$

\item $g_{\bomega_{EH,s}}$ and $g_{\bomega_{EH,1}}$ (hence similarly for $\bomega_{EH,s}$, $\bomega_{EH,1}$) are isometric up to scaling, i.e. $$R_{s^{\frac{1}{2}}}^{*}\paren*{\frac{1}{s}g_{\bomega_{EH,s}}} = g_{\bomega_{EH,1}}$$

In fact, we may say more: from Kronheimer's classification of \hka ALE gravitational instantons \cite{Kronheimer1} \cite{Kronheimer2} (see \cite{JoyceBook}), \textbf{all \hka ALE spaces asymptotic to $\bbR^{4}/\bbZ_{2}$ are isomorphic to $\paren*{T^{*}S^{2}, \bomega_{EH,1}, g_{\bomega_{EH,1}}} $ up to dilation/scaling \& \hka rotations}, namely $$\paren*{T^{*}S^{2}, \bomega_{t,A}, g_{\bomega_{t,A}}}$$ where the isomorphism is given by a diffeomorphism $P \in \Diff\paren*{T^{*}S^{2}}$ satisfying $$P^{*}\paren*{g_{\bomega_{t,A}}} = t^{2} g_{A\bomega_{EH,1}}\qquad \text{and}\qquad P^{*}\paren*{\bomega_{t,A}} = t^{2}A \bomega_{EH,1}$$ for some $\paren*{t,A} \in \bbR_{>0}\times \SO\paren*{3,\bbR}$, where (recalling our column convention from Notation \ref{triple of 2 forms as columns!}) the RHS denotes matrix multiplication.

Recalling that \hka metrics have an $S^{2}(1)$s worth of \ka forms which are compatible with it, that $\SO\paren*{3,\bbR}$ acts transitively on $S^{2}(1)\subset \bbR^{3}$, and that WLOG the direction in $S^{2}$ corresponding to $\bomega_{EH,1}$ is $e_{1}$, we thus have that upon \textbf{counting parameters} that we have $t \in \bbR_{>0}$ and a choice of direction $e \in S^{2}(1)\subset \bbR^{3}$ to which $A \in \SO\paren*{3,\bbR}$ sends $e_{1} \mapsto e$. 

In other words, \textit{all \hka ALE spaces asymptotic to $\bbR^{4}/\bbZ_{2}$ form a \textbf{3-parameter family}, as each is given by a choice of dilation/scaling constant $t \in \bbR_{>0}$ to scale $g_{\bomega_{EH,1}}$ and a choice of direction $e \in S^{2}(1)$ to send $\bomega_{EH,1}$ to under ``rotation''.} For future reference, denote the $3$-parameter family of all \hka ALE spaces asymptotic to $\bbR^{4}/\bbZ_{2}$ as $$\set*{\paren*{T^{*}S^{2}, \bomega^{ALE}_{t,e}, g_{\bomega^{ALE}_{t,e}}}}_{\paren*{t,e} \in \bbR_{>0}\times S^{2}}$$ where $\bomega^{ALE}_{t,e}, g_{\bomega^{ALE}_{t,e}} \overset{\text{isometric}}{\cong} A^{e}\bomega_{EH,t^{2}}, g_{A^{e}\bomega_{EH, t^{2}}} \overset{\text{isometric}}{\cong} t^{2}A^{e}\bomega_{EH,1}, t^{2}g_{A^{e}\bomega_{EH, 1}} $ where $A^{e} \in \SO\paren*{3,\bbR}$ maps the $\bomega_{EH,1}$ direction in $S^{2}$ to $e \in S^{2}$.
\end{itemize}\end{prop}

Let $\epsilon > 0$ be the gluing parameter and $0 < \lambda < 1$ the gluing rate from Data \ref{constant epsilon and lambda introduction}. The analogue of Proposition \ref{preglueEH BG interpolation!} is \begin{prop}\label{HKA preglueEH BG interpolation!} For $\epsilon > 0$ from Data \ref{constant epsilon and lambda introduction} \textit{sufficiently small} so that $\lwhat{\omega_{EH-0, \epsilon}}$ is a positive $\paren*{1,1}$-form hence a \ka form on $\paren*{T^{*}\CP^{1}, J_{\cO\paren*{-2}}}$ compatible with $J_{\cO\paren*{-2}}$. 

Then letting $\lwhat{\bomega_{EH-0, \epsilon}} \coloneq \paren*{\lwhat{\omega_{EH-0, \epsilon}}, \Re \Omega_{T^{*} \bbC P^{1}}, \Im \Omega_{T^{*} \bbC P^{1}}}$ and $g_{\lwhat{\bomega_{EH-0, \epsilon}}} \coloneq \lwhat{g_{EH-0, \epsilon}}$, we have that $\paren*{T^{*}S^{2}, \lwhat{\bomega_{EH-0, \epsilon}}, \lwhat{g_{EH-0, \epsilon}}}$ is a \hka manifold such that $$\lwhat{\bomega_{EH-0, \epsilon}} = \begin{cases}
\bomega_{EH,1} & \text{on } \set*{\abs*{\cdot}^{g_{\bomega_{0}}}_{S^{2}} \leq \frac{1}{\epsilon^{\lambda}} } \\
\bomega_{0} & \text{on }\set*{\frac{2}{\epsilon^{\lambda}} \leq \abs*{\cdot}^{g_{\bomega_{0}}}_{S^{2}}}
\end{cases}$$ By 2-out-of-3, all of the above hold verbatim for the associated Riemannian metric $g_{\lwhat{\bomega_{EH-0, \epsilon}}}$. Moreover, $$\abs*{\nabla_{g_{\bomega_{0}}}^{k}\paren*{g_{\lwhat{\bomega_{EH-0, \epsilon}}} - g_{\bomega_{0}}}}_{g_{\bomega_{0}}} \leq c_{2}(k) \epsilon^{\paren*{4+k}\lambda}\qquad\text{and}\qquad\abs*{\nabla_{g_{\bomega_{0}}}^{k}\paren*{\lwhat{\bomega_{EH-0, \epsilon}} - \bomega_{0}}}_{g_{\bomega_{0}}} \leq c_{1}(k+2) \epsilon^{\paren*{4+k}\lambda}$$ on the annuli $\set*{\frac{1}{\epsilon^{\lambda}} \leq \abs*{\cdot}^{g_{\bomega_{0}}}_{S^{2}} \leq \frac{2}{\epsilon^{\lambda}}}$ when $\epsilon>0$ is \textit{sufficiently small} (specifically $R_{EH} < \frac{1}{\epsilon^{\lambda}} \Leftrightarrow \epsilon < R_{EH}^{-\frac{1}{\lambda}}$), for the same positive constants $c_{1}\paren*{k} > 0, k \in \bbN_{0}$ in Proposition \ref{preglueEH BG interpolation!}.
\end{prop}

The analogue of Proposition \ref{preglueEH} is \begin{prop}\label{HKA preglueEH} For gluing rates $\lambda < \frac{2}{3}$ and $\epsilon > 0$ from Data \ref{constant epsilon and lambda introduction} \textit{sufficiently small} so that $\wtilde{\omega_{EH, \epsilon}}$ is a positive $\paren*{1,1}$-form hence a \ka form on $\paren*{T^{*}\CP^{1}, J_{\cO\paren*{-2}}}$ compatible with $J_{\cO\paren*{-2}}$. 

Then letting $\wtilde{\bomega_{EH, \epsilon}} \coloneq \paren*{\wtilde{\omega_{EH, \epsilon}}, \Re \Omega_{T^{*} \bbC P^{1}}, \Im \Omega_{T^{*} \bbC P^{1}}}$ and $g_{\wtilde{\bomega_{EH, \epsilon}}} \coloneq \wtilde{g_{EH, \epsilon}}$, we have that $\paren*{T^{*}S^{2}, \wtilde{\bomega_{EH, \epsilon}}, \wtilde{g_{EH, \epsilon}}}$ is a \hka manifold such that\begin{itemize}
\itemsep0em 
\item $$\wtilde{\bomega_{EH, \epsilon}} = \begin{cases}
\bomega_{EH,\epsilon^{2}} & \text{on } \set*{\abs*{\cdot}^{g_{\bomega_{0}}}_{S^{2}} \leq \epsilon^{\lambda}} \\
\bomega_{0} & \text{on }\set*{2\epsilon^{\lambda} \leq \abs*{\cdot}^{g_{\bomega_{0}}}_{S^{2}}}
\end{cases}\qquad\text{and}\qquad g_{\wtilde{\bomega_{EH, \epsilon}}} = \begin{cases}
g_{\bomega_{EH,\epsilon^{2}}} & \text{on } \set*{\abs*{\cdot}^{g_{\bomega_{0}}}_{S^{2}} \leq \epsilon^{\lambda}} \\
g_{\bomega_{0}} & \text{on }\set*{2\epsilon^{\lambda} \leq \abs*{\cdot}^{g_{\bomega_{0}}}_{S^{2}}}
\end{cases}$$
\item With the former interpreted componentwise, $$\abs*{\nabla_{g_{\bomega_{0}}}^{k}\paren*{\wtilde{\bomega_{EH, \epsilon}} - \bomega_{0}}}_{g_{\bomega_{0}}} \leq c_{1}(k+2) \epsilon^{4 - \paren*{4+k}\lambda}\qquad\text{and}\qquad\abs*{\nabla_{g_{0}}^{k}\paren*{g_{\wtilde{\bomega_{EH, \epsilon}}} - g_{\bomega_{0}}}}_{g_{\bomega_{0}}} \leq c_{2}(k) \epsilon^{4-\paren*{4+k}\lambda}$$ on the annuli $\set*{\epsilon^{\lambda} \leq \abs*{\cdot}^{g_{\bomega_{0}}}_{S^{2}} \leq 2\epsilon^{\lambda}} = \set*{\frac{1}{\epsilon^{1-\lambda}} \leq \frac{\abs*{\cdot}^{g_{\bomega_{0}}}_{S^{2}}}{\epsilon} \leq \frac{2}{\epsilon^{1-\lambda}}}$ when $\epsilon>0$ is \textit{sufficiently small} (specifically $R_{EH} < \frac{1}{\epsilon^{1-\lambda}} \Longleftrightarrow \epsilon < R_{EH}^{-\frac{1}{1-\lambda}}$), for the same positive constants $c_{1}\paren*{k}, c_{2}(k) > 0, k \in \bbN_{0}$ as in Proposition \ref{preglueEH}.

\item With the former interpreted componentwise, $$\wtilde{\bomega_{EH,\epsilon}} = \epsilon^{2} R^{*}_{\frac{1}{\epsilon}} \paren*{\lwhat{\bomega_{EH-0, \epsilon}}}\qquad\text{and}\qquad g_{\wtilde{\bomega_{EH,\epsilon}}} = \epsilon^{2} R^{*}_{\frac{1}{\epsilon}} \paren*{g_{\lwhat{\bomega_{EH-0, \epsilon}}}}$$

hence $$\wtilde{\bomega_{EH, \epsilon}} = \begin{cases}
\epsilon^{2} R^{*}_{\frac{1}{\epsilon}} \paren*{\bomega_{EH,1}} & \text{on } \set*{\abs*{\cdot}^{g_{\bomega_{0}}}_{S^{2}} \leq \epsilon^{\lambda}} \\
\bomega_{0} = \epsilon^{2} R^{*}_{\frac{1}{\epsilon}} \paren*{\bomega_{0}} & \text{on }\set*{2\epsilon^{\lambda} \leq \abs*{\cdot}^{g_{\bomega_{0}}}_{S^{2}}}
\end{cases}\qquad\text{and}\qquad g_{\wtilde{\bomega_{EH, \epsilon}}} = \begin{cases}
\epsilon^{2} R^{*}_{\frac{1}{\epsilon}} \paren*{g_{\bomega_{EH,1}}} & \text{on } \set*{\abs*{\cdot}^{g_{\bomega_{0}}}_{S^{2}} \leq \epsilon^{\lambda}} \\
g_{\bomega_{0}} = \epsilon^{2} R^{*}_{\frac{1}{\epsilon}} \paren*{g_{\bomega_{0}}} & \text{on }\set*{2\epsilon^{\lambda} \leq \abs*{\cdot}^{g_{\bomega_{0}}}_{S^{2}}}
\end{cases}$$

\end{itemize} 
Thus we also have, \textit{for $\epsilon < 1$}, $$\dV_{g_{\wtilde{\bomega_{EH, \epsilon}}}} = \begin{cases}
\dV_{g_{\bomega_{EH,\epsilon^{2}}}} & \text{on } \set*{\abs*{\cdot}^{g_{\bomega_{0}}}_{S^{2}} \leq \epsilon^{\lambda}} \\
\paren*{1 + O\paren*{\epsilon^{4-4\lambda}}}\dV_{g_{\bomega_{0}}} & \text{on }\set*{\epsilon^{\lambda} \leq \abs*{\cdot}^{g_{\bomega_{0}}}_{S^{2}}}
\end{cases}$$ This directly implies that the associated intersection matrix satisfies (componentwise) $$\Abs*{Q_{\wtilde{\bomega_{EH, \epsilon}}} - \id}_{C^{0}\paren*{T^{*}S^{2}}} \leq \cE_{1} \epsilon^{4-4\lambda}$$ for some uniform constant $\cE_{1} > 0$. More specifically, $$\abs*{Q_{\wtilde{\bomega_{EH, \epsilon}}} - \id} \begin{cases}
\leq \cE_{1} \epsilon^{4-4\lambda} & \text{on }\set*{\epsilon^{\lambda} \leq \abs*{\cdot}^{g_{\bomega_{0}}}_{S^{2}} \leq 2\epsilon^{\lambda}}\\
= 0 & \text{on }\set*{\abs*{\cdot}^{g_{\bomega_{0}}}_{S^{2}} \leq \epsilon^{\lambda}} \sqcup \set*{2\epsilon^{\lambda} \leq \abs*{\cdot}^{g_{\bomega_{0}}}_{S^{2}}}
\end{cases}$$\end{prop}

\begin{remark}\label{HKA lambda < 2/3 bound} Verbatim as in \textit{Remark} \ref{lambda < 2/3 bound}, the reason why we now impose $\lambda < \frac{2}{3}$ on the gluing rate in Proposition \ref{HKA preglueEH} is to ensure that the curvature of $g_{\wtilde{\bomega_{EH,\epsilon}}}$ stays \textit{bounded} in the region $\set*{\epsilon^{\lambda} \leq \abs*{\cdot}^{g_{\bomega_{0}}}_{S^{2}} \leq 2\epsilon^{\lambda}} = \set*{\frac{1}{\epsilon^{1-\lambda}} \leq \frac{\abs*{\cdot}^{g_{\bomega_{0}}}_{S^{2}}}{\epsilon} \leq \frac{2}{\epsilon^{1-\lambda}}}$.\end{remark}

\subsection{\textsection \ The ``Kummer Construction'' \ \textsection}\label{HKA The Kummer Construction}

Let $\bbT^{4} = \bbR^{4}/\Lambda$ for a full rank lattice $\Lambda \cong \bbZ^{4}$ in $\bbR^{4}$. Upon equipping $\bbR^{4}$ with the standard Euclidean \hka metric $\bomega_{0},g_{\bomega_{0}}$, as each of these tensors are $\bbZ^{4}$-invariant, hence by Proposition \ref{CoveringSpacesInvariantForms} descends down to $\bbT^{4}$ yielding a \textbf{flat \hka $4$-torus $\paren*{\bbT^{4}, \bomega_{0}, g_{\bomega_{0}}}$}. Moreover, Proposition \ref{flat tori facts} still holds for our flat \hka $4$-torus $\paren*{\bbT^{4}, \bomega_{0}, g_{\bomega_{0}}}$, whence the moduli of all flat metrics on $\bbT^{4}$ is $\GL\paren*{4,\bbZ}\backslash \GL\paren*{4,\bbR}/\O(4)$, which has dimension $\frac{4\paren*{4+1}}{2} = 10$.

Just as in Section \ref{The Kummer Construction}, upon passing the flat \hka datum $\bomega_{0},g_{\bomega_{0}}$ on $\bbT^{4}$ to the quotient orbifold $\bbT^{4}/\bbZ_{2}$ with the standard involution action $p \mapsto -p$ on $\bbT^{4}$, we get a \textbf{flat \hka orbifold $\paren*{\bbT^{4}/\bbZ_{2}, \bomega_{0}, g_{\bomega_{0}}}$} with $16$ isolated orbifold point singularities $\cS \coloneq \Fix\paren*{\bbZ_{2}}$, each modeled on $\bbR^{4}/\bbZ_{2}$. 

Let $\abs*{\cdot}^{g_{\bomega_{0}}}_{\cS} \coloneq \dist_{g_{\bomega_{0}}}\paren*{\cS,\cdot}$ be the usual metric distance away from $\cS$. Just as in Data \ref{constant epsilon patching and WLOG lattice}, WLOG let $\Lambda = \bbZ^{4}$, and choose our sufficiently small $\epsilon > 0$ from Data \ref{constant epsilon and lambda introduction} and Proposition \ref{HKA preglueEH} to be even smaller so that $3\epsilon^{\lambda} < \frac{1}{2}$.

Now consider 16 copies of the closed tubular neighborhood of radius $3\epsilon^{\lambda}$ around the exceptional divisor on $T^{*}S^{2}$, i.e. $\bigsqcup_{i \in \set*{1,\dots, 16}} T^{*} S^{2} - \set*{\abs*{\cdot}^{g_{\bomega_{0}}}_{S^{2}} > 3\epsilon^{\lambda}} = \bigsqcup_{i \in \set*{1,\dots, 16}} \set*{\abs*{\cdot}^{g_{\bomega_{0}}}_{S^{2}} \leq 3\epsilon^{\lambda}} = \bigsqcup_{i \in \set*{1,\dots, 16}} D^{g_{\bomega_{0}}}_{3\epsilon^{\lambda}}\paren*{S^{2}}$. Now define \begin{al}
\paren*{\Km_{\epsilon}, \bomega_{\epsilon}, g_{\bomega_{\epsilon}}} &\coloneq \frac{\paren*{\bbT^{4}/\bbZ_{2}  , \bomega_{0}, g_{\bomega_{0}}} \bigsqcup_{i \in \set*{1,\dots, 16}} \paren*{T^{*} S^{2} - \set*{\abs*{\cdot}^{g_{\bomega_{0}}}_{S^{2}} > 3\epsilon^{\lambda}}, \wtilde{\bomega_{EH, \epsilon}}, g_{\wtilde{\bomega_{EH, \epsilon}}}} }{\sim}
\end{al} (where we have $g_{\bomega_{\epsilon}} = g_{\epsilon}$) with the equivalence relation being the identification $\set*{\abs*{\cdot}^{g_{\bomega_{0}}}_{\cS} \leq 3\epsilon^{\lambda}} \sim \set*{\abs*{\cdot}^{g_{\bomega_{0}}}_{S^{2}} \leq 3\epsilon^{\lambda}}$. That is, each of the 16 connected components of $\set*{\abs*{\cdot}^{g_{\bomega_{0}}}_{\cS} \leq 3\epsilon^{\lambda}}$, which is $\set*{\abs*{\cdot}^{g_{\bomega_{0}}}_{p} \leq 3\epsilon^{\lambda}}$ for each $p \in \cS$, gets identified with a single copy of $\set*{\abs*{\cdot}^{g_{\bomega_{0}}}_{S^{2}} \leq 3\epsilon^{\lambda}} \subset T^{*} S^{2}$, and this identification is done via first lifting/pulling back $\set*{\abs*{\cdot}^{g_{\bomega_{0}}}_{p} \leq 3\epsilon^{\lambda}} \subset \bbT^{4}/\bbZ_{2} $ to $\bbT^{4} $ via the (branched) double cover/orbifold projection, then to (a proper subset of the fundamental domain in) $\bbR^{4}$ via the quotient $\bbR^{4} \onto \bbR^{4}/\Lambda = \bbT^{4}$, translating so that $p \mapsto 0$, then pushing down via $\bbR^{4} \onto \bbR^{4}/\bbZ_{2}$ onto $\bbR^{4}/\bbZ_{2}$, then lifting/pulling back via the \textit{underlying smooth map of the} crepant resolution map $T^{*} \bbC P^{1} \onto \bbC^{2}/\bbZ_{2}$ to get $\set*{\abs*{\cdot}^{g_{\bomega_{0}}}_{S^{2}} \leq 3\epsilon^{\lambda}} \subset T^{*} S^{2}$. Thus the \textit{underlying} projection map $\pi: \Km_{\epsilon} \onto \bbT^{4}/\bbZ_{2}$ is given by the identity on $\set*{3\epsilon^{\lambda} < \abs*{\cdot}^{g_{\bomega_{0}}}_{\cS}}$, and on $\set*{\abs*{\cdot}^{g_{\bomega_{0}}}_{\cS} \leq 3\epsilon^{\lambda}} \sim \set*{\abs*{\cdot}^{g_{\bomega_{0}}}_{S^{2}} \leq 3\epsilon^{\lambda}}$ via \textit{the underlying smooth map of the} crepant resolution map $\pi: \paren*{T^{*}\CP^{1}, J_{\cO\paren*{-2}}}\onto \paren*{\bbC^{2}/\bbZ_{2},J_{0}}$. As this construction is just the same as back in Section \ref{The Kummer Construction} but done \textit{on the underlying oriented smooth manifold level}, we thus still get a Kummer $K3$ surface $\Km_{\epsilon}$

Now to define our closed/hypersymplectic definite triple $\bomega_{\epsilon}$, letting $E_{i} \cong S^{2}$ denote one of the 16 exceptional divisors and $\pi^{-1}\paren*{\cS} = \sqcup_{i\in \set*{1,\dots, 16}}E_{i}$, we set $$\bomega_{\epsilon} \coloneq \begin{cases}
\bomega_{0} & \text{on } \Km_{\epsilon} - \sqcup_{i\in \set*{1,\dots, 16}} D^{g_{\bomega_{0}}}_{3\epsilon^{\lambda}}\paren*{E_{i}} = \Km_{\epsilon} - \set*{\abs*{\cdot}^{g_{\bomega_{0}}}_{\pi^{-1}\paren*{\cS}} \leq 3\epsilon^{\lambda}} = \set*{ 3\epsilon^{\frac{1}{2}} \leq \abs*{\cdot}^{g_{\bomega_{0}}}_{\pi^{-1}\paren*{\cS}}} \\
\wtilde{\bomega_{EH, \epsilon}} & \text{on } \set*{\abs*{\cdot}^{g_{\bomega_{0}}}_{\pi^{-1}\paren*{\cS}} \leq 3\epsilon^{\lambda}}
\end{cases}	$$ which is clearly well defined and smooth on all of $\Km_{\epsilon}$ by Proposition \ref{HKA preglueEH}, and we now have our $K3$ surface $\paren*{\Km_{\epsilon}, \bomega_{\epsilon}, g_{\bomega_{\epsilon}}}$ with our closed definite triple $\bomega_{\epsilon}$ and associated Riemannian metric $g_{\bomega_{\epsilon}}$, \textit{depending on the small $\epsilon > 0$.}

We still have the analogue of \textit{Remark} \ref{identification} by forgetting the complex structure and working on the underlying smooth manifold level.

The analogue of Proposition \ref{approximatemetricproperties} is \begin{prop}\label{HKA approximatemetricproperties} For $\epsilon > 0$ satisfying Data \ref{constant epsilon patching and WLOG lattice}, we have that $\bomega_{\epsilon}$ satisfies
$$\bomega_{\epsilon} = \begin{cases}
\bomega_{EH,\epsilon^{2}} & \text{on }\set*{\abs*{\cdot}^{g_{\bomega_{0}}}_{\pi^{-1}\paren*{\cS}} \leq \epsilon^{\lambda}} \\
\bomega_{0} & \text{on }\set*{2\epsilon^{\lambda} \leq \abs*{\cdot}^{g_{\bomega_{0}}}_{\pi^{-1}\paren*{\cS}}} 
\end{cases}\qquad\text{and}\qquad g_{\bomega_{\epsilon}} = \begin{cases}
g_{\bomega_{EH,\epsilon^{2}}} & \text{on } \set*{\abs*{\cdot}^{g_{\bomega_{0}}}_{\pi^{-1}\paren*{\cS}} \leq \epsilon^{\lambda}} \\
g_{\bomega_{0}} & \text{on }\set*{2\epsilon^{\lambda} \leq \abs*{\cdot}^{g_{\bomega_{0}}}_{\pi^{-1}\paren*{\cS}}}
\end{cases}$$ and, for sufficiently small $\epsilon > 0$ and $\lambda < \frac{2}{3}$ to satisfy Proposition \ref{HKA preglueEH} (and with the former interpreted componentwise), $$\abs*{\nabla_{g_{\bomega_{0}}}^{k}\paren*{\bomega_{\epsilon} - \bomega_{0}}}_{g_{\bomega_{0}}} \leq c_{1}(k+2) \epsilon^{4-\paren*{4+k}\lambda}\qquad\text{and}\qquad \abs*{\nabla_{g_{0}}^{k}\paren*{g_{\bomega_{\epsilon}} - g_{\bomega_{0}}}}_{g_{\bomega_{0}}} \leq c_{2}(k) \epsilon^{4-\paren*{4+k}\lambda}$$ on $\set*{\epsilon^{\lambda}\leq \abs*{\cdot}^{g_{\bomega_{0}}}_{\pi^{-1}\paren*{\cS}} \leq 2\epsilon^{\lambda}} = \set*{\frac{1}{\epsilon^{1-\lambda}}\leq \frac{\abs*{\cdot}^{g_{\bomega_{0}}}_{\pi^{-1}\paren*{\cS}}}{\epsilon} \leq \frac{2}{\epsilon^{1-\lambda}}}$.

Thus we also have, \textit{for $\epsilon < 1$} (which Data \ref{constant epsilon patching and WLOG lattice} guarantees), $$\dV_{g_{\bomega_{\epsilon}}} = \begin{cases}
\dV_{g_{\bomega_{EH,\epsilon^{2}}}} & \text{on } \set*{\abs*{\cdot}^{g_{\bomega_{0}}}_{\pi^{-1}\paren*{\cS}} \leq \epsilon^{\lambda}} \\
\paren*{1 + O\paren*{\epsilon^{4-4\lambda}}}\dV_{g_{\bomega_{0}}} & \text{on }\set*{\epsilon^{\lambda} \leq \abs*{\cdot}^{g_{\bomega_{0}}}_{\pi^{-1}\paren*{\cS}}}
\end{cases}$$ and that the cohomology class $[\bomega_{\epsilon}] \in H^{2}\paren*{K3}\otimes \bbR^{3}$ is independent of the cutoff function $\chi$ in its construction (\textit{Remark} \ref{cutoff function dependence}).

Moreover, this volume form decay directly implies that the associated intersection matrix satisfies (componentwise) $$\Abs*{Q_{\bomega_{\epsilon}} - \id}_{C^{0}\paren*{\Km_{\epsilon}}} \leq \cE_{1} \epsilon^{4-4\lambda}$$ for some uniform constant $\cE_{1} > 0$. More specifically, $$\abs*{Q_{\bomega_{\epsilon}} - \id} \begin{cases}
\leq \cE_{1} \epsilon^{4-4\lambda} & \text{on }\set*{\epsilon^{\lambda} \leq \abs*{\cdot}^{g_{\bomega_{0}}}_{\pi^{-1}\paren*{\cS}} \leq 2\epsilon^{\lambda}}\\
= 0 & \text{on }\set*{\abs*{\cdot}^{g_{\bomega_{0}}}_{\pi^{-1}\paren*{\cS}} \leq \epsilon^{\lambda}} \sqcup \set*{2\epsilon^{\lambda} \leq \abs*{\cdot}^{g_{\bomega_{0}}}_{\pi^{-1}\paren*{\cS}}}
\end{cases}$$ 
\end{prop}

We thus still have that each pair $\Km_{\epsilon_{1}}, \Km_{\epsilon_{2}}$ are diffeomorphic for $\epsilon_{1}, \epsilon_{2} > 0$, and \textit{Remark} \ref{GHremark} analogously gives us \begin{remark}\label{HKA GHremark}
From the construction of $\Km_{\epsilon}$ and $\bomega_{\epsilon}$ itself, we have that if $\epsilon \searrow 0$, that as Riemannian manifolds $\paren*{\Km_{\epsilon}, \bomega_{\epsilon}, g_{\bomega_{\epsilon}}} \xrightarrow{GH} \paren*{\bbT^{4}/\bbZ_{2}, \bomega_{0}, g_{\bomega_{0}}}$ converges in the Gromov-Hausdorff topology to the orbifold $\paren*{\bbT^{4}/\bbZ_{2}, \bomega_{0}, g_{\bomega_{0}}}$ with the singular orbifold \hka metric $\bomega_{0}$ with conical $\bbR^{4}/\bbZ_{2}$ singularities at each of the 16 singular points. Moreover, this convergence satisfies \textbf{volume non-collapse}, that is, $\Vol_{g_{\omega_{\epsilon}}}\paren*{B_{r}^{g_{\omega_{\epsilon}}}\paren*{p}} \geq Cr^{4}$ uniformly in $p$, $r$, and $\epsilon$. Note that outside of the set $\cS$ of 16 singular points, we have that $\bomega_{0}$ is the flat \hka metric on $\bbT^{4}/\bbZ_{2} - \cS$ from before.\end{remark}

Moreover, the analogue of \textit{Remark} \ref{ricciestimates} is \begin{remark}\label{HKA ricciestimates}

Recall from \textit{Remark} \ref{CYvsHKA} that for $\paren*{M^{4},\bomega, g_{\bomega}}$ which is \textit{hypersymplectic}, that the resulting $\paren*{M^{4}, J_{e}, g_{\bomega}, \omega_{e}}$ is K\"{a}hler \textit{for each direction $e \in S^{2}$} (since $Q_{\bomega} = \id$ only comes in to show that $g_{\bomega}$ is Ricci-flat), i.e. each $\paren*{J_{e},\omega_{e}}_{e \in S^{2}}$ is \ka WRT the \textit{fixed} metric $g_{\bomega}$. Hence WLOG upon fixing a direction $e \in S^{2}$, say $e_{1}$, we have that we may do the same estimates as in Proposition \ref{presmall} for the Ricci-potential of $\Ric_{g_{\bomega_{\epsilon}}}$ WRT $J_{e_{1}}$ and argue in the same exact same way as \textit{Remark} \ref{ricciestimates} to get that $$\abs*{\nabla_{g_{\bomega_{0}}}^{k}\Ric_{g_{\bomega_{\epsilon}}}}_{g_{\bomega_{0}}} \leq c_{4}(k) \epsilon^{4-\paren*{6+k}\lambda}$$ holds on the annulus $\set*{\epsilon^{\lambda} \leq \abs*{\cdot}^{g_{\bomega_{0}}}_{\pi^{-1}\paren*{\cS}} \leq 2\epsilon^{\lambda}}$ for some positive constants $c_{4}(k) >0,k \in \bbN_{0}$, that $\Ric_{g_{\bomega_{\epsilon}}}$ vanishes outside $\set*{\epsilon^{\lambda} \leq \abs*{\cdot}^{g_{\bomega_{0}}}_{\pi^{-1}\paren*{\cS}} \leq 2\epsilon^{\lambda}}$, and that $$\abs*{\Ric_{g_{\bomega_{\epsilon}}}}_{g_{\bomega_{\epsilon}}} \leq c_{4}(0) \epsilon^{4-6\lambda}$$ holds on the annulus $\set*{\epsilon^{\lambda} \leq \abs*{\cdot}^{g_{\bomega_{0}}}_{\pi^{-1}\paren*{\cS}} \leq 2\epsilon^{\lambda}}$, whence we easily get that $$-3\Lambda g_{\bomega_{\epsilon}}\leq \Ric_{g_{\bomega_{\epsilon}}} \leq 3 \Lambda g_{\bomega_{\epsilon}}$$ holds \textit{everywhere} on $\Km_{\epsilon}$ for $\Lambda \coloneq \frac{c_{4}(0)\epsilon^{4-6\lambda}}{3}$. Whence since $\lambda< \frac{2}{3}$, $\paren*{\Km_{\epsilon}, \bomega_{\epsilon}, g_{\bomega_{\epsilon}}}$ has \textbf{uniformly bounded Ricci curvature}.\end{remark}

\subsection{\textsection \ Weighted (H\"{o}lder) Spaces \ \textsection}\label{HKA Weighted Setup}

As remarked back in Section \ref{Weighted Setup}, since H\"{o}lder spaces only depend on the Riemannian metric, the same weighted H\"{o}lder spaces as defined back in Section \ref{Weighted Setup} still hold in the \hka category, whence we may speak of $C^{k,\alpha}_{\delta, g}$ for $\paren*{p,q}$-tensors on $\paren*{T^{*}S^{2}, \bomega_{EH,s}, g_{\bomega_{EH,s}}}$, $\paren*{\bbR^{4}/\bbZ_{2}, \bomega_{0}, g_{\bomega_{0}}}$, $\paren*{\bbT^{4}, \bomega_{0}, g_{\bomega_{0}}}$, $\paren*{\bbT^{4} - \cS, \bomega_{0}, g_{\bomega_{0}}}$, $\paren*{\bbT^{4}/\bbZ_{2} - \cS, \bomega_{0}, g_{\bomega_{0}}}$, and $\paren*{\Km_{\epsilon}, \bomega_{\epsilon}, g_{\bomega_{\epsilon}}}$. 

For reference, this is \begin{al}
\Abs*{\fs}_{C^{k,\alpha}_{\delta, g}}&\coloneq \sum_{j \leq k} \Abs*{\nabla_{g}^{j} \fs}_{C^{0}_{\delta,g}} + \sup_{\substack{x \neq y\\ d_{g}\paren*{x,y} < \sigma_{g}\min\set*{\rho\paren*{x},\rho\paren*{y}}}}\min\set*{\rho(x),\rho(y)}^{-\delta + k + \alpha} \frac{\abs*{\restr{\nabla_{g}^{k}\fs}{x} - \restr{\nabla_{g}^{k}\fs}{y}}_{g}}{d_{g}\paren*{x , y}^{\alpha}} \\
&\coloneq \sum_{j \leq k} \Abs*{\rho^{-\delta + j}\nabla_{g}^{j} \fs}_{C^{0}_{g}} + \sup_{\substack{x \neq y\\ d_{g}\paren*{x,y} < \sigma_{g}\min\set*{\rho\paren*{x},\rho\paren*{y}}}}\min\set*{\rho(x),\rho(y)}^{-\delta + k + \alpha} \frac{\abs*{\restr{\nabla_{g}^{k}\fs}{x} - \restr{\nabla_{g}^{k}\fs}{y}}_{g}}{d_{g}\paren*{x , y}^{\alpha}}
\end{al} where $\sigma_{g}> 0$ is a constant chosen such that the inequality $\inj_{g}\paren*{p} \geq \sigma_{g}\rho\paren*{p}>0$ holds uniformly $\forall p \in M$, where $\inj_{g}\paren*{p}$ is the injectivity radius of $g$ about $p$, whence we interpret the difference in the numerator of the (weighted) H\"{o}lder seminorm via parallel transport along the unique minimal geodesic connecting $x$ and $y$. Note that shrinking $\sigma_{g}$ induces an equivalent norm. The space of weighted H\"{o}lder $\paren*{p,q}$-tensors $C^{k,\alpha}_{\delta,g}$ is thus $C^{k,\alpha}_{\delta,g}\coloneq \set*{\fs \in C^{k,\alpha}_{g}  : \Abs*{s}_{C^{k,\alpha}_{\delta,g}} < \infty}$ but with the $\Abs*{\cdot}_{C^{k,\alpha}_{\delta,g}}$ norm instead.

As the weight function $\rho_{\epsilon}$ on $\Km_{\epsilon}$ needed for the $C^{k,\alpha}_{\delta,g}$-norm on $\paren*{\Km_{\epsilon}, g_{\bomega_{\epsilon}}}$ requires $2\epsilon^{\lambda} < \frac{3}{25}$ and $2\epsilon < \epsilon^{\lambda}$, we thus still have Data \ref{constant epsilon 1-16} in effect.

\begin{notation}
For notational clarity (as we're no longer only dealing with scalar-valued functions), we let $C^{k,\alpha}_{\delta, g}\paren*{M}$ denote the weighted H\"{o}lder space of $\bbR$-valued functions on $M$, and we let $C^{k,\alpha}_{\delta, g}\paren*{E}$ denote the weighted H\"{o}lder space of sections of a metric $\bbK$-vector bundle $E\onto M$.

Sometimes, especially when doing estimates, whenever the tensor type is obvious from the context we will resort back to the usual notation $C^{k,\alpha}_{\delta, g}\paren*{M}$ for norms of tensors.
\end{notation}

\begin{remark}\label{HKA weighted tuple notation}

In the sequel, we will be dealing with the following function spaces: \begin{al}
\cH^{+}_{g} \otimes \bbR^{3} &\cong \bbR^{9}\\
C^{1,\alpha}_{\delta,g}\paren*{\mathring{\Omega}^{1}_{g}\paren*{M^{4}}} &\\
C^{0,\alpha}_{\delta - 1,g}\paren*{\Lambda_{g}^{+}T^{*}M^{4}}\otimes\bbR^{3}&\\		
\paren*{C^{1,\alpha}_{\delta,g}\paren*{T^{*}M^{4}}\oplus \cH^{+}_{g}} \otimes \bbR^{3} &= \paren*{C^{1,\alpha}_{\delta,g}\paren*{T^{*}M^{4}}\otimes \bbR^{3}}\oplus \paren*{\cH^{+}_{g}\otimes \bbR^{3}}\\
C^{0,\alpha}_{\delta - 1,g}\paren*{\bbR\oplus \Lambda_{g}^{+}T^{*}M^{4}}\otimes\bbR^{3} &= \paren*{C^{0,\alpha}_{\delta - 1,g}\paren*{M^{4}} \otimes\bbR^{3}} \oplus \paren*{C^{0,\alpha}_{\delta - 1,g}\paren*{\Lambda_{g}^{+}T^{*}M^{4}} \otimes \bbR^{3}}\\
\paren*{C^{1,\alpha}_{\delta,g}\paren*{\mathring{T^{*}M^{4}}}\oplus \cH^{+}_{g}} \otimes \bbR^{3} &= \paren*{C^{1,\alpha}_{\delta,g}\paren*{\mathring{T^{*}M^{4}}}\otimes \bbR^{3}}\oplus \paren*{\cH^{+}_{g}\otimes \bbR^{3}}
\end{al} where we recall that $\cH^{+}_{g}$ denotes the space of self-dual harmonic forms WRT $g$ and\footnote{Note the abuse of notation here, as $\Omega^{1}$, $\mathring{\Omega}^{1}_{g}$ usually denotes \textit{smooth} regularity.} $$C^{1,\alpha}_{\delta,g}\paren*{\mathring{\Omega}^{1}_{g}\paren*{M^{4}}} \coloneq \set*{\fa \in C^{1,\alpha}_{\delta,g}\paren*{T^{*}M^{4}} : d_{g}^{*}\fa = 0} = C^{1,\alpha}_{\delta,g}\paren*{T^{*}M^{4}} \cap \ker d_{g}^{*}$$ is in fact a \textit{closed} Banach subspace of $C^{1,\alpha}_{\delta,g}\paren*{T^{*}M^{4}}$ WRT the $\Abs*{\cdot}_{C^{1,\alpha}_{\delta,g}\paren*{T^{*}M^{4}}}$-norm since the (formal) adjoint of the exterior derivative $d_{g}^{*} : C^{1,\alpha}_{\delta,g}\paren*{T^{*}M^{4}} \rightarrow C^{0,\alpha}_{\delta-1,g}\paren*{M^{4}}$ is continuous. That is, these function spaces consists of either a single or a pair of \textit{$3$-tuples of elements}. In addition to the $\Abs*{\cdot}_{C^{k,\alpha}_{\delta,g}}$-norms which we've already defined, we therefore use the following norms when working with the above: \begin{al}
\Abs*{b}_{\cH^{+}_{g}} &= \Abs*{b}_{L^{2}_{g}} \coloneq \paren*{\int_{M^{4}} b \wedge *_{g} b}^{\frac{1}{2}}\\
\Abs*{\bm{a}}_{\cA \otimes \bbR^{3}} &= \paren*{\Abs*{a_{1}}^{2}_{\cA}+ \Abs*{a_{2}}^{2}_{\cA}+ \Abs*{a_{3}}^{2}_{\cA}}^{\frac{1}{2}}\\
\Abs*{\paren*{a, b}}_{\cA \oplus \cB} &= \Abs*{a}_{\cA} + \Abs*{b}_{\cB}		
\end{al} That is, we put the $L^{2}_{g}$-norm on $\cH^{+}_{g}$, the product $L^{2}$-norm on $3$-tuples, and the product $L^{1}$-norm on direct sums.\end{remark}

\subsection{\textsection \ Nonlinear Setup \ \textsection}\label{HKA Nonlinear Setup}

Let us now bring everything together and find our \hka triple $\wtilde{\bomega_{\epsilon}}$ by perturbing our closed definite triple $\bomega_{\epsilon}$ using the setup in Section \ref{HKA prelims}.

\begin{notation}
As we will be working with a \textit{fixed} $\epsilon > 0$, denote the sought after \hka triple $\wtilde{\bomega_{\epsilon}} = \bomega_{\epsilon} + \bm{\eta}$ as $\bomega$, \textit{keeping in mind that as this will end up being the unique \hka triple that's perturbed from $\bomega$, that \textit{it depends on the gluing parameter $\epsilon>0$} as the starting closed definite triple $\bomega_{\epsilon}$ depends on $\epsilon>0$}.
\end{notation}

From Section \ref{HKA prelims}, since $\Km_{\epsilon}$ is a closed manifold with $b_{1}\paren*{\Km_{\epsilon}} = 0$ we want to solve the following nonlinear \textit{elliptic} system with singular limit (see \textit{Remark} \ref{HKA GHremark}): \begin{al}
\Phi_{\epsilon}:\paren*{C^{1,\alpha}_{\delta,g_{\bomega_{\epsilon}}}\paren*{\mathring{\Omega}^{1}_{g_{\bomega_{\epsilon}}}\paren*{\Km_{\epsilon}}}\oplus \cH^{+}_{g_{\bomega_{\epsilon}}}} \otimes \bbR^{3} &\rightarrow C^{0,\alpha}_{\delta - 1,g_{\bomega_{\epsilon}}}\paren*{\Lambda_{g_{\bomega_{\epsilon}}}^{+}T^{*}\Km_{\epsilon}}\otimes\bbR^{3}\\
\paren*{\bm{a},\bm{\zeta}} &\mapsto d_{g_{\bomega_{\epsilon}}}^{+}\bm{a}+ \bm{\zeta} - \cF\paren*{\paren*{- Q_{\bomega_{\epsilon}} - d^{-}_{g_{\bomega_{\epsilon}}}\bm{a} \star d^{-}_{g_{\bomega_{\epsilon}}}\bm{a}}_{0}}\bomega_{\epsilon}
\end{al} where we solve \begin{al}
\Phi_{\epsilon}\paren*{\bm{a},\bm{\zeta}} &\coloneq	d_{g_{\bomega_{\epsilon}}}^{+}\bm{a}+ \bm{\zeta} - \cF\paren*{\paren*{- Q_{\bomega_{\epsilon}} - d^{-}_{g_{\bomega_{\epsilon}}}\bm{a} \star d^{-}_{g_{\bomega_{\epsilon}}}\bm{a}}_{0}}\bomega_{\epsilon}\\
&= \underbrace{-\cF\paren*{ \paren*{- Q_{\bomega_{\epsilon}}}_{0} }\bomega_{\epsilon}}_{=\Phi_{\epsilon}(0,0)} + \underbrace{d_{g_{\bomega_{\epsilon}}}^{+}\bm{a}+ \bm{\zeta}}_{=D_{0}\Phi_{\epsilon}\paren*{\bm{a},\bm{\zeta}}} + \underbrace{\cF\paren*{\paren*{- Q_{\bomega_{\epsilon}}}_{0}}\bomega_{\epsilon} - \cF\paren*{\paren*{- Q_{\bomega_{\epsilon}} - d^{-}_{g_{\bomega_{\epsilon}}}\bm{a} \star d^{-}_{g_{\bomega_{\epsilon}}}\bm{a}}_{0}}\bomega_{\epsilon}}_{\text{nonlinearity}}\\
&= 0
\end{al} Recall that $\cH^{+}_{g_{\bomega_{\epsilon}}}$ denotes the space of self-dual harmonic $2$-forms WRT $g_{\bomega_{\epsilon}}$, whence consists of constant linear combinations of the $3$ components of $\bomega_{\epsilon}$.

Similar to the Calabi-Yau case, our strategy to prove the existence of $\paren*{\bm{a},\bm{\zeta}}$ solving $\Phi_{\epsilon}\paren*{\bm{a},\bm{\zeta}} = 0$ above is via the implicit function theorem (Theorem \ref{IFT}).

\subsection{\textsection \ Blow-up analysis \ \textsection}\label{HKA Blow-up Analysis}

First, we recall that the linearization at the origin of $\Phi_{\epsilon}$, $D_{0}\Phi_{\epsilon} = \paren*{d_{g_{\bomega_{\epsilon}}}^{+} \oplus \id_{\cH^{+}_{g_{\bomega_{\epsilon}}}}}\otimes \bbR^{3}$ always an isomorphism since $\Km_{\epsilon}$ is closed and satisfies $b_{1}\paren*{\Km_{\epsilon}} = 0$. However, while $D_{0}\Phi_{\epsilon}$ is always an isomorphism, the key is to have the operator norm of its inverse be \textit{uniformly bounded in $\epsilon > 0$}. Our strategy is to first prove uniform estimates for our Dirac operator $$D_{g_{\bomega_{\epsilon}}} \coloneq d^{*}_{g_{\bomega_{\epsilon}}} \oplus d^{+}_{g_{\bomega_{\epsilon}}}:C^{1,\alpha}_{\delta, g_{\bomega_{\epsilon}}}\paren*{T^{*}\Km_{\epsilon}}\rightarrow C^{0,\alpha}_{\delta-1, g_{\bomega_{\epsilon}}}\paren*{\bbR \oplus \Lambda_{g_{\bomega_{\epsilon}}}^{+}T^{*}\Km_{\epsilon}}$$ for certain weights $\delta$, then show that these imply the sought after uniform estimates for our linearization $$D_{0}\Phi_{\epsilon}= \paren*{d_{g_{\bomega_{\epsilon}}}^{+} \oplus \id_{\cH^{+}_{g_{\bomega_{\epsilon}}}}}\otimes \bbR^{3}:\paren*{C^{1,\alpha}_{\delta,g_{\bomega_{\epsilon}}}\paren*{\mathring{\Omega}^{1}_{g_{\bomega_{\epsilon}}}\paren*{\Km_{\epsilon}}}\oplus \cH^{+}_{g_{\bomega_{\epsilon}}}} \otimes \bbR^{3} \rightarrow C^{0,\alpha}_{\delta - 1,g_{\bomega_{\epsilon}}}\paren*{\Lambda_{g_{\bomega_{\epsilon}}}^{+}T^{*}\Km_{\epsilon}}\otimes\bbR^{3}$$

The analogue of Proposition \ref{scalingproperties} is the following, which follows from $d^{*}_{C^{2}}g = \frac{1}{C^{2}}d^{*}_{g}$, $*_{C^{2}g} = *_{g}$ on $2$-forms, and the conformal invariance $\Lambda_{C^{2}g}^{+} T^{*}M = \Lambda_{g}^{+}T^{*}M$ of self-dual $2$-forms: \begin{prop}\label{HKA scalingproperties} Let $\paren*{M^{n},g}$ be an arbitrary Riemannian manifold. Let $C>0$ be a constant. Let $\rho$ be a smooth nonnegative function, let $\delta \in \bbR$, let $\fw$ be a $\paren*{0,m}$-tensor, let $\abs*{\cdot}_{g}$ denote all higher order tensor norms WRT $g$, and recall $\Abs*{\nabla^{j}_{g} \fw}_{C^{0}_{\delta,g}} \coloneq \Abs*{\rho^{-\delta+ j}\nabla_{g}^{j}\fw}_{C^{0}_{g}}$. Recall from Proposition \ref{scalingproperties}(\ref{scalingproperties covariant derivative rescale}) that $\nabla_{C^{2}g} = \nabla_{g}$. \begin{enumerate}
\itemsep0em 
\item\label{HKA scalingproperties pointwise tensor norm rescale} Under a conformal rescaling $g\mapsto C^{2}g$, we have that $$\abs*{\nabla^{j}\fw}_{C^{2}g} = \paren*{\frac{1}{C}}^{j+m}\abs*{\nabla^{j}\fw}_{g}$$

\item\label{HKA scalingproperties Hodge Laplacian rescale} Under a conformal rescaling $g\mapsto C^{2}g$, we have for the Hodge Laplacian $\Delta_{g}\coloneq d d_{g}^{*} + d_{g}^{*}d$ that $$\Delta_{C^{2}g} = \frac{1}{C^{2}} \Delta_{g}$$

\item\label{HKA scalingproperties Dirac rescale} When $n = 4$ and $\paren*{M^{4},g}$ is oriented, under a conformal rescaling $g\mapsto C^{2}g$ we have that $$D_{C^{2}g} = \paren*{\frac{1}{C^{2}} d_{g}^{*}, d_{g}^{+}} = \paren*{\frac{1}{C^{2}} d_{g}^{*}} \oplus d_{g}^{+}$$

\item\label{HKA scalingproperties weighted H\"{o}lder tensor seminorm rescale} Immediately from Proposition \ref{scalingproperties}(\ref{scalingproperties covariant derivative rescale}) and (\ref{HKA scalingproperties pointwise tensor norm rescale}), under a conformal rescaling $g\mapsto C^{2}g$ we have that $$ \sqparen*{\fw}_{C^{k,\alpha}_{\delta,C^{2}g}} = \paren*{\frac{1}{C}}^{\delta + m}\sqparen*{\fw}_{C^{k,\alpha}_{\delta,g}}$$ with the weight of the RHS $C^{k,\alpha}_{\delta,g}$ being $\frac{\rho}{C}$, i.e.
\begin{al}
\sup_{\substack{x \neq y\\ d_{C^{2}g}\paren*{x,y} < \sigma_{g}\min\set*{\rho\paren*{x},\rho\paren*{y}}}}\min\set*{\rho(x),\rho(y)}^{-\delta + k + \alpha} \frac{\abs*{\restr{\nabla_{g}^{k}\fw}{x} - \restr{\nabla_{g}^{k}\fw}{y}}_{C^{2}g}}{d_{C^{2}g}\paren*{x,y}^{\alpha}} \\ = \paren*{\frac{1}{C}}^{\delta + m}\sup_{\substack{x \neq y\\ d_{g}\paren*{x,y} < \sigma_{g}\min\set*{\frac{\rho\paren*{x}}{C},\frac{\rho\paren*{y}}{C}}}}\min\set*{\restr{\frac{\rho}{C}}{x},\restr{\frac{\rho}{C}}{y}}^{-\delta + k + \alpha} \frac{\abs*{\restr{\nabla_{g}^{k}\fw}{x} - \restr{\nabla_{g}^{k}\fw}{y}}_{g}}{d_{g}\paren*{x,y}^{\alpha}} 
\end{al} 

\item\label{HKA scalingproperties weighted H\"{o}lder tensor norm rescale} Under a conformal rescaling $g\mapsto C^{2}g$, we have from (\ref{HKA scalingproperties weighted H\"{o}lder tensor seminorm rescale}) that $$\Abs*{\fw}_{C^{k,\alpha}_{\delta,C^{2}g}} = \paren*{\frac{1}{C}}^{\delta + m} \Abs*{\fw}_{C^{k,\alpha}_{\delta,g}}\overset{g\leftrightarrow \frac{1}{C^{2}}g}{\Longleftrightarrow} \Abs*{\fw}_{C^{k,\alpha}_{\delta,g}} = \paren*{\frac{1}{C}}^{\delta + m} \Abs*{\fw}_{C^{k,\alpha}_{\delta,\frac{1}{C^{2}}g}}$$

where again the weight of the RHS is $\frac{\rho}{C}$.

\item\label{HKA scalingproperties standard H\"{o}lder tensor norm rescale} Under a conformal rescaling $g\mapsto C^{2}g$, we have from Proposition \ref{scalingproperties}(\ref{scalingproperties covariant derivative rescale}) and (\ref{HKA scalingproperties pointwise tensor norm rescale}) that for the \textit{unweighted/standard} H\"{o}lder norm that \begin{al}
\Abs*{\fw}_{C^{k,\alpha}_{g}}  \xrightarrow{\text{conformal rescaling }g\mapsto C^{2}g} \Abs*{\fw}_{C^{k,\alpha}_{C^{2}g}} &= \sum_{j \leq k} \paren*{\frac{1}{C}}^{j+m}\Abs*{\nabla_{g}^{j} \fw}_{C^{0}_{g}} \\ &+ \paren*{\frac{1}{C}}^{k+m + \alpha}\sup_{\substack{x \neq y\\d_{g}\paren*{x,y} < \sigma_{g}\min\set*{\inj_{g}\paren*{x},\inj_{g}\paren*{y}}}} \frac{\abs*{\restr{\nabla_{g}^{k}\fw}{x} - \restr{\nabla_{g}^{k}\fw}{y}}_{g}}{d_{g}\paren*{x,y}^{\alpha}}
\end{al}

\item\label{HKA scalingproperties Schauder Hodge Laplacian rescale} Transferring the standard Schauder estimate $\Abs*{\fa}_{C^{2,\alpha}_{g}} \leq \cC\paren*{\Abs*{\fa}_{C^{0}_{g}} + \Abs*{\Delta_{g}\fa}_{C^{0,\alpha}_{g}}}$ for $1$-forms $\fa$ on metric balls $\paren*{B^{g}_{r_{1}}\paren*{p}, B^{g}_{r_{2}}\paren*{p}}$ of radii $\paren*{r_{1},r_{2}}$ (clearly with $r_{1} < r_{2}$) under the conformal rescaling $g\mapsto C^{2}g$, we end up getting via (\ref{HKA scalingproperties Hodge Laplacian rescale}) and (\ref{HKA scalingproperties standard H\"{o}lder tensor norm rescale}):
\begin{al}
\Abs*{\fa}_{C^{2,\alpha}_{C^{2}g}} \leq \cC\paren*{\Abs*{\fa}_{C^{0}_{C^{2}g}} + \Abs*{\Delta_{C^{2}g}\fa}_{C^{0,\alpha}_{C^{2}g}}}
\end{al}
which is equal to
\begin{al}
\frac{1}{C}\Abs*{\fa}_{C^{0}_{g}} &+ \paren*{\frac{1}{C}}^{2} \Abs*{\nabla_{g}\fa}_{C^{0}_{g}} + \paren*{\frac{1}{C}}^{3}\Abs*{\nabla^{2}_{g} \fa}_{C^{0}_{g}} + \paren*{\frac{1}{C}}^{3 + \alpha}\sqparen*{\nabla^{2}_{g} \fa}_{C^{0,\alpha}_{g}}\\ &\leq \cC\paren*{\frac{1}{C}\Abs*{\fa}_{C^{0}_{g}} + \paren*{\frac{1}{C}}^{3} \Abs*{\Delta_{g}\fa}_{C^{0}_{g}} + \paren*{\frac{1}{C}}^{3 + \alpha} \sqparen*{\Delta_{g} \fa}_{C^{0,\alpha}_{g}}  }
\end{al} on the \textit{same} metric balls $\paren*{B^{g}_{r_{1}}\paren*{p}, B^{g}_{r_{2}}\paren*{p}} = \paren*{B^{C^{2}g}_{Cr_{1}}\paren*{p}, B^{C^{2}g}_{Cr_{2}}\paren*{p}}$, but with different radii WRT the \textit{new} metric $C^{2}g$.

Conversely, if we already had the standard Schauder estimate for $C^{2}g$, namely $\Abs*{\fa}_{C^{2,\alpha}_{C^{2}g}} \leq \cC\paren*{\Abs*{\fa}_{C^{0}_{C^{2}g}} + \Abs*{\Delta_{C^{2}g}\fa}_{C^{0,\alpha}_{C^{2}g}}}$ on $\paren*{B^{C^{2}g}_{Cr_{1}}\paren*{p}, B^{C^{2}g}_{Cr_{2}}\paren*{p}}$, then unraveling it would yield the same estimate above on $\paren*{B^{g}_{r_{1}}\paren*{p}, B^{g}_{r_{2}}\paren*{p}} = \paren*{B^{C^{2}g}_{Cr_{1}}\paren*{p}, B^{C^{2}g}_{Cr_{2}}\paren*{p}}$.

\item\label{HKA scalingproperties Dirac norm rescale} When $n = 4$ and $\paren*{M^{4},g}$ is oriented, under a conformal rescaling $g\mapsto C^{2}g$, upon recalling the $L^{2}$-norm on direct sums convention from \textit{Remark} \ref{HKA weighted tuple notation} (since $D_{g}$ lands in $\Omega^{0} \oplus \Omega_{g}^{+} \subset \Omega^{0} \oplus \Omega^{2}$), on $1$-forms $\fa$ we have from (\ref{HKA scalingproperties Dirac rescale}) and upon setting $m = 0$ and $m = 2$ in (\ref{HKA scalingproperties standard H\"{o}lder tensor norm rescale}) that $$\Abs*{D_{C^{2}g} \fa }_{C^{0,\alpha}_{C^{2}g}} = \paren*{\frac{1}{C}}^{2}\Abs*{d^{*}_{g}\fa}_{C^{0}_{g}} + \paren*{\frac{1}{C}}^{2+\alpha}\sqparen*{d^{*}_{g}\fa}_{C^{0,\alpha}_{g}} + \paren*{\frac{1}{C}}^{2}\Abs*{d^{+}_{g}\fa}_{C^{0}_{g}} + \paren*{\frac{1}{C}}^{2+\alpha}\sqparen*{d^{+}_{g}\fa}_{C^{0,\alpha}_{g}} $$ \footnote{Indeed, since $\Abs*{D_{C^{2}g} \fa }_{C^{0,\alpha}_{C^{2}g}} = \frac{1}{C^{2}} \Abs*{d^{*}_{g}\fa}_{C^{0,\alpha}_{C^{2}g}} + \Abs*{d^{+}_{g}\fa}_{C^{0,\alpha}_{C^{2}g}}$, $\Abs*{d^{*}_{g}\fa}_{C^{0,\alpha}_{C^{2}g}} = \Abs*{d^{*}_{g}\fa}_{C^{0}_{g}} + \paren*{\frac{1}{C}}^{\alpha}\sqparen*{d^{*}_{g}\fa}_{C^{0,\alpha}_{g}}$, and $\Abs*{d^{+}_{g}\fa}_{C^{0,\alpha}_{C^{2}g}} = \paren*{\frac{1}{C}}^{2}\Abs*{d^{+}_{g}\fa}_{C^{0}_{g}} + \paren*{\frac{1}{C}}^{2+\alpha}\sqparen*{d^{+}_{g}\fa}_{C^{0,\alpha}_{g}}$.}Therefore, since we have from (\ref{HKA scalingproperties standard H\"{o}lder tensor norm rescale}) that $\Abs*{\nabla_{g}\fa}_{C^{0,\alpha}_{C^{2}g}} = \paren*{\frac{1}{C}}^{2}\Abs*{\nabla_{g} \fa}_{C^{0}_{g}} + \paren*{\frac{1}{C}}^{2 + \alpha}\sqparen*{\nabla_{g}\fa}_{C^{0,\alpha}_{g}}$ for $\fa$ a $1$-form, the heuristic here is that \textit{the $C^{0,\alpha}_{g}$-norm of $D_{g}\fa$ and $\nabla_{g}\fa$ have the same scaling behavior under the conformal rescaling $g\mapsto C^{2}g$}, even though $D_{g}$ takes values in sections of the bundle $\bbR\oplus \Lambda_{g}^{+}T^{*}M^{4} \onto M^{4}$.

\item\label{HKA scalingproperties Schauder Dirac rescale} When $n = 4$ and $\paren*{M^{4},g}$ is oriented, transferring the standard Schauder estimate $\Abs*{\fa}_{C^{1,\alpha}_{g}} \leq \cC\paren*{\Abs*{\fa}_{C^{0}_{g}} + \Abs*{D_{g}\fa}_{C^{0,\alpha}_{g}}}$ for $1$-forms $\fa$ on metric balls $\paren*{B^{g}_{r_{1}}\paren*{p}, B^{g}_{r_{2}}\paren*{p}}$ of radii $\paren*{r_{1},r_{2}}$ (clearly with $r_{1} < r_{2}$) under the conformal rescaling $g\mapsto C^{2}g$, we end up getting via (\ref{HKA scalingproperties Dirac rescale}), (\ref{HKA scalingproperties standard H\"{o}lder tensor norm rescale}) and (\ref{HKA scalingproperties Dirac norm rescale}):
\begin{al}
\Abs*{\fa}_{C^{1,\alpha}_{C^{2}g}} \leq \cC\paren*{\Abs*{\fa}_{C^{0}_{C^{2}g}} + \Abs*{D_{C^{2}g}\fa}_{C^{0,\alpha}_{C^{2}g}}}
\end{al}
which is equal to
\begin{al}
\frac{1}{C}\Abs*{\fa}_{C^{0}_{g}} &+ \paren*{\frac{1}{C}}^{2} \Abs*{\nabla_{g}\fa}_{C^{0}_{g}} + \paren*{\frac{1}{C}}^{2 + \alpha}\sqparen*{\nabla_{g} \fa}_{C^{0,\alpha}_{g}}\\ &\leq \cC\paren*{\frac{1}{C}\Abs*{\fa}_{C^{0}_{g}} + \paren*{\frac{1}{C}}^{2}\Abs*{d^{*}_{g}\fa}_{C^{0}_{g}} + \paren*{\frac{1}{C}}^{2+\alpha}\sqparen*{d^{*}_{g}\fa}_{C^{0,\alpha}_{g}} + \paren*{\frac{1}{C}}^{2}\Abs*{d^{+}_{g}\fa}_{C^{0}_{g}} + \paren*{\frac{1}{C}}^{2+\alpha}\sqparen*{d^{+}_{g}\fa}_{C^{0,\alpha}_{g}}   }
\end{al} on the \textit{same} metric balls $\paren*{B^{g}_{r_{1}}\paren*{p}, B^{g}_{r_{2}}\paren*{p}} = \paren*{B^{C^{2}g}_{Cr_{1}}\paren*{p}, B^{C^{2}g}_{Cr_{2}}\paren*{p}}$, but with different radii WRT the \textit{new} metric $C^{2}g$.

Conversely, if we already had the standard Schauder estimate for $C^{2}g$, namely $\Abs*{\fa}_{C^{1,\alpha}_{C^{2}g}} \leq \cC\paren*{\Abs*{\fa}_{C^{0}_{C^{2}g}} + \Abs*{D_{C^{2}g}\fa}_{C^{0,\alpha}_{C^{2}g}}}$ on $\paren*{B^{C^{2}g}_{Cr_{1}}\paren*{p}, B^{C^{2}g}_{Cr_{2}}\paren*{p}}$, then unraveling it would yield the same estimate above on $\paren*{B^{g}_{r_{1}}\paren*{p}, B^{g}_{r_{2}}\paren*{p}} = \paren*{B^{C^{2}g}_{Cr_{1}}\paren*{p}, B^{C^{2}g}_{Cr_{2}}\paren*{p}}$.

\end{enumerate}

\end{prop}

Propositions \ref{preciseweightfunction} and \ref{preciseweightfunction EH almost constancy} still hold as these are just estimates on the weight functions $\rho_{\epsilon}$ and $\wtilde{\rho_{0}}$ on $\Km_{\epsilon}$ and $T^{*}S^{2}$, respectively.

We therefore arrive at the analogue of the first key proposition, namely the local weighted Schauder estimate (Proposition \ref{localWSE}), but for $D_{g_{\bomega_{\epsilon}}}$: \begin{prop}[Local Weighted Schauder Estimate]\label{HKA localWSE} Let $\epsilon > 0$ be from Data \ref{constant epsilon 1-16}. Then $\forall \fa \in C^{1,\alpha}_{\delta, g_{\bomega_{\epsilon}}}\paren*{T^{*}\Km_{\epsilon}}$ and $\forall p \in \Km_{\epsilon}$, there exists an $r = r_{p} > 0$ (hence may depend on $p$) such that: 
\begin{al}
\Abs*{\fa}_{C^{1,\alpha}_{\delta, g_{\bomega_{\epsilon}}} \paren*{T^{*}B^{g_{\bomega_{\epsilon}}}_{r_{p}}\paren*{p}}} &\leq \cC_{11}\paren*{\Abs*{\fa}_{C^{0}_{\delta, g_{\bomega_{\epsilon}}} \paren*{T^{*}B^{g_{\bomega_{\epsilon}}}_{2r_{p}}\paren*{p}}} + \Abs*{D_{g_{\bomega_{\epsilon}}}\fa}_{C^{0,\alpha}_{\delta - 1, g_{\bomega_{\epsilon}}} \paren*{\bbR\oplus \Lambda_{g_{\bomega_{\epsilon}}}^{+}T^{*}B^{g_{\bomega_{\epsilon}}}_{2r_{p}}\paren*{p}}}}\\
&\coloneq \cC_{11}\paren*{\Abs*{\rho_{\epsilon}^{-\delta}\fa}_{C^{0}_{g_{\bomega_{\epsilon}}} \paren*{T^{*}B^{g_{\bomega_{\epsilon}}}_{2r_{p}}\paren*{p}}} + \Abs*{D_{g_{\bomega_{\epsilon}}}\fa}_{C^{0,\alpha}_{\delta - 1, g_{\bomega_{\epsilon}}} \paren*{\bbR\oplus \Lambda_{g_{\bomega_{\epsilon}}}^{+}T^{*}B^{g_{\bomega_{\epsilon}}}_{2r_{p}}\paren*{p}}}}
\end{al}
for $\cC_{11} > 0$ \textit{independent of $\epsilon$ and $p\in \Km_{\epsilon}$.}
\end{prop}

\begin{proof}

The proof strategy is the exact same as the proof of Proposition \ref{localWSE}, but with some key differences due to working with $1$-forms and with $D_{g}$ instead of the scalar Laplacian $\Delta_{g}$.

In more detail, let $p \in \Km_{\epsilon}$. We consider 3 cases: $p \in \set*{\abs*{\cdot}^{g_{\bomega_{0}}}_{\pi^{-1}\paren*{\cS}} < 2\epsilon}$, $p \in \set*{\epsilon < \abs*{\cdot}^{g_{\bomega_{0}}}_{\pi^{-1}\paren*{\cS}} < \frac{1}{8}}$, and $p \in \set*{\frac{3}{25} < \abs*{\cdot}^{g_{\bomega_{0}}}_{\pi^{-1}\paren*{\cS}}}$. Clearly these 3 regions cover all of $\Km_{\epsilon}$.

\begin{enumerate}[wide, labelwidth=!, labelindent=0pt]
\item \textbf{Case 1; $p \in \set*{\abs*{\cdot}^{g_{\bomega_{0}}}_{\pi^{-1}\paren*{\cS}} < 2\epsilon}$:} WLOG let us focus on only 1 of the 16 connected components of $\set*{\abs*{\cdot}^{g_{\bomega_{0}}}_{\pi^{-1}\paren*{\cS}} < 2\epsilon}$, hence suppose that $\pi^{-1}\paren*{\cS} = E$ consists of only a single exceptional $E \cong S^{2}$. By construction (Proposition \ref{HKA approximatemetricproperties} as well as Data \ref{constant epsilon 1-16} telling us that $2\epsilon < \epsilon^{\lambda}$), on this region $g_{\bomega_{\epsilon}} = g_{\bomega_{EH,\epsilon^{2}}}$. Recalling that $R_{\epsilon}^{*}\paren*{\frac{1}{\epsilon^{2}}g_{\bomega_{EH,\epsilon^{2}}}} = g_{\bomega_{EH,1}}$ from Proposition \ref{HKA EHproperties} as well as the identification $\set*{\abs*{\cdot}^{g_{\bomega_{0}}}_{\pi^{-1}\paren*{\cS}} \leq 3\epsilon^{\lambda}} \sim \set*{\abs*{\cdot}^{g_{\bomega_{0}}}_{S^{2}} \leq 3\epsilon^{\lambda}}$ from the underlying construction of $\Km_{\epsilon}$, we have that upon doing a conformal rescaling $g_{\bomega_{\epsilon}} \mapsto \frac{1}{\epsilon^{2}}g_{\bomega_{\epsilon}}$ that $\paren*{\set*{\abs*{\cdot}^{g_{\bomega_{0}}}_{\pi^{-1}\paren*{\cS}} < 2\epsilon}, \frac{1}{\epsilon^{2}} g_{\bomega_{\epsilon}}}$ is isometric under $R_{\epsilon}$ to $\paren*{\set*{\abs*{\cdot}^{g_{\bomega_{0}}}_{S^{2}} < 2}, g_{\bomega_{EH,1}}} \subset \paren*{T^{*}S^{2}, g_{\bomega_{EH,1}}}$, which is a precompact region of the Eguchi-Hanson space $\paren*{T^{*}S^{2}, g_{\bomega_{EH,1}}}$. Since $\paren*{T^{*}S^{2}, g_{\bomega_{EH,1}}}$ has bounded geometry, we thus have, for $r_{2} > r_{1} > 0$ small enough so that $B^{g_{\bomega_{EH,1}}}_{r_{2}}\paren*{p} \subset \paren*{\set*{\abs*{\cdot}^{g_{\bomega_{0}}}_{S^{2}} < 2}, g_{\bomega_{EH,1}}}$ (hence both $r_{1}, r_{2}$ may depend on $p$) where we abuse notation and denote the point where $p$ lands in $\set*{\abs*{\cdot}^{g_{\bomega_{0}}}_{S^{2}} < 2}$ as $p$, the standard Schauder estimate $\Abs*{k\fa}_{C^{1,\alpha}_{g_{\bomega_{EH,1}}}} \leq \cC_{11.1}\paren*{\Abs*{k\fa}_{C^{0}_{g_{\bomega_{EH,1}}}} + \Abs*{kD_{g_{\bomega_{EH,1}}}\fa}_{C^{0,\alpha}_{g_{\bomega_{EH,1}} } }}$ on the metric balls $\paren*{B^{g_{\bomega_{EH,1}}}_{r_{1}}\paren*{p}, B^{g_{\bomega_{EH,1}}}_{r_{2}}\paren*{p}}$ (where we have multiplied our $1$-form by a positive scalar $k > 0$, to be determined), \textit{where $\cC_{11.1}> 0$ is independent of $p$}.

Transferring everything back to $\paren*{\set*{\abs*{\cdot}^{g_{\bomega_{0}}}_{\pi^{-1}\paren*{\cS}} < 2\epsilon}, \frac{1}{\epsilon^{2}} g_{\bomega_{\epsilon}}}$ via applying $R_{\epsilon}$ to $\paren*{\set*{\abs*{\cdot}^{g_{\bomega_{0}}}_{S^{2}} < 2}, g_{\bomega_{EH,1}}}$ and using $\frac{1}{\epsilon^{2}}g_{\bomega_{EH,\epsilon^{2}}} = R_{\frac{1}{\epsilon}}^{*}\paren*{g_{\bomega_{EH,1}}}$ and the identification $\set*{\abs*{\cdot}^{g_{\bomega_{0}}}_{\pi^{-1}\paren*{\cS}} \leq 3\epsilon^{\lambda}} \sim \set*{\abs*{\cdot}^{g_{\bomega_{0}}}_{S^{2}} \leq 3\epsilon^{\lambda}}$, we have that our standard Schauder estimate is now $\Abs*{k\fa}_{C^{1,\alpha}_{\frac{1}{\epsilon^{2}}g_{\bomega_{\epsilon}}}} \leq \cC_{11.1}\paren*{\Abs*{k\fa}_{C^{0}_{\frac{1}{\epsilon^{2}}g_{\bomega_{\epsilon}}}} + \Abs*{kD_{\frac{1}{\epsilon^{2}}g_{\bomega_{\epsilon}}}\fa}_{C^{0,\alpha}_{\frac{1}{\epsilon^{2}}g_{\bomega_{\epsilon}} } }}$ on the metric balls $\paren*{B^{\frac{1}{\epsilon^{2}}g_{\bomega_{\epsilon}}}_{R_{1}}\paren*{p}, B^{\frac{1}{\epsilon^{2}}g_{\bomega_{\epsilon}}}_{R_{2}}\paren*{p}}$ for some $0 < R_{1} < R_{2}$ (which may depend on $p$) such that $B^{\frac{1}{\epsilon^{2}}g_{\bomega_{\epsilon}}}_{R_{2}}\paren*{p} \subset \set*{\abs*{\cdot}^{g_{\bomega_{0}}}_{\pi^{-1}\paren*{\cS}} < 2\epsilon}$ and $B^{\frac{1}{\epsilon^{2}}g_{\bomega_{\epsilon}}}_{R_{2}}\paren*{p} \subset R_{\epsilon} \paren*{B^{g_{\bomega_{EH,1}}}_{r_{2}}\paren*{p}} $ (and vice versa for $R_{1}$). Now apply Proposition \ref{HKA scalingproperties}(\ref{HKA scalingproperties Schauder Dirac rescale}) with $C = \frac{1}{\epsilon}$ and $g = g_{\bomega_{\epsilon}}$, set $k = \epsilon^{-\delta-1}$, and upon noting that on $\set*{\abs*{\cdot}^{g_{\bomega_{0}}}_{\pi^{-1}\paren*{\cS}} < 2\epsilon}$ we have the \textit{uniform} estimate $\epsilon \leq \rho_{\epsilon} \leq 2\epsilon$ (Proposition \ref{preciseweightfunction}(\ref{preciseweightfunction first region})), we get upon recalling the definition of the weighted H\"{o}lder norms our sought after \textit{weighted} Schauder estimate $\Abs*{\fa}_{C^{1,\alpha}_{\delta, g_{\bomega_{\epsilon}}}} \leq \wtilde{\cC_{11.1}}\paren*{\Abs*{\fa}_{C^{0}_{\delta, g_{\bomega_{\epsilon}}}} + \Abs*{D_{g_{\bomega_{\epsilon}}}\fa}_{C^{0,\alpha}_{\delta - 1, g_{\bomega_{\epsilon}} } }}$ on the metric balls $\paren*{B^{\frac{1}{\epsilon^{2}}g_{\bomega_{\epsilon}}}_{R_{1}}\paren*{p}, B^{\frac{1}{\epsilon^{2}}g_{\bomega_{\epsilon}}}_{R_{2}}\paren*{p}} = \paren*{B^{g_{\bomega_{\epsilon}}}_{\epsilon R_{1}}\paren*{p}, B^{g_{\bomega_{\epsilon}}}_{\epsilon R_{2}}\paren*{p}}$ upon setting $r_{p} \coloneq \epsilon R_{1}$ and $R_{2} \coloneq 2 R_{1}$. Note that $\wtilde{\cC_{11.1}} > 0$ is still independent of $p$ \textit{as well as $\epsilon > 0$} since $\epsilon \leq \rho_{\epsilon} \leq 2\epsilon$ tells us that we may replace $\rho_{\epsilon}$ with $\epsilon$, up to introducing harmless factors of $2$ and $\frac{1}{2}$ into $\cC_{11.1}$.

\item \textbf{Case 2; $p \in \set*{\epsilon < \abs*{\cdot}^{g_{\bomega_{0}}}_{\pi^{-1}\paren*{\cS}} < \frac{1}{8}}$:} Again WLOG let us focus only on 1 of the 16 connected components of $\set*{\epsilon < \abs*{\cdot}^{g_{\bomega_{0}}}_{\pi^{-1}\paren*{\cS}} < \frac{1}{8}}$. Hence $g_{\bomega_{\epsilon}}$ is an interpolation of $g_{\bomega_{EH,\epsilon^{2}}}$ and $g_{\bomega_{0}}$ on $\set*{\epsilon < \abs*{\cdot}^{g_{\bomega_{0}}}_{\pi^{-1}\paren*{\cS}} < \frac{1}{8}}$, namely $g_{\bomega_{\epsilon}} = \begin{cases}
g_{\wtilde{\bomega_{EH, \epsilon}}} & \text{on } \set*{\abs*{\cdot}^{g_{\bomega_{0}}}_{\pi^{-1}\paren*{\cS}} \leq 3\epsilon^{\lambda}}\\
g_{\bomega_{0}} & \text{on } \set*{ 3\epsilon^{\lambda} \leq \abs*{\cdot}^{g_{\bomega_{0}}}_{\pi^{-1}\paren*{\cS}}} 
\end{cases}	$ and where $g_{\wtilde{\bomega_{EH, \epsilon}}} = \begin{cases}
g_{\bomega_{EH,\epsilon^{2}}} & \text{on } \set*{\abs*{\cdot}^{g_{\bomega_{0}}}_{\pi^{-1}\paren*{\cS}} \leq \epsilon^{\lambda}} \\
g_{\bomega_{0}} & \text{on }\set*{2\epsilon^{\lambda} \leq \abs*{\cdot}^{g_{\bomega_{0}}}_{\pi^{-1}\paren*{\cS}}}
\end{cases}$.

Now by \textit{Remark} \ref{identification}, we may identify $\set*{\epsilon < \abs*{\cdot}^{g_{\bomega_{0}}}_{\pi^{-1}\paren*{\cS}} < \frac{1}{8}}$ with $\set*{\epsilon < \abs*{\cdot}^{g_{\bomega_{0}}}_{S^{2}} < \frac{1}{8}}\subset T^{*}S^{2}$, and upon recalling $g_{\wtilde{\bomega_{EH, \epsilon}}} = \epsilon^{2} R^{*}_{\frac{1}{\epsilon}} \paren*{g_{\lwhat{\bomega_{EH-0, \epsilon}}}} \Leftrightarrow R_{\epsilon}^{*}\paren*{\frac{1}{\epsilon^{2}}g_{\wtilde{\bomega_{EH, \epsilon}}}} = g_{\lwhat{\bomega_{EH-0, \epsilon}}}$ and $g_{\bomega_{0}} = \epsilon^{2}R^{*}_{\frac{1}{\epsilon}} \paren*{g_{\bomega_{0}}}\Leftrightarrow R^{*}_{\epsilon}\paren*{\frac{1}{\epsilon^{2}}g_{\bomega_{0}}} = g_{\bomega_{0}}$ from Proposition \ref{HKA preglueEH} directly implying by construction of $g_{\lwhat{\bomega_{EH-0, \epsilon}}}$ from Proposition \ref{HKA preglueEH BG interpolation!} (as well as by construction of $g_{\bomega_{\epsilon}}$) that $R^{*}_{\epsilon}\paren*{\frac{1}{\epsilon^{2}}g_{\bomega_{\epsilon}}} = g_{\lwhat{\bomega_{EH-0, \epsilon}}}$, we have that upon doing a conformal rescaling $g_{\bomega_{\epsilon}} \mapsto \frac{1}{\epsilon^{2}}g_{\bomega_{\epsilon}}$ that $\paren*{\set*{\epsilon < \abs*{\cdot}^{g_{\bomega_{0}}}_{\pi^{-1}\paren*{\cS}} < \frac{1}{8}}, \frac{1}{\epsilon^{2}} g_{\bomega_{\epsilon}}}$ is isometric under $R_{\epsilon}$ to $\paren*{\set*{1 < \abs*{\cdot}^{g_{\bomega_{0}}}_{S^{2}} < \frac{1}{8\epsilon}}, g_{\lwhat{\bomega_{EH-0, \epsilon}}}} \subset \paren*{T^{*}S^{2}, g_{\lwhat{\bomega_{EH-0, \epsilon}}}}$, which is a precompact region of the Riemannian manifold $\paren*{T^{*}S^{2},g_{\lwhat{\bomega_{EH-0, \epsilon}}}}$. Now by Proposition \ref{HKA preglueEH BG interpolation!}, since $g_{\lwhat{\bomega_{EH-0, \epsilon}}} = \begin{cases}
g_{\bomega_{EH,1}} & \text{on } \set*{\abs*{\cdot}^{g_{\bomega_{0}}}_{S^{2}} \leq \frac{1}{\epsilon^{\frac{1}{2}}} } \\
g_{\bomega_{0}} & \text{on }\set*{\frac{2}{\epsilon^{\frac{1}{2}}} \leq \abs*{\cdot}^{g_{\bomega_{0}}}_{S^{2}}}
\end{cases}$ aka $g_{\lwhat{\bomega_{EH-0, \epsilon}}}$ is an interpolation of $g_{\bomega_{EH,1}}$ on a large compact region and with the flat Euclidean metric $g_{\bomega_{0}}$ on the complement of a (slightly) larger (pre)compact region, it is immediate that the Riemannian manifold $\paren*{T^{*}S^{2},g_{\lwhat{\bomega_{EH-0, \epsilon}}}}$ is complete with bounded geometry.

Next, note that upon transferring over the weight function $\rho_{\epsilon}$ via the above process, we end up getting $R^{*}_{\epsilon}\paren*{\rho_{\epsilon}}$, and since this satisfies $R^{*}_{\epsilon}\paren*{\rho_{\epsilon}} = \begin{cases}
\epsilon\abs*{\cdot}^{g_{\bomega_{0}}}_{S^{2}} & \text{ on }\set*{2 \leq \abs*{\cdot}^{g_{\bomega_{0}}}_{S^{2}}< \frac{1}{8\epsilon}}\\
\epsilon & \text{ on }\set*{\abs*{\cdot}^{g_{\bomega_{0}}}_{S^{2}} \leq 1}
\end{cases}$, we immediately see that $R^{*}_{\epsilon}\paren*{\frac{1}{\epsilon}\rho_{\epsilon}} = \wtilde{\rho_{0}}$ on $\set*{1 < \abs*{\cdot}^{g_{\bomega_{0}}}_{S^{2}} < \frac{1}{8\epsilon}}$, the weight function used to defined the weighted H\"{o}lder spaces on $T^{*}S^{2}$.

Now recall that our point $p$ is now $p \in \set*{1 < \abs*{\cdot}^{g_{\bomega_{0}}}_{S^{2}} < \frac{1}{8\epsilon}}$, where we abuse notation and denote the point where $p$ lands in $\set*{1 < \abs*{\cdot}^{g_{\bomega_{0}}}_{S^{2}} < \frac{1}{8\epsilon}}$ as $p$. Let $\delta_{2,p}> 0$ be the same constant as in Proposition \ref{preciseweightfunction EH almost constancy} but shrunken so that we have that $\set*{\abs*{p}^{g_{\bomega_{0}}}_{S^{2}} - \delta_{2,p}\leq \abs*{\cdot}^{g_{\bomega_{0}}}_{S^{2}} \leq \abs*{p}^{g_{\bomega_{0}}}_{S^{2}} + \delta_{2,p}} \subset \set*{1 < \abs*{\cdot}^{g_{\bomega_{0}}}_{S^{2}} < \frac{1}{8\epsilon}}$ by openness. Now let $\sigma \in \paren*{0,1}$ be chosen such that, upon letting $R \coloneq \wtilde{\rho_{0}}(p)$ for notation's sake, we have that $B_{2\sigma R}^{g_{\lwhat{\bomega_{EH-0, \epsilon}}}}(p) \subset \set*{\abs*{p}^{g_{\bomega_{0}}}_{S^{2}} - \delta_{2,p}\leq \abs*{\cdot}^{g_{\bomega_{0}}}_{S^{2}} \leq \abs*{p}^{g_{\bomega_{0}}}_{S^{2}} + \delta_{2,p}} \subset \set*{1 < \abs*{\cdot}^{g_{\bomega_{0}}}_{S^{2}} < \frac{1}{8\epsilon}}$. Therefore, by Proposition \ref{preciseweightfunction EH almost constancy}, we have that $\frac{999}{1000} \leq \frac{\wtilde{\rho_{0}}(q)}{\wtilde{\rho_{0}}(p)} = \frac{\wtilde{\rho_{0}}(q)}{R} \leq \frac{1001}{1000}$ holds $\forall q \in B_{2\sigma R}^{g_{\lwhat{\bomega_{EH-0, \epsilon}}}}(p)$. 

Now because $\paren*{T^{*}S^{2},g_{\lwhat{\bomega_{EH-0, \epsilon}}}}$ has bounded geometry and we're on a precompact region $\paren*{\set*{1 < \abs*{\cdot}^{g_{\bomega_{0}}}_{S^{2}} < \frac{1}{8\epsilon}}, g_{\lwhat{\bomega_{EH-0, \epsilon}}}}$, we have the standard Schauder estimate $\Abs*{k\fa}_{C^{1,\alpha}_{g_{\lwhat{\bomega_{EH-0, \epsilon}}}}} \leq \cK \paren*{ \Abs*{k\fa}_{C^{0}_{g_{\lwhat{\bomega_{EH-0, \epsilon}}}}} + \Abs*{k D_{g_{\lwhat{\bomega_{EH-0, \epsilon}}}}\fa}_{C^{0,\alpha}_{g_{\lwhat{\bomega_{EH-0, \epsilon}}}}} }$ on the metric balls $\paren*{B^{g_{\lwhat{\bomega_{EH-0, \epsilon}}}}_{\sigma R}(p), B^{g_{\lwhat{\bomega_{EH-0, \epsilon}}}}_{2\sigma R}(p)}$ (where we have multiplied our function by a positive scalar $k > 0$, to be determined), \textit{where $\cK> 0$ is independent of $p$}.

Performing a conformal rescaling $g_{\lwhat{\bomega_{EH-0, \epsilon}}} \mapsto \frac{1}{R^{2}}g_{\lwhat{\bomega_{EH-0, \epsilon}}}$ and then applying Proposition \ref{HKA scalingproperties}(\ref{HKA scalingproperties Schauder Dirac rescale}) with $C = \frac{1}{R}$ and $g = g_{\lwhat{\bomega_{EH-0, \epsilon}}}$ gives us \begin{al}
R\Abs*{k\fa}_{C^{0}_{g_{\lwhat{\bomega_{EH-0, \epsilon}}}}} &+ R^{2} \Abs*{\nabla_{g_{\lwhat{\bomega_{EH-0, \epsilon}}}}k\fa}_{C^{0}_{g_{\lwhat{\bomega_{EH-0, \epsilon}}}}} + R^{2 + \alpha}\sqparen*{\nabla_{g_{\lwhat{\bomega_{EH-0, \epsilon}}}} k\fa}_{C^{0,\alpha}_{g_{\lwhat{\bomega_{EH-0, \epsilon}}}}}\\ &\leq \cK\Biggl( R\Abs*{k\fa}_{C^{0}_{g_{\lwhat{\bomega_{EH-0, \epsilon}}}}} + R^{2}\Abs*{d^{*}_{g_{\lwhat{\bomega_{EH-0, \epsilon}}}}k \fa}_{C^{0}_{g_{\lwhat{\bomega_{EH-0, \epsilon}}}}} + R^{2+\alpha}\sqparen*{d^{*}_{g_{\lwhat{\bomega_{EH-0, \epsilon}}}}k \fa}_{C^{0,\alpha}_{g_{\lwhat{\bomega_{EH-0, \epsilon}}}}}\\
&\qquad\qquad\qquad\qquad\qquad\qquad \ +R^{2}\Abs*{d^{+}_{g_{\lwhat{\bomega_{EH-0, \epsilon}}}}k \fa}_{C^{0}_{g_{\lwhat{\bomega_{EH-0, \epsilon}}}}} + R^{2+\alpha}\sqparen*{d^{+}_{g_{\lwhat{\bomega_{EH-0, \epsilon}}}}k \fa}_{C^{0,\alpha}_{g_{\lwhat{\bomega_{EH-0, \epsilon}}}}}  \Biggr)
\end{al} on the metric balls $\paren*{B^{g_{\lwhat{\bomega_{EH-0, \epsilon}}}}_{\sigma R}(p), B^{g_{\lwhat{\bomega_{EH-0, \epsilon}}}}_{2\sigma R}(p)} = \paren*{B^{\frac{1}{R^{2}}g_{\lwhat{\bomega_{EH-0, \epsilon}}}}_{\sigma }(p), B^{\frac{1}{R^{2}}g_{\lwhat{\bomega_{EH-0, \epsilon}}}}_{2\sigma }(p)}$.

Now recalling that $\frac{999}{1000} \leq \frac{\wtilde{\rho_{0}}(q)}{R} \leq \frac{1001}{1000}$ holds $\forall q \in B_{2\sigma R}^{g_{\lwhat{\bomega_{EH-0, \epsilon}}}}(p)$ and hence $\forall q \in B_{\sigma R}^{g_{\lwhat{\bomega_{EH-0, \epsilon}}}}(p)$ as well, we get upon recalling the definition of the weighted H\"{o}lder norms the following weighted Schauder estimate $$\Abs*{R^{\delta+1}k\fa}_{C^{1,\alpha}_{\delta, g_{\lwhat{\bomega_{EH-0, \epsilon}}}}} \leq \wtilde{\cC_{11.2}} \paren*{ \Abs*{R^{\delta+1}k\fa}_{C^{0}_{\delta, g_{\lwhat{\bomega_{EH-0, \epsilon}}}}} + \Abs*{R^{\delta+1}k D_{g_{\lwhat{\bomega_{EH-0, \epsilon}}}}\fa}_{C^{0,\alpha}_{\delta-1, g_{\lwhat{\bomega_{EH-0, \epsilon}}}}} }$$ on the metric balls $\paren*{B^{g_{\lwhat{\bomega_{EH-0, \epsilon}}}}_{\sigma R}(p), B^{g_{\lwhat{\bomega_{EH-0, \epsilon}}}}_{2\sigma R}(p)} = \paren*{B^{\frac{1}{R^{2}}g_{\lwhat{\bomega_{EH-0, \epsilon}}}}_{\sigma }(p), B^{\frac{1}{R^{2}}g_{\lwhat{\bomega_{EH-0, \epsilon}}}}_{2\sigma }(p)}$, where the weight function here is $\wtilde{\rho_{0}}$. Similarly as in the first case, $\wtilde{\cC_{11.2}}> 0$ is still independent of all the parameters involved because we've only modified $\cK$ by factors of $\frac{999}{1000}$ and $\frac{1001}{1000}$ since we've applied the uniform estimate $\frac{999}{1000} \leq \frac{\wtilde{\rho_{0}}(q)}{R} \leq \frac{1001}{1000}$ on those metric balls.

Now let us transfer everything back to $\paren*{\set*{\epsilon < \abs*{\cdot}^{g_{\bomega_{0}}}_{\pi^{-1}\paren*{\cS}} < \frac{1}{8}}, \frac{1}{\epsilon^{2}} g_{\bomega_{\epsilon}}}$ via applying $R_{\epsilon}$ to $\paren*{\set*{1 < \abs*{\cdot}^{g_{\bomega_{0}}}_{S^{2}} < \frac{1}{8\epsilon}}, g_{\lwhat{\bomega_{EH-0, \epsilon}}}}$ and using the identification $\set*{\epsilon < \abs*{\cdot}^{g_{\bomega_{0}}}_{\pi^{-1}\paren*{\cS}} < \frac{1}{8}}\sim \set*{\epsilon < \abs*{\cdot}^{g_{\bomega_{0}}}_{S^{2}} < \frac{1}{8}}\subset T^{*}S^{2}$, as well as using the relation $\frac{1}{\epsilon^{2}}g_{\bomega_{\epsilon}} = R^{*}_{\frac{1}{\epsilon}} \paren*{g_{\lwhat{\bomega_{EH-0, \epsilon}}}}$ and the relation $\frac{1}{\epsilon}\rho_{\epsilon} = R^{*}_{\frac{1}{\epsilon}}\paren*{\wtilde{\rho_{0}}}$. We get that our weighted Schauder estimate has now become $\Abs*{R^{\delta+1}k\fa}_{C^{1,\alpha}_{\delta, \frac{1}{\epsilon^{2}}g_{\bomega_{\epsilon}}}} \leq \wtilde{\cC_{11.2}} \paren*{ \Abs*{R^{\delta+1}k\fa}_{C^{0}_{\delta, \frac{1}{\epsilon^{2}}g_{\bomega_{\epsilon}}}} + \Abs*{R^{\delta+1}k D_{\frac{1}{\epsilon^{2}}g_{\bomega_{\epsilon}}}\fa}_{C^{0,\alpha}_{\delta-1, \frac{1}{\epsilon^{2}}g_{\bomega_{\epsilon}} }} }$ on the metric balls $\paren*{B^{\frac{1}{\epsilon^{2}}g_{\bomega_{\epsilon}}}_{R_{1}}(p), B^{\frac{1}{\epsilon^{2}}g_{\bomega_{\epsilon}}}_{R_{2}}(p)}$ for some $0 < R_{1} < R_{2}$ (which may depend on $p$) such that $B^{\frac{1}{\epsilon^{2}}g_{\bomega_{\epsilon}}}_{R_{2}}(p) \subset \set*{\epsilon < \abs*{\cdot}^{g_{\bomega_{0}}}_{\pi^{-1}\paren*{\cS}} < \frac{1}{8}}$ and $B^{\frac{1}{\epsilon^{2}}g_{\bomega_{\epsilon}}}_{R_{2}}(p) \subset R_{\epsilon}\paren*{B^{g_{\lwhat{\bomega_{EH-0, \epsilon}}}}_{2\sigma R}(p)}$ (and vice versa for $R_{1}$), \textit{where the weight function used here is $\frac{1}{\epsilon}\rho_{\epsilon}$}. Now we undo the original conformal rescaling of $\frac{1}{\epsilon^{2}}g_{\bomega_{\epsilon}}$ via performing a conformal rescaling $\frac{1}{\epsilon^{2}}g_{\bomega_{\epsilon}} \mapsto \epsilon^{2}\frac{1}{\epsilon^{2}}g_{\bomega_{\epsilon}} = g_{\bomega_{\epsilon}}$ \textit{on the metric balls $\paren*{B^{\frac{1}{\epsilon^{2}}g_{\bomega_{\epsilon}}}_{R_{1}}(p), B^{\frac{1}{\epsilon^{2}}g_{\bomega_{\epsilon}}}_{R_{2}}(p)}$ hence yielding $\paren*{B^{\frac{1}{\epsilon^{2}}g_{\bomega_{\epsilon}}}_{R_{1}}\paren*{p}, B^{\frac{1}{\epsilon^{2}}g_{\bomega_{\epsilon}}}_{R_{2}}\paren*{p}} = \paren*{B^{g_{\bomega_{\epsilon}}}_{\epsilon R_{1}}\paren*{p}, B^{g_{\bomega_{\epsilon}}}_{\epsilon R_{2}}\paren*{p}}$}, applying Proposition \ref{HKA scalingproperties}(\ref{HKA scalingproperties weighted H\"{o}lder tensor norm rescale}) with $C = \epsilon$ and $g = \frac{1}{\epsilon^{2}}g_{\bomega_{\epsilon}}$ on each of the 3 weighted H\"{o}lder norm terms \textit{where $g$ denotes the metric used in the weighted H\"{o}lder norm}, and then setting $k = \epsilon^{\delta + 1}R^{-\delta-1}$, yielding the weighted Schauder estimate $\Abs*{\fa}_{C^{1,\alpha}_{\delta, g_{\bomega_{\epsilon}}}} \leq \wtilde{\cC_{11.2}} \paren*{ \Abs*{\fa}_{C^{0}_{\delta, g_{\bomega_{\epsilon}}}} + \Abs*{D_{\frac{1}{\epsilon^{2}}g_{\bomega_{\epsilon}}}\fa}_{C^{0,\alpha}_{\delta-1, g_{\bomega_{\epsilon}} }} }$ on the metric balls $\paren*{B^{\frac{1}{\epsilon^{2}}g_{\bomega_{\epsilon}}}_{R_{1}}(p), B^{\frac{1}{\epsilon^{2}}g_{\bomega_{\epsilon}}}_{R_{2}}(p)}= \paren*{B^{g_{\bomega_{\epsilon}}}_{\epsilon R_{1}}\paren*{p}, B^{g_{\bomega_{\epsilon}}}_{\epsilon R_{2}}\paren*{p}}$. Now Proposition \ref{HKA scalingproperties}(\ref{HKA scalingproperties Dirac rescale}) tells us that $D_{\frac{1}{\epsilon^{2}} g_{\bomega_{\epsilon}}} = \paren*{\epsilon^{2} d_{g_{\bomega_{\epsilon}}}^{*}, d_{g_{\bomega_{\epsilon}}}^{+}} = \paren*{\epsilon^{2}d_{g_{\bomega_{\epsilon}}}^{*}} \oplus d_{g_{\bomega_{\epsilon}}}^{+}$, and since Data \ref{constant epsilon 1-16} tells us that $\epsilon^{2} < \frac{1}{16} < 1$, we finally get our sought after weighted Schauder estimate $\Abs*{\fa}_{C^{1,\alpha}_{\delta, g_{\bomega_{\epsilon}}}} \leq \wtilde{\cC_{11.2}} \paren*{ \Abs*{\fa}_{C^{0}_{\delta, g_{\bomega_{\epsilon}}}} + \Abs*{D_{g_{\bomega_{\epsilon}}}\fa}_{C^{0,\alpha}_{\delta-1, g_{\bomega_{\epsilon}} }} }$ on the metric balls $\paren*{B^{\frac{1}{\epsilon^{2}}g_{\bomega_{\epsilon}}}_{R_{1}}\paren*{p}, B^{\frac{1}{\epsilon^{2}}g_{\bomega_{\epsilon}}}_{R_{2}}\paren*{p}} = \paren*{B^{g_{\bomega_{\epsilon}}}_{\epsilon R_{1}}\paren*{p}, B^{g_{\bomega_{\epsilon}}}_{\epsilon R_{2}}\paren*{p}}$ upon setting $r_{p} \coloneq \epsilon R_{1}$ and $R_{2} \coloneq 2 R_{1}$.

\item \textbf{Case 3; $p \in \set*{\frac{3}{25} < \abs*{\cdot}^{g_{\bomega_{0}}}_{\pi^{-1}\paren*{\cS}}}$:} By Data \ref{constant epsilon 1-16}, since $2\epsilon^{\lambda} < \frac{3}{25}$, by construction of $g_{\bomega_{\epsilon}}$ (Proposition \ref{HKA approximatemetricproperties}) we have that on this region $g_{\bomega_{\epsilon}} = g_{\bomega_{0}}$ the flat metric. This, combined with Proposition \ref{preciseweightfunction}(\ref{preciseweightfunction past 3/25 region}) which tells us that our weight function satisfies on this region the \textit{uniform} estimate $\frac{3}{25} \leq \rho_{\epsilon} \leq 1$, the proof of the weighted Schauder estimate is exactly as in the first case, but this time using the identification $\set*{\frac{3}{25} < \abs*{\cdot}^{g_{\bomega_{0}}}_{\pi^{-1}\paren*{\cS}}} \sim \set*{\frac{3}{25} < \abs*{\cdot}^{g_{\bomega_{0}}}_{\cS}} \subset \bbT^{4} - \cS\subset \bbT^{4}$ from \textit{Remark} \ref{identification} and without having to do any conformal rescaling.

In more detail, upon transferring everything over from $\paren*{\set*{\frac{3}{25} < \abs*{\cdot}^{g_{\bomega_{0}}}_{\pi^{-1}\paren*{\cS}}}, g_{\bomega_{\epsilon}}}$ to (the obviously precompact region) $\paren*{\set*{\frac{3}{25} < \abs*{\cdot}^{g_{\bomega_{0}}}_{\cS}}, g_{\bomega_{0}}}\subset \paren*{\bbT^{4}, g_{\bomega_{0}}}$, and since flat tori clearly have bounded geometry, we thus automatically have the standard Schauder estimate $\Abs*{\fa}_{C^{1,\alpha}_{g_{\bomega_{0}}}} \leq \cC_{11.3} \paren*{ \Abs*{\fa}_{C^{0}_{g_{\bomega_{0}}}} + \Abs*{D_{g_{\bomega_{0}}}\fa}_{C^{0,\alpha}_{g_{\bomega_{0}} }} }$ on the metric balls $\paren*{B^{g_{\bomega_{0}}}_{R_{1}}\paren*{p}, B^{g_{\bomega_{0}}}_{R_{2}}\paren*{p}}$ where $0 < R_{1} < R_{2}$ is chosen such that $B^{g_{\bomega_{0}}}_{R_{2}}\paren*{p} \subset \set*{\frac{3}{25} < \abs*{\cdot}^{g_{\bomega_{0}}}_{\cS}}$, and \textit{where $\cC_{11.3}> 0$ is independent of $p$}. Undoing the identification $\set*{\frac{3}{25} < \abs*{\cdot}^{g_{\bomega_{0}}}_{\pi^{-1}\paren*{\cS}}} \sim \set*{\frac{3}{25} < \abs*{\cdot}^{g_{\bomega_{0}}}_{\cS}} $ gives us that same estimate over $\set*{\frac{3}{25} < \abs*{\cdot}^{g_{\bomega_{0}}}_{\pi^{-1}\paren*{\cS}}}$ on the metric balls $\paren*{B^{g_{\bomega_{0}}}_{R_{1}}\paren*{p}, B^{g_{\bomega_{0}}}_{R_{2}}\paren*{p}}$ (which are clearly preserved under the identification). Since we have our uniform estimate $\frac{3}{25} \leq \rho_{\epsilon} \leq 1$ on $\set*{\frac{3}{25} < \abs*{\cdot}^{g_{\bomega_{0}}}_{\pi^{-1}\paren*{\cS}}}$, we have on any subset of $\set*{\frac{3}{25} < \abs*{\cdot}^{g_{\bomega_{0}}}_{\pi^{-1}\paren*{\cS}}}$ that the weighted H\"{o}lder norms WRT $\rho_{\epsilon}$ and the standard H\"{o}lder norms are \textit{equivalent norms}, whence (up to changing the constant $\cC_{11.3}$ by harmless factors) we immediately get our sought after weighted Schauder estimate $\Abs*{\fa}_{C^{1,\alpha}_{\delta, g_{\bomega_{\epsilon}}}} \leq \wtilde{\cC_{11.3}} \paren*{ \Abs*{\fa}_{C^{0}_{\delta, g_{\bomega_{\epsilon}}}} + \Abs*{D_{g_{\bomega_{\epsilon}}}\fa}_{C^{0,\alpha}_{\delta-1, g_{\bomega_{\epsilon}} }} }$ on the metric balls $\paren*{B^{g_{\bomega_{\epsilon}}}_{R_{1}}\paren*{p}, B^{g_{\bomega_{\epsilon}}}_{R_{2}}\paren*{p}}$ upon setting $r_{p} \coloneq R_{1}$ and $R_{2} \coloneq 2 R_{1}$, with $\wtilde{\cC_{11.3}}>0$ still independent of all parameters involved.

\end{enumerate} Setting $\cC_{11} \coloneq \max\set*{\wtilde{\cC_{11.1}},\wtilde{\cC_{11.2}},\wtilde{\cC_{11.3}}}$ and noting that these $3$ constants are both \textit{independent of $p$} as well as \textit{independent of $\epsilon > 0$} yields us our local weighted Schauder estimate $\forall p \in \Km_{\epsilon}$, as was to be shown.\end{proof}

Thus we get the analogue of Corollary \ref{WSE} for $D_{g_{\bomega_{\epsilon}}}$ whose proof is verbatim from before: \begin{cor}[Weighted Schauder Estimate]\label{HKA WSE} Let $\epsilon > 0$ be from Data \ref{constant epsilon 1-16} and sufficiently small.
We have the following estimate $\forall \fa \in C^{1,\alpha}_{\delta, g_{\bomega_{\epsilon}}}\paren*{T^{*}\Km_{\epsilon}}$:
\begin{al}
\Abs*{\fa}_{C^{1,\alpha}_{\delta, g_{\bomega_{\epsilon}}} \paren*{T^{*}\Km_{\epsilon}}} &\leq \cC_{12}\paren*{\Abs*{\fa}_{C^{0}_{\delta, g_{\bomega_{\epsilon}}} \paren*{T^{*}\Km_{\epsilon}}} + \Abs*{D_{g_{\bomega_{\epsilon}}}\fa}_{C^{0,\alpha}_{\delta - 1, g_{\bomega_{\epsilon}}} \paren*{\bbR\oplus \Lambda_{g_{\bomega_{\epsilon}}}^{+}T^{*}\Km_{\epsilon}}}}\\
&\coloneq \cC_{12}\paren*{\Abs*{\rho_{\epsilon}^{-\delta}\fa}_{C^{0}_{g_{\bomega_{\epsilon}}} \paren*{T^{*}\Km_{\epsilon}}} + \Abs*{D_{g_{\bomega_{\epsilon}}}\fa}_{C^{0,\alpha}_{\delta - 1, g_{\bomega_{\epsilon}}} \paren*{\bbR\oplus \Lambda_{g_{\bomega_{\epsilon}}}^{+}T^{*}\Km_{\epsilon}}}}
\end{al}
for $\cC_{12} > 0$ \textit{independent of $\epsilon$.}
\end{cor}

Thus since we want to improve the above weighted H\"{o}lder estimate for $D_{g_{\bomega_{\epsilon}}}$ to uniform invertibility of $D_{g_{\bomega_{\epsilon}}}$, the analogue of Proposition \ref{IWSE} is \begin{prop}[Improved Weighted Schauder Estimate]\label{HKA IWSE}  Let $\epsilon > 0$ be as from Data \ref{constant epsilon 1-16}. Let $\lambda< \frac{2}{3}$. \textit{Let the rate/weight parameter satisfy $\boxed{\delta \in \paren*{-2,0}}$.}

Then for $\epsilon > 0$ sufficiently small there exists a constant $\cC_{13} > 0$ \textit{independent of $\epsilon>0$} such that we have the following estimate for $D_{g_{\bomega_{\epsilon}}}:C^{1,\alpha}_{\delta, g_{\bomega_{\epsilon}}}\paren*{T^{*}\Km_{\epsilon}}\rightarrow C^{0,\alpha}_{\delta-1, g_{\bomega_{\epsilon}}}\paren*{\bbR \oplus \Lambda_{g_{\bomega_{\epsilon}}}^{+}T^{*}\Km_{\epsilon}}$:
$$\Abs*{\fa}_{C^{1,\alpha}_{\delta, g_{\bomega_{\epsilon}}} \paren*{T^{*}\Km_{\epsilon}}} \leq \cC_{13} \Abs*{D_{g_{\bomega_{\epsilon}}}\fa}_{C^{0,\alpha}_{\delta - 1, g_{\bomega_{\epsilon}}} \paren*{\bbR\oplus \Lambda_{g_{\bomega_{\epsilon}}}^{+}T^{*}\Km_{\epsilon}}},\qqfa \fa \in C^{1,\alpha}_{\delta, g_{\bomega_{\epsilon}}} \paren*{T^{*}\Km_{\epsilon}}$$
\end{prop}

\begin{remark}\label{HKA closedrangedetector}
This inequality precisely means that when $\delta \in \paren*{\max\set*{-2,-4\lambda},0}$ and the gluing parameter $\epsilon > 0$ is sufficiently small, $D_{g_{\bomega_{\epsilon}}}:C^{1,\alpha}_{\delta, g_{\bomega_{\epsilon}}}\paren*{T^{*}\Km_{\epsilon}}\rightarrow C^{0,\alpha}_{\delta-1, g_{\bomega_{\epsilon}}}\paren*{\bbR \oplus \Lambda_{g_{\bomega_{\epsilon}}}^{+}T^{*}\Km_{\epsilon}}$ is injective, \textit{uniformly} bounded below, and has \textit{closed range}.\end{remark}

\begin{proof} 

Just like for Proposition \ref{IWSE}, we prove this by contradiction via a ``blow-up'' analysis dichotomy. However, as we are dealing with $1$-forms and with $D_{g}$, there are several differences that manifest in the analysis.

By our weighted Schauder estimate proved above (Corollary \ref{HKA WSE}), we may instead prove $$ \Abs*{\fa}_{C^{0}_{\delta, g_{\bomega_{\epsilon}}} \paren*{T^{*}\Km_{\epsilon}}} \leq \cC_{14} \Abs*{D_{g_{\bomega_{\epsilon}}}\fa}_{C^{0,\alpha}_{\delta - 1, g_{\bomega_{\epsilon}}} \paren*{\bbR\oplus \Lambda_{g_{\bomega_{\epsilon}}}^{+}T^{*}\Km_{\epsilon}}}$$ with the same assumptions ($\delta \in \paren*{\max\set*{-2,-4\lambda},0}$, etc.), whence $\cC_{13}\coloneq \cC_{12}\paren*{\cC_{14}+ 1}$. Now assume for contradiction that the statement for this latter inequality fails. Then this implies by countable choice that $\exists \epsilon_{i} \in \paren*{0,\frac{1}{i}}$ whence $\exists \epsilon_{i} \searrow 0$ and $\exists \fa_{i} \in C^{1,\alpha}_{\delta, g_{\bomega_{\epsilon_{i}}}} \paren*{T^{*}\Km_{\epsilon_{i}}}$ a sequence such that (after normalizing) we get a tuple of equations \begin{al}
\Abs*{\fa_{i}}_{C^{0}_{{\delta, g_{\bomega_{\epsilon_{i}}}} \paren*{T^{*}\Km_{\epsilon_{i}}}} } &= 1\\
\Abs*{D_{g_{\bomega_{\epsilon_{i}}}}\fa_{i}}_{C^{0,\alpha}_{\delta - 1, g_{\bomega_{\epsilon_{i}}}} \paren*{\bbR\oplus \Lambda_{g_{\bomega_{\epsilon_{i}}}}^{+}T^{*}\Km_{\epsilon_{i}}}} &\leq \frac{1}{i}
\end{al} Note that our weighted Schauder estimate Corollary \ref{HKA WSE} immediately gives us $$\Abs*{\fa_{i}}_{C^{1,\alpha}_{\delta, g_{\bomega_{\epsilon_{i}}}} \paren*{T^{*}\Km_{\epsilon_{i}}}} \leq 2\cC_{12}$$

Let us now examine the following two cases:

\begin{enumerate}[wide, labelwidth=!, labelindent=0pt]
\item \textbf{Case 1; bulk region:} Firstly, from \textit{Remark} \ref{identification} let us identify each region $\set*{\epsilon_{i}^{\lambda} < \abs*{\cdot}^{g_{\bomega_{0}}}_{\pi^{-1}\paren*{\cS}}} \subset \Km_{\epsilon_{i}}$ with $\set*{\epsilon_{i}^{\lambda} < \abs*{\cdot}^{g_{\bomega_{0}}}_{\cS}} \subset \bbT^{4}/\bbZ_{2} - \cS$ (hence we're in a \textit{fixed} ambient manifold). Similarly from that remark, let us denote the lift of $\set*{\epsilon_{i}^{\lambda} < \abs*{\cdot}^{g_{\bomega_{0}}}_{\cS}} \subset \bbT^{4}/\bbZ_{2} - \cS$ to $\bbT^{4} - \cS$ via the double cover as $\lwhat{\set*{\epsilon_{i}^{\lambda} < \abs*{\cdot}^{g_{\bomega_{0}}}_{\cS}}} \subset \bbT^{4} - \cS$, and denote the resulting lift of the metric $g_{\bomega_{\epsilon_{i}}}$ as $\what{g_{\bomega_{\epsilon_{i}}}}$.

Now since both $\set*{\epsilon_{i}^{\lambda} < \abs*{\cdot}^{g_{\bomega_{0}}}_{\cS}}$, $\lwhat{\set*{\epsilon_{i}^{\lambda} < \abs*{\cdot}^{g_{\bomega_{0}}}_{\cS}}}$ are in \textit{fixed} ambient manifolds $\bbT^{4}/\bbZ_{2} - \cS$, $\bbT^{4} - \cS$ respectively, and since $\bbZ_{2}$ acts on $\bbT^{4} - \cS$ and on each region $\lwhat{\set*{\epsilon_{i}^{\lambda} < \abs*{\cdot}^{g_{0}}_{\cS}}} \subset \bbT^{4} - \cS$ in the same way, by construction (Proposition \ref{HKA approximatemetricproperties}) we have that for all compact subsets $K\Subset \bbT^{4} - \cS$, $\exists i \gg 1$ such that $\forall j \geq i$ we have that $K \Subset \lwhat{\set*{\epsilon_{j}^{\lambda} < \abs*{\cdot}^{g_{\bomega_{0}}}_{\cS}}}$ and $\restr{\what{g_{\bomega_{\epsilon_{j}}}}}{K} \xrightarrow{C^{\infty}_{g_{\bomega_{0}}}} \restr{g_{\bomega_{0}}}{K}$.

Abuse notation and denote by $\fa_{i}$ to mean \textit{both} the restriction of each $\fa_{i} \in C^{1,\alpha}_{\delta, g_{\bomega_{\epsilon_{i}}}} \paren*{T^{*}\Km_{\epsilon_{i}}}$ to $\set*{\epsilon_{i}^{\lambda} < \abs*{\cdot}^{g_{\bomega_{0}}}_{\pi^{-1}\paren*{\cS}}}$ as well as the transferal of $\fa_{i}$ to $\set*{\epsilon_{i}^{\lambda} < \abs*{\cdot}^{g_{\bomega_{0}}}_{\cS}} \subset \bbT^{4}/\bbZ_{2} - \cS$, and denote the lift of each $\fa_{i}$ to $\lwhat{\set*{\epsilon_{i}^{\lambda} < \abs*{\cdot}^{g_{\bomega_{0}}}_{\cS}}}$ as $\lwhat{\fa_{i}}$. Hence $\lwhat{\fa_{i}}$ is even/$\bbZ_{2}$-invariant by Proposition \ref{CoveringSpacesInvariantForms}.

Now since we have $\Abs*{\fa_{i}}_{C^{1,\alpha}_{\delta, g_{\bomega_{\epsilon_{i}}}} \paren*{T^{*}\Km_{\epsilon_{i}}}} \leq 2\cC_{12}$, we thus get $\Abs*{\fa_{i}}_{C^{1,\alpha}_{\delta, g_{\bomega_{\epsilon_{i}}}} \paren*{\set*{\epsilon_{i}^{\lambda} < \abs*{\cdot}^{g_{\bomega_{0}}}_{\cS}}}} \leq 2\cC_{12}$ and hence the same bound for each $\lwhat{\fa_{i}}$, namely $$\Abs*{\lwhat{\fa_{i}}}_{C^{1,\alpha}_{\delta, \what{g_{\bomega_{\epsilon_{i}}}}} \paren*{\lwhat{\set*{\epsilon_{i}^{\lambda} < \abs*{\cdot}^{g_{\bomega_{0}}}_{\cS}}}}} \leq 2\cC_{12}$$ where we also lift the weight function $\rho_{\epsilon_{i}}$ to $\lwhat{\rho_{\epsilon_{i}}}$.

Therefore, by Arzel\`{a}-Ascoli, upon relabeling the original sequence of $\fa_{i}$ we have that \begin{gat}
\lwhat{\fa_{i}} \xrightarrow{C^{1,\alpha'}_{g_{\bomega_{0}, loc}}} \lwhat{\fa_{\infty}}\\
\Abs*{\lwhat{\fa_{\infty}}}_{C^{1,\alpha}_{\delta, g_{\bomega_{0}}} \paren*{\bbT^{4} - \cS}} \leq 2\cC_{2}
\end{gat} with $0 < \alpha' < \alpha < 1$, for some $1$-form $\lwhat{\fa_{\infty}}$ on $\bbT^{4} - \cS$. Here, the $loc$ convergence means on compact subsets. Moreover, since everything immediately works $\bbZ_{2}$-equivariantly, we have that $\lwhat{\fa_{\infty}}$ is even/$\bbZ_{2}$-invariant and hence descends down to a $1$-form $\fa_{\infty}$ on $\bbT^{4}/\bbZ_{2} - \cS$ by Proposition \ref{CoveringSpacesInvariantForms}. Moreover, the weight function used in $C^{1,\alpha}_{\delta, g_{\bomega_{0}}} \paren*{\bbT^{4} - \cS}$ is precisely the limiting weight function, which (as before) is $\rho_{0}$ the weight function used to define $C^{k,\alpha}_{\delta,g_{\bomega_{0}}} \paren*{\bbT^{4} - \cS}$ back in Sections \ref{Weighted Setup} and \ref{HKA Weighted Setup}. 

Playing the same game with $\Abs*{D_{g_{\bomega_{\epsilon_{i}}}}\fa_{i}}_{C^{0,\alpha}_{\delta - 1, g_{\bomega_{\epsilon_{i}}}} \paren*{\bbR\oplus \Lambda_{g_{\bomega_{\epsilon_{i}}}}^{+}T^{*}\Km_{\epsilon_{i}}}} \leq \frac{1}{i}$, namely that this bound directly implies the bounds $\Abs*{D_{g_{\bomega_{\epsilon_{i}}}}\fa_{i}}_{C^{0,\alpha}_{\delta - 1, g_{\bomega_{\epsilon_{i}}}} \paren*{\bbR\oplus \Lambda_{g_{\bomega_{\epsilon_{i}}}}^{+}T^{*}\set*{\epsilon_{i}^{\lambda} < \abs*{\cdot}^{g_{\bomega_{0}}}_{\cS}}}} \leq \frac{1}{i}$ and $\Abs*{D_{\what{g_{\bomega_{\epsilon_{i}}}}}\what{\fa_{i}}}_{C^{0,\alpha}_{\delta - 1, \what{g_{\bomega_{\epsilon_{i}}}}} \paren*{\bbR\oplus \Lambda_{g_{\bomega_{\epsilon_{i}}}}^{+}T^{*}\lwhat{\set*{\epsilon_{i}^{\lambda} < \abs*{\cdot}^{g_{\bomega_{0}}}_{\cS}}}}} \leq \frac{1}{i}$, that we may take the limit as $i\nearrow \infty$ and get $$D_{g_{\bomega_{0}}} \lwhat{\fa_{\infty}} = 0$$ on $\paren*{\bbT^{4} - \cS,g_{\bomega_{0}}}$ (since $D_{g_{\bomega_{0}}} \lwhat{\fa_{\infty}} = 0$ on all compact subsets of $\bbT^{4} - \cS$, hence on all of $\bbT^{4} - \cS$). Now as explained back in Sections \ref{Weighted Setup} and \ref{HKA Weighted Setup}, $\Abs*{\lwhat{\fa_{\infty}}}_{C^{1,\alpha}_{\delta, g_{\bomega_{0}}} \paren*{\bbT^{4} - \cS}} \coloneq \Abs*{\lwhat{\fa_{\infty}}}_{C^{0}_{\delta, g_{\bomega_{0}}} \paren*{\bbT^{4} - \cS}} + \Abs*{\nabla_{g_{\bomega_{0}}}\lwhat{\fa_{\infty}}}_{C^{0}_{\delta, g_{\bomega_{0}}} \paren*{\bbT^{4} - \cS}} + \sqparen*{ \lwhat{\fa_{\infty}}}_{C^{1,\alpha}_{\delta, g_{\bomega_{0}}} \paren*{\bbT^{4} - \cS}} \leq 2\cC_{2} < \infty$ means that $\lwhat{\fa_{\infty}}$ satisfies the decay $$\nabla^{j}_{g_{\bomega_{0}}}\lwhat{\fa_{\infty}} = O_{g_{\bomega_{0}}}\paren*{\rho_{0}^{\delta - j}}, \qqfa j \leq 1$$ That is, by the above properties of $\rho_{0}$, $\lwhat{\fa_{\infty}}$ is bounded on $\set*{\frac{1}{8} \leq \abs*{\cdot}^{g_{\bomega_{0}}}_{\cS}}$ and satisfies $$\abs*{\nabla^{j}_{g_{\bomega_{0}}}\lwhat{\fa_{\infty}}}_{g_{\bomega_{0}}} \leq C_{15,j}\paren*{\abs*{\cdot}^{g_{\bomega_{0}}}_{\cS}}^{\delta - j}, \qqfa j \leq 1$$ on $\set*{0 < \abs*{\cdot}^{g_{\bomega_{0}}}_{\cS}\leq \frac{1}{8}} \subset \bbT^{4} - \cS$ for some $C_{15,j} > 0$.

Now since $\bbT^{4} \coloneq \bbR^{4}/\Lambda$ where $\Lambda \cong \bbZ^{4}$, since $\bbR^{4}$ has $4$ \textbf{nowhere vanishing parallel $1$-forms} $dx^{1},dx^{2}, dx^{3},dx^{4}$, namely the \textbf{global coordinate coframe}, which are parallel WRT $g_{\bomega_{0}}$ and since all flat metrics on $\bbT^{4}$ come from flat metrics descending upstairs from $\bbR^{4}$, we have that $\bbT^{4}$ is \textbf{parallelizable} and we may globally trivialize $T^{*}\bbT^{4}$ with $4$ nowhere vanishing parallel $1$-forms $d\theta^{1}, d\theta^{2}, d\theta^{3}, d\theta^{4}$, namely the descendants of $dx^{1},dx^{2},dx^{3},dx^{4}$. Hence $\theta^{1},\theta^{2},\theta^{3},\theta^{4}$ are the $4$ angular coordinate functions of $\bbT^{4} \cong S^{1}\times S^{1} \times S^{1}\times S^{1}$. Hence since the Levi-Civita connection and the Hodge star of a fixed metric commute, we have that $d\theta^{1}, d\theta^{2}, d\theta^{3}, d\theta^{4}$ give us a global trivialization of $T^{*}\bbT^{4}$ by \textit{harmonic} $1$-forms. Thus since Proposition \ref{HKA Dirac implies harmonic} tells us that $$D_{g_{\bomega_{0}}} \lwhat{\fa_{\infty}} = 0 \Longrightarrow \Delta_{g_{\bomega_{0}}} \lwhat{\fa_{\infty}} = 0$$ upon writing $$\lwhat{\fa_{\infty}} = \what{f_{1}}d\theta^{1} + \what{f_{2}}d\theta^{2} + \what{f_{3}}d\theta^{3} + \what{f_{4}}d\theta^{4}$$ for $f_{1},f_{2},f_{3},f_{4}$ coefficient \textit{functions} on $\bbT^{4} - \cS$, that harmonicity of each $d\theta^{i}$ WRT $g_{\bomega_{0}}$ and $\Delta_{g_{\bomega_{0}}} \lwhat{\fa_{\infty}} = 0$ forces \begin{al}
\Delta_{g_{\bomega_{0}}} f_{1} &=0\\
\Delta_{g_{\bomega_{0}}} f_{2} &=0\\
\Delta_{g_{\bomega_{0}}} f_{3} &=0\\
\Delta_{g_{\bomega_{0}}} f_{4} &=0
\end{al}

In other words, $f_{1},f_{2},f_{3}, f_{4}$ are four harmonic \textit{functions} on $\bbT^{4} - \cS$. Moreover, since $\lwhat{\fa_{\infty}}$ is bounded on $\set*{\frac{1}{8} \leq \abs*{\cdot}^{g_{\bomega_{0}}}_{\cS}}$ and satisfies $\abs*{\nabla^{j}_{g_{\bomega_{0}}}\lwhat{\fa_{\infty}}}_{g_{\bomega_{0}}} \leq C_{15,j}\paren*{\abs*{\cdot}^{g_{\bomega_{0}}}_{\cS}}^{\delta - j}$, $\forall j \leq 1$ on $\set*{0 < \abs*{\cdot}^{g_{\bomega_{0}}}_{\cS}\leq \frac{1}{8}} \subset \bbT^{4} - \cS$, this means that \textit{the four coefficient functions satisfy these very same bounds as well}. 

Therefore we may play the same game as in the proof of Proposition \ref{IWSE} (namely elliptic regularity yielding each $f_{i} \in C^{\infty}\paren*{\bbT^{4} - \cS}$, then Laurent series expansion near $\cS$, $-2 < \delta$ with the 2nd bound near $\cS$ implying removable singularities near $\cS$, then the above bounds away from \& near $\cS$ implying each $f_{i}$ is a bounded), yielding that each of the four coefficient functions $f_{1}, f_{2}, f_{3}, f_{4}$ of $\lwhat{\fa_{\infty}}$ are \textit{constant} by maximum principle. Thus we have that $\lwhat{\fa_{\infty}} \in \Omega^{1}\paren*{\bbT^{4}}$ is just a constant $1$-form (hence smooth). 

Now recall that $\lwhat{\fa_{\infty}} = f_{1}d\theta^{1} + f_{2} d\theta^{2} + f_{3}d\theta^{3} + f_{4}d\theta^{4}$ is \textit{even/$\bbZ_{2}$-invariant} and (upon restricting to $\bbT^{4} - \cS$) descends down to a $1$-form $\fa_{\infty}$ on $\bbT^{4}/\bbZ_{2} - \cS$. But our involution $\bbZ^{2} \Lactson \bbT^{4}$ acts via $x \mapsto -x$, and \textit{since each $d\theta^{i}$ is a coordinate $1$-form for the angular/$S^{1}$ coordinate $\theta^{i}$}, we have that \textit{each $d\theta^{i}$ is odd WRT $\bbZ_{2}$}, aka satisfies $\restr{d\theta^{i}}{-x} = - \restr{d\theta^{i}}{x}$ since the involution acts via $\theta^{i} \mapsto -\theta^{i}$. Thus to preserve even-ness/$\bbZ_{2}$-invariance of $\lwhat{\fa_{\infty}}$, \textit{we must conclude that each coefficient function $f_{i}$ is odd WRT $\bbZ_{2}$, that is, $f_{i}(-x) = -f_{i}(x)$}. But each $f_{i} \in \bbR$ is a \textit{constant}, whence odd-ness forces $f_{1} = f_{2} = f_{3} = f_{4} = 0$.

Whence we have concluded that $\lwhat{\fa_{\infty}} = 0$.

Thus we have that $\lwhat{\fa_{i}} \xrightarrow{C^{1,\alpha'}_{g_{\bomega_{0}}, loc}} 0$ aka $\Abs*{\what{\fa_{i}}}_{C^{1,\alpha'}_{g_{\bomega_{0}}} \paren*{K}} \searrow 0$ for all compact subsets $K \Subset\bbT^{4} - \cS$. I now claim that $\Abs*{\what{\fa_{i}}}_{C^{1,\alpha'}_{\delta, \what{g_{\bomega_{\epsilon_{i}}}}} \paren*{K}} \searrow 0$ holds for all compact subsets $K \Subset\bbT^{4} - \cS$. Suppose not. Then for some fixed compact subset $K\Subset \bbT^{4} - \cS$, we have that $\Abs*{\what{\fa_{i}}}_{C^{1,\alpha'}_{\delta, \what{g_{\bomega_{\epsilon_{i}}}}} \paren*{K}} \geq \cC_{2.5} > 0$ uniform in $i$. Now since for sufficiently large $i$ we have $K \Subset \lwhat{\set*{\epsilon_{i}^{\lambda} < \abs*{\cdot}^{g_{\bomega_{0}}}_{\cS}}} \subset \bbT^{4} - \cS$, pick a sequence of points $p_{i} \in K \cap \lwhat{\set*{\epsilon_{i}^{\lambda} < \abs*{\cdot}^{g_{\bomega_{0}}}_{\cS}}}$ such that $p_{i} \to p_{\infty} \in K$ and in which $$ \abs*{\lwhat{\rho_{\epsilon_{i}}}^{-\delta}\paren*{p_{i}} \restr{\lwhat{\fa_{i}}}{p_{i}}}_{\what{g_{\bomega_{\epsilon_{i}}}}} = \lwhat{\rho_{\epsilon_{i}}}^{-\delta}\paren*{p_{i}}\abs*{\restr{\lwhat{\fa_{i}}}{p_{i}}}_{\what{g_{\bomega_{\epsilon_{i}}}}} \geq \cC_{2.5} > 0$$ But upon taking the limit $i\nearrow\infty$ we get $$\rho_{0}^{-\delta}\paren*{p_{\infty}}\abs*{\restr{\lwhat{\fa_{\infty}}}{p_{\infty}}}_{g_{\bomega_{0}}} \geq \cC_{2.5} > 0$$ and since $p_{\infty} \in K \Subset \bbT^{4} - \cS$, by compactness of $K$ and continuity of $\rho_{0}$ we thus have that $\rho_{0}^{-\delta}\paren*{p_{\infty}} \neq 0$ and whence we get $$\abs*{\restr{\lwhat{\fa_{\infty}}}{p_{\infty}}}_{g_{\bomega_{0}}} > 0$$ contradicting $\lwhat{\fa_{\infty}} = 0$.

Therefore, we have that $\Abs*{\lwhat{\fa_{i}}}_{C^{1,\alpha'}_{\delta, \what{g_{\bomega_{\epsilon_{i}}}}} \paren*{K}} \searrow 0$ for all compact subsets $K \Subset \lwhat{\set*{\epsilon_{i}^{\lambda} < \abs*{\cdot}^{g_{\bomega_{0}}}_{\cS}}}\subset \bbT^{4} - \cS$, whence transferring everything back to $\Km_{\epsilon_{i}}$ and using $\Abs*{\cdot}_{C^{0}_{\delta, g_{\bomega_{\epsilon_{i}}}}} \leq \Abs*{\cdot}_{C^{1,\alpha'}_{\delta,g_{\bomega_{\epsilon_{i}}}}}$, we thus have that $$\Abs*{\fa_{i}}_{C^{0}_{\delta, g_{\bomega_{\epsilon_{i}}}} \paren*{K}} \searrow 0$$ holds for all compact subsets $K \Subset \set*{\epsilon_{i}^{\lambda} < \abs*{\cdot}^{g_{\bomega_{0}}}_{\pi^{-1}\paren*{\cS}}} \subset \Km_{\epsilon_{i}}$.

\item \textbf{Case 2; bubble region:} Recall that $\Abs*{\fa_{i}}_{C^{0}_{\delta, g_{\bomega_{\epsilon_{i}}}} \paren*{\Km_{\epsilon_{i}}}} = 1$ yet \textbf{Case 1} above shows $\Abs*{\fa_{i}}_{C^{0}_{\delta, g_{\bomega_{\epsilon_{i}}}} \paren*{K}} \searrow 0$ holds for all compact subsets $K \Subset \set*{0 < \abs*{\cdot}^{g_{\bomega_{0}}}_{\pi^{-1}\paren*{\cS}}} \subset \Km_{\epsilon_{i}}$. Since singletons are compact, and $1 = \Abs*{\fa_{i}}_{C^{0}_{\delta, g_{\bomega_{\epsilon_{i}}}} \paren*{\Km_{\epsilon_{i}}}} = \max\set*{\Abs*{\fa_{i}}_{C^{0}_{\delta, g_{\bomega_{\epsilon_{i}}}} \paren*{\set*{\frac{1}{8} < \abs*{\cdot}^{g_{\bomega_{0}}}_{\pi^{-1}\paren*{\cS}}}}}, \Abs*{\fa_{i}}_{C^{0}_{\delta, g_{\bomega_{\epsilon_{i}}}} \paren*{\set*{\abs*{\cdot}^{g_{\bomega_{0}}}_{\pi^{-1}\paren*{\cS}}\leq \frac{1}{8}}}}}$ we therefore must have that $$\Abs*{\fa_{i}}_{C^{0}_{\delta, g_{\bomega_{\epsilon_{i}}}} \paren*{\set*{\abs*{\cdot}^{g_{\bomega_{0}}}_{\pi^{-1}\paren*{\cS}}\leq \frac{1}{8}}}} = 1$$ WLOG we restrict to a single connected component of $\set*{\abs*{\cdot}^{g_{\bomega_{0}}}_{\pi^{-1}\paren*{\cS}}\leq \frac{1}{8}}$ in which the above norm $=1$.

By construction and by Proposition \ref{HKA preglueEH}, we have that on $\set*{\abs*{\cdot}^{g_{\bomega_{0}}}_{\pi^{-1}\paren*{\cS}} \leq \frac{1}{8}}$ that $g_{\bomega_{\epsilon_{i}}} = \wtilde{g_{\bomega_{EH, \epsilon_{i}}}}$. Moreover, by Proposition \ref{HKA preglueEH} and Proposition \ref{HKA approximatemetricproperties} we have that $\wtilde{g_{\bomega_{EH,\epsilon_{i}}}} = \epsilon_{i}^{2} R^{*}_{\frac{1}{\epsilon_{i}}} \paren*{\lwhat{g_{\bomega_{EH-0, \epsilon_{i}}}}} \Longleftrightarrow R_{\epsilon_{i}}^{*}\paren*{\frac{1}{\epsilon_{i}^{2}} \wtilde{g_{\bomega_{EH,\epsilon_{i}}}}} = \lwhat{g_{\bomega_{EH-0, \epsilon_{i}}}}$. Therefore, just as in proving \textbf{Case 2} of Proposition \ref{HKA localWSE}, upon performing a conformal rescaling $g_{\bomega_{\epsilon_{i}}}\mapsto \frac{1}{\epsilon_{i}^{2}}g_{\bomega_{\epsilon_{i}}}$, we have that each Riemannian manifold $\paren*{\set*{\abs*{\cdot}^{g_{\bomega_{0}}}_{\pi^{-1}\paren*{\cS}} < \frac{1}{8}}, \frac{1}{\epsilon_{i}^{2}}g_{\bomega_{\epsilon_{i}}}}$ is isometric under $R_{\epsilon_{i}}$ to the Riemannian submanifold $\paren*{\set*{\abs*{\cdot}^{g_{\bomega_{0}}}_{S^{2}} < \frac{1}{8\epsilon_{i}}}, g_{\lwhat{\bomega_{EH-0, \epsilon_{i}}}}} \subset \paren*{T^{*}S^{2},g_{\lwhat{\bomega_{EH-0, \epsilon_{i}}}}}$.

Hence upon taking the limit as $i\nearrow \infty$, we have that, since we're in a \textit{fixed} ambient manifold and upon inspection Proposition \ref{HKA preglueEH BG interpolation!} tells us that $\lwhat{g_{\bomega_{EH-0, \epsilon_{i}}}}$ converges to $g_{\bomega_{EH,1}}$ on compact subsets aka $\lwhat{g_{\bomega_{EH-0, \epsilon_{i}}}} \xrightarrow{C^{\infty}_{g_{\bomega_{EH,1}},loc}} g_{\bomega_{EH,1}}$, the Riemannian manifolds $\paren*{\set*{\abs*{\cdot}^{g_{\bomega_{0}}}_{S^{2}} \leq \frac{1}{8\epsilon_{i}}}, \lwhat{g_{\bomega_{EH-0, \epsilon_{i}}}}}$ exhausts to $\paren*{T^{*}S^{2},g_{\bomega_{EH,1}}}$ in the limit.

Transferring over the weight function $\rho_{\epsilon_{i}}$ on $\set*{\abs*{\cdot}^{g_{\bomega_{0}}}_{\pi^{-1}\paren*{\cS}} < \frac{1}{8}}$ via the above isometry gives us $R^{*}_{\epsilon_{i}}\paren*{\rho_{\epsilon_{i}}}$ which satisfies $R^{*}_{\epsilon_{i}}\paren*{\frac{1}{\epsilon_{i}}\rho_{\epsilon_{i}}} = \begin{cases}
\frac{1}{\epsilon_{i}} & \text{ on }\set*{\frac{1}{5\epsilon_{i}} \leq \abs*{\cdot}^{g_{\bomega_{0}}}_{S^{2}}}\\
\abs*{\cdot}^{g_{\bomega_{0}}}_{S^{2}} & \text{ on }\set*{2\leq \abs*{\cdot}^{g_{\bomega_{0}}}_{S^{2}} \leq \frac{1}{8\epsilon_{i}}}\\
1 & \text{ on }\set*{\abs*{\cdot}^{g_{\bomega_{0}}}_{S^{2}} \leq 1}
\end{cases}$. Call $\wtilde{\rho_{\epsilon_{i}}} \coloneq R^{*}_{\epsilon_{i}}\paren*{\frac{1}{\epsilon_{i}}\rho_{\epsilon_{i}}}$ for notation.

Therefore, as $\Abs*{\fa_{i}}_{C^{0}_{\delta, g_{\bomega_{\epsilon_{i}}}} \paren*{\Km_{\epsilon_{i}}}} \coloneq \Abs*{\rho_{\epsilon_{i}}^{-\delta}\fa_{i}}_{C^{0}_{g_{\bomega_{\epsilon_{i}}}} \paren*{\Km_{\epsilon_{i}}}} = \Abs*{\epsilon_{i}\rho_{\epsilon_{i}}^{-\delta}\epsilon_{i}^{\delta} \epsilon_{i}^{-\delta-1}\fa_{i}}_{C^{0}_{g_{\bomega_{\epsilon_{i}}}} \paren*{\Km_{\epsilon_{i}}}}$, we must also scale our $1$-forms $\fa_{i}\mapsto \epsilon_{i}^{-\delta-1}\fa_{i}$ to preserve the weighted H\"{o}lder norm upon transferring to $\paren*{\set*{\abs*{\cdot}^{g_{\bomega_{0}}}_{S^{2}} < \frac{1}{8\epsilon_{i}}}, g_{\lwhat{\bomega_{EH-0, \epsilon_{i}}}}}$ (after first restricting to $\set*{\abs*{\cdot}^{g_{\bomega_{0}}}_{\pi^{-1}\paren*{\cS}} < \frac{1}{8}}$). Thus call $\wtilde{\fa_{i}} \coloneq \epsilon_{i}^{-\delta-1}\fa_{i}$ for notation. Note that $\wtilde{\fa}_{i} = \restr{\wtilde{\fa}_{i}}{\set*{\abs*{\cdot}^{g_{\bomega_{0}}}_{S^{2}} \leq \frac{1}{8\epsilon_{i}}}}$ is defined on $\set*{\abs*{\cdot}^{g_{\bomega_{0}}}_{S^{2}} \leq \frac{1}{8\epsilon_{i}}}$.

Therefore, since Case (\ref{HKA scalingproperties weighted H\"{o}lder tensor norm rescale}) of Proposition \ref{HKA scalingproperties} gives $$\Abs*{\fa_{i}}_{C^{k,\alpha}_{\delta, g_{\bomega_{\epsilon_{i}}}} \paren*{\set*{\abs*{\cdot}^{g_{\bomega_{0}}}_{\pi^{-1}\paren*{\cS}}\leq \frac{1}{8}}}} = \Abs*{\wtilde{\fa_{i}}}_{C^{k,\alpha}_{\delta, \lwhat{g_{\bomega_{EH-0,\epsilon_{i}}}}} \paren*{\set*{\abs*{\cdot}^{g_{\bomega_{0}}}_{S^{2}} \leq \frac{1}{8\epsilon_{i}}}}}$$ with weight functions $\rho_{\epsilon_{i}}$ and $\wtilde{\rho_{\epsilon_{i}}}$ respectively, the inequality $\Abs*{\fa_{i}}_{C^{1,\alpha}_{\delta, g_{\bomega_{\epsilon_{i}}}} \paren*{\Km_{\epsilon_{i}}}} \leq 2\cC_{12}$ implies we have the uniform bound $\Abs*{\wtilde{\fa_{i}}}_{C^{2,\alpha}_{\delta, \lwhat{g_{\bomega_{EH-0,\epsilon_{i}}}}} \paren*{\set*{\abs*{\cdot}^{g_{\bomega_{0}}}_{S^{2}} \leq \frac{1}{8\epsilon_{i}}}}} \leq 2\cC_{12}$ where again the weight function used is $\wtilde{\rho_{\epsilon_{i}}}$. Whence since $T^{*}S^{2} = \bigcup_{i \in \bbN} \set*{\abs*{\cdot}^{g_{\bomega_{0}}}_{S^{2}} \leq \frac{1}{8\epsilon_{i}}}$ is a nested exhaustion of compact subsets, we therefore have by Arzel\`{a}-Ascoli and the convergence $\lwhat{g_{\bomega_{EH-0,\epsilon_{i}}}} \xrightarrow{C^{\infty}_{g_{\bomega_{EH,1}},loc}} g_{\bomega_{EH,1}}$ that (upon relabeling the original sequence of $\fa_{i}$ once more) \begin{gat}
\wtilde{\fa_{i}} \xrightarrow{C^{1,\alpha'}_{g_{\bomega_{EH,1}}, loc}} \wtilde{\fa_{\infty}}\\
\Abs*{\wtilde{\fa_{\infty}}}_{C^{1,\alpha}_{\delta, g_{\bomega_{EH,1}}} \paren*{T^{*}S^{2}}} \leq 2\cC_{12}
\end{gat} with $0 < \alpha' < \alpha < 1$, for some function $\wtilde{\fa_{\infty}}$ on $T^{*}\CP^{1}$ and where the convergence $\xrightarrow{C^{1,\alpha'}_{g_{\bomega_{EH,1}}, loc}}$ means that $\Abs*{\wtilde{\fa_{i}} - \wtilde{\fa_{\infty}}}_{C^{1,\alpha'}_{g_{\bomega_{EH,1}}}\paren*{K}} \searrow 0$ hence $\Abs*{\wtilde{\fa_{i}}}_{C^{1,\alpha'}_{g_{\bomega_{EH,1}}}\paren*{K}} \to \Abs*{\wtilde{\fa_{\infty}}}_{C^{1,\alpha'}_{g_{\bomega_{EH,1}}}\paren*{K}}$, for all compact subsets $K\Subset T^{*}S^{2}$. Moreover, the weight function used in $C^{1,\alpha}_{\delta, g_{\bomega_{EH,1}}} \paren*{T^{*}S^{2}}$ is precisely the limiting weight function, which (as before) is $\wtilde{\rho_{0}}$ the weight function used to define the weighted H\"{o}lder spaces for $T^{*}S^{2}$ back in Sections \ref{Weighted Setup} and \ref{HKA Weighted Setup}.

Playing the same game with $\Abs*{D_{g_{\bomega_{\epsilon_{i}}}}\fa_{i}}_{C^{0,\alpha}_{\delta - 1, g_{\bomega_{\epsilon_{i}}}} \paren*{\Km_{\epsilon_{i}}}} \leq \frac{1}{i}$, namely that this bound directly implies the bounds $\Abs*{D_{g_{\lwhat{\bomega_{EH-0, \epsilon_{i}}}}}\fa_{i}}_{C^{0,\alpha}_{\delta - 1, g_{\lwhat{\bomega_{EH-0, \epsilon_{i}}}}} \paren*{\set*{\abs*{\cdot}^{g_{\bomega_{0}}}_{S^{2}} < \frac{1}{8\epsilon_{i}}}}} \leq \frac{1}{i}$, that we may take the limit as $i\nearrow \infty$ and get $$D_{g_{\bomega_{EH,1}}} \wtilde{\fa_{\infty}} = 0$$ on $\paren*{T^{*}S^{2},g_{\bomega_{EH,1}}}$, hence by Proposition \ref{HKA Dirac implies harmonic} we get $$\Delta_{g_{\bomega_{EH,1}}} \wtilde{\fa_{\infty}} = 0$$ and thus by elliptic regularity $\wtilde{\fa_{\infty}} \in \Omega^{1}\paren*{T^{*}S^{2}}$ is a smooth $1$-form as it is harmonic.

Now $\Abs*{\wtilde{\fa_{\infty}}}_{C^{1,\alpha'}_{\delta, g_{\bomega_{EH,1}}} \paren*{T^{*}S^{2}}} \leq 2\cC_{12}$ means that $\wtilde{\fa_{\infty}}$ satisfies the decay $$\nabla_{g_{\bomega_{EH,1}}}^{j}\wtilde{\fa_{\infty}} = O_{g_{\bomega_{EH,1}}}\paren*{\wtilde{\rho_{0}}^{\delta - j}},\qqfa j\leq 1$$ hence by the properties of $\wtilde{\rho_{0}}$, we have that $\abs*{\wtilde{\fa_{\infty}}}_{g_{\bomega_{EH,1}}}$ hence $\wtilde{\fa_{\infty}}$ is bounded \textit{WRT the norm induced from $g_{\bomega_{EH,1}}$} on $\set*{\abs*{\cdot}^{g_{\bomega_{0}}}_{S^{2}} \leq 2}$ and satisfies $$\abs*{\nabla^{j}_{g_{\bomega_{EH,1}}}\wtilde{\fa_{\infty}}}_{g_{\bomega_{EH,1}}} \leq C_{16,j} \paren*{\abs*{\cdot}^{g_{\bomega_{0}}}_{S^{2}}}^{\delta - j},\qqfa j\leq 1$$ on $\set*{2\leq \abs*{\cdot}^{g_{\bomega_{0}}}_{S^{2}}}$ for some $C_{16,j} > 0$. But since $\delta < 0$, we have that $$\abs*{\wtilde{\fa_{\infty}}}_{g_{\bomega_{EH,1}}} \leq C_{6,0} \paren*{\abs*{\cdot}^{g_{\bomega_{0}}}_{S^{2}}}^{\delta} \leq C_{16,0} 2^{\delta} < \infty$$ on $\set*{2\leq \abs*{\cdot}^{g_{\bomega_{0}}}_{S^{2}}}$, and hence $\abs*{\wtilde{\fa_{\infty}}}_{g_{\bomega_{EH,1}}}$ is bounded on all of $T^{*}S^{2}$. Moreover, $\abs*{\wtilde{\fa_{\infty}}}_{g_{\bomega_{EH,1}}} \leq C_{16,0} \paren*{\abs*{\cdot}^{g_{\bomega_{0}}}_{S^{2}}}^{\delta}$ and $\delta < 0$ means that, \textit{since $\wtilde{\fa_{\infty}}$ is a $1$-form}, $\abs*{\wtilde{\fa_{\infty}}}_{g_{\bomega_{EH,1}}}\searrow 0$ as $\abs*{\cdot}^{g_{\bomega_{0}}}_{S^{2}} \nearrow \infty$. 

Thus we have established that $\abs*{\wtilde{\fa_{\infty}}}_{g_{\bomega_{EH,1}}}$, hence $\abs*{\wtilde{\fa_{\infty}}}_{g_{\bomega_{EH,1}}}^{2}$, is bounded on all of $T^{*}S^{2}$ and $\searrow 0$ as $\abs*{\cdot}^{g_{\bomega_{0}}}_{S^{2}} \nearrow \infty$. Now from the \textbf{Bochner-Weitzenb\"{o}ck formula for $1$ forms} $-\frac{1}{2}\Delta_{g}\abs*{\alpha}_{g}^{2} = \abs*{\nabla_{g} \alpha}_{g}^{2} - \Inner*{\Delta_{g} \alpha,\alpha}_{g}   + \Ric_{g}\paren*{\alpha,\alpha}$, since the Eguchi-Hanson space $\paren*{T^{*}S^{2}, g_{\bomega_{EH,1}}}$ is Ricci-flat, we get from $\Delta_{g_{\bomega_{EH,1}}} \wtilde{\fa_{\infty}} = 0$/harmonicity of $\wtilde{\fa_{\infty}}$ that $$-\frac{1}{2}\Delta_{g_{\bomega_{EH,1}}} \abs*{\wtilde{\fa_{\infty}}}_{g_{\bomega_{EH,1}}}^{2} = \abs*{\nabla_{g_{\bomega_{EH,1}}}\wtilde{\fa_{\infty}}}_{g_{\bomega_{EH,1}}}^{2} \geq 0$$ That is, $\abs*{\wtilde{\fa_{\infty}}}_{g_{\bomega_{EH,1}}}^{2}$ is \textbf{subharmonic} on $T^{*}S^{2}$ (recall that $\Delta_{g}$ has \textit{nonnegative spectrum}), and since $\abs*{\wtilde{\fa_{\infty}}}_{g_{\bomega_{EH,1}}}^{2}$ is bounded on all of $T^{*}S^{2}$ it therefore has a global maximum whence by the maximum principle $\abs*{\wtilde{\fa_{\infty}}}_{g_{\bomega_{EH,1}}}^{2}$ is \textit{constant}, and since $\abs*{\wtilde{\fa_{\infty}}}_{g_{\bomega_{EH,1}}}^{2}\searrow 0$ as $\abs*{\cdot}^{g_{\bomega_{0}}}_{S^{2}} \nearrow \infty$, we therefore conclude $$\abs*{\wtilde{\fa_{\infty}}}_{g_{\bomega_{EH,1}}}^{2} = 0 \Longleftrightarrow \wtilde{\fa_{\infty}} = 0$$

Thus we have that $\wtilde{\fa_{i}} \xrightarrow{C^{1,\alpha'}_{g_{\bomega_{EH,1}}, loc}} 0$ aka $\Abs*{\wtilde{\fa_{i}}}_{C^{1,\alpha'}_{g_{\bomega_{EH,1}}} \paren*{K}} \searrow 0$ for all compact subsets $K \Subset T^{*}S^{2}$. I now claim that $\Abs*{\wtilde{\fa_{i}}}_{C^{1,\alpha'}_{\delta, \lwhat{g_{\bomega_{EH-0, \epsilon_{i}}}}} \paren*{K}} \searrow 0$ holds for all compact subsets $K \Subset T^{*}S^{2}$. Suppose not. Then for some fixed compact subset $K\Subset T^{*}S^{2}$, we have that $\Abs*{\wtilde{\fa_{i}}}_{C^{1,\alpha'}_{\delta, \lwhat{g_{\bomega_{EH-0, \epsilon_{i}}}}} \paren*{K}} \geq \cC_{12.5} > 0$ uniform in $i$. Now since for sufficiently large $i$ we have $K \Subset \set*{\abs*{\cdot}^{g_{\bomega_{0}}}_{S^{2}} \leq \frac{1}{8\epsilon_{i}}} \subset T^{*}S^{2}$, pick a sequence of points $p_{i} \in K \cap \set*{\abs*{\cdot}^{g_{\bomega_{0}}}_{S^{2}} \leq \frac{1}{8\epsilon_{i}}}$ such that $p_{i} \to p_{\infty} \in K$ and in which $$ \abs*{\wtilde{\rho_{\epsilon_{i}}}^{-\delta}\paren*{p_{i}} \restr{\wtilde{\fa_{i}}}{p_{i}}} = \wtilde{\rho_{\epsilon_{i}}}^{-\delta}\paren*{p_{i}}\abs*{ \restr{\wtilde{\fa_{i}}}{p_{i}}} \geq \cC_{12.5} > 0$$ But upon taking the limit $i\nearrow\infty$ we get $$\wtilde{\rho_{0}}^{-\delta}\paren*{p_{\infty}}\abs*{ \restr{\wtilde{\fa_{\infty}}}{p_{\infty}}} \geq \cC_{12.5} > 0$$ and since $p_{\infty} \in K \Subset T^{*}S^{2}$, by compactness of $K$ and continuity of $\wtilde{\rho_{0}}$ we thus have that $\wtilde{\rho_{0}}^{-\delta}\paren*{p_{\infty}} \neq 0$ and whence we get $$\abs*{ \restr{\wtilde{\fa_{\infty}}}{p_{\infty}}} > 0$$ contradicting $\wtilde{\fa_{\infty}} = 0$. Therefore $\Abs*{\wtilde{\fa_{i}}}_{C^{1,\alpha'}_{\delta, \lwhat{g_{\bomega_{EH-0, \epsilon_{i}}}}}\paren*{K}} \searrow 0$ whence $\Abs*{\wtilde{\fa_{i}}}_{C^{0,\alpha'}_{\delta, \lwhat{g_{\bomega_{EH-0, \epsilon_{i}}}}}\paren*{K}} \searrow 0$ holds for all compact subsets $K\Subset T^{*}S^{2}$.

Now we further dichotomize: either \textit{globally} $\Abs*{\wtilde{\fa_{i}}}_{C^{0,\alpha'}_{\delta, \lwhat{g_{\bomega_{EH-0, \epsilon_{i}}}}}\paren*{T^{*}S^{2}}} \searrow 0$, or not. If so, then transferring everything back to $\Km_{\epsilon_{i}}$ via $\Abs*{\wtilde{\fa_{i}}}_{C^{0,\alpha'}_{\delta, \lwhat{g_{\bomega_{EH-0, \epsilon_{i}}}}} \paren*{\set*{\abs*{\cdot}^{g_{\bomega_{0}}}_{S^{2}} \leq \frac{1}{8\epsilon_{i}}}}} = \Abs*{\fa_{i}}_{C^{0,\alpha'}_{\delta, g_{\bomega_{\epsilon_{i}}}} \paren*{\set*{\abs*{\cdot}^{g_{\bomega_{0}}}_{\pi^{-1}\paren*{\cS}}\leq \frac{1}{8}}}}$ gives us $$\Abs*{\fa_{i}}_{C^{0}_{\delta, g_{\bomega_{\epsilon_{i}}}} } \searrow 0$$ holds on $\set*{\abs*{\cdot}^{g_{\bomega_{0}}}_{\pi^{-1}\paren*{\cS}} \leq \frac{1}{8}} \subset \Km_{\epsilon_{i}}$. But this directly contradicts $\Abs*{\fa_{i}}_{C^{0}_{\delta, g_{\bomega_{\epsilon_{i}}}} \paren*{\set*{\abs*{\cdot}^{g_{\bomega_{0}}}_{\pi^{-1}\paren*{\cS}}\leq \frac{1}{8}}}} = 1$.

So instead suppose that $\Abs*{\wtilde{\fa_{i}}}_{C^{0,\alpha'}_{\delta, \lwhat{g_{\bomega_{EH-0, \epsilon_{i}}}}}\paren*{T^{*}S^{2}}}$ does \textit{not} converge to $0$ as $i\nearrow\infty$. Then $\exists \wtilde{C_{16}} > 0$ uniform in $i$ such that $0 < \wtilde{C_{16}} \leq \Abs*{\wtilde{\fa_{i}}}_{C^{0,\alpha'}_{\delta, \lwhat{g_{\bomega_{EH-0, \epsilon_{i}}}}}\paren*{T^{*}S^{2}}}$. Now since $\Abs*{\wtilde{\fa_{i}}}_{C^{0,\alpha'}_{\delta, g_{\bomega_{EH,1}}}\paren*{K}} \searrow 0$ holds $\forall K \Subset T^{*}S^{2}$, we have upon writing $T^{*}S^{2}$ as a nested exhaustion of compact subsets that $\exists \set*{p_{i}} \subset T^{*}S^{2}$ a sequence of diverging points (escaping all compact subsets) such that $p_{i} \in \set*{\abs*{\cdot}^{g_{\bomega_{0}}}_{S^{2}} \leq \frac{1}{8\epsilon_{i}}}$ and $$0 < \wtilde{C_{16}} \leq \abs*{\wtilde{\rho_{0}}\paren*{p_{i}}^{-\delta} \restr{\wtilde{\fa_{i}}}{p_{i}}} = \wtilde{\rho_{0}}\paren*{p_{i}}^{-\delta}\abs*{ \restr{\wtilde{\fa_{i}}}{p_{i}}}$$ Now since the points $\set*{p_{i}}$ are diverging, WLOG after assuming that they're all at least distance $\geq 2$ away from the zero section $S^{2}$ we thus have that $\wtilde{\rho_{0}}\paren*{p_{i}} \coloneq \abs*{p_{i}}^{g_{\bomega_{0}}}_{S^{2}} \nearrow\infty$.

Now by dichotomy, either $\sup_{i}\frac{1}{\epsilon_{i} \wtilde{\rho_{0}}\paren*{p_{i}}} < \infty$ or $\sup_{i}\frac{1}{\epsilon_{i} \wtilde{\rho_{0}}\paren*{p_{i}}} = \infty$.

Suppose the former, whence $\inf_{i} \epsilon_{i} \wtilde{\rho_{0}}\paren*{p_{i}} > 0$. Since $p_{i} \in \set*{\abs*{\cdot}^{g_{\bomega_{0}}}_{S^{2}} \leq \frac{1}{8\epsilon_{i}}}$ we have by definition that actually $p_{i} \in \set*{\frac{\inf_{i} \epsilon_{i} \wtilde{\rho_{0}}\paren*{p_{i}}}{\epsilon_{i}} \leq \abs*{\cdot}^{g_{\bomega_{0}}}_{S^{2}} \leq \frac{1}{8\epsilon_{i}}}$. Therefore transferring everything back to $\Km_{\epsilon_{i}}$ via $\Abs*{\wtilde{\fa_{i}}}_{C^{0,\alpha'}_{\delta, \lwhat{g_{\bomega_{EH-0, \epsilon_{i}}}}} \paren*{\set*{\abs*{\cdot}^{g_{\bomega_{0}}}_{S^{2}} \leq \frac{1}{8\epsilon_{i}}}}} = \Abs*{\fa_{i}}_{C^{0,\alpha'}_{\delta, g_{\bomega_{\epsilon_{i}}}} \paren*{\set*{\abs*{\cdot}^{g_{\bomega_{0}}}_{\pi^{-1}\paren*{\cS}}\leq \frac{1}{8}}}}$ tells us (upon recalling $\wtilde{\fa}_{i} \coloneq \epsilon_{i}^{-\delta-1}\fa_{i}$) that $$0<\wtilde{C_{16}} \leq \paren*{\abs*{p_{i}}^{g_{\bomega_{0}}}_{\pi^{-1}\paren*{\cS}}}^{-\delta}\abs*{ \restr{\fa_{i}}{p_{i}}}$$ where now $p_{i} \in \set*{\inf_{i}\epsilon_{i} \wtilde{\rho_{0}}\paren*{p_{i}} \leq \abs*{\cdot}^{g_{\bomega_{0}}}_{\pi^{-1}\paren*{\cS}} \leq \frac{1}{8}}$. Since $\inf_{i} \epsilon_{i} \wtilde{\rho_{0}}\paren*{p_{i}} > 0$, whence $\set*{\inf_{i}\epsilon_{i} \wtilde{\rho_{0}}\paren*{p_{i}} \leq \abs*{\cdot}^{g_{\bomega_{0}}}_{\pi^{-1}\paren*{\cS}} \leq \frac{1}{8}} \Subset \set*{0 < \abs*{\cdot}^{g_{\bomega_{0}}}_{\pi^{-1}\paren*{\cS}}} \subset \Km_{\epsilon_{i}}$. Whence since $\rho_{\epsilon_{i}} = \abs*{\cdot}^{g_{\bomega_{0}}}_{\pi^{-1}\paren*{\cS}} $ on $\set*{2\epsilon_{i}\leq \abs*{\cdot}^{g_{\bomega_{0}}}_{\pi^{-1}\paren*{\cS}} \leq \frac{1}{8}}$, for sufficiently large $i$ we have that $2\epsilon_{i} \leq \inf_{i}\epsilon_{i} \wtilde{\rho_{0}}\paren*{p_{i}}$ and so we get $$0<\wtilde{C_{16}} \leq \rho_{\epsilon_{i}}\paren*{p_{i}}^{-\delta}\abs*{ \restr{\fa_{i}}{p_{i}}} \leq \Abs*{\fa_{i}}_{C^{0}_{\delta, g_{\bomega_{\epsilon_{i}}}} \paren*{\set*{\inf_{i}\epsilon_{i} \wtilde{\rho_{0}}\paren*{p_{i}} \leq \abs*{\cdot}^{g_{\bomega_{0}}}_{\pi^{-1}\paren*{\cS}} \leq \frac{1}{8}}}}$$ However, \textbf{Case 1} establishes that $\Abs*{\fa_{i}}_{C^{0}_{\delta, g_{\bomega_{\epsilon_{i}}}} \paren*{K}} \searrow 0$ holds for all compact subsets $K \Subset \set*{0 < \abs*{\cdot}^{g_{\bomega_{0}}}_{\pi^{-1}\paren*{\cS}}} \subset \Km_{\epsilon_{i}}$, in particular for $K = \set*{\inf_{i}\epsilon_{i} \wtilde{\rho_{0}}\paren*{p_{i}} \leq \abs*{\cdot}^{g_{\bomega_{0}}}_{\pi^{-1}\paren*{\cS}} \leq \frac{1}{8}}$, contradiction.

Ergo $\sup_{i}\frac{1}{\epsilon_{i} \wtilde{\rho_{0}}\paren*{p_{i}}} = \infty$.

Now consider the conformally rescaled Riemannian manifolds $\paren*{\set*{\abs*{\cdot}^{g_{\bomega_{0}}}_{S^{2}} \leq \frac{1}{8\epsilon_{i}}}, \frac{1}{\wtilde{\rho_{0}}\paren*{p_{i}}^{2}}\lwhat{g_{\bomega_{EH-0, \epsilon_{i}}}}}$. By Proposition \ref{HKA preglueEH BG interpolation!} we have that $\frac{1}{\wtilde{\rho_{0}}\paren*{p_{i}}^{2}}\lwhat{g_{\bomega_{EH-0, \epsilon_{i}}}} = \begin{cases}
\frac{1}{\wtilde{\rho_{0}}\paren*{p_{i}}^{2}}g_{\bomega_{EH,1}} & \text{on } \set*{\abs*{\cdot}^{g_{\bomega_{0}}}_{S^{2}} \leq \frac{1}{\epsilon_{i}^{\lambda}} } \\
\frac{1}{\wtilde{\rho_{0}}\paren*{p_{i}}^{2}}g_{\bomega_{0}} & \text{on }\set*{\frac{2}{\epsilon_{i}^{\lambda}} \leq \abs*{\cdot}^{g_{\bomega_{0}}}_{S^{2}}}
\end{cases}$. Upon pulling back by the scaling $R_{\wtilde{\rho_{0}}\paren*{p_{i}}}$, using the scaling of the Euclidean metric $R_{s}^{*}g_{\bomega_{0}} = s^{2}g_{\bomega_{0}}$ as well as $R_{s}^{*} \paren*{ \frac{1}{s^{2}} g_{\bomega_{EH,1}}} = g_{\bomega_{EH,\frac{1}{s^{2}}}}$ from Proposition \ref{HKA EHproperties}, we have that $\paren*{\set*{\abs*{\cdot}^{g_{\bomega_{0}}}_{S^{2}} \leq \frac{1}{8\epsilon_{i}}}, \frac{1}{\wtilde{\rho_{0}}\paren*{p_{i}}^{2}}\lwhat{g_{\bomega_{EH-0, \epsilon_{i}}}}}$ is isometric under $R_{\wtilde{\rho_{0}}\paren*{p_{i}}}$ to $\paren*{\set*{\abs*{\cdot}^{g_{\bomega_{0}}}_{S^{2}} \leq \frac{1}{8\epsilon_{i}\wtilde{\rho_{0}}\paren*{p_{i}}}}, \frac{1}{\wtilde{\rho_{0}}\paren*{p_{i}}^{2}}R^{*}_{\wtilde{\rho_{0}}\paren*{p_{i}}}\paren*{\lwhat{g_{\bomega_{EH-0, \epsilon_{i}}}}}}$ where $ \frac{1}{\wtilde{\rho_{0}}\paren*{p_{i}}^{2}}R^{*}_{\wtilde{\rho_{0}}\paren*{p_{i}}}\paren*{\lwhat{g_{\bomega_{EH-0, \epsilon_{i}}}}} = \begin{cases}
g_{\bomega_{EH, \frac{1}{\wtilde{\rho_{0}}\paren*{p_{i}}^{2}}}} & \text{on } \set*{\abs*{\cdot}^{g_{\bomega_{0}}}_{S^{2}} \leq \frac{1}{\epsilon_{i}^{\lambda}\wtilde{\rho_{0}}\paren*{p_{i}}} } \\
g_{0} & \text{on }\set*{\frac{2}{\epsilon_{i}^{\lambda}\wtilde{\rho_{0}}\paren*{p_{i}}} \leq \abs*{\cdot}^{g_{\bomega_{0}}}_{S^{2}}}
\end{cases}$.

Thus, as $\set*{\wtilde{\rho_{0}}\paren*{p_{i}}}\subset\bbR_{>0}$ is a diverging sequence of real numbers, as $i\nearrow \infty$ the rescaled Eguchi-Hanson metric $g_{\bomega_{EH, \frac{1}{\wtilde{\rho_{0}}\paren*{p_{i}}^{2}}}}$ gets blown down to the flat metric $g_{\bomega_{0}}$ on $\bbR^{4}/\bbZ_{2}$ since the diameter of the $S^{2}$ bolt shrinks (Proposition \ref{HKA EHproperties}). Ergo by Proposition \ref{HKA preglueEH BG interpolation!} and $\sup_{i}\frac{1}{\epsilon_{i} \wtilde{\rho_{0}}\paren*{p_{i}}} = \infty$, in the limit as $i\nearrow\infty$ we have that $$\paren*{\set*{\abs*{\cdot}^{g_{\bomega_{0}}}_{S^{2}} \leq \frac{1}{8\epsilon_{i}\wtilde{\rho_{0}}\paren*{p_{i}}}}, \frac{1}{\wtilde{\rho_{0}}\paren*{p_{i}}^{2}}R^{*}_{\wtilde{\rho_{0}}\paren*{p_{i}}}\paren*{\lwhat{g_{\bomega_{EH-0, \epsilon_{i}}}}}} \xrightarrow{C^{\infty}_{g_{\bomega_{0}},loc}} \paren*{\bbR^{4}/\bbZ_{2} - \set*{0}, g_{0}}$$ That is, $\paren*{\bbR^{4}/\bbZ_{2}, g_{\bomega_{0}}}$ is the \textbf{tangent cone at infinity} to the (modified) Eguchi-Hanson metric $\paren*{T^{*}S^{2}, \lwhat{g_{\bomega_{EH-0, \epsilon_{i}}}}}$.

Now by Case (\ref{HKA scalingproperties weighted H\"{o}lder tensor norm rescale}) of Proposition \ref{HKA scalingproperties} we have upon transferring under the isometry $R_{\wtilde{\rho_{0}}\paren*{p_{i}}}$ that $$\Abs*{\wtilde{\fa_{i}}}_{C^{k,\alpha}_{\delta, \lwhat{g_{\bomega_{EH-0, \epsilon_{i}}}}} \paren*{\set*{\abs*{\cdot}^{g_{\bomega_{0}}}_{S^{2}} \leq \frac{1}{8\epsilon_{i}}}}} = \Abs*{\wtilde{\rho_{0}}\paren*{p_{i}}^{-\delta}\wtilde{\fa_{i}}}_{C^{k,\alpha}_{\delta, \frac{1}{\wtilde{\rho_{0}}\paren*{p_{i}}^{2}}\lwhat{g_{\bomega_{EH-0, \epsilon_{i}}}}} \paren*{\set*{\abs*{\cdot}^{g_{\bomega_{0}}}_{S^{2}} \leq \frac{1}{8\epsilon_{i}\wtilde{\rho_{0}}\paren*{p_{i}}}}}}$$ where the weight function on the RHS is $R_{\wtilde{\rho_{0}}\paren*{p_{i}}}^{*}\paren*{\frac{1}{\wtilde{\rho_{0}}\paren*{p_{i}}}\wtilde{\rho_{\epsilon_{i}}}} = \begin{cases}
\frac{1}{\epsilon_{i}\wtilde{\rho_{0}}\paren*{p_{i}}} & \text{ on }\set*{\frac{1}{5\epsilon_{i}\wtilde{\rho_{0}}\paren*{p_{i}}} \leq \abs*{\cdot}^{g_{\bomega_{0}}}_{S^{2}}}\\
\abs*{\cdot}^{g_{\bomega_{0}}}_{S^{2}} & \text{ on }\set*{\frac{2}{\wtilde{\rho_{0}}\paren*{p_{i}}}\leq \abs*{\cdot}^{g_{\bomega_{0}}}_{S^{2}} \leq \frac{1}{8\epsilon_{i}\wtilde{\rho_{0}}\paren*{p_{i}}}}\\
\frac{1}{\wtilde{\rho_{0}}\paren*{p_{i}}} & \text{ on }\set*{\abs*{\cdot}^{g_{\bomega_{0}}}_{S^{2}} \leq \frac{1}{\wtilde{\rho_{0}}\paren*{p_{i}}}}
\end{cases}$. Therefore transferring over $\Abs*{\wtilde{\fa_{i}}}_{C^{1,\alpha}_{\delta, \lwhat{g_{\bomega_{EH-0,\epsilon_{i}}}}} \paren*{\set*{\abs*{\cdot}^{g_{\bomega_{0}}}_{S^{2}} \leq \frac{1}{8\epsilon_{i}}}}} \leq 2\cC_{2}$ and $\Abs*{\Delta_{\lwhat{g_{\bomega_{EH-0,\epsilon_{i}}}}}\fa_{i}}_{C^{0,\alpha}_{\delta - 2, \lwhat{g_{\bomega_{EH-0,\epsilon_{i}}}}} \paren*{\set*{\abs*{\cdot}^{g_{\bomega_{0}}}_{S^{2}} \leq \frac{1}{8\epsilon_{i}}}}} \leq \frac{1}{i}$ under this identification, an application of Arzel\`{a}-Ascoli yields \begin{gat}
\wtilde{\rho_{0}}\paren*{p_{i}}^{-\delta}\wtilde{\fa_{i}} \xrightarrow{C^{1,\alpha'}_{g_{\bomega_{0}},loc}} \overline{\fa_{\infty}}\\
\Abs*{\overline{\fa_{\infty}}}_{C^{1,\alpha}_{\delta,g_{\bomega_{0}}}\paren*{\bbR^{4}/\bbZ_{2} - \set*{0}}} \leq 2\cC_{2}\\
\Delta_{g_{\bomega_{0}}}\overline{\fa_{\infty}} = 0
\end{gat} for $0 < \alpha' < \alpha < 1$ and where the limit weight function is simply the radial weight function $\abs*{\cdot}^{g_{\bomega_{0}}}_{0}$ on $\bbR^{4}/\bbZ_{2}$.

Ergo $\overline{\fa_{\infty}}$ is a harmonic \textit{1-form} on $\bbR^{4}/\bbZ_{2}- \set*{0}$ that satisfies $\overline{\fa_{\infty}} = O\paren*{\paren*{\abs*{\cdot}^{g_{\bomega_{0}}}_{0}}^{\delta}}$. The same exact argument as in \textbf{Case 1} where we lift $\overline{\fa_{\infty}}$ along $\bbR^{4} - \set*{0} \overset{2:1}{\onto} \bbR^{4}/\bbZ_{2}- \set*{0}$ and then decompose it using the global parallel coordinate coframe on $\bbR^{4}$ into a 4-tuple of harmonic functions tells us that the 4-tuple of harmonic functions on $\bbR^{4} - \set*{0}$ satisfy the decay $O\paren*{\paren*{\abs*{\cdot}^{g_{\bomega_{0}}}_{0}}^{\delta}}$. But since $-2 < \delta < 0$ and harmonic functions on $\bbR^{4} - \set*{0}$ cannot decay slower than the \textbf{Green's function}, such harmonic functions must \textit{vanish} on $\bbR^{4} - \set*{0}$. Ergo $\overline{\fa_{\infty}} = 0$.

However, upon transferring $p_{i} \in \set*{\abs*{\cdot}^{g_{\bomega_{0}}}_{S^{2}} \leq \frac{1}{8\epsilon_{i}}}$ to $p_{i} \in \set*{\abs*{\cdot}^{g_{\bomega_{0}}}_{S^{2}} \leq \frac{1}{8\epsilon_{i}\wtilde{\rho_{0}}\paren*{p_{i}}}}$, we get that $p_{i} \in \set*{\abs*{\cdot}^{g_{\bomega_{0}}}_{S^{2}} = 1}$, whence in the limit $p_{\infty} \in \set*{\abs*{\cdot}^{g_{\bomega_{0}}}_{0} = 1} \Subset \bbR^{4}/\bbZ_{2} - \set*{0}$ lands in a compact subset. Thus taking $i\nearrow\infty$ of $0 < \wtilde{C_{6}} \leq \abs*{\wtilde{\rho_{0}}\paren*{p_{i}}^{-\delta} \wtilde{\fa_{i}}\paren*{p_{i}}}$ yields $$0 < \wtilde{C_{6}} \leq \abs*{\overline{\fa_{\infty}}\paren*{p_{\infty}}}$$ contradicting $\overline{\fa_{\infty}} = 0$.\end{enumerate}Whence we have reached a contradiction to which the proof follows, as was to be shown.\end{proof}

With these in hand, we may finally prove

\begin{theorem}\label{HKA biginverse} For $\lambda< \frac{2}{3}$, $\epsilon > 0$ from Data \ref{constant epsilon 1-16} sufficiently small, and $\boxed{\delta \in \paren*{-1,0}}$, the following is a bounded linear isomorphism:
$$D_{0}\Phi_{\epsilon}:\paren*{C^{1,\alpha}_{\delta,g_{\bomega_{\epsilon}}}\paren*{\mathring{\Omega}^{1}_{g_{\bomega_{\epsilon}}}\paren*{\Km_{\epsilon}}}\oplus \cH^{+}_{g_{\bomega_{\epsilon}}}} \otimes \bbR^{3} \rightarrow C^{0,\alpha}_{\delta - 1,g_{\bomega_{\epsilon}}}\paren*{\Lambda_{g_{\bomega_{\epsilon}}}^{+}T^{*}\Km_{\epsilon}}\otimes\bbR^{3}$$ where $D_{0}\Phi_{\epsilon}= \paren*{d_{g_{\bomega_{\epsilon}}}^{+} \oplus \id_{\cH^{+}_{g_{\bomega_{\epsilon}}}}}\otimes \bbR^{3}$.

More crucially, we have that the operator norm of the inverse is bounded by $\cL> 0$ a constant independent of $\epsilon > 0$:
$$ \Abs*{\paren*{D_{0}\Phi_{\epsilon}}^{-1}} \leq \cL$$ That is, under these conditions, $$\forall \bm{t} \in C^{0,\alpha}_{\delta - 1,g_{\bomega_{\epsilon}}}\paren*{\Lambda_{g_{\bomega_{\epsilon}}}^{+}T^{*}\Km_{\epsilon}}\otimes\bbR^{3}$$ there exists a unique pair $$\paren*{\bm{a},\bm{\zeta}} \in \paren*{C^{1,\alpha}_{\delta,g_{\bomega_{\epsilon}}}\paren*{\mathring{\Omega}^{1}_{g_{\bomega_{\epsilon}}}\paren*{\Km_{\epsilon}}}\oplus \cH^{+}_{g_{\bomega_{\epsilon}}}} \otimes \bbR^{3}$$ such that $$D_{0}\Phi_{\epsilon}\paren*{\bm{a},\bm{\zeta}} = d^{+}_{g_{\bomega_{\epsilon}}}\bm{a} + \bm{\zeta} = \bm{t}$$ and $$ \Abs*{\bm{a}}_{C^{1,\alpha}_{\delta,g_{\bomega_{\epsilon}}}\paren*{\mathring{\Omega}^{1}_{g_{\bomega_{\epsilon}}}\paren*{\Km_{\epsilon}}} \otimes \bbR^{3}} + \Abs*{\bm{\zeta}}_{\cH^{+}_{g_{\bomega_{\epsilon}}} \otimes \bbR^{3}} \leq \cL \Abs*{\bm{t}}_{C^{0,\alpha}_{\delta - 1,g_{\bomega_{\epsilon}}}\paren*{\Lambda_{g_{\bomega_{\epsilon}}}^{+}T^{*}\Km_{\epsilon}}\otimes\bbR^{3}} $$\end{theorem} \begin{proof}
Firstly, the existence of a unique $\paren*{\bm{a},\bm{\zeta}}$ for any $\bm{t}$ is clear since the linearization $D_{0}\Phi_{\epsilon}$ is always an isomorphism hence bijective. We now just need to show the uniform operator norm estimate for $\paren*{D_{0}\Phi_{\epsilon}}^{-1}$. WLOG we prove this on the level of forms, whence by \textit{Remark} \ref{HKA weighted tuple notation} the estimate follows on $3$-tuples (up to changing the constant).

So suppose we do not have a uniform constant $\cL_{1} > 0$ such that $$\Abs*{\fa}_{C^{1,\alpha}_{\delta,g_{\bomega_{\epsilon}}}\paren*{\mathring{\Omega}^{1}_{g_{\bomega_{\epsilon}}}\paren*{\Km_{\epsilon}}}} + \Abs*{\zeta}_{L^{2}_{g_{\bomega_{\epsilon}}}} \leq \cL_{1} \Abs*{\ft}_{C^{0,\alpha}_{\delta - 1,g_{\bomega_{\epsilon}}}\paren*{\Lambda_{g_{\bomega_{\epsilon}}}^{+}T^{*}\Km_{\epsilon}}} $$ for all $\ft$ such that $d^{+}_{g_{\bomega_{\epsilon}}} \fa + \zeta = \ft$. Then this implies by countable choice that $\exists \epsilon_{i} \in \paren*{0,\frac{1}{i}}$ whence $\exists \epsilon_{i} \searrow 0$ and $\exists \paren*{\fa_{i}, \zeta_{i}} \in C^{1,\alpha}_{\delta,g_{\bomega_{\epsilon_{i}}}}\paren*{\mathring{T^{*}\Km_{\epsilon_{i}}}}\oplus \cH^{+}_{g_{\bomega_{\epsilon_{i}}}}$ a sequence such that (after normalizing) we get a tuple of equations \begin{al}
\Abs*{\fa_{i}}_{C^{1,\alpha}_{\delta,g_{\bomega_{\epsilon_{i}}}}\paren*{\mathring{\Omega}^{1}_{g_{\bomega_{\epsilon_{i}}}}\paren*{\Km_{\epsilon_{i}}}}} + \Abs*{\zeta_{i}}_{L^{2}_{g_{\bomega_{\epsilon_{i}}}}}  &= 1\\
\Abs*{  d^{+}_{g_{\bomega_{\epsilon_{i}}}}\fa_{i} + \zeta_{i}  }_{C^{0,\alpha}_{\delta - 1, g_{\bomega_{\epsilon_{i}}}} \paren*{\Lambda_{g_{\bomega_{\epsilon_{i}}}}^{+}T^{*}\Km_{\epsilon_{i}}}} &\leq \frac{1}{i}
\end{al} Thus as $i\nearrow \infty$ we clearly have that $$\Abs*{  d^{+}_{g_{\bomega_{\epsilon_{i}}}}\fa_{i} + \zeta_{i}  }_{C^{0,\alpha}_{\delta - 1, g_{\bomega_{\epsilon_{i}}}} \paren*{\Lambda_{g_{\bomega_{\epsilon_{i}}}}^{+}T^{*}\Km_{\epsilon_{i}}}} \searrow 0$$

Now since each $\Km_{\epsilon_{i}}$ is \textit{closed}, being harmonic is equivalent to being closed \& co-closed by integration by parts. Whence using integration by parts and the fact that each $\zeta_{i}$ is both self-dual and harmonic, we have that $\Inner*{d^{+}_{g_{\bomega_{\epsilon_{i}}}}\fa_{i} + \zeta_{i}, \zeta_{i}}_{L^{2}_{g_{\bomega_{\epsilon_{i}}}}\paren*{\Km_{\epsilon_{i}}}} \coloneq \int_{\Km_{\epsilon_{i}}} \paren*{d^{+}_{g_{\bomega_{\epsilon_{i}}}}\fa_{i} + \zeta_{i}}\wedge *_{g_{\bomega_{\epsilon_{i}}}} \zeta_{i} = \int_{\Km_{\epsilon_{i}}} \paren*{d^{+}_{g_{\bomega_{\epsilon_{i}}}}\fa_{i} + \zeta_{i}}\wedge \zeta_{i} = \Abs*{\zeta_{i}}^{2}_{L^{2}_{g_{\bomega_{\epsilon_{i}}}}\paren*{\Km_{\epsilon_{i}}}}$. Now since $\Abs*{  d^{+}_{g_{\bomega_{\epsilon_{i}}}}\fa_{i} + \zeta_{i}  }_{C^{0,\alpha}_{\delta - 1, g_{\bomega_{\epsilon_{i}}}} \paren*{\Lambda_{g_{\bomega_{\epsilon_{i}}}}^{+}T^{*}\Km_{\epsilon_{i}}}} \leq \frac{1}{i}$, recalling the definition of the $\Abs*{\cdot}_{C^{0,\alpha}_{\delta - 1, g_{\bomega_{\epsilon_{i}}}}}$-norm we have that $$\abs*{d^{+}_{g_{\bomega_{\epsilon_{i}}}}\fa_{i} + \zeta_{i}}_{g_{\bomega_{\epsilon_{i}}}} \leq C_{17} \frac{1}{i} \rho_{\epsilon}^{\delta-1}$$

Therefore by Cauchy-Schwartz applied to the \textit{pointwise} Hodge-inner product and H\"{o}lder's inequality for $L^{2}_{g}$, \begin{al}
0\leq \Abs*{\zeta_{i}}^{2}_{L^{2}_{g_{\bomega_{\epsilon_{i}}}}\paren*{\Km_{\epsilon_{i}}}} &= \Inner*{d^{+}_{g_{\bomega_{\epsilon_{i}}}}\fa_{i} + \zeta_{i}, \zeta_{i}}_{L^{2}_{g_{\bomega_{\epsilon_{i}}}}\paren*{\Km_{\epsilon_{i}}}}\\
&= \int_{\Km_{\epsilon_{i}}} \paren*{d^{+}_{g_{\bomega_{\epsilon_{i}}}}\fa_{i} + \zeta_{i}}\wedge \zeta_{i}\\
&= \int_{\Km_{\epsilon_{i}}} \paren*{d^{+}_{g_{\bomega_{\epsilon_{i}}}}\fa_{i} + \zeta_{i}, \zeta_{i}}_{g_{\bomega_{\epsilon}}} dV_{g_{\bomega_{\epsilon_{i}}}} \\
&\leq C_{17}\frac{1}{i} \int_{\Km_{\epsilon_{i}}} \rho_{\epsilon_{i}}^{\delta-1}\abs*{\zeta_{i}}_{g_{\bomega_{\epsilon_{i}}}} dV_{g_{\bomega_{\epsilon_{i}}}}\\
&\leq C_{17}\frac{1}{i} \Abs*{\zeta_{i}}_{L^{2}_{g_{\bomega_{\epsilon_{i}}}}\paren*{\Km_{\epsilon_{i}}}} \paren*{\int_{\Km_{\epsilon_{i}}} \rho_{\epsilon_{i}}^{2\delta - 2} dV_{g_{\bomega_{\epsilon_{i}}}}}^{\frac{1}{2}}\\
&= C_{17}\frac{1}{i} \Abs*{\zeta_{i}}_{L^{2}_{g_{\bomega_{\epsilon_{i}}}}\paren*{\Km_{\epsilon_{i}}}} \Abs*{\rho_{\epsilon_{i}}^{\delta - 1}}_{L^{2}_{g_{\bomega_{\epsilon_{i}}}}\paren*{\Km_{\epsilon_{i}}}}
\end{al} From the volume-noncollapsed convergence $\paren*{\Km_{\epsilon_{i}}, g_{\bomega_{\epsilon_{i}}}} \xrightarrow{GH} \paren*{\bbT^{4}/\bbZ_{2}, g_{\bomega_{0}}}$ (\textit{Remark} \ref{HKA GHremark}) and directly inspecting $\dV_{g_{\bomega_{\epsilon_{i}}}} = \dV_{g_{\bomega_{0}}} + O_{g_{\bomega_{0}}}\paren*{\epsilon_{i}^{2}}$ on $\set*{\epsilon_{i}^{\lambda} \leq \abs*{\cdot}^{g_{\bomega_{0}}}_{\pi^{-1}\paren*{\cS}}}$, we have that $$\Abs*{\rho_{\epsilon_{i}}^{\delta - 1}}^{2}_{L^{2}_{g_{\bomega_{\epsilon_{i}}}}\paren*{\Km_{\epsilon_{i}}}} = \int_{\Km_{\epsilon_{i}}} \rho_{\epsilon_{i}}^{2\delta - 2} \dV_{g_{\bomega_{\epsilon_{i}}}} \rightarrow \int_{\bbT^{4}/\bbZ_{2}} \rho_{0}^{2\delta - 2}\dV_{g_{\bomega_{0}}} = \int_{\bbT^{4}/\bbZ_{2} - \cS} \rho_{0}^{2\delta - 2}\dV_{g_{\bomega_{0}}} = \frac{1}{2}\int_{\bbT^{4} - \cS} \rho_{0}^{2\delta - 2}\dV_{g_{\bomega_{0}}}$$ since $\cS$ has measure zero and $\lim_{i\nearrow\infty}\rho_{\epsilon_{i}} = \rho_{0}$ the limiting weight function is the weight function used to define the weighted H\"{o}lder spaces on $\bbT^{4}$ and is manifestly even whence descends down $\bbT^{4} -\cS\overset{2:1}{\onto} \bbT^{4}/\bbZ_{2}- \cS$.

Thus since $2\delta - 2 < 0 $ because $\delta < 0$ and since $\rho_{0} = \begin{cases}
1 & \text{ on }\set*{\frac{1}{5}\leq \abs*{\cdot}^{g_{\bomega_{0}}}_{\cS}}\\
\abs*{\cdot}^{g_{\bomega_{0}}}_{\cS} & \text{ on }\set*{\abs*{\cdot}^{g_{\bomega_{0}}}_{\cS} \leq \frac{1}{8}}
\end{cases}$, we need to worry about controlling $\rho_{0}^{2\delta - 2}$ near $\cS$. Thus on $\set*{\abs*{\cdot}^{g_{\bomega_{0}}}_{\cS} \leq \frac{1}{8}}$, that upon using exponential polar coordinates WRT $g_{\bomega_{0}}$ centered around each point $p \in \cS$ we have that $\rho_{0} = \abs*{\cdot}^{g_{\bomega_{0}}}_{p}$ is the radial distance function of those polar coordinates, and whence \begin{al}
\int_{\set*{\abs*{\cdot}^{g_{\bomega_{0}}}_{\cS} \leq \frac{1}{8}}} \rho_{0}^{2\delta - 2} \dV_{g_{0}} &= 16\int_{\set*{\abs*{\cdot}^{g_{\bomega_{0}}}_{p} \leq \frac{1}{8}}} \paren*{\abs*{\cdot}^{g_{\bomega_{0}}}_{p}}^{2\delta - 2}\dV_{g_{0}}\\
&=  16 \int_{0}^{\frac{1}{8}} r^{2\delta - 2}r^{4 - 1} dr \int_{S^{3}}d\theta\\
&= O\paren*{1} \int_{0}^{\frac{1}{8}} r^{2\delta + 1} dr= O\paren*{1} \Eval{\frac{r^{2\delta + 2}}{2\delta + 2}}{0}{\frac{1}{8}}= O\paren*{1} < \infty
\end{al} precisely because $0 < 2\delta + 2 \Longleftrightarrow -1 < \delta$. Whence we have that $\int_{\bbT^{4} - \cS} \rho_{0}^{2\delta - 2}\dV_{g_{\bomega_{0}}} = O\paren*{1}$, whence since we have thus shown that $$\Abs*{\rho_{\epsilon_{i}}^{\delta - 1}}^{2}_{L^{2}_{g_{\bomega_{\epsilon_{i}}}}\paren*{\Km_{\epsilon_{i}}}} \rightarrow \frac{1}{2} \Abs*{\rho_{0}^{\delta - 1}}^{2}_{L^{2}_{g_{\bomega_{0}}}\paren*{\bbT^{4} - \cS}} < \infty$$ that is, the sequence of real numbers $\set*{\Abs*{\rho_{\epsilon_{i}}^{\delta - 1}}^{2}_{L^{2}_{g_{\bomega_{\epsilon_{i}}}}\paren*{\Km_{\epsilon_{i}}}}}_{i} \subset \bbR$ is convergent, we therefore have that $$\Abs*{\rho_{\epsilon_{i}}^{\delta - 1}}^{2}_{L^{2}_{g_{\bomega_{\epsilon_{i}}}}\paren*{\Km_{\epsilon_{i}}}} \leq \cM$$ is \textit{uniformly} bounded above as convergent sequences are bounded. Whence $$ \Abs*{\zeta_{i}}_{L^{2}_{g_{\bomega_{\epsilon_{i}}}}\paren*{\Km_{\epsilon_{i}}}}\leq C_{17}\frac{1}{i} \cM^{\frac{1}{2}}$$ whence $$\Abs*{\zeta_{i}}_{L^{2}_{g_{\bomega_{\epsilon_{i}}}}\paren*{\Km_{\epsilon_{i}}}} \searrow 0$$ as $i\nearrow \infty$.

Next, recall that our $3$-dimensional space of self-dual harmonic $2$-forms is $\cH^{+}_{g_{\bomega_{\epsilon_{i}}}} = \spano_{\bbR}\paren*{\bomega_{\epsilon_{i}}} = \spano_{\bbR}\set*{\omega^{\epsilon_{i}}_{1}, \omega^{\epsilon_{i}}_{2}, \omega^{\epsilon_{i}}_{3} }$. Now I claim that $\omega^{\epsilon_{i}}_{1}, \omega^{\epsilon_{i}}_{2}, \omega^{\epsilon_{i}}_{3} $ each have \textit{uniformly} bounded $C^{0,\alpha}_{\delta-1, g_{\bomega_{\epsilon_{i}}}}$-norm, where the bound is \textit{uniform in $i$}. Indeed, since $Q_{\bomega_{\epsilon_{i}}} = \id$ on $\set*{\abs*{\cdot}^{g_{\bomega_{0}}}_{\pi^{-1}\paren*{\cS}} \leq \epsilon_{i}^{\lambda}} \sqcup \set*{2\epsilon_{i}^{\lambda} \leq \abs*{\cdot}^{g_{\bomega_{0}}}_{\pi^{-1}\paren*{\cS}}}$, we have that $\bomega_{\epsilon_{i}} = \paren*{\omega^{\epsilon_{i}}_{1}, \omega^{\epsilon_{i}}_{2}, \omega^{\epsilon_{i}}_{3}}$ is \hka and thus each $\omega^{\epsilon_{i}}_{1}, \omega^{\epsilon_{i}}_{2}, \omega^{\epsilon_{i}}_{3}$ is parallel WRT $\nabla_{g_{\bomega_{\epsilon_{i}}}}$ and bounded WRT $g_{\bomega_{\epsilon_{i}}}$. Moreover we have from Proposition \ref{HKA approximatemetricproperties} that $\abs*{\bomega_{\epsilon_{i}} - \bomega_{0}}_{g_{\bomega_{0}}} \leq c_{1}(2) \epsilon_{i}^{4-4\lambda}$ aka componentwise $\bomega_{\epsilon_{i}} = \bomega_{0} + O_{g_{\bomega_{0}}}\paren*{\epsilon_{i}^{4-4\lambda}}$ on $\set*{\epsilon_{i}^{\lambda}\leq \abs*{\cdot}^{g_{\bomega_{0}}}_{\pi^{-1}\paren*{\cS}} \leq 2\epsilon_{i}^{\lambda}} $. Whence using a local ON frame calculation as well as the decay $g_{\bomega_{\epsilon_{i}}} = g_{\bomega_{0}} + O_{g_{\bomega_{0}}}\paren*{\epsilon_{i}^{4-4\lambda}}$ on $\set*{\epsilon_{i}^{\lambda}\leq \abs*{\cdot}^{g_{\bomega_{0}}}_{\pi^{-1}\paren*{\cS}} \leq 2\epsilon_{i}^{\lambda}} $, we compute that (the $3$ components of) $\abs*{\bomega_{\epsilon_{i}}}_{g_{\bomega_{\epsilon_{i}}}}$ is bounded uniformly in $i$ on the annular region $\set*{\epsilon_{i}^{\lambda}\leq \abs*{\cdot}^{g_{\bomega_{0}}}_{\pi^{-1}\paren*{\cS}} \leq 2\epsilon_{i}^{\lambda}}$. Combined with the fact that $0 < -\delta + 1 \Longleftrightarrow \delta < 1$ which we have because $\delta < 0$, we get that from $\epsilon_{i} \leq \rho_{\epsilon_{i}} \leq 1$ that $\rho_{\epsilon_{i}}^{-\delta + 1} \leq 1$ is uniformly bounded above, whence we conclude that each $\omega^{\epsilon_{i}}_{1}, \omega^{\epsilon_{i}}_{2}, \omega^{\epsilon_{i}}_{3} $ has \textit{uniformly} bounded $C^{0,\alpha}_{\delta-1, g_{\bomega_{\epsilon_{i}}}}$-norm, i.e. $\bomega_{\epsilon_{i}}$ has \textit{uniformly} bounded $C^{0,\alpha}_{\delta - 1,g_{\bomega_{\epsilon_{i}}}}\paren*{\Lambda_{g_{\bomega_{\epsilon_{i}}}}^{+}T^{*}\Km_{\epsilon_{i}}}\otimes\bbR^{3}$-norm from \textit{Remark} \ref{HKA weighted tuple notation}.

Next up, for $m,n \in \set*{1,2,3}$ we have by definition of the associated intersection matrix $\paren*{Q_{\bomega_{\epsilon_{i}}}}_{mn} = \frac{1}{2}\omega^{\epsilon_{i}}_{m}\wedge  \omega^{\epsilon_{i}}_{n}$ and from Proposition \ref{HKA approximatemetricproperties} the fact that $Q_{\bomega_{\epsilon_{i}}} = \id + O\paren*{\epsilon_{i}^{4-4\lambda}}$ on $\set*{\epsilon_{i}^{\lambda} \leq \abs*{\cdot}^{g_{\bomega_{0}}}_{\pi^{-1}\paren*{\cS}} \leq 2\epsilon_{i}^{\lambda}}$, that $\dV_{g_{\bomega_{\epsilon_{i}}}} = \paren*{1+O\paren*{\epsilon_{i}^{4-4\lambda}}}\dV_{g_{\bomega_{0}}}$ on $\set*{\epsilon_{i}^{\lambda} \leq \abs*{\cdot}^{g_{\bomega_{0}}}_{\pi^{-1}\paren*{\cS}}}$ (with straight up equality with $\dV_{g_{\bomega_{0}}}$ on $\set*{2\epsilon_{i}^{\lambda} \leq \abs*{\cdot}^{g_{\bomega_{0}}}_{\pi^{-1}\paren*{\cS}}}$), that \begin{al}\Inner*{\omega^{\epsilon_{i}}_{m}, \omega^{\epsilon_{i}}_{n}}_{L^{2}_{g_{\bomega_{\epsilon_{i}}}}} &= \int_{\Km_{\epsilon_{i}}} \omega^{\epsilon_{i}}_{m}\wedge *_{g_{\bomega_{\epsilon_{i}}}} \omega^{\epsilon_{i}}_{n} \\
&= \int_{\Km_{\epsilon_{i}}} \omega^{\epsilon_{i}}_{m}\wedge  \omega^{\epsilon_{i}}_{n} \\
&= \int_{\set*{\abs*{\cdot}^{g_{\bomega_{0}}}_{\pi^{-1}\paren*{\cS}} \leq \epsilon_{i}^{\lambda}}}\omega^{\epsilon_{i}}_{m}\wedge  \omega^{\epsilon_{i}}_{n} + \int_{\set*{\epsilon_{i}^{\lambda} \leq \abs*{\cdot}^{g_{\bomega_{0}}}_{\pi^{-1}\paren*{\cS}} \leq 2\epsilon_{i}^{\lambda}}}\omega^{\epsilon_{i}}_{m}\wedge  \omega^{\epsilon_{i}}_{n} + \int_{\set*{2\epsilon_{i}^{\lambda} \leq \abs*{\cdot}^{g_{\bomega_{0}}}_{\pi^{-1}\paren*{\cS}}}}\omega^{\epsilon_{i}}_{m}\wedge  \omega^{\epsilon_{i}}_{n} \\ 
&=  2\int_{\set*{\abs*{\cdot}^{g_{\bomega_{0}}}_{\pi^{-1}\paren*{\cS}} \leq \epsilon_{i}^{\lambda}}}\delta_{mn}\dV_{g_{\bomega_{\epsilon_{i}}}} + 2\int_{\set*{\epsilon_{i}^{\lambda} \leq \abs*{\cdot}^{g_{\bomega_{0}}}_{\pi^{-1}\paren*{\cS}} \leq 2\epsilon_{i}^{\lambda}}}\paren*{\delta_{mn}+ O\paren*{\epsilon_{i}^{4-4\lambda}}}\dV_{g_{\bomega_{\epsilon_{i}}}} + 2\int_{\set*{2\epsilon_{i}^{\lambda} \leq \abs*{\cdot}^{g_{\bomega_{0}}}_{\pi^{-1}\paren*{\cS}}}}\delta_{mn}\dV_{g_{\bomega_{\epsilon_{i}}}}\\ 
&= 2\delta_{mn} \Vol_{g_{\bomega_{\epsilon_{i}}}}\paren*{\Km_{\epsilon_{i}}} + 2\int_{\set*{\epsilon_{i}^{\lambda} \leq \abs*{\cdot}^{g_{\bomega_{0}}}_{\pi^{-1}\paren*{\cS}} \leq 2\epsilon_{i}^{\lambda}}} O\paren*{\epsilon_{i}^{4-4\lambda}}\paren*{1+O\paren*{\epsilon_{i}^{4-4\lambda}}}\dV_{g_{\bomega_{0}}} \\ &= 2\delta_{mn} \Vol_{g_{\bomega_{\epsilon_{i}}}}\paren*{\Km_{\epsilon_{i}}} + O\paren*{\epsilon_{i}^{4-4\lambda}}\Vol_{g_{\bomega_{0}}}\paren*{\set*{\epsilon_{i}^{\lambda} \leq \abs*{\cdot}^{g_{\bomega_{0}}}_{\pi^{-1}\paren*{\cS}} \leq 2\epsilon_{i}^{\lambda}}} \\
&= 2\delta_{mn} \Vol_{g_{\bomega_{\epsilon_{i}}}}\paren*{\Km_{\epsilon_{i}}} + O\paren*{\epsilon_{i}^{4}} < \infty \end{al} and where this bound is \textit{uniform in $i$} since $\Km_{\epsilon_{i}}$ is volume non-collapsed and $\Vol_{g_{\bomega_{\epsilon_{i}}}}\paren*{\Km_{\epsilon_{i}}} \to \Vol_{g_{\bomega_{0}}}\paren*{\bbT^{4}/\bbZ_{2}} < \infty$.

Whence we have shown that $$ \Abs*{\omega^{\epsilon_{i}}_{l}}_{C^{0,\alpha}_{\delta - 1, g_{\bomega_{\epsilon_{i}}}} \paren*{\Lambda_{g_{\bomega_{\epsilon_{i}}}}^{+}T^{*}\Km_{\epsilon_{i}}}}, \Inner*{\omega^{\epsilon_{i}}_{m}, \omega^{\epsilon_{i}}_{n}}_{L^{2}_{g_{\bomega_{\epsilon_{i}}}}} \qqfa l, m,n \in \set*{1,2,3}$$ are all \textit{uniformly} bounded in $i$,

Now let us produce via Gram-Schmidt an $L^{2}_{g_{\bomega_{\epsilon_{i}}}}$-orthonormal basis $\overline{\omega^{\epsilon_{i}}_{1}}, \overline{\omega^{\epsilon_{i}}_{2}}, \overline{\omega^{\epsilon_{i}}_{3}}$ of $\cH^{+}_{g_{\bomega_{\epsilon_{i}}}}$ from $\omega^{\epsilon_{i}}_{1}, \omega^{\epsilon_{i}}_{2}, \omega^{\epsilon_{i}}_{3} $. Because each $\overline{\omega^{\epsilon_{i}}_{1}}, \overline{\omega^{\epsilon_{i}}_{2}}, \overline{\omega^{\epsilon_{i}}_{3}}$ consists of linear combinations of $\omega^{\epsilon_{i}}_{1}, \omega^{\epsilon_{i}}_{2}, \omega^{\epsilon_{i}}_{3}$ with coefficients involving $L^{2}_{g_{\bomega_{\epsilon_{i}}}}$-inner products $\Inner*{\omega^{\epsilon_{i}}_{m}, \omega^{\epsilon_{i}}_{n}}_{L^{2}_{g_{\bomega_{\epsilon_{i}}}}}, \forall m,n \in \set*{1,2,3}$, we have that $$\Abs*{  \overline{\omega^{\epsilon_{i}}_{m}}  }_{C^{0,\alpha}_{\delta - 1, g_{\bomega_{\epsilon_{i}}}} \paren*{\Lambda_{g_{\bomega_{\epsilon_{i}}}}^{+}T^{*}\Km_{\epsilon_{i}}}} \qqfa m \in \set*{1,2,3}$$ are all \textit{uniformly} bounded in $i$. Therefore, upon writing out $\zeta_{i} = \lambda^{i}_{1}\overline{\omega^{\epsilon_{i}}_{1}} + \lambda^{i}_{2}\overline{\omega^{\epsilon_{i}}_{2}} + \lambda^{i}_{3} \overline{\omega^{\epsilon_{i}}_{3}}$ where each $\lambda^{i}_{k} = \int_{\Km_{\epsilon_{i}}} \zeta_{i} \wedge \overline{\omega^{\epsilon_{i}}_{k}} \in \bbR$, we have that $$\Abs*{\zeta_{i}}^{2}_{L^{2}_{g_{\bomega_{\epsilon_{i}}}}\paren*{\Km_{\epsilon_{i}}}} = \paren*{\lambda^{i}_{1}}^{2} + \paren*{\lambda^{i}_{2}}^{2} + \paren*{\lambda^{i}_{3}}^{2} $$ whence $$\lambda^{i}_{1}, \lambda^{i}_{2}, \lambda^{i}_{3} \to 0$$ as $i\nearrow \infty$. Whence by the triangle inequality we get \begin{al}
\Abs*{  d^{+}_{g_{\bomega_{\epsilon_{i}}}}\fa_{i}  }_{C^{0,\alpha}_{\delta - 1, g_{\bomega_{\epsilon_{i}}}} \paren*{\Lambda_{g_{\bomega_{\epsilon_{i}}}}^{+}T^{*}\Km_{\epsilon_{i}}}} &= \Abs*{  d^{+}_{g_{\bomega_{\epsilon_{i}}}}\fa_{i} + \zeta_{i} - \zeta_{i} }_{C^{0,\alpha}_{\delta - 1, g_{\bomega_{\epsilon_{i}}}} \paren*{\Lambda_{g_{\bomega_{\epsilon_{i}}}}^{+}T^{*}\Km_{\epsilon_{i}}}} \\
&\leq \Abs*{  d^{+}_{g_{\bomega_{\epsilon_{i}}}}\fa_{i} + \zeta_{i}  }_{C^{0,\alpha}_{\delta - 1, g_{\bomega_{\epsilon_{i}}}} \paren*{\Lambda_{g_{\bomega_{\epsilon_{i}}}}^{+}T^{*}\Km_{\epsilon_{i}}}} + \Abs*{  \zeta_{i}  }_{C^{0,\alpha}_{\delta - 1, g_{\bomega_{\epsilon_{i}}}} \paren*{\Lambda_{g_{\bomega_{\epsilon_{i}}}}^{+}T^{*}\Km_{\epsilon_{i}}}} \\
&\leq \frac{1}{i} + \Abs*{  \zeta_{i}  }_{C^{0,\alpha}_{\delta - 1, g_{\bomega_{\epsilon_{i}}}} \paren*{\Lambda_{g_{\bomega_{\epsilon_{i}}}}^{+}T^{*}\Km_{\epsilon_{i}}}} = \frac{1}{i} + \Abs*{  \lambda^{i}_{1}\overline{\omega^{\epsilon_{i}}_{1}} + \lambda^{i}_{2}\overline{\omega^{\epsilon_{i}}_{2}} + \lambda^{i}_{3} \overline{\omega^{\epsilon_{i}}_{3}}  }_{C^{0,\alpha}_{\delta - 1, g_{\bomega_{\epsilon_{i}}}} \paren*{\Lambda_{g_{\bomega_{\epsilon_{i}}}}^{+}T^{*}\Km_{\epsilon_{i}}}}\\
&\leq \frac{1}{i} + \abs*{\lambda^{i}_{1}} \Abs*{  \overline{\omega^{\epsilon_{i}}_{1}}  }_{C^{0,\alpha}_{\delta - 1, g_{\bomega_{\epsilon_{i}}}} \paren*{\Lambda_{g_{\bomega_{\epsilon_{i}}}}^{+}T^{*}\Km_{\epsilon_{i}}}} + \abs*{\lambda^{i}_{2}} \Abs*{  \overline{\omega^{\epsilon_{i}}_{2}}  }_{C^{0,\alpha}_{\delta - 1, g_{\bomega_{\epsilon_{i}}}} \paren*{\Lambda_{g_{\bomega_{\epsilon_{i}}}}^{+}T^{*}\Km_{\epsilon_{i}}}} \\
& + \abs*{\lambda^{i}_{3}} \Abs*{  \overline{\omega^{\epsilon_{i}}_{3}}  }_{C^{0,\alpha}_{\delta - 1, g_{\bomega_{\epsilon_{i}}}} \paren*{\Lambda_{g_{\bomega_{\epsilon_{i}}}}^{+}T^{*}\Km_{\epsilon_{i}}}}\\
&\leq \frac{1}{i} + O\paren*{\abs*{\lambda^{i}_{1}}} + O\paren*{\abs*{\lambda^{i}_{2}}} + O\paren*{\abs*{\lambda^{i}_{3}}}
\end{al} where we've used the fact that each $\Abs*{  \overline{\omega^{\epsilon_{i}}_{m}}  }_{C^{0,\alpha}_{\delta - 1, g_{\bomega_{\epsilon_{i}}}} \paren*{\Lambda_{g_{\bomega_{\epsilon_{i}}}}^{+}T^{*}\Km_{\epsilon_{i}}}}$ is uniformly bounded in $i$ to absorb those terms into $O\paren*{\abs*{\lambda^{i}_{m}}} $ as harmless constants.

Therefore since $\Abs*{\fa_{i}}_{C^{1,\alpha}_{\delta,g_{\bomega_{\epsilon_{i}}}}\paren*{\mathring{\Omega}^{1}_{g_{\bomega_{\epsilon_{i}}}}\paren*{\Km_{\epsilon_{i}}}}} = 1 -  \Abs*{\zeta_{i}}_{L^{2}_{g_{\bomega_{\epsilon_{i}}}}}$, we have upon normalizing each $\fa_{i}$ to have unit $C^{1,\alpha}_{\delta,g_{\bomega_{\epsilon_{i}}}}\paren*{\mathring{\Omega}^{1}_{g_{\bomega_{\epsilon_{i}}}}\paren*{\Km_{\epsilon_{i}}}}$-norm since $\Abs*{\zeta_{i}}_{L^{2}_{g_{\bomega_{\epsilon_{i}}}}} \searrow 0$, we have finally shown that \begin{al}
\Abs*{\fa_{i}}_{C^{1,\alpha}_{\delta,g_{\bomega_{\epsilon_{i}}}}\paren*{\mathring{\Omega}^{1}_{g_{\bomega_{\epsilon_{i}}}}\paren*{\Km_{\epsilon_{i}}}}} &= 1\\
\Abs*{  d^{+}_{g_{\bomega_{\epsilon_{i}}}}\fa_{i}  }_{C^{0,\alpha}_{\delta - 1, g_{\bomega_{\epsilon_{i}}}} \paren*{\Lambda_{g_{\bomega_{\epsilon_{i}}}}^{+}T^{*}\Km_{\epsilon_{i}}}} &\searrow 0
\end{al} as $i\nearrow \infty$. But this contradicts Proposition \ref{HKA IWSE} since $\Abs*{\fa_{i}}_{C^{1,\alpha}_{\delta,g_{\bomega_{\epsilon_{i}}}}\paren*{\mathring{\Omega}^{1}_{g_{\bomega_{\epsilon_{i}}}}\paren*{\Km_{\epsilon_{i}}}}} = \Abs*{\fa_{i}}_{C^{1,\alpha}_{\delta,g_{\bomega_{\epsilon_{i}}}}\paren*{T^{*}\Km_{\epsilon_{i}}}}$ with $d^{*}_{g_{\bomega_{\epsilon_{i}}}}\fa_{i} = 0$, whence we have reached a contradiction, as was to be shown.\end{proof}

\subsection{\textsection \ Finishing the Proof \& Discussion \ \textsection}\label{HKA Finishing the proof}

Recall from Section \ref{HKA Nonlinear Setup} our nonlinear problem: \begin{al}
\Phi_{\epsilon}:\paren*{C^{1,\alpha}_{\delta,g_{\bomega_{\epsilon}}}\paren*{\mathring{\Omega}^{1}_{g_{\bomega_{\epsilon}}}\paren*{\Km_{\epsilon}}}\oplus \cH^{+}_{g_{\bomega_{\epsilon}}}} \otimes \bbR^{3} &\rightarrow C^{0,\alpha}_{\delta - 1,g_{\bomega_{\epsilon}}}\paren*{\Lambda_{g_{\bomega_{\epsilon}}}^{+}T^{*}\Km_{\epsilon}}\otimes\bbR^{3}\\
\paren*{\bm{a},\bm{\zeta}} &\mapsto d_{g_{\bomega_{\epsilon}}}^{+}\bm{a}+ \bm{\zeta} - \cF\paren*{\paren*{- Q_{\bomega_{\epsilon}} - d^{-}_{g_{\bomega_{\epsilon}}}\bm{a} \star d^{-}_{g_{\bomega_{\epsilon}}}\bm{a}}_{0}}\bomega_{\epsilon}
\end{al} where we solve \begin{al}
\Phi_{\epsilon}\paren*{\bm{a},\bm{\zeta}} &\coloneq	d_{g_{\bomega_{\epsilon}}}^{+}\bm{a}+ \bm{\zeta} - \cF\paren*{\paren*{- Q_{\bomega_{\epsilon}} - d^{-}_{g_{\bomega_{\epsilon}}}\bm{a} \star d^{-}_{g_{\bomega_{\epsilon}}}\bm{a}}_{0}}\bomega_{\epsilon}\\
&= \underbrace{-\cF\paren*{ \paren*{- Q_{\bomega_{\epsilon}}}_{0} }\bomega_{\epsilon}}_{=\Phi_{\epsilon}(0,0)} + \underbrace{d_{g_{\bomega_{\epsilon}}}^{+}\bm{a}+ \bm{\zeta}}_{=D_{0}\Phi_{\epsilon}\paren*{\bm{a},\bm{\zeta}}} + \underbrace{\cF\paren*{\paren*{- Q_{\bomega_{\epsilon}}}_{0}}\bomega_{\epsilon} - \cF\paren*{\paren*{- Q_{\bomega_{\epsilon}} - d^{-}_{g_{\bomega_{\epsilon}}}\bm{a} \star d^{-}_{g_{\bomega_{\epsilon}}}\bm{a}}_{0}}\bomega_{\epsilon}}_{\text{nonlinearity}}\\
&= 0
\end{al} Applying Theorem \ref{IFT}, we now arrive at: \begin{theorem}[Big Theorem II]\label{bigtheorem2} Let $\lambda < \frac{2}{3}$ so that the curvature of $g_{\bomega_{\epsilon}}$ remains bounded (\textit{Remark} \ref{HKA lambda < 2/3 bound}). Let $\delta \in (-1,0)$ for Proposition \ref{HKA IWSE}. Let $\epsilon > 0$ from Data \ref{constant epsilon and lambda introduction} be\begin{enumerate}
\item sufficiently small so that each pre-glued Eguchi-Hanson metric from Propositions \ref{HKA preglueEH BG interpolation!} and \ref{HKA preglueEH} are \hka and satisfy their annular estimates, and $\epsilon < 1$ to satisfy the volume form estimates.
\item $3\epsilon^{\lambda} < \frac{1}{2}$ from Data \ref{constant epsilon patching and WLOG lattice} for the construction of $\Km_{\epsilon}$.
\item $2\epsilon^{\lambda} < \frac{3}{25}$ and $2\epsilon < \epsilon^{\lambda}$ from Data \ref{constant epsilon 1-16} for the construction of $\rho_{\epsilon}$.
\item sufficiently small for the weighted Schauder estimate from Corollary \ref{HKA WSE}.
\item sufficiently small for the improved Schauder estimate from Proposition \ref{HKA IWSE}.
\item sufficiently small for the the uniform bounded linear isomorphism $D_{0}\Phi_{\epsilon}:\paren*{C^{1,\alpha}_{\delta,g_{\bomega_{\epsilon}}}\paren*{\mathring{\Omega}^{1}_{g_{\bomega_{\epsilon}}}\paren*{\Km_{\epsilon}}}\oplus \cH^{+}_{g_{\bomega_{\epsilon}}}} \otimes \bbR^{3} \rightarrow C^{0,\alpha}_{\delta - 1,g_{\bomega_{\epsilon}}}\paren*{\Lambda_{g_{\bomega_{\epsilon}}}^{+}T^{*}\Km_{\epsilon}}\otimes\bbR^{3}$ where $D_{0}\Phi_{\epsilon}= \paren*{d_{g_{\bomega_{\epsilon}}}^{+} \oplus \id_{\cH^{+}_{g_{\bomega_{\epsilon}}}}}\otimes \bbR^{3}$ from Theorem \ref{HKA biginverse}.
\item $\epsilon^{\paren*{1-\lambda}\paren*{3+\delta}} < \frac{1}{4L^{2}C_{18}C_{19}}$ for $L, C_{18}, C_{19} > 0$ constants which will be defined in the proof.
\end{enumerate}

Then for each such $\epsilon > 0$, the equation $\Phi_{\epsilon}\paren*{\bm{a},\bm{\zeta}} = 0$ has a \textbf{unique} solution $\paren*{\bm{a},\bm{\zeta}}$ satisfying:

\begin{itemize}
\item $\bm{a} \in \Omega^{1}\paren*{\Km_{\epsilon}}\otimes \bbR^{3}$ (\textbf{smoothness})
\item $\Abs*{\paren*{\bm{a},\bm{\zeta}}}_{\paren*{C^{1,\alpha}_{\delta,g_{\bomega_{\epsilon}}}\paren*{\mathring{\Omega}^{1}_{g_{\bomega_{\epsilon}}}\paren*{\Km_{\epsilon}}}\oplus \cH^{+}_{g_{\bomega_{\epsilon}}}} \otimes \bbR^{3}} \leq 2L \Abs*{\Phi_{\epsilon}\paren*{0,0}}_{C^{0,\alpha}_{\delta - 1,g_{\bomega_{\epsilon}}}\paren*{\Lambda_{g_{\bomega_{\epsilon}}}^{+}T^{*}\Km_{\epsilon}}\otimes\bbR^{3}} < C_{20}\epsilon^{4-\lambda\paren*{3+\delta}}$ for $C_{20} > 0$ chosen so that $C_{20} > 2L C_{19}$ (\textbf{smallness of norm})
\end{itemize}

Hence by Section \ref{HKA prelims} and Ansatz \ref{HKA perturbation ansatz}, for each such $\epsilon > 0$ we have that $\bomega_{\epsilon} + d\bm{a} + \bm{\zeta} = \bomega_{\epsilon} + d^{+}_{g_{\bomega_{\epsilon}}}\bm{a} + \bm{\zeta}$ is the \textbf{unique} \hka triple perturbed from the given closed and ``almost \hka triple'' $\bomega_{\epsilon}$ on $\Km_{\epsilon}$, whence gives us a \hka metric $g_{\bomega_{\epsilon} + d\bm{a} + \bm{\zeta}} = g_{\bomega_{\epsilon} + d^{+}_{g_{\bomega_{\epsilon}}}\bm{a} + \bm{\zeta}}$. 
\end{theorem} \begin{proof} We setup the input needed for the implicit function theorem (Theorem \ref{IFT}). \begin{enumerate}[wide, labelwidth=!, labelindent=0pt]
\item We have from Theorem \ref{HKA biginverse} that $D_{0}\Phi_{\epsilon}:\paren*{C^{1,\alpha}_{\delta,g_{\bomega_{\epsilon}}}\paren*{\mathring{\Omega}^{1}_{g_{\bomega_{\epsilon}}}\paren*{\Km_{\epsilon}}}\oplus \cH^{+}_{g_{\bomega_{\epsilon}}}} \otimes \bbR^{3} \rightarrow C^{0,\alpha}_{\delta - 1,g_{\bomega_{\epsilon}}}\paren*{\Lambda_{g_{\bomega_{\epsilon}}}^{+}T^{*}\Km_{\epsilon}}\otimes\bbR^{3}$ is a bounded linear isomorphism with $\Abs*{\paren*{D_{0}\Phi_{\epsilon}}^{-1}} \leq \cL \eqcolon L$ and $L>0$ \textit{uniform in $\epsilon > 0$}.
\item Let $r_{1} > 0$ be some fixed positive number which will be specified later. Equip $\Sym_{0}^{2}\paren*{\bbR^{3}}$ with the Frobenius norm $\Abs*{M}_{\text{F}} \coloneq \paren*{\sum_{i,j} \abs*{M_{ij}}^{2}}^{\frac{1}{2}} = \paren*{\Tr\paren*{M^{T}M}}^{\frac{1}{2}}$. Recall that $\cF: \Sym_{0}^{2}\paren*{\bbR^{3}} \rightarrow \Sym_{0}^{2}\paren*{\bbR^{3}}$ is a smooth map. Whence precomposing $\cF$ by a translation, say $\cF\paren*{M_{1} + \cdot}$ for $M_{1} \in \Sym_{0}^{2}\paren*{\bbR^{3}}$, remains smooth on $\Sym_{0}^{2}\paren*{\bbR^{3}}$. Therefore by the mean value theorem, we have that $\Abs*{\cF\paren*{M_{1} + A_{1}} - \cF\paren*{M_{1} + A_{2}}}_{\text{F}} \leq \cG\Abs*{A_{1} - A_{2}}_{\text{F}}$ for when $A_{1}, A_{2} \in D_{r_{1}}\paren*{0} \subset \Sym_{0}^{2}\paren*{\bbR^{3}}$ the closed disk of radius $r_{1}>0$ around the origin WRT the Frobenius norm, which is \textit{compact} since $\Sym_{0}^{2}\paren*{\bbR^{3}}$ is finite dimensional, whence the constant $\cG > 0$ is \textit{uniform} in $A_{1},A_{2}\in D_{r_{1}}\paren*{0}$. 

In our case, since $\cN\paren*{\bm{a},\bm{\zeta}} = \cN\paren*{\bm{a}} \coloneq \cF\paren*{\paren*{- Q_{\bomega_{\epsilon}}}_{0}}\bomega_{\epsilon} - \cF\paren*{\paren*{- Q_{\bomega_{\epsilon}} - d^{-}_{g_{\bomega_{\epsilon}}}\bm{a} \star d^{-}_{g_{\bomega_{\epsilon}}}\bm{a}}_{0}}\bomega_{\epsilon}$ only depends on $\bm{a}$ and that $\cN\paren*{\bm{a_{1}}} - \cN\paren*{\bm{a_{2}}} = \paren*{\cF\paren*{\paren*{- Q_{\bomega_{\epsilon}} - d^{-}_{g_{\bomega_{\epsilon}}}\bm{a_{2}} \star d^{-}_{g_{\bomega_{\epsilon}}}\bm{a_{2}}}_{0}} - \cF\paren*{\paren*{- Q_{\bomega_{\epsilon}} - d^{-}_{g_{\bomega_{\epsilon}}}\bm{a_{1}} \star d^{-}_{g_{\bomega_{\epsilon}}}\bm{a_{1}}}_{0}}}\bomega_{\epsilon}$, we thus have upon letting $M_{1} \coloneq - \paren*{Q_{\bomega_{\epsilon}}}_{0}$, $A_{1} \coloneq - \paren*{d^{-}_{g_{\bomega_{\epsilon}}}\bm{a_{1}} \star d^{-}_{g_{\bomega_{\epsilon}}}\bm{a_{1}}}_{0}$, and $A_{2} \coloneq - \paren*{d^{-}_{g_{\bomega_{\epsilon}}}\bm{a_{2}} \star d^{-}_{g_{\bomega_{\epsilon}}}\bm{a_{2}}}_{0}$, that \textit{pointwise} in $\Sym_{0}^{2}\paren*{\bbR^{3}}$ we have that \begin{al}
\abs*{\cN\paren*{\bm{a_{1}}} - \cN\paren*{\bm{a_{2}}}}_{g_{\epsilon}\otimes \bbR^{3}} &= \abs*{ \paren*{\cF\paren*{\paren*{- Q_{\bomega_{\epsilon}} - d^{-}_{g_{\bomega_{\epsilon}}}\bm{a_{1}} \star d^{-}_{g_{\bomega_{\epsilon}}}\bm{a_{1}}}_{0}} - \cF\paren*{\paren*{- Q_{\bomega_{\epsilon}} - d^{-}_{g_{\bomega_{\epsilon}}}\bm{a_{2}} \star d^{-}_{g_{\bomega_{\epsilon}}}\bm{a_{2}}}_{0}}}\bomega_{\epsilon}}_{g_{\epsilon}\otimes \bbR^{3}}\\
&\leq \Abs*{\cF\paren*{\paren*{- Q_{\bomega_{\epsilon}} - d^{-}_{g_{\bomega_{\epsilon}}}\bm{a_{1}} \star d^{-}_{g_{\bomega_{\epsilon}}}\bm{a_{1}}}_{0}} - \cF\paren*{\paren*{- Q_{\bomega_{\epsilon}} - d^{-}_{g_{\bomega_{\epsilon}}}\bm{a_{2}} \star d^{-}_{g_{\bomega_{\epsilon}}}\bm{a_{2}}}_{0}}}_{\text{F}} \abs*{\bomega_{\epsilon}}_{g_{\epsilon}\otimes \bbR^{3}} \\
&\leq \cG \Abs*{\paren*{d^{-}_{g_{\bomega_{\epsilon}}}\bm{a_{1}} \star d^{-}_{g_{\bomega_{\epsilon}}}\bm{a_{1}} - d^{-}_{g_{\bomega_{\epsilon}}}\bm{a_{2}} \star d^{-}_{g_{\bomega_{\epsilon}}}\bm{a_{2}}}_{0}}_{\text{F}}\abs*{\bomega_{\epsilon}}_{g_{\epsilon}\otimes \bbR^{3}}\\
&\leq \cG \Abs*{d^{-}_{g_{\bomega_{\epsilon}}}\bm{a_{1}} \star d^{-}_{g_{\bomega_{\epsilon}}}\bm{a_{1}} - d^{-}_{g_{\bomega_{\epsilon}}}\bm{a_{2}} \star d^{-}_{g_{\bomega_{\epsilon}}}\bm{a_{2}}}_{\text{F}}\abs*{\bomega_{\epsilon}}_{g_{\epsilon}\otimes \bbR^{3}}
\end{al} where $\abs*{\bm{v}}_{g_{\epsilon}\otimes \bbR^{3}} \coloneq \paren*{\abs*{v_{1}}_{g_{\epsilon}} + \abs*{v_{2}}_{g_{\epsilon}} + \abs*{v_{3}}_{g_{\epsilon}} }^{\frac{1}{2}}$ (recall \textit{Remark} \ref{HKA weighted tuple notation}), that $\Abs*{M}_{\text{F}, \cA} \coloneq \paren*{\sum_{i,j} \abs*{M_{ij}}_{\cA}^{2}}^{\frac{1}{2}}$ the Frobenius norm coupled with the norm $\Abs*{\cdot}_{\cA}$, that $\Abs*{M}_{\text{F}, \cA} = \Abs*{M}_{\cA \otimes \bbR^{3}}$ for a $3\times 1$ or $1\times 3$ form valued matrix $M$, $\Abs*{AB}_{\text{F}}\leq \Abs*{A}_{\text{F}}\Abs*{B}_{\text{F}}$, and $\Abs*{A_{0}}_{\text{F}}\leq \Abs*{A}_{\text{F}}$.

Now recall that $\abs*{\dV_{g}}_{g} = 1$ directly implies that $\abs*{\Theta}_{g} = \abs*{\frac{\Theta}{\dV_{g}} \dV_{g}}_{g}=  \abs*{\frac{\Theta}{\dV_{g}}}$ for any $4$-form $\Theta$ and that $\bm{\gamma} \star \bm{\gamma} \coloneq \frac{1}{2} \frac{\bm{\gamma} \wedge\bm{\gamma}^{T}}{\mu_{\bomega}} = \frac{1}{2} \frac{\bm{\gamma}\wedge \bm{\gamma}^{T}}{\dV_{g_{\bomega}}}$ for any triple $\bm{\gamma}$ of $2$-forms. Thus $\abs*{\bm{\gamma}\star\bm{\gamma}}_{ij} = \frac{1}{2} \abs*{\frac{\gamma_{i}\wedge \gamma_{j}}{\dV_{g_{\bomega}}}} = \frac{1}{2} \abs*{\gamma_{i}\wedge \gamma_{j}}_{g_{\bomega}} = \frac{1}{2}\abs*{\paren*{\bm{\gamma}\wedge \bm{\gamma}^{T}}_{ij}}_{g_{\bomega_{\epsilon}}}$, and since $AA^{T} - BB^{T} = \fS\sqparen*{\paren*{A-B}\paren*{A+B}^{T}}$ where $\fS\sqparen*{M} \coloneq \frac{M+M^{T}}{2}$ is the symmetrization map, upon using the definition of the Frobenius norm and continuity of the symmetrization map we have that \begin{al}
\Abs*{d^{-}_{g_{\bomega_{\epsilon}}}\bm{a_{1}} \star d^{-}_{g_{\bomega_{\epsilon}}}\bm{a_{1}} - d^{-}_{g_{\bomega_{\epsilon}}}\bm{a_{2}} \star d^{-}_{g_{\bomega_{\epsilon}}}\bm{a_{2}}}_{\text{F}} &= \Abs*{d^{-}_{g_{\bomega_{\epsilon}}}\bm{a_{1}} \wedge \paren*{d^{-}_{g_{\bomega_{\epsilon}}}\bm{a_{1}}}^{T} - d^{-}_{g_{\bomega_{\epsilon}}}\bm{a_{2}} \wedge \paren*{d^{-}_{g_{\bomega_{\epsilon}}}\bm{a_{2}}}^{T}}_{\text{F}, g_{\bomega_{\epsilon}}}\\
&= \Abs*{\fS\sqparen*{\paren*{d^{-}_{g_{\bomega_{\epsilon}}}\bm{a_{1}} - d^{-}_{g_{\bomega_{\epsilon}}}\bm{a_{2}}}\wedge \paren*{d^{-}_{g_{\bomega_{\epsilon}}}\bm{a_{1}} + d^{-}_{g_{\bomega_{\epsilon}}}\bm{a_{2}}}^{T}}}_{\text{F}, g_{\bomega_{\epsilon}}}\\
&\leq \Abs*{\fS}_{\text{op}}\Abs*{\paren*{d^{-}_{g_{\bomega_{\epsilon}}}\bm{a_{1}} - d^{-}_{g_{\bomega_{\epsilon}}}\bm{a_{2}}}\wedge \paren*{d^{-}_{g_{\bomega_{\epsilon}}}\bm{a_{1}} + d^{-}_{g_{\bomega_{\epsilon}}}\bm{a_{2}}}^{T}}_{\text{F}, g_{\bomega_{\epsilon}}}
\end{al}

Thus we have the \textit{pointwise} inequality $$\abs*{\cN\paren*{\bm{a_{1}}} - \cN\paren*{\bm{a_{2}}}}_{g_{\epsilon}\otimes \bbR^{3}} \leq \cG \Abs*{\fS}_{\text{op}}\abs*{\bomega_{\epsilon}}_{g_{\epsilon}\otimes \bbR^{3}}\Abs*{\paren*{d^{-}_{g_{\bomega_{\epsilon}}}\bm{a_{1}} - d^{-}_{g_{\bomega_{\epsilon}}}\bm{a_{2}}}\wedge \paren*{d^{-}_{g_{\bomega_{\epsilon}}}\bm{a_{1}} + d^{-}_{g_{\bomega_{\epsilon}}}\bm{a_{2}}}^{T}}_{\text{F}, g_{\bomega_{\epsilon}}} $$ Thus we multiply each side by $\rho_{\epsilon}^{-\delta + 1}$, and upon recalling that $\abs*{\bomega_{\epsilon}}_{C^{0}_{g_{\epsilon}\otimes \bbR^{3}}}$ is uniformly bounded in $\epsilon$ from the proof of Theorem \ref{HKA biginverse} and the definition of the weighted H\"{o}lder norm, we finally have upon taking $\sup_{\Km_{\epsilon}}$ on both sides and arguing all of the above in a similar fashion for the H\"{o}lder seminorm that $$\Abs*{\cN\paren*{\bm{a_{1}}} - \cN\paren*{\bm{a_{2}}}}_{C^{0,\alpha}_{\delta - 1,g_{\bomega_{\epsilon}}}\paren*{\Lambda_{g_{\bomega_{\epsilon}}}^{+}T^{*}\Km_{\epsilon}}\otimes\bbR^{3}} \leq \cG_{1} \Abs*{\paren*{d^{-}_{g_{\bomega_{\epsilon}}}\bm{a_{1}} - d^{-}_{g_{\bomega_{\epsilon}}}\bm{a_{2}}}\wedge \paren*{d^{-}_{g_{\bomega_{\epsilon}}}\bm{a_{1}} + d^{-}_{g_{\bomega_{\epsilon}}}\bm{a_{2}}}^{T}}_{\text{F}, C^{0,\alpha}_{\delta - 1, g_{\bomega_{\epsilon}}}}$$ for some \textit{uniform} constant $\cG_{1}> 0$ and with the $\Abs*{\cdot}_{\text{F}, C^{0,\alpha}_{\delta - 1, g_{\bomega_{\epsilon}}}}$-norm having the obvious meaning. Thus, using Proposition \ref{weighted products} with $l = 1$ on the product in the RHS term, noting the definition of the weighted H\"{o}lder norms (namely that the weight is $\rho_{\epsilon}^{-\delta + 1}$ for the $\nabla_{g_{\bomega_{\epsilon}}}$-term), applying the triangle inequality, and adding the remaining terms in the norm of $\paren*{C^{1,\alpha}_{\delta,g_{\bomega_{\epsilon}}}\paren*{\mathring{\Omega}^{1}_{g_{\bomega_{\epsilon}}}\paren*{\Km_{\epsilon}}}\oplus \cH^{+}_{g_{\bomega_{\epsilon}}}} \otimes \bbR^{3}$ onto the RHS (all of which are \textit{nonnegative}), we finally get that \begin{al}
&\Abs*{\cN\paren*{\bm{a_{1}}, \bm{\zeta_{1}}} - \cN\paren*{\bm{a_{2}}, \bm{\zeta_{2}}}}_{C^{0,\alpha}_{\delta - 1,g_{\bomega_{\epsilon}}}\paren*{\Lambda_{g_{\bomega_{\epsilon}}}^{+}T^{*}\Km_{\epsilon}}\otimes\bbR^{3}} \\
&\leq C_{18}\epsilon^{\delta - 1} \Abs*{\paren*{\bm{a_{1}}, \bm{\zeta_{1}}} - \paren*{\bm{a_{2}}, \bm{\zeta_{2}}}}_{\paren*{C^{1,\alpha}_{\delta,g_{\bomega_{\epsilon}}}\paren*{\mathring{\Omega}^{1}_{g_{\bomega_{\epsilon}}}\paren*{\Km_{\epsilon}}}\oplus \cH^{+}_{g_{\bomega_{\epsilon}}}} \otimes \bbR^{3}}\\
&\quad \cdot \paren*{\Abs*{\paren*{\bm{a_{1}}, \bm{\zeta_{1}}}}_{\paren*{C^{1,\alpha}_{\delta,g_{\bomega_{\epsilon}}}\paren*{\mathring{\Omega}^{1}_{g_{\bomega_{\epsilon}}}\paren*{\Km_{\epsilon}}}\oplus \cH^{+}_{g_{\bomega_{\epsilon}}}} \otimes \bbR^{3}} + \Abs*{\paren*{\bm{a_{2}}, \bm{\zeta_{2}}}}_{\paren*{C^{1,\alpha}_{\delta,g_{\bomega_{\epsilon}}}\paren*{\mathring{\Omega}^{1}_{g_{\bomega_{\epsilon}}}\paren*{\Km_{\epsilon}}}\oplus \cH^{+}_{g_{\bomega_{\epsilon}}}} \otimes \bbR^{3}}}
\end{al} for some \textit{uniform} constant $C_{18}>0$.

Lastly, we in fact set $r_{1}\coloneq r^{2}_{0}$ and have that this above inequality makes sense only when $$\paren*{\bm{a_{1}}, \bm{\zeta_{1}}}, \paren*{\bm{a_{2}}, \bm{\zeta_{2}}} \in D_{r_{0}}^{\paren*{C^{1,\alpha}_{\delta,g_{\bomega_{\epsilon}}}\paren*{\mathring{\Omega}^{1}_{g_{\bomega_{\epsilon}}}\paren*{\Km_{\epsilon}}}\oplus \cH^{+}_{g_{\bomega_{\epsilon}}}} \otimes \bbR^{3}}\paren*{0}$$ because by definition of the weighted H\"{o}lder norm, $\rho_{\epsilon}\leq 1$, and $0 < -\delta + 1$, we have that $\Abs*{\rho_{\epsilon}^{-\delta + 1}\nabla_{g_{\bomega_{\epsilon}}} \bm{a}}_{C^{0}_{g_{\bomega_{\epsilon}}}\paren*{\Km_{\epsilon}}\otimes\bbR^{3}} = \Abs*{\rho_{\epsilon}^{-\delta + 1}}_{C^{0}\paren*{\Km_{\epsilon}}}\Abs*{\nabla_{g_{\bomega_{\epsilon}}} \bm{a}}_{C^{0}_{g_{\bomega_{\epsilon}}}\paren*{\Km_{\epsilon}}\otimes\bbR^{3}}  = \Abs*{\nabla_{g_{\bomega_{\epsilon}}} \bm{a}}_{C^{0}_{g_{\bomega_{\epsilon}}}\paren*{\Km_{\epsilon}}\otimes\bbR^{3}} $ and $$\Abs*{\nabla_{g_{\bomega_{\epsilon}}} \bm{a}}_{C^{0}_{g_{\bomega_{\epsilon}}}\paren*{\Km_{\epsilon}}\otimes\bbR^{3}} \leq \Abs*{\paren*{\bm{a},\bm{\zeta}}}_{\paren*{C^{1,\alpha}_{\delta,g_{\bomega_{\epsilon}}}\paren*{\mathring{\Omega}^{1}_{g_{\bomega_{\epsilon}}}\paren*{\Km_{\epsilon}}}\oplus \cH^{+}_{g_{\bomega_{\epsilon}}}} \otimes \bbR^{3}} \leq r_{0}$$ and for $\bm{\gamma}$ a triple of $2$-forms $$\Abs*{\bm{\gamma}\star\bm{\gamma}}_{\text{F}} = \Abs*{\bm{\gamma}\wedge \bm{\gamma}^{T}}_{\text{F}, g} \leq \Abs*{\bm{\gamma}}^{2}_{\text{F}, g} = \Abs*{\bm{\gamma}}^{2}_{g\otimes \bbR^{3}}$$ Thus we have that $N \coloneq C_{18}\epsilon^{\delta - 1}$ and $r_{0} > 0$ to be determined later. 
\item Since $\Phi_{\epsilon}\paren*{0,0} = -\cF\paren*{ \paren*{- Q_{\bomega_{\epsilon}}}_{0} }\bomega_{\epsilon}$, we have by the same arguments as for establishing the nonlinearity estimate above coupled with the fact that $\cF\paren*{0} = 0$ that \begin{al}
\Abs*{\Phi_{\epsilon}\paren*{0,0}}_{C^{0,\alpha}_{\delta - 1,g_{\bomega_{\epsilon}}}\paren*{\Lambda_{g_{\bomega_{\epsilon}}}^{+}T^{*}\Km_{\epsilon}}\otimes\bbR^{3}}
&\leq \cG_{2}\Abs*{\paren*{Q_{\bomega_{\epsilon}}}_{0}}_{\text{F}, C^{0,\alpha}_{\delta - 1,g_{\bomega_{\epsilon}}}}\\
&= \cG_{2}\Abs*{Q_{\bomega_{\epsilon}} - \frac{1}{3} \Tr\paren*{Q_{\bomega_{\epsilon}}}\cdot \id}_{\text{F}, C^{0,\alpha}_{\delta - 1,g_{\bomega_{\epsilon}}}}\\
&= \cG_{2}\Abs*{\id + O\paren*{\epsilon^{4-4\lambda}} - \paren*{1+O\paren*{\epsilon^{4-4\lambda}}}\id }_{\text{F}, C^{0,\alpha}_{\delta - 1,g_{\bomega_{\epsilon}}}}\\
&= \cG_{2}\Abs*{O\paren*{\epsilon^{4-4\lambda}}}_{\text{F}, C^{0,\alpha}_{\delta - 1,g_{\bomega_{\epsilon}}}}\\
&\leq C_{19} \epsilon^{4-4\lambda + \lambda\paren*{1-\delta}} \\
&= C_{19} \epsilon^{4 - 3\lambda -\lambda\delta}= C_{19} \epsilon^{4 - \lambda\paren*{3+\delta}}
\end{al} where we have used the fact that $Q_{\bomega_{\epsilon}} = \id + O\paren*{\epsilon^{4-4\lambda}}$ hence $\Tr\paren*{Q_{\bomega_{\epsilon}}} = \paren*{1+O\paren*{\epsilon^{4-4\lambda}}}\Tr\paren*{\id} = \paren*{1+O\paren*{\epsilon^{4-4\lambda}}}\cdot 3$ holding both componentwise and globally on $\Km_{\epsilon}$ (Proposition \ref{HKA approximatemetricproperties}), as well as the fact that $\supp \paren*{Q_{\bomega_{\epsilon}}}_{0} \subseteq \set*{\epsilon^{\lambda} \leq \abs*{\cdot}^{g_{\bomega_{0}}}_{\pi^{-1}\paren*{\cS}} \leq 2\epsilon^{\lambda}}$ since outside that region $Q_{\bomega_{\epsilon}} = \id$, whence the factor of $\delta$ comes in because on $\set*{\epsilon^{\lambda} \leq \abs*{\cdot}^{g_{\bomega_{0}}}_{\pi^{-1}\paren*{\cS}} \leq 2\epsilon^{\lambda}}$ our weight function $\rho_{\epsilon}$ when combined with Data \ref{constant epsilon 1-16} (namely $2\epsilon < \epsilon^{\lambda}$) satisfies $\epsilon^{\lambda} \leq \rho_{\epsilon} \leq 2\epsilon^{\lambda}$ whence since $0 < -\delta + 1$ we have that $\rho_{\epsilon}^{-\delta + 1} = O\paren*{\epsilon^{\lambda\paren*{1-\delta}}}$. All of this holds for some \textit{uniform} constant $C_{19}> 0$.
\end{enumerate}

Set $r_{0} \coloneq C_{20} \epsilon^{4-\lambda\paren*{3+\delta}}$ with $C_{20} > 0$ chosen so that $C_{20} > 2LC_{19}$. We now need to show that $$\Abs*{\Phi_{\epsilon}\paren*{0,0}}_{C^{0,\alpha}_{\delta - 1,g_{\bomega_{\epsilon}}}\paren*{\Lambda_{g_{\bomega_{\epsilon}}}^{+}T^{*}\Km_{\epsilon}}\otimes\bbR^{3}}\leq \frac{r}{2L}$$ with $r < \min\set*{r_{0}, \frac{1}{2NL}} = \min\set*{C_{20} \epsilon^{4-\lambda\paren*{3+\delta}}, \frac{\epsilon^{1-\delta}}{2LC_{18}}}$. We have that \begin{al}
\Abs*{\Phi_{\epsilon}\paren*{0,0}}_{C^{0,\alpha}_{\delta - 1,g_{\bomega_{\epsilon}}}\paren*{\Lambda_{g_{\bomega_{\epsilon}}}^{+}T^{*}\Km_{\epsilon}}\otimes\bbR^{3}}\leq C_{19}\epsilon^{4-\lambda\paren*{3+\delta}} = \frac{2LC_{19}\epsilon^{4-\lambda\paren*{3+\delta}}}{2L}
\end{al} and so we choose our $r \coloneq 2LC_{19}\epsilon^{4-\lambda\paren*{3+\delta}}$. Clearly $r < r_{0}$, and it remains to show that $r < \frac{\epsilon^{1-\delta}}{2LC_{18}}$, i.e. $\epsilon^{4-\lambda\paren*{3+\delta} + \delta - 1} < \frac{1}{4L^{2}C_{18}C_{19}}$. However, $4-\lambda\paren*{3+\delta} +\delta - 1 = 3 - \lambda\paren*{3+\delta} + \delta = \paren*{1-\lambda}\paren*{3+\delta} > 0$ is \textit{positive} since $\lambda \in \paren*{0,\frac{2}{3}}$ and $\delta \in \paren*{-1,0}$. Ergo making $\epsilon > 0$ small enough so that $\epsilon^{\paren*{1-\lambda}\paren*{3+\delta}} < \frac{1}{4L^{2}C_{18}C_{19}}$ forces $r < \frac{\epsilon^{1-\delta}}{2LC_{18}}$. Therefore, \textit{for each such $\epsilon > 0$} Theorem \ref{IFT} gives us the existence of a unique solution $\paren*{\bm{a},\bm{\zeta}} \in D_{2L \Abs*{\Phi_{\epsilon}\paren*{0}}_{C^{0,\alpha}_{\delta - 1,g_{\bomega_{\epsilon}}}\paren*{\Lambda_{g_{\bomega_{\epsilon}}}^{+}T^{*}\Km_{\epsilon}}\otimes\bbR^{3}}}^{\paren*{C^{1,\alpha}_{\delta,g_{\bomega_{\epsilon}}}\paren*{\mathring{\Omega}^{1}_{g_{\bomega_{\epsilon}}}\paren*{\Km_{\epsilon}}}\oplus \cH^{+}_{g_{\bomega_{\epsilon}}}} \otimes \bbR^{3}}\paren*{0}$ to $\Phi_{\epsilon}\paren*{\bm{a},\bm{\zeta}} = 0$, giving us existence, uniqueness, and the bound $$\Abs*{\paren*{\bm{a},\bm{\zeta}}}_{\paren*{C^{1,\alpha}_{\delta,g_{\bomega_{\epsilon}}}\paren*{\mathring{\Omega}^{1}_{g_{\bomega_{\epsilon}}}\paren*{\Km_{\epsilon}}}\oplus \cH^{+}_{g_{\bomega_{\epsilon}}}} \otimes \bbR^{3}} \leq 2L \Abs*{\Phi_{\epsilon}\paren*{0}}_{C^{0,\alpha}_{\delta - 1,g_{\bomega_{\epsilon}}}\paren*{\Lambda_{g_{\bomega_{\epsilon}}}^{+}T^{*}\Km_{\epsilon}}\otimes\bbR^{3}} < C_{20}\epsilon^{4-\lambda\paren*{3+\delta}}$$ since $r < r_{0}$. Note that the RHS remains \textit{bounded} as $\epsilon \ll 1$ since $\lambda < \frac{4}{3+\delta}$ holds since $1 < \frac{4}{3+\delta}$ from $\delta \in \paren*{-1,0}$. Smoothness follows from elliptic regularity, as our nonlinear equation has smooth coefficients \& $1$st order elliptic linearization and the produced solution $\paren*{\bm{a},\bm{\zeta}}$ is in $C^{1,\alpha}\subset C^{1}$. Hence the setup in Section \ref{HKA prelims} as well as Anzatz \ref{HKA perturbation ansatz} finish off the existence of our unique Ricci-flat \hka metric for each sufficiently small $\epsilon > 0$, as was to be shown.\end{proof}

Now we have from the smallness of the norm $\Abs*{\paren*{\bm{a},\bm{\zeta}}}_{\paren*{C^{1,\alpha}_{\delta,g_{\bomega_{\epsilon}}}\paren*{\mathring{\Omega}^{1}_{g_{\bomega_{\epsilon}}}\paren*{\Km_{\epsilon}}}\oplus \cH^{+}_{g_{\bomega_{\epsilon}}}} \otimes \bbR^{3}} < C_{20}\epsilon^{4-\lambda\paren*{3+\delta}}$ of the solution to the nonlinear problem above, the definition of the weighted H\"{o}lder spaces, the global bound $\epsilon \leq \rho_{\epsilon}$ and the fact that $\delta \in \paren*{\max\set*{-2,-4\lambda},0}$ gives that \begin{al}
\Abs*{\bomega_{\epsilon} + d^{+}_{g_{\bomega_{\epsilon}}}\bm{a} + \bm{\zeta} - \bomega_{\epsilon}}_{\paren*{C^{0}_{g_{\bomega_{\epsilon}}}\paren*{\mathring{\Omega}^{1}_{g_{\bomega_{\epsilon}}}\paren*{\Km_{\epsilon}}}\oplus \cH^{+}_{g_{\bomega_{\epsilon}}}} \otimes \bbR^{3}} &= \Abs*{d^{+}_{g_{\bomega_{\epsilon}}}\bm{a} + \bm{\zeta}}_{\paren*{C^{0}_{g_{\bomega_{\epsilon}}}\paren*{\mathring{\Omega}^{1}_{g_{\bomega_{\epsilon}}}\paren*{\Km_{\epsilon}}}\oplus \cH^{+}_{g_{\bomega_{\epsilon}}}} \otimes \bbR^{3}} \\
&\leq C_{21}\epsilon^{4-\lambda\paren*{3+\delta}} \epsilon^{\delta - 1} \\ &= C_{21}\epsilon^{\paren*{1-\lambda}\paren*{3+\delta}} 
\end{al} for some \textit{uniform} constant $C_{21} > 0$, whence $\Abs*{\bomega_{\epsilon} + d^{+}_{g_{\bomega_{\epsilon}}}\bm{a} + \bm{\zeta} - \bomega_{\epsilon}}_{\paren*{C^{0}_{g_{\bomega_{\epsilon}}}\paren*{\mathring{\Omega}^{1}_{g_{\bomega_{\epsilon}}}\paren*{\Km_{\epsilon}}}\oplus \cH^{+}_{g_{\bomega_{\epsilon}}}} \otimes \bbR^{3}}$ tends to $0$ as $\epsilon\searrow 0$ due to $\lambda \in \paren*{0,\frac{2}{3}}$ and $\delta \in \paren*{-1,0}$.

Combining these with \textit{Remark} \ref{HKA GHremark}, the two convergences $g_{\bomega_{\epsilon}} \xrightarrow{C^{\infty}_{g_{\bomega_{0}},loc}\paren*{\set*{0 < \abs*{\cdot}^{g_{\bomega_{0}}}_{\pi^{-1}\paren*{\cS}}}}} g_{\bomega_{0}}$ and $R_{\epsilon}^{*}\paren*{\frac{1}{\epsilon^{2}} g_{\bomega_{\epsilon}}} = \lwhat{g_{\bomega_{EH-0, \epsilon}}} \xrightarrow{C^{\infty}_{g_{\bomega_{EH,1}},loc}\paren*{\set*{\abs*{\cdot}^{g_{\bomega_{0}}}_{S^{2}} \leq \frac{1}{8\epsilon}}}} g_{\bomega_{EH,1}}$ used in the proof of Proposition \ref{HKA IWSE}, Proposition \ref{HKA EHproperties} regarding the curvature blow-up of the bolt of $g_{\bomega_{EH,\epsilon^{2}}}$ and Propositions \ref{HKA preglueEH} and \ref{HKA approximatemetricproperties} on the construction of $g_{\bomega_{\epsilon}}$, and the triangle inequality, we thus have the analogue of Theorem \ref{maintheorem1} in the \hka context:

\begin{theorem}[Main Theorem IIa]\label{maintheorem2a} Let $K3$ denote the $K3$ surface. Equip $K3$ with the family of closed definite triples $\bomega_{\epsilon}$ constructed in Section \ref{HKA The Kummer Construction}, hence getting the family of almost \hka metrics $\paren*{\Km_{\epsilon}, \bomega_{\epsilon}, g_{\bomega_{\epsilon}}}$ on a family of Kummer surfaces $\Km_{\epsilon}$.

Then denoting by $\wtilde{\bomega_{\epsilon}} \coloneq \bomega_{\epsilon} + d\bm{a}+ \bm{\zeta} = \bomega_{\epsilon} + d^{+}_{g_{\bomega_{\epsilon}}}\bm{a}+ \bm{\zeta}$ and $g_{\wtilde{\bomega_{\epsilon}}}$ the unique Ricci-flat \hka metric produced from Theorem \ref{bigtheorem2} for each $\epsilon > 0$, we have that $0 < \exists \epsilon_{0} \ll 1$ such that \textbf{there exists a 1-parameter family of \hka triples and \hka metrics} $$\set*{\paren*{\wtilde{\bomega_{\epsilon}}, g_{\wtilde{\bomega_{\epsilon}}}}}_{\epsilon \in \paren*{0,\epsilon_{0}}}$$ on $K3$.

Moreover when $\epsilon \searrow 0$, $$\paren*{\Km_{\epsilon}, \bomega_{\epsilon}, g_{\bomega_{\epsilon}}} \xrightarrow{GH} \paren*{\bbT^{4}/\bbZ_{2}, \bomega_{0}, g_{\bomega_{0}}}$$ converges in the Gromov-Hausdorff topology to the flat \hka orbifold $\paren*{\bbT^{4}/\bbZ_{2}, \bomega_{0}, g_{\bomega_{0}}}$ with the singular orbifold \hka metric $\bomega_{0}, g_{\bomega_{0}}$ which is equal to the flat \hka metric on $\bbT^{4}/\bbZ_{2} - \cS$.

Moreover, we may decompose $K3$ into a union of sets\footnote{say, $\Km_{\epsilon}^{reg} \coloneq \set*{\frac{\epsilon^{\lambda}}{2} < \abs*{\cdot}^{g_{\bomega_{0}}}_{\pi^{-1}\paren*{\cS}}}$ and each $\Km_{\epsilon}^{p}$ being the 16 connected components of $\set*{\abs*{\cdot}^{g_{\bomega_{0}}}_{\pi^{-1}\paren*{\cS}}\leq \epsilon^{\lambda}}$.} $\Km_{\epsilon}^{reg} \cup \bigsqcup_{p \in \cS}\Km_{\epsilon}^{b_{p}}$ such that \begin{enumerate}
\itemsep0em 
\item $\paren*{\Km_{\epsilon}^{reg}, \bomega_{\epsilon}, g_{\bomega_{\epsilon}}}$ converges in $C^{\infty}_{g_{\bomega_{0}},loc}$ to the flat \hka manifold $\paren*{\bbT^{4}/\bbZ_{2} - \cS, \bomega_{0}, g_{\bomega_{0}}}$ with bounded curvature away from $\cS$ (\textbf{regular region}).
\item For each $p \in \cS$, after applying a $\frac{1}{\epsilon}$-dilation $\paren*{R_{\frac{1}{\epsilon}}\paren*{\Km_{\epsilon}^{b_{p}}}, R_{\epsilon}^{*}\paren*{\frac{1}{\epsilon^{2}}g_{\bomega_{\epsilon}}}}$ converges to $\paren*{T^{*}S^{2}, g_{\bomega_{EH,1}}}$ in $C^{\infty}_{g_{\bomega_{EH,1}},loc}$ (\textbf{ALE bubble region}).
\end{enumerate}\end{theorem}

In fact, what we have proved yields more:

Enumerate the 16 orbifold singularities $\cS = \set*{p_{1},\dots, p_{16}}$ of $\bbT^{4}/\bbZ_{2}$. Pick a fixed flat \hka metric $\bomega_{0}, g_{\bomega_{0}}$ on $\bbT^{4}/\bbZ_{2}$, as from Proposition \ref{flat tori facts} there is a $\frac{4(4+1)}{2} = 10$-parameter family of flat metrics on $\bbT^{4}/\bbZ_{2}$. For any subset $I\subseteq \set*{1,\dots, 16}$, denote its complement as $I^{c} \coloneq \set*{1,\dots, 16} - I$ whence $I \sqcup I^{c} = \set*{1,\dots, 16}$. 

For a fixed $I \subseteq \set*{1,\dots, 16}$, repeat the gluing construction in Section \ref{HKA The Kummer Construction} on the base orbifold $\bbT^{4}/\bbZ_{2}$, but instead we\begin{itemize}
\itemsep0em 
\item pick a flat \hka metric $\paren*{\bomega_{f}, g_{\bomega_{f}}}$ on $\bbT^{4}/\bbZ_{2}$ where $f \in \GL\paren*{4,\bbZ}\backslash \GL\paren*{4,\bbR}/\O(4)$ the moduli space of flat metrics on $\bbT^{4}/\bbZ_{2}$ which is a $\frac{4(4+1)}{2} = 10$-parameter/dimensional space (Proposition \ref{flat tori facts}).
\item \textit{leave $p_{i} \in \cS\subset \bbT^{4}/\bbZ_{2}$ for $i \in I^{c}$ unresolved}.
\item recalling from Proposition \ref{HKA EHproperties} $$\set*{\paren*{T^{*}S^{2}, \bomega^{ALE}_{t,e}, g_{\bomega^{ALE}_{t,e}}}}_{\paren*{t,e} \in \bbR_{>0}\times S^{2}}$$ the $3$-parameter family of \hka ALE spaces asymptotic to $\bbR^{4}/\bbZ_{2}$ and where $$\bomega^{ALE}_{t,e}, g_{\bomega^{ALE}_{t,e}} \overset{\text{isometric}}{\cong} A^{e}\bomega_{EH,t^{2}}, g_{A^{e}\bomega_{EH, t^{2}}} \overset{\text{isometric}}{\cong} t^{2}A^{e}\bomega_{EH,1}, t^{2}g_{A^{e}\bomega_{EH, 1}} $$ for each $i \in I$ we choose a pair $\paren*{\epsilon_{i}, e_{i}} \in \bbR_{>0} \times S^{2}$ and glue $\paren*{T^{*}S^{2}, \bomega^{ALE}_{\epsilon_{i},e_{i}}, g_{\bomega^{ALE}_{\epsilon_{i},e_{i}}}}$ around the $i$-th singular point $p_{i} \in \cS$ using the isometry above to equivalently glue $\paren*{T^{*}S^{2},A^{e_{i}}\bomega_{EH,\epsilon_{i}^{2}}, g_{A^{e_{i}}\bomega_{EH, \epsilon_{i}^{2}}}}$ via the same cutoff function interpolation of Propositions \ref{HKA preglueEH BG interpolation!} and \ref{HKA preglueEH}.
\end{itemize}

This yields a compact \textit{orbifold} $\Km^{\abs*{I}}_{f, \epsilon_{I}, e_{I}}$ with $16 - \abs*{I}$ orbifold points $p_{i} \in \Km^{\abs*{I}}_{f, \epsilon_{I}, e_{I}}, i \in I^{c}$ each locally modeled on $\bbR^{4}/\bbZ_{2}$ which we call the \textbf{$I$-th partial smoothing $\Km^{\abs*{I}} = \Km^{\abs*{I}}_{f, \epsilon_{I}, e_{I}}$ of $\bbT^{4}/\bbZ_{2}$}, equipped with the constructed closed \textit{orbifold} definite triple which is \textit{almost} hyper-K\"{a}hler. Since the resulting closed \textit{orbifold} definite triple on $\Km^{\abs*{I}}_{f, \epsilon_{I}, e_{I}}$ that gets constructed is a genuine \hka triple near both the $\abs*{I}$ exceptional $S^{2}$s and near each orbifold point $p_{i} \in \Km^{\abs*{I}}_{f, \epsilon_{I}, e_{I}}, i \in I^{c}$, namely the \textit{smooth} Eguchi-Hanson \hka triple $A^{e_{i}}\bomega_{EH,\epsilon_{i}^{2}}$ and the flat \textit{orbifold} \hka triple $\bomega_{0}$ near each $S^{2} \cong E_{i}, i \in I$ and $p_{j} \in \Km_{\epsilon_{I}}, j \in I^{c}$ respectively, and satisfies the \textit{same} associated intersection matrix decay $Q = \id + O\paren*{\epsilon^{4-4\lambda}}$ on each annulus $\set*{\epsilon_{i}^{\lambda} \leq \abs*{\cdot}^{g_{\bomega_{0}}}_{E_{i}}  \leq 2\epsilon_{i}^{\lambda}}$ around each exceptional $S^{2} \cong E_{i}$ for $i \in I$, the analysis developed in Sections \ref{HKA Weighted Setup}, \ref{HKA Nonlinear Setup}, \ref{HKA Blow-up Analysis}, and in the proof of Theorem \ref{bigtheorem2} to perturb approximate \hka triples to genuine \hka triples via solving the nonlinear elliptic PDE system may be carried out with only cosmetic adjustments in the presence of flat \hka orbifold singularities (via lifting to a $\bbZ_{2}$-invariant problem on a \textit{smooth} manifold since all orbifold singularities in our context have isotropy group $\bbZ_{2}$) to yield an \textbf{orbifold \hka triple} $\paren*{\bomega_{f, \epsilon_{I}, e_{I}}, g_{\bomega_{f, \epsilon_{I}, e_{I}}}} $ on the compact orbifold $\Km^{\abs*{I}}_{f, \epsilon_{I}, e_{I}}$ for when each of the gluing parameters $\set*{\epsilon_{i} : i \in I} $ are sufficiently small.

Whence since the proofs of both the \textbf{regular region} and \textbf{ALE bubble region} structures in Theorem \ref{maintheorem2a} above follow from the $C^{\infty}_{loc}$ convergences in the proof of Proposition \ref{HKA IWSE}, which followed via examining the metric behavior as $\epsilon\searrow 0$ of the regions of $\Km_{\epsilon}$ in the bulk and \textit{near} each exceptional $S^{2}$ \textit{after a $\frac{1}{\epsilon}$-dilation}, respectively, we immediately have that such convergences in the bulk and near each exceptional $S^{2}$ follow verbatim for the compact \hka orbifolds $\paren*{\Km^{\abs*{I}}_{f, \epsilon_{I}, e_{I}}, \bomega_{f, \epsilon_{I}, e_{I}}, g_{\bomega_{f, \epsilon_{I}, e_{I}}}}$ because the region near each exceptional $S^{2}$ is \textit{localized} near $S^{2}$, whence the convergence remains the same irrespective of whether or not there are flat orbifold singularities in its complement left unresolved. Thus these cosmetic adjustments yield us Theorem \ref{MainTheoremB}, namely \begin{theorem}[Main Theorem II]\label{maintheorem2} Enumerate the $16$ singular points $\cS = \set*{p_{1},\dots, p_{16}}$ of $\bbT^{4}/\bbZ_{2}$.

Then $\forall I \subseteq \set*{1,\dots, 16}$, $0 < \exists \epsilon^{0}_{I} \ll 1$ such that \textbf{there exists a $10 + 3\abs*{I}$-parameter family of (orbifold) \hka triples and (orbifold) \hka metrics} $$\set*{\paren*{\bomega_{f, \epsilon_{I}, e_{I}}, g_{\bomega_{f, \epsilon_{I}, e_{I}}}}}_{\substack{f \in \GL\paren*{4,\bbZ}\backslash \GL\paren*{4,\bbR}/\O(4)\\ \epsilon_{i} \in \paren*{0,\epsilon_{I}^{0}} \\ e_{i} \in S^{2} \\ i \in I}}$$ on the $I$-th partial smoothing $\Km^{\abs*{I}} = \Km^{\abs*{I}}_{f, \epsilon_{I}, e_{I}}$ of $\bbT^{4}/\bbZ_{2}$.

Moreover, as $\epsilon^{0}_{I}\searrow 0$, $$\paren*{\Km^{\abs*{I}}_{f, \epsilon_{I}, e_{I}}, \bomega_{f, \epsilon_{I}, e_{I}}, g_{\bomega_{f, \epsilon_{I}, e_{I}}}} \xrightarrow{GH} \paren*{\bbT^{4}/\bbZ_{2}, \bomega_{f}, g_{\bomega_{f}}}$$ converges in the Gromov-Hausdorff topology to the flat \hka orbifold $\paren*{\bbT^{4}/\bbZ_{2}, \bomega_{f}, g_{\bomega_{f}}}$ with the singular orbifold \hka metric $\bomega_{f}, g_{\bomega_{f}}$. Moreover, we may decompose each compact orbifold $\Km^{\abs*{I}}$ into a union of sets\footnote{say, $\Km^{\abs*{I},reg}_{f, \epsilon_{I}, e_{I}} \coloneq \set*{\frac{\paren*{\max \epsilon_{I}}^{\lambda}}{2} < \abs*{\cdot}^{g_{\bomega_{f}}}_{\bigsqcup_{i \in I}E_{i}}}$ and each $\Km^{\abs*{I},b_{i}}_{f, \epsilon_{I}, e_{I}}$ being $\set*{\abs*{\cdot}^{g_{\bomega_{f}}}_{E_{i}}\leq \paren*{\max \epsilon_{I}}^{\lambda}}$.} $\Km^{\abs*{I},reg}_{f, \epsilon_{I}, e_{I}} \cup \bigsqcup_{i \in I}\Km^{\abs*{I},b_{i}}_{f, \epsilon_{I}, e_{I}}$ such that \begin{enumerate}
\itemsep0em 
\item $\paren*{\Km^{\abs*{I},reg}_{f, \epsilon_{I}, e_{I}}, \bomega_{f, \epsilon_{I}, e_{I}}, g_{\bomega_{f, \epsilon_{I}, e_{I}}}}$ converges in $C^{\infty}_{g_{\bomega_{0}},loc}$ to the flat \hka orbifold $\paren*{\bbT^{4}/\bbZ_{2} - \set*{p_{i}}_{i\in I}, \bomega_{0}, g_{\bomega_{0}}}$ with bounded curvature away from $p_{i} \in \cS$ for $i \in I$ (\textbf{regular region}).
\item For each $i \in I$, after applying a $\frac{1}{\epsilon_{i}}$-dilation $\paren*{R_{\frac{1}{\epsilon_{i}}}\paren*{\Km^{\abs*{I},b_{i}}_{f, \epsilon_{I}, e_{I}}}, R^{*}_{\epsilon_{i}}\paren*{\frac{1}{\epsilon_{i}^{2}} \bomega_{f, \epsilon_{I}, e_{I}}}, R^{*}_{\epsilon_{i}}\paren*{\frac{1}{\epsilon_{i}^{2}}g_{\bomega_{f, \epsilon_{I}, e_{I}}}}}$ converges to $\paren*{T^{*}S^{2}, A^{e_{i}}\bomega_{EH,1}, g_{A^{e_{i}}\bomega_{EH,1}}}$ in $C^{\infty}_{g_{A^{e_{i}}\bomega_{EH,1}},loc}$ (\textbf{ALE bubble region}).
\end{enumerate}

In particular, when $I = \set*{1,\dots, 16}$, $0 < \exists \epsilon^{0} \ll 1$ such that \textbf{there exists a $10 + 3\cdot 16 = 58$-parameter family of \hka triples and \hka metrics} $$\set*{\paren*{\bomega_{f, \epsilon_{I}, e_{I}}, g_{\bomega_{f, \epsilon_{I}, e_{I}}}}}_{\substack{f \in \GL\paren*{4,\bbZ}\backslash \GL\paren*{4,\bbR}/\O(4)\\ \epsilon_{i} \in \paren*{0,\epsilon^{0}} \\ e_{i} \in S^{2} \\ i =1,\dots, 16}}$$ on the \textbf{K3 surface} $\Km^{16} = K3 = \Km_{f, \epsilon_{I}, e_{I}}$.

Moreover, as $\epsilon^{0} \searrow 0$, $$\paren*{\Km_{f, \epsilon_{I}, e_{I}}, \bomega_{f, \epsilon_{I}, e_{I}}, g_{\bomega_{f, \epsilon_{I}, e_{I}}}} \xrightarrow{GH} \paren*{\bbT^{4}/\bbZ_{2}, \bomega_{f}, g_{\bomega_{f}}}$$ converges in the Gromov-Hausdorff topology to the flat \hka orbifold $\paren*{\bbT^{4}/\bbZ_{2}, \bomega_{f}, g_{\bomega_{f}}}$ with the singular orbifold \hka metric $\bomega_{f}, g_{\bomega_{f}}$. Moreover, we may decompose $K3$ into a union of sets $\Km_{f, \epsilon_{I}, e_{I}}^{reg} \cup \bigsqcup_{i =1}^{16}\Km_{f, \epsilon_{I}, e_{I}}^{b_{i}}$ such that \begin{enumerate}
\itemsep0em 
\item $\paren*{\Km_{f, \epsilon_{I}, e_{I}}^{reg}, \bomega_{f, \epsilon_{I}, e_{I}}, g_{\bomega_{f, \epsilon_{I}, e_{I}}}}$ converges in $C^{\infty}_{g_{\bomega_{0}},loc}$ to the flat \hka manifold $\paren*{\bbT^{4}/\bbZ_{2} - \cS, \bomega_{0}, g_{\bomega_{0}}}$ with bounded curvature away from $\cS$ (\textbf{regular region}).
\item For each $i =1,\dots, 16$, after applying a $\frac{1}{\epsilon_{i}}$-dilation $\paren*{R_{\frac{1}{\epsilon_{i}}}\paren*{\Km_{f, \epsilon_{I}, e_{I}}^{b_{i}}}, R_{\epsilon_{i}}^{*}\paren*{\frac{1}{\epsilon_{i}^{2}} \bomega_{f, \epsilon_{I}, e_{I}}}, R_{\epsilon_{i}}^{*}\paren*{\frac{1}{\epsilon_{i}^{2}}g_{\bomega_{f, \epsilon_{I}, e_{I}}}}}$ converges to $\paren*{T^{*}S^{2}, A^{e_{i}}\bomega_{EH,1}, g_{A^{e_{i}}\bomega_{EH,1}}}$ in $C^{\infty}_{g_{A^{e_{i}}\bomega_{EH,1}},loc}$ (\textbf{ALE bubble region}).
\end{enumerate}

Therefore, letting $k \in \set*{0,\dots, 16}$ be a chosen nonnegative integer (the \textbf{depth}), upon considering the following strictly nested sequence $$\set*{1,\dots, 16} = I_{0}\supsetneq I_{1}\supsetneq \cdots \supsetneq I_{k}$$ with $I_{16} \coloneq \emptyset$, then we have the following sequence of \textbf{iterated volume non-collapsed degenerations} of our 58-parameter family of \hka metrics $\bomega_{f, \epsilon_{I_{0}}, e_{I_{0}}}, g_{\bomega_{f, \epsilon_{I_{0}}, e_{I_{0}}}}$ on the $K3$ surface $\Km^{\abs*{I_{0}}}_{f, \epsilon_{I_{0}}, e_{I_{0}}} =\Km_{f, \epsilon_{I_{0}}, e_{I_{0}}}$:\\ 

\hspace*{-2.698cm}\scalebox{0.93}{\begin{tikzcd}[ampersand replacement=\&]
\&\& {\paren*{\Km_{f, \epsilon_{I_{0}}, e_{I_{0}}}, g_{\bomega_{f, \epsilon_{I_{0}}, e_{I_{0}}}}}} \\
{\paren*{\Km^{\abs*{I_{1}}}_{f, \epsilon_{I_{1}}, e_{I_{1}}}, g_{\bomega_{f, \epsilon_{I_{1}}, e_{I_{1}}}}}} \& {\paren*{\Km^{\abs*{I_{2}}}_{f, \epsilon_{I_{2}}, e_{I_{2}}}, g_{\bomega_{f, \epsilon_{I_{2}}, e_{I_{2}}}}}} \& \cdots \& {\paren*{\Km^{\abs*{I_{k-1}}}_{f, \epsilon_{I_{k-1}}, e_{I_{k-1}}}, g_{\bomega_{f, \epsilon_{I_{k-1}}, e_{I_{k-1}}}}}} \& {\paren*{\Km^{\abs*{I_{k}}}_{f, \epsilon_{I_{k}}, e_{I_{k}}}, g_{\bomega_{f, \epsilon_{I_{k}}, e_{I_{k}}}}}} \\
\&\& {\paren*{\bbT^{4}/\bbZ_{2}, g_{\bomega_{f}}}}
\arrow["GH"', from=1-3, to=2-1]
\arrow["GH", from=1-3, to=2-2]
\arrow["GH", from=1-3, to=2-3]
\arrow["GH", from=1-3, to=2-4]
\arrow["GH", from=1-3, to=2-5]
\arrow["GH", from=2-1, to=2-2]
\arrow["GH", from=2-1, to=3-3]
\arrow["GH", from=2-2, to=2-3]
\arrow["GH", from=2-2, to=3-3]
\arrow["GH", from=2-3, to=2-4]
\arrow["GH", from=2-3, to=3-3]
\arrow["GH", from=2-4, to=2-5]
\arrow["GH"', from=2-4, to=3-3]
\arrow["GH", from=2-5, to=3-3]
\end{tikzcd}}

where we\begin{itemize}
\itemsep0em 
\item degenerate from our \hka $K3$ surface $\paren*{\Km_{f, \epsilon_{I_{0}}, e_{I_{0}}}, g_{\bomega_{f, \epsilon_{I_{0}}, e_{I_{0}}}}}$ to $\paren*{\Km^{\abs*{I_{j}}}_{f, \epsilon_{I_{j}}, e_{I_{j}}}, \bomega_{f, \epsilon_{I_{j}}, e_{I_{j}}}, g_{\bomega_{f, \epsilon_{I_{j}}, e_{I_{j}}}}}$ via only collapsing the gluing parameters $\epsilon_{i}>0$ corresponding to $i \in I_{0} - I_{j} = I_{j}^{c}$
\item degenerate from $\paren*{\Km^{\abs*{I_{j-1}}}_{f, \epsilon_{I_{j-1}}, e_{I_{j-1}}}, \bomega_{f, \epsilon_{I_{j-1}}, e_{I_{j-1}}}, g_{\bomega_{f, \epsilon_{I_{j-1}}, e_{I_{j-1}}}}}$ to $\paren*{\Km^{\abs*{I_{j}}}_{f, \epsilon_{I_{j}}, e_{I_{j}}}, \bomega_{f, \epsilon_{I_{j}}, e_{I_{j}}}, g_{\bomega_{f, \epsilon_{I_{j}}, e_{I_{j}}}}}$ via only collapsing the gluing parameters $\epsilon_{i}>0$ corresponding to $i \in I_{j-1} - I_{j}$
\item and degenerate from $\paren*{\Km^{\abs*{I_{j}}}_{f, \epsilon_{I_{j}}, e_{I_{j}}}, \bomega_{f, \epsilon_{I_{j}}, e_{I_{j}}}, g_{\bomega_{f, \epsilon_{I_{j}}, e_{I_{j}}}}}$ to $\paren*{\bbT^{4}/\bbZ_{2}, \bomega_{f}, g_{\bomega_{f}}}$ via collapsing all gluing parameters $\epsilon_{i} > 0, i \in I_{j}$.\end{itemize}\end{theorem}

\begin{remark}\label{HKA big main theorem 2 consistency with codimension 3 for depth 1}
Consider the depth $k=1$ case of the iterated volume non-collapsed degenerations of our $58$-parameter family of \hka metrics on $K3$ and where $\abs*{I_{1}} = 15$. That is, from $\paren*{\Km_{f, \epsilon_{I}, e_{I}}, \bomega_{f, \epsilon_{I}, e_{I}}, g_{\bomega_{f, \epsilon_{I}, e_{I}}}}$ with $I \coloneq \set*{1,\dots, 16}$ we only collapse a \textit{single} gluing parameter $\epsilon_{i}$ for $i \in I - I_{1} = I_{1}^{c}$, resulting in a compact \hka orbifold $\paren*{\Km^{\abs*{I_{1}}}_{f, \epsilon_{I_{1}}, e_{I_{1}}}, \bomega_{f, \epsilon_{I_{1}}, e_{I_{1}}}, g_{\bomega_{f, \epsilon_{I_{1}}, e_{I_{1}}}}}$ with a single orbifold singularity modeled on $\bbR^{4}/\bbZ_{2}$. This resulting compact orbifold $\Km^{\abs*{I_{1}}}$ thus has a $10 + 3\abs*{I_{1}} = 10 + 3\cdot 15 = 55$-parameter family of orbifold \hka metrics $\paren*{\bomega_{f, \epsilon_{I_{1}}, e_{I_{1}}}, g_{\bomega_{f, \epsilon_{I_{1}}, e_{I_{1}}}}}$. As remarked back in Section \ref{Introduction}, the volume non-collapsed limits of (unit volume) \hka metrics on $K3$ corresponds under the period map to points in the \textit{codimension 3} subset $$\bigcup_{\substack{\Sigma \in H_{2}\paren*{K3;\bbZ}\\ \Sigma \cdot \Sigma = -2}} \set*{H \in \Gr^{+}\paren*{3,19}/\Gamma : [\alpha]\paren*{\Sigma}= 0, \forall [\alpha] \in H}$$ of $\Gr^{+}\paren*{3,19}/\Gamma$ in which the original period map $\cP: \sM^{K3}_{\Vol=1} \rightarrow \Gr^{+}\paren*{3,19}/\Gamma$ fails to be surjective.

Therefore, upon rescaling the metrics involved to have unit volume and thus sacrificing a choice of parameter (namely where we fix the volume of the starting flat \hka orbifold $\paren*{\bbT^{4}/\bbZ_{2}, \bomega_{f}, g_{\bomega_{f}}}$), Theorem \ref{maintheorem2} therefore produces a $54$-parameter family of orbifold \hka metrics $\paren*{\Km^{\abs*{I_{1}}}_{f, \epsilon_{I_{1}}, e_{I_{1}}}, \bomega_{f, \epsilon_{I_{1}}, e_{I_{1}}}, g_{\bomega_{f, \epsilon_{I_{1}}, e_{I_{1}}}}}$ which arise as volume non-collapsed Gromov-Hausdorff limit spaces of a $57$-parameter family of \hka metrics $\paren*{\Km_{f, \epsilon_{I}, e_{I}}, \bomega_{f, \epsilon_{I}, e_{I}}, g_{\bomega_{f, \epsilon_{I}, e_{I}}}}$ on $K3$. Thus since $$57 - 54 = 3$$ we have that this is an explicit example of the codimension $3$ nature of the ``holes'' that the range of the original period map misses in its codomain $\Gr^{+}\paren*{3,19}/\Gamma$. And indeed, from Proposition \ref{HKA EHproperties} and construction of $\paren*{\Km_{f, \epsilon_{I}, e_{I}}, \bomega_{f, \epsilon_{I}, e_{I}}, g_{\bomega_{f, \epsilon_{I}, e_{I}}}}$ we have that the $i$th exceptional $E_{i} \cong S^{2}$ in the $K3$ surface $\Km_{f, \epsilon_{I}, e_{I}}$ has volume $\pi\epsilon_{i}^{2}$ and diameter $\epsilon_{i}$ and self-intersection number $E_{i}\cdot E_{i} = S^{2}\cdot S^{2} = -2$, so collapsing $\epsilon_{i}$ for $i \in I - I_{1} = I_{1}^{c}$ to get to our compact \hka orbifold $\paren*{\Km^{\abs*{I_{1}}}_{f, \epsilon_{I_{1}}, e_{I_{1}}}, \bomega_{f, \epsilon_{I_{1}}, e_{I_{1}}}, g_{\bomega_{f, \epsilon_{I_{1}}, e_{I_{1}}}}}$ \textit{precisely} gives us our $\Sigma \in H_{2}\paren*{K3;\bbZ}$ with $\Sigma \cdot \Sigma = -2$ and a ``large''/\textit{top dimensional} subset of the points $H \in \Gr^{+}\paren*{3,19}/\Gamma$ such that $[\alpha]\paren*{\Sigma} = 0, \forall [\alpha] \in H$, namely: $$\Sigma \coloneq E_{i}$$ where $i \in I - I_{1} = I_{1}^{c}$ and $$\set*{\spano_{\bbR} \paren*{\bomega_{f, \epsilon_{j}, e_{j}}} : f \in \text{unit vol flat}, \epsilon_{j} \in \paren*{0,\epsilon^{0}_{I_{1}}}, e_{j} \in S^{2}, j \in I_{1} } \subseteq \set*{H \in \Gr^{+}\paren*{3,19}/\Gamma : [\alpha]\paren*{\Sigma}= 0, \forall [\alpha] \in H}$$ and \begin{al}
&\dim_{\bbR} \set*{H \in \Gr^{+}\paren*{3,19}/\Gamma : [\alpha]\paren*{\Sigma}= 0, \forall [\alpha] \in H} \\
&= \dim_{\bbR} \set*{\spano_{\bbR} \paren*{\bomega_{f, \epsilon_{j}, e_{j}}} : f \in \text{unit vol flat}, \epsilon_{j} \in \paren*{0,\epsilon^{0}_{I_{1}}}, e_{j} \in S^{2}, j \in I_{1} } \\
&= 54
\end{al} Here we are implicitly viewing $\bomega_{f, \epsilon_{j}, e_{j}}$ as a \textit{singular} \hka triple on $K3$ gotten by pullback along the smoothing map $\Km_{f, \epsilon_{I}, e_{I}} \onto \Km^{\abs*{I_{1}}}_{f, \epsilon_{I_{1}}, e_{I_{1}}}$.\end{remark}

\begin{remark}\label{HKA big main theorem 2 consistency with codimension 3 for depth k}
Therefore generalizing \textit{Remark} \ref{HKA big main theorem 2 consistency with codimension 3 for depth 1} for arbitrary depth $k \in \set*{1,\dots, 16}$ (excluding the trivial $k=0$ case) whence $I_{k} \subsetneq \set*{1,\dots, 16}$, we thus have that upon rescaling all metrics to have unit volume, that Theorem \ref{maintheorem2} produces a $9 + 3\cdot \abs*{I_{k}}$-parameter family of orbifold \hka metrics $\paren*{\Km^{\abs*{I_{k}}}_{f, \epsilon_{I_{k}}, e_{I_{k}}}, \bomega_{f, \epsilon_{I_{k}}, e_{I_{k}}}, g_{\bomega_{f, \epsilon_{I_{k}}, e_{I_{k}}}}}$ which arise as volume non-collapsed Gromov-Hausdorff limit spaces of a $57$-parameter family of \hka metrics $\paren*{\Km_{f, \epsilon_{I}, e_{I}}, \bomega_{f, \epsilon_{I}, e_{I}}, g_{\bomega_{f, \epsilon_{I}, e_{I}}}}$ on $K3$. Thus since \begin{al}
57 - \paren*{9 + 3\cdot \abs*{I_{k}}} &= 48 - 3\cdot \abs*{I_{k}}\\
&= 3\cdot \paren*{16 - \abs*{I_{k}}}\\
&> 3
\end{al} since $I_{k}\subsetneq \set*{1,\dots, 16}$, we have that these explicit examples elucidate/``further probe down'' into the codimension $3$ holes $\bigcup_{\substack{\Sigma \in H_{2}\paren*{K3;\bbZ}\\ \Sigma \cdot \Sigma = -2}} \set*{H \in \Gr^{+}\paren*{3,19}/\Gamma : [\alpha]\paren*{\Sigma}= 0, \forall [\alpha] \in H}$ of $\Gr^{+}\paren*{3,19}/\Gamma$ in which the original period map $\cP: \sM^{K3}_{\Vol=1} \rightarrow \Gr^{+}\paren*{3,19}/\Gamma$ fails to be surjective \textit{by exhibiting a }\textbf{stratification}\textit{ of that codimension $3$ subset according to the size of $I_{k}^{c}$}. Thus our large family of orbifold \hka metrics $\paren*{\Km^{\abs*{I_{k}}}_{f, \epsilon_{I_{k}}, e_{I_{k}}}, \bomega_{f, \epsilon_{I_{k}}, e_{I_{k}}}, g_{\bomega_{f, \epsilon_{I_{k}}, e_{I_{k}}}}}$ corresponds inside the codimesion $3$ subset $\bigcup_{\substack{\Sigma \in H_{2}\paren*{K3;\bbZ}\\ \Sigma \cdot \Sigma = -2}} \set*{H \in \Gr^{+}\paren*{3,19}/\Gamma : [\alpha]\paren*{\Sigma}= 0, \forall [\alpha] \in H} \subset \Gr^{+}\paren*{3,19}/\Gamma$ to $$\bigcup_{i \in I - I_{k} = I_{k}^{c}} \set*{\spano_{\bbR}\paren*{\bomega_{f, \epsilon_{j}, e_{j}}} : f \in \text{unit vol flat}, \epsilon_{j} \in \paren*{0,\epsilon^{0}_{I_{k}}}, e_{j} \in S^{2}, j \in I_{k} }$$ i.e. corresponds to the union of $16 - \abs*{I_{k}}$ subsets, each with dimension $9+ 3\cdot \abs*{I_{k}}$ and each corresponding to the $16 - \abs*{I_{k}}$ exceptional $E_{i}\cong S^{2}$s for $i \in I_{k}^{c}$, which are clearly 2-cycles inside $K3$ with self-intersection $-2$, volume $\pi\epsilon_{i}^{2}$, and diameter $\epsilon_{i}$.\end{remark}

\begin{remark}\label{HKA codimension 4 conjecture remark}
Since Ricci-flat metrics are Einstein, and since each $\paren*{\Km_{\epsilon}, J, \wtilde{g_{\epsilon}}, \wtilde{\omega_{\epsilon}}}$ in Theorem \ref{maintheorem1} and each $\paren*{\Km_{f, \epsilon_{I}, e_{I}}, \bomega_{f, \epsilon_{I}, e_{I}}, g_{\bomega_{f, \epsilon_{I}, e_{I}}}} $ in Theorem \ref{maintheorem2} have uniformly bounded diameter and satisfies (uniform) \textit{volume non-collapsing} (\textit{Remark} \ref{GHremark}, \textit{Remark} \ref{HKA GHremark}), the convergences \begin{al}
\paren*{\Km_{\epsilon}, J, \wtilde{g_{\epsilon}}, \wtilde{\omega_{\epsilon}}} &\xrightarrow{GH} \paren*{\bbT^{4}/\bbZ_{2}, J_{0}, g_{0}, \omega_{0}}\\
\paren*{\Km_{f, \epsilon_{I}, e_{I}}, \bomega_{f, \epsilon_{I}, e_{I}}, g_{\bomega_{f, \epsilon_{I}, e_{I}}}} &\xrightarrow{GH} \paren*{\bbT^{4}/\bbZ_{2}, \bomega_{f}, g_{\bomega_{f}}}\\
\paren*{\Km_{f, \epsilon_{I_{0}}, e_{I_{0}}}, g_{\bomega_{f, \epsilon_{I_{0}}, e_{I_{0}}}}} &\xrightarrow{GH}\paren*{\Km^{\abs*{I_{j}}}_{f, \epsilon_{I_{j}}, e_{I_{j}}}, \bomega_{f, \epsilon_{I_{j}}, e_{I_{j}}}, g_{\bomega_{f, \epsilon_{I_{j}}, e_{I_{j}}}}}
\end{al} proven in those two theorems exhibit examples in which the codimension 4 conjecture is \textit{sharp}, since the orbifold singularities of $\paren*{\bbT^{4}/\bbZ_{2}, J_{0}, g_{0}, \omega_{0}}$, $\paren*{\bbT^{4}/\bbZ_{2}, \bomega_{f}, g_{\bomega_{f}}}$ and $\paren*{\Km^{\abs*{I_{j}}}_{f, \epsilon_{I_{j}}, e_{I_{j}}}, \bomega_{f, \epsilon_{I_{j}}, e_{I_{j}}}, g_{\bomega_{f, \epsilon_{I_{j}}, e_{I_{j}}}}}$ clearly have (real) codimension $4$. Indeed, the \textbf{codimension 4 conjecture}, now a theorem by Cheeger and Naber \cite{CN}, states that the (pointed) Gromov-Hausdorff limit spaces of \textit{non-collapsed} Riemannian manifolds with bounded Ricci curvature are smooth away from a closed set of (real) Hausdorff codimension 4. The compact Einstein case, proven by Anderson \cite{Anderson}, Bando-Kasue-Nakajima \cite{BKNa}, and Nakajima \cite{Nakajima}, in fact tells us that when the sequence of manifolds in question are Einstein 4-manifolds having \textit{uniform} Ricci upper bound, uniform diameter upper bound, uniform Euler characteristic upper bound, and uniform volume \textit{lower} bound, a subsequence converges to an Einstein \textit{orbifold} with finitely many isolated singular points. Whence Theorem \ref{maintheorem1}, Theorem \ref{maintheorem2a}, and each of the convergences  in Theorem \ref{maintheorem2} gives us many explicit examples of this, as each of the underlying smooth manifolds in $\paren*{\Km_{\epsilon}, J, \wtilde{g_{\epsilon}}, \wtilde{\omega_{\epsilon}}}$ and $\paren*{\Km_{f, \epsilon_{I_{0}}, e_{I_{0}}}, g_{\bomega_{f, \epsilon_{I_{0}}, e_{I_{0}}}}}$ are the $K3$ surface having Euler characteristic $24$.
\end{remark}

\printbibliography[heading=bibintoc,title={References}]

% Praise God from whom all blessings flow!
% Praise Him all creatures here below!
% Praise Him above ye heavenly hosts!
% Praise Father Son and Holy Ghost!
% Amen!

\end{document}